\documentclass[11pt]{amsart}
\usepackage{amsmath, amsfonts, amssymb, amsthm, caption, dsfont, geometry, graphics, graphicx, mathtools, multicol, multirow, enumerate, float, floatflt, fullpage, import,tikz, outlines, url, subcaption, tabularx, caption}
\usepackage{amssymb,amsthm,amsfonts,mathrsfs}
\usepackage[all]{xy}

\usepackage{amsmath,amscd}
\usepackage{hyperref}
\usepackage{color}
\usepackage{graphicx}
\usepackage{stmaryrd}  
  
\let\oldtocsection=\tocsection
\let\oldtocsubsection=\tocsubsection
\renewcommand{\tocsection}[2]{\hspace{0em}\oldtocsection{#1}{#2}}
\renewcommand{\tocsubsection}[2]{\hspace{1em}\oldtocsubsection{#1}{#2}}

\newtheorem{theorem}{Theorem}
\newtheorem{lemma}[theorem]{Lemma}
\newtheorem{con}[theorem]{Conjecture}
\newtheorem{proposition}[theorem]{Proposition}
\newtheorem{prop}[theorem]{Proposition}

\theoremstyle{definition}
\newtheorem{definition}[theorem]{Definition}
\newtheorem{remark}[theorem]{Remark}
\newtheorem{example}[theorem]{Example}
\newtheorem*{remark*}{Remark}

\newcommand{\ssubsection}[1]{%
  \subsubsection*[#1]{\bfseries\itshape #1}}

\def\mc{\mathcal}
\newcommand\A{{\mathcal A}}
\newcommand\B{{\mathcal B}}
\newcommand\del{\partial}
\newcommand\oop{{\mathrm{op}}}

\newcommand\Z{{\mathbb{Z}}}
\newcommand\mcA{\mathcal{A}}
\newcommand\mcB{\mathcal{B}}
\newcommand\mcF{{\mathcal{F}}}

\newcommand{\ov}[1]{\overline{#1}}

\newcommand\wU{\mathbb{U}}  

\newcommand\wUcirc{\wU_{\circ}}  
\newcommand\wB{\mathbb{B}}    
\newcommand{\wA}{\mathbb{A}}
\newcommand{\wBout}[1]{\wB^{\mathrm{out}}_{#1}}
\newcommand{\wUout}[1]{\wU^{\mathrm{out}}_{#1}}

\newcommand\wSU{\mathbb{SU}}
\newcommand\wDSU{\mathbb{DSU}}

\newcommand{\wDU}{\mathbb{DU}}

\newcommand{\wUd}{{}^{\bullet}\wU} 

\def\lra{\longrightarrow}
\def\Hom{\mathrm{Hom}}
\def\HomA{\Hom_{\mcA}}

\def\End{\mathrm{End}}
\def\Ide{\mathrm{Id}}
\def\ide{\mathrm{id}} 
\def\Kar{\mathrm{Kar}} 
\def\Ob{\mathrm{Ob}}
\def\one{\mathds{1}}
\def\ObA{\Ob(\mcA)}
\def\rank{\mathrm{rank}}
\def\Tr{\mathrm{Tf}}
\def\tr{\mathrm{tr}}

\def\arc{\mathrm{arc}}

\newcommand\mcEnd{\mathcal{E}nd}

\def\circles{\mathrm{cir}} 
\def\mcr{\mathrm{cir}} 
\def\formap{\mathrm{for}} 
\def\trees{\mathbb{T}}  
\def\forests{\mathbb{T}^{\ast}}  
\def\mfsl{\mathfrak{sl}}

\newcommand{\CTrip}{\mathrm{CTrip}} 

\def\kk{\mathbf{k}}  

\newcommand{\scr}[1]{{\scriptstyle #1}}
\newcommand{\un}[1]{\underline{#1}}

\def\cF{\mathcal{F}}
\def\cG{\mathcal{G}}

\def\R{\mathbb R}
\def\Q{\mathbb Q}
\def\Z{\mathbb Z}
 
\def\C{\mathbb C}

\def\SS{\mathbb{S}}    

\let\emptyset\varnothing

\title{Planar diagrammatics of self-adjoint functors and recognizable tree series}
\author{Mikhail Khovanov}
 \address{Department of Mathematics, Columbia University, New York, NY 10027, USA}
 \email{\href{mailto:khovanov@math.columbia.edu}{khovanov@math.columbia.edu}}
\author{Robert Laugwitz}
\address{School of Mathematical Sciences,
University of Nottingham, Nottingham, NG7 2RD, UK}
 \email{\href{mailto:robert.laugwitz@nottingham.ac.uk}{robert.laugwitz@nottingham.ac.uk}}

\date{October 6, 2021}

\begin{document}

\begin{abstract}
A pair of biadjoint functors between two categories produces a collection of elements in the centers of these categories, one for each isotopy class of nested circles in the plane. If the centers are equipped with a trace map into the ground field, then one assigns an element of that field to a diagram of nested circles. We focus on the self-adjoint functor case of this construction and study the reverse problem of recovering such a functor and a category given values associated to diagrams of nested circles. 
\end{abstract}
\subjclass[2020]{
Primary 18M30; 
Secondary 18A40, 
18M10, 
57K16,  
11B85   
}
\keywords{Self-adjoint functor, universal construction, monoidal category, Temperley--Lieb category, recognizable tree series}
\dedicatory{Dedicated to the memory of Vaughan F.~R. Jones}

\maketitle
\tableofcontents

%
%

\section{Introduction}

\smallskip

A biadjoint pair $(F,G)$ of functors consists of two functors $F,G$ and choices of natural transformations making $G$ both left and right adjoint to $F$. 
Biadjoint pairs of functors are common in today's mathematics and mathematical physics. Their popularity is related, in part, to their natural appearance  in extended topological field theories. Such an extended theory may associate 
\begin{itemize}
\item a category $C(K)$ to a closed $(n-2)$-manifold $K$, 
\item a functor $C(N)$ to an $(n-1)$-cobordism $N$ between $(n-2)$-manifolds, 
\item a natural transformation $C(M)$ to an $n$-cobordism $M$ with corners between $(n-1)$-cobordisms. 
\end{itemize}
The functors $C(N)$ and $C(N^{\ast})$ are then  naturally biadjoint, where $N^{\ast}$ is the reverse cobordism of $N$, see~\cite{Kh1} for instance. 

\vspace{0.1in} 

Furthermore, biadjoint functors appear throughout representation theory, algebra, and geometry, including as 
\begin{itemize}
\item projective functors in highest weight categories \cite{BG}, 
\item Zuckerman and Bernstein derived functors in highest weight categories \cite{BFK}, 
\item various functors in modular representation theory~\cite{Kl}, 
\item Fourier-Mukai kernels between Calabi-Yau varieties~\cite{Or},
\item functors of tensor products with matrix factorizations, see e.g. \cite{CM}, 
\item functors of convolution with Lagrangians between Fukaya--Floer categories, 
\item suitable convolution functors in geometric representation theory and in categories of sheaves on manifolds and stratified spaces, 
\item generating functors for categorifications of Hecke algebras~\cite{EK}, quantum groups~\cite{CR,KhL,R,L2}, and Heisenberg algebras \cite{Kh4},
\item functors of tensoring with objects of pivotal categories \cite{S}, 
\item 1-morphisms in Mazorchuk--Miemietz fiat 2-categories~\cite{MM1}.
\end{itemize} 

A biadjoint pair $(F,G)$ between categories $\mcA,\mcB$ gives rise to a planar diagrammatic calculus of collections of arcs and circles in the plane, as reviewed below in Section~\ref{subsec_diag_biadj}. Such diagrams describe natural transformations between compositions of $F$ and $G$ built from the four biadjointness natural transformations. Regions of these diagrams are labelled by categories $\mcA$ and $\mcB$ in checkerboard fashion. Closed diagrams in the plane, that is, collections of nested circles in the plane, give rise to elements of the centers $Z(\mcA), Z(\mcB)$, depending on whether the outside region is labelled  $\mcA$ or $\mcB$. The centers $Z(\mcA),Z(\mcB)$ are commutative monoids, and potentially, there is a lot of freedom in associating elements of these commutative monoids to closed diagrams. 

In this paper we investigate the case when the categories $\mcA,\mcB$ are pre-additive or additive. To further simplify matters, we assume that the categories and functors are $\kk$-linear, over  a field $\kk$, and, in particular, their centers $Z(\mcA),Z(\mcB)$ are commutative $\kk$-algebras. One can further assume that the centers come with suitably non-degenerate trace maps $Z(\mcA)\lra \kk$, $Z(\mcB)\lra\kk$ to the ground field. Applying these trace maps to central elements encoded by nested circle diagrams produces a collection of elements of $\kk$, one for each nested diagram together with a label for an outer region. Given such data, one can turn around and build a ``minimal'' non-degenerate system of such categories, biadjoint functors, and trace maps on centers in a straightforward way. 

We explain these constructions in detail in the slightly different case of a \emph{self-adjoint} endofunctor $F\colon \mcA\lra \mcA$ rather than a biadjoint pair $(F,G)$. In the self-adjoint case there is only one category $\mcA$, each region is labelled by $\mcA$, and that label can be omitted. The evaluation data is given by assigning an element of $\kk$ to each nested circle diagram. 

\vspace{0.1in} 

In the self-adjoint case, the center $Z(\A)$ is a commutative algebra that comes with a $\kk$-linear map $\omega\colon Z(\A)\lra Z(\A)$, corresponding to the operator of wrapping a circle around a diagram representing an element of $Z(\A)$. This wrapping operator is the trace morphism for the self-adjoint functor $F$, see \cite{B} and \eqref{eq_trace_map}.

\vspace{0.1in} 

In Sections~\ref{sec-BZ} and~\ref{sec_pairings} we discuss various monoidal categories one can assign to the data $(\kk,Z,\omega)$ of a commutative $\kk$-algebra $Z$ and a $\kk$-linear map $\omega$, generalizing in some cases from a field $\kk$ to a commutative ring $R$. These categories come from a suitable pairing  between the $\kk$-vector space generated by diagrams of arcs, circles, and elements of $Z$ embedded in a disk and a similar space spanned by such diagrams in an annulus, see Section~\ref{subsec_pairings_neg}. When $Z$ is finite-dimensional over $\kk$, the morphism spaces in the resulting categories are finite-dimensional and the categories and functors between them are recorded in diagram (\ref{overview-diag}) in Section~\ref{sec:commsquare}. 

\vspace{0.1in} 

A related setup emerges when 
$Z$ and $\omega$ are hidden and instead there is a trace map $\varepsilon\colon Z\lra \kk$ into the ground field. 
Then, to a nested diagram $u$ of circles one can associate the element of $\kk$ given by evaluating $u$ to an element of $Z$ via $\omega$ and then applying the trace map $\varepsilon.$ 
A collection of these evaluations can be encoded into an analogue of power series $\alpha$, called \emph{circular series}, where each nested circle diagram carries a coefficient. In Section~\ref{subsec_circ_s} and the latter half of Section~\ref{subsec_arcs} we discuss reconstructing the data of a category and a self-adjoint functor from such circular series and single out \emph{recognizable} series, which yield finite-dimensional morphism spaces between the objects in the resulting categories. 
The situation discussed here is similar to that of universal construction of topological theories, see~\cite{Kh2,KKO,Kh3}, for instance. 
From that viewpoint, the current paper deals with the case when the ambient manifold is   $\R^2$  
or $\SS^2$ together with defect circles (submanifolds or \emph{defects} of codimension one). 

\vspace{0.1in} 

To simplify computations of the pairing between vector spaces of diagrams it is sometimes natural to assume that the coefficients of the formal power series $\alpha$ depend only on the isotopy type of the nested circle diagrams in $\SS^2$ rather than in $\R^2$. Such \emph{spherical} power series are considered in Section~\ref{subsec_arcs}, with examples in Sections~\ref{subset_ss_2d},~\ref{subset_ss_spherical}.

\vspace{0.1in} 

In Section~\ref{sec_trees} we review the well-known correspondence between collections of nested circles in the plane and trees and forests of a suitable type. Section~\ref{subsec_settheory} contains a brief discussion of the set-theoretical version of our construction. 

A general theory and understanding of the monoidal categories defined in Sections~\ref{sec_pairings} and~\ref{sec_trees} seems to be currently absent. Some examples are considered in Section~\ref{sec_adj_ex}. 

In Section~\ref{sec_miscell} several modifications of our constructions are discussed, offering, in particular, a common generalization of tensor envelopes of noncommutative power series as introduced in~\cite{Kh3} and some structures from the present paper.  

\vspace{0.1in} 

Given a recognizable circular series $\alpha$ as above, that  is, an assignment of an  element $\alpha(u)$ of the field $\kk$ to each isotopy class $u$ of planar  diagrams of nested circles, one key construction is the associated category $\wU_{\alpha}$. The monoidal category $\wU_{\alpha}$ has non-negative integers $n$ as objects and morphisms from $n$ to $m$ are given by linear combinations of planar diagrams of arcs and circles with $n$ bottom and $m$ top endpoints. The skein relations in the category $\wU_{\alpha}$ are defined via  the universal construction and  may be difficult to write down for a given $\alpha$. 

The finite-dimensional endomorphism rings $TL_{\alpha,n}:=\End_{\wU_{\alpha}}(n)$ can be viewed as generalizations of the \emph{Temperley--Lieb algebra} \cite{Jo1,Ka}. When $\alpha$ ignores the nesting and evaluates any diagram of $k$ circles to $d^k$, where $d\in \kk$ and $\mathrm{char}(\kk)=0$, then  $TL_{\alpha,n}$ is isomorphic to the Temperley--Lieb algebra $TL_n(d)$ for generic $d$ and to the Jones quotient of $TL_n(d)$  when $d=q+q^{-1}$, where $q$ is a root of unity (the quotient by the ideal of negligible morphisms). 

When $\alpha$ is spherical, the Jones quotient of the algebra $TL_{\alpha,n}$ is, in addition, a Frobenius algebra. 
These generalized Temperley--Lieb algebras may be an interesting topic for further investigation.

The present paper proposes a framework for generalized  Temperley--Lieb  algebras and associated categories but does not try to work out the general theory. An incomplete treatment of some examples can be found in Section~\ref{sec_adj_ex}. 

\vspace{0.1in} 

Vaughan Jones discovered and  developed many remarkable structures in mathematics and mathematical physics intricately  related to the notion of the Temperley--Lieb algebra. These 
structures include the index for subfactors~\cite{Jo1,Jo3}, the Jones polynomial of links~\cite{Jo2}, Hecke algebras~\cite{Jo3,Jo4}, 
models of statistical mechanics~\cite{Jo5}, and planar algebras~\cite{Jo7}, among others. We dedicate this  paper to his memory.

\vspace{0.1in}

\subsection*{Acknowledgments}
M.~K. was partially supported via NSF grant DMS-1807425. R.~L. acknowledges support by a Nottingham Research Fellowship. The Figures have been created using Inkscape\footnote{\url{https://inkscape.org/}}. 

%
%

\section{Self-adjoint functors  and circle diagrams}\label{sec_selfadj}


\subsection{Diagrammatics for biadjoint pairs}\label{subsec_diag_biadj} 

\quad
\smallskip 

Given categories $\A,\B$ and functors
\begin{equation}
    F\colon \A\to \B,\  G\colon \B\to \A,
\end{equation} 
the pair $(F,G)$ is called \emph{biadjoint} if there are isomorphisms
\begin{align}
 \Hom_\A(GN,M)&\cong \Hom_\B(N,FN) \label{adjunction1}\\
 \Hom_\A(M,GM)&\cong \Hom_\B(FM,N) \label{adjunction2}
\end{align}
which are natural in $M\in \Ob\A$ and $N\in \Ob\B$. We consider biadjoint pairs $(F,G)$ together with a choice of natural isomorphisms \eqref{adjunction1}, \eqref{adjunction2}. The natural isomorphism in (\ref{adjunction1}) can be described by the unit and counit natural transformations
\begin{align*}
    \delta_1\colon \Ide_{\B}\Longrightarrow FG,&& \mu_1\colon GF\Longrightarrow\Ide_{\A},
\end{align*}
which satisfy the relations
\begin{align}\label{eq_cond_1}
    (1_F \, \mu_1)\circ (\delta_1 \, 1_F) &= 1_F, & ( \mu_1 \, 1_G)\circ (1_G\,\delta_1)&=1_G,
\end{align}
where $1_F\colon F \Longrightarrow F$ is the identity  natural transformation from $F$ to itself, and analogously  for $1_G$.  
Likewise, the natural isomorphism in (\ref{adjunction2}) can be described by the unit and counit natural transformations
\begin{align*}
     \delta_2\colon \Ide_{\A}\Longrightarrow GF,&& \mu_2\colon FG\Longrightarrow \Ide_{\B}, 
\end{align*}
which satisfy the relations 
\begin{align}\label{eq_cond_2}
    (1_G \, \mu_2)\circ (\delta_2\, 1_G) &= 1_G, & ( \mu_2 \, 1_F)\circ (1_F\, \delta_2)&=1_F.
\end{align}
A pair of functors $(F,G)$ may have more than one collection of natural transformations $(\delta_1,\mu_1,\delta_2,\mu_2)$ satisfying  these conditions.  By a pair of biajoint functors $(F,G)$ we mean a pair of  functors as  above together  with a choice of such four natural transformations. We refer to~\cite{Kh1,Ba,L1} for more details on biadjoint functors and their diagrammatics. 

\vspace{0.1in} 

We use oriented planar diagrams, read from bottom to top, to denote these adjunctions. The identity natural transformation  $1_F$ of the functor $F$ is denoted by a  line oriented up and the identity  transformation $1_G$ by a line oriented down, see Figure~\ref{fig1_1}. 
The unit and counit transformations are denoted by oriented cup and cap morphisms, see Figure~\ref{fig1_2}.

\begin{figure}
     \centering
     \begin{subfigure}[htb]{0.25\textwidth}
\begingroup%
  \makeatletter%
  \providecommand\color[2][]{%
    \errmessage{(Inkscape) Color is used for the text in Inkscape, but the package 'color.sty' is not loaded}%
    \renewcommand\color[2][]{}%
  }%
  \providecommand\transparent[1]{%
    \errmessage{(Inkscape) Transparency is used (non-zero) for the text in Inkscape, but the package 'transparent.sty' is not loaded}%
    \renewcommand\transparent[1]{}%
  }%
  \providecommand\rotatebox[2]{#2}%
  \newcommand*\fsize{\dimexpr\f@size pt\relax}%
  \newcommand*\lineheight[1]{\fontsize{\fsize}{#1\fsize}\selectfont}%
  \ifx\svgwidth\undefined%
    \setlength{\unitlength}{45.00000283bp}%
    \ifx\svgscale\undefined%
      \relax%
    \else%
      \setlength{\unitlength}{\unitlength * \real{\svgscale}}%
    \fi%
  \else%
    \setlength{\unitlength}{\svgwidth}%
  \fi%
  \global\let\svgwidth\undefined%
  \global\let\svgscale\undefined%
  \makeatother%
  \begin{picture}(1,1.55455654)%
    \lineheight{1}%
    \setlength\tabcolsep{0pt}%
    \put(0.66262729,0.68112417){\color[rgb]{0,0,0}\makebox(0,0)[lt]{\smash{\begin{tabular}[t]{l}$\A$\end{tabular}}}}%
    \put(0.41573097,1.3519524){\color[rgb]{0,0,0}\makebox(0,0)[lt]{\smash{\begin{tabular}[t]{l}$F$\end{tabular}}}}%
    \put(0.16560267,0.68596469){\color[rgb]{0,0,0}\makebox(0,0)[lt]{\smash{\begin{tabular}[t]{l}$\B$\end{tabular}}}}%
    \put(0,0){\includegraphics[width=\unitlength,page=1]{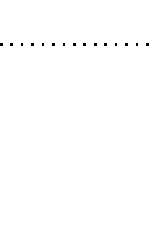}}%
    \put(0.41560265,0.01959635){\color[rgb]{0,0,0}\makebox(0,0)[lt]{\smash{\begin{tabular}[t]{l}$F$\end{tabular}}}}%
    \put(0,0){\includegraphics[width=\unitlength,page=2]{idF.pdf}}%
  \end{picture}%
\endgroup%
 
      \centering
         \caption{$1_F$}
         \label{fig:idF}
     \end{subfigure}
     \begin{subfigure}[htb]{0.25\textwidth}
        \centering
\begingroup%
  \makeatletter%
  \providecommand\color[2][]{%
    \errmessage{(Inkscape) Color is used for the text in Inkscape, but the package 'color.sty' is not loaded}%
    \renewcommand\color[2][]{}%
  }%
  \providecommand\transparent[1]{%
    \errmessage{(Inkscape) Transparency is used (non-zero) for the text in Inkscape, but the package 'transparent.sty' is not loaded}%
    \renewcommand\transparent[1]{}%
  }%
  \providecommand\rotatebox[2]{#2}%
  \newcommand*\fsize{\dimexpr\f@size pt\relax}%
  \newcommand*\lineheight[1]{\fontsize{\fsize}{#1\fsize}\selectfont}%
  \ifx\svgwidth\undefined%
    \setlength{\unitlength}{45.00000283bp}%
    \ifx\svgscale\undefined%
      \relax%
    \else%
      \setlength{\unitlength}{\unitlength * \real{\svgscale}}%
    \fi%
  \else%
    \setlength{\unitlength}{\svgwidth}%
  \fi%
  \global\let\svgwidth\undefined%
  \global\let\svgscale\undefined%
  \makeatother%
  \begin{picture}(1,1.55455654)%
    \lineheight{1}%
    \setlength\tabcolsep{0pt}%
    \put(0.66262729,0.68112417){\color[rgb]{0,0,0}\makebox(0,0)[lt]{\smash{\begin{tabular}[t]{l}$\B$\end{tabular}}}}%
    \put(0.41573097,1.3519524){\color[rgb]{0,0,0}\makebox(0,0)[lt]{\smash{\begin{tabular}[t]{l}$G$\end{tabular}}}}%
    \put(0.16560267,0.68596469){\color[rgb]{0,0,0}\makebox(0,0)[lt]{\smash{\begin{tabular}[t]{l}$\A$\end{tabular}}}}%
    \put(0,0){\includegraphics[width=\unitlength,page=1]{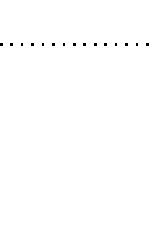}}%
    \put(0.41560265,0.01959635){\color[rgb]{0,0,0}\makebox(0,0)[lt]{\smash{\begin{tabular}[t]{l}$G$\end{tabular}}}}%
    \put(0,0){\includegraphics[width=\unitlength,page=2]{idG.pdf}}%
  \end{picture}%
\endgroup%
 
         \caption{$1_G$}
         \label{fig:idG}
     \end{subfigure}
     \caption{Diagrams  of $1_F$ and  $1_G$. In  this diagrammatics, the planar region between two parallel horizontal dashed lines describes a natural transformation from the composition of functors  read off the bottom dashed line  to  the composition given by boundary points at the top  dashed  line. Regions of the  diagram correspond to categories. }
     \label{fig1_1}
\end{figure}

\begin{figure}
     \centering
     \begin{subfigure}[htb]{0.2\textwidth}
\begingroup%
  \makeatletter%
  \providecommand\color[2][]{%
    \errmessage{(Inkscape) Color is used for the text in Inkscape, but the package 'color.sty' is not loaded}%
    \renewcommand\color[2][]{}%
  }%
  \providecommand\transparent[1]{%
    \errmessage{(Inkscape) Transparency is used (non-zero) for the text in Inkscape, but the package 'transparent.sty' is not loaded}%
    \renewcommand\transparent[1]{}%
  }%
  \providecommand\rotatebox[2]{#2}%
  \newcommand*\fsize{\dimexpr\f@size pt\relax}%
  \newcommand*\lineheight[1]{\fontsize{\fsize}{#1\fsize}\selectfont}%
  \ifx\svgwidth\undefined%
    \setlength{\unitlength}{76.17767154bp}%
    \ifx\svgscale\undefined%
      \relax%
    \else%
      \setlength{\unitlength}{\unitlength * \real{\svgscale}}%
    \fi%
  \else%
    \setlength{\unitlength}{\svgwidth}%
  \fi%
  \global\let\svgwidth\undefined%
  \global\let\svgscale\undefined%
  \makeatother%
  \begin{picture}(1,0.92707171)%
    \lineheight{1}%
    \setlength\tabcolsep{0pt}%
    \put(0,0){\includegraphics[width=\unitlength,page=1]{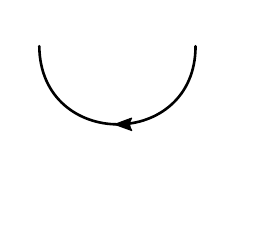}}%
    \put(0.39230927,0.56125302){\color[rgb]{0,0,0}\makebox(0,0)[lt]{\smash{\begin{tabular}[t]{l}$\A$\end{tabular}}}}%
    \put(0.09948418,0.80738851){\color[rgb]{0,0,0}\makebox(0,0)[lt]{\smash{\begin{tabular}[t]{l}$F$\end{tabular}}}}%
    \put(0.69040818,0.80738851){\color[rgb]{0,0,0}\makebox(0,0)[lt]{\smash{\begin{tabular}[t]{l}$G$\end{tabular}}}}%
    \put(0.39582506,0.26629231){\color[rgb]{0,0,0}\makebox(0,0)[lt]{\smash{\begin{tabular}[t]{l}$\B$\end{tabular}}}}%
    \put(0,0){\includegraphics[width=\unitlength,page=2]{delta1.pdf}}%
    \put(0.34308178,0.01857551){\color[rgb]{0,0,0}\makebox(0,0)[lt]{\smash{\begin{tabular}[t]{l}$\Ide_\B$\end{tabular}}}}%
  \end{picture}%
\endgroup%
 
      \centering
         \caption{$\delta_1$}
         \label{fig:delta1}
     \end{subfigure}
     \begin{subfigure}[htb]{0.2\textwidth}
        \centering
\begingroup%
  \makeatletter%
  \providecommand\color[2][]{%
    \errmessage{(Inkscape) Color is used for the text in Inkscape, but the package 'color.sty' is not loaded}%
    \renewcommand\color[2][]{}%
  }%
  \providecommand\transparent[1]{%
    \errmessage{(Inkscape) Transparency is used (non-zero) for the text in Inkscape, but the package 'transparent.sty' is not loaded}%
    \renewcommand\transparent[1]{}%
  }%
  \providecommand\rotatebox[2]{#2}%
  \newcommand*\fsize{\dimexpr\f@size pt\relax}%
  \newcommand*\lineheight[1]{\fontsize{\fsize}{#1\fsize}\selectfont}%
  \ifx\svgwidth\undefined%
    \setlength{\unitlength}{73.89416297bp}%
    \ifx\svgscale\undefined%
      \relax%
    \else%
      \setlength{\unitlength}{\unitlength * \real{\svgscale}}%
    \fi%
  \else%
    \setlength{\unitlength}{\svgwidth}%
  \fi%
  \global\let\svgwidth\undefined%
  \global\let\svgscale\undefined%
  \makeatother%
  \begin{picture}(1,0.90173871)%
    \lineheight{1}%
    \setlength\tabcolsep{0pt}%
    \put(0,0){\includegraphics[width=\unitlength,page=1]{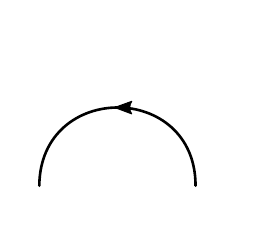}}%
    \put(0.40443258,0.5839098){\color[rgb]{0,0,0}\makebox(0,0)[lt]{\smash{\begin{tabular}[t]{l}$\A$\end{tabular}}}}%
    \put(0.40805701,0.2798341){\color[rgb]{0,0,0}\makebox(0,0)[lt]{\smash{\begin{tabular}[t]{l}$\B$\end{tabular}}}}%
    \put(0,0){\includegraphics[width=\unitlength,page=2]{mu1.pdf}}%
    \put(0.35368384,0.84004787){\color[rgb]{0,0,0}\makebox(0,0)[lt]{\smash{\begin{tabular}[t]{l}$\Ide_\A$\end{tabular}}}}%
    \put(0.10437172,0.02386754){\color[rgb]{0,0,0}\makebox(0,0)[lt]{\smash{\begin{tabular}[t]{l}$G$\end{tabular}}}}%
    \put(0.71319306,0.02567975){\color[rgb]{0,0,0}\makebox(0,0)[lt]{\smash{\begin{tabular}[t]{l}$F$\end{tabular}}}}%
  \end{picture}%
\endgroup%
 
         \caption{$\mu_1$}
         \label{fig:mu1}
     \end{subfigure}
     \begin{subfigure}[htb]{0.2\textwidth}
        \centering
\begingroup%
  \makeatletter%
  \providecommand\color[2][]{%
    \errmessage{(Inkscape) Color is used for the text in Inkscape, but the package 'color.sty' is not loaded}%
    \renewcommand\color[2][]{}%
  }%
  \providecommand\transparent[1]{%
    \errmessage{(Inkscape) Transparency is used (non-zero) for the text in Inkscape, but the package 'transparent.sty' is not loaded}%
    \renewcommand\transparent[1]{}%
  }%
  \providecommand\rotatebox[2]{#2}%
  \newcommand*\fsize{\dimexpr\f@size pt\relax}%
  \newcommand*\lineheight[1]{\fontsize{\fsize}{#1\fsize}\selectfont}%
  \ifx\svgwidth\undefined%
    \setlength{\unitlength}{73.78704627bp}%
    \ifx\svgscale\undefined%
      \relax%
    \else%
      \setlength{\unitlength}{\unitlength * \real{\svgscale}}%
    \fi%
  \else%
    \setlength{\unitlength}{\svgwidth}%
  \fi%
  \global\let\svgwidth\undefined%
  \global\let\svgscale\undefined%
  \makeatother%
  \begin{picture}(1,0.95710789)%
    \lineheight{1}%
    \setlength\tabcolsep{0pt}%
    \put(0,0){\includegraphics[width=\unitlength,page=1]{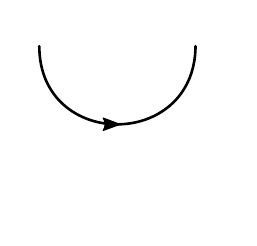}}%
    \put(0.4050197,0.57943705){\color[rgb]{0,0,0}\makebox(0,0)[lt]{\smash{\begin{tabular}[t]{l}$\B$\end{tabular}}}}%
    \put(0.10270737,0.83354708){\color[rgb]{0,0,0}\makebox(0,0)[lt]{\smash{\begin{tabular}[t]{l}$G$\end{tabular}}}}%
    \put(0.71277671,0.83354708){\color[rgb]{0,0,0}\makebox(0,0)[lt]{\smash{\begin{tabular}[t]{l}$F$\end{tabular}}}}%
    \put(0.40864939,0.27491991){\color[rgb]{0,0,0}\makebox(0,0)[lt]{\smash{\begin{tabular}[t]{l}$\A$\end{tabular}}}}%
    \put(0,0){\includegraphics[width=\unitlength,page=2]{delta2.pdf}}%
    \put(0.35238214,0.01917734){\color[rgb]{0,0,0}\makebox(0,0)[lt]{\smash{\begin{tabular}[t]{l}$\Ide_\A$\end{tabular}}}}%
  \end{picture}%
\endgroup%

         \caption{$\delta_2$}
         \label{fig:delta2}
     \end{subfigure}
     \begin{subfigure}[htb]{0.2\textwidth}
        \centering
\begingroup%
  \makeatletter%
  \providecommand\color[2][]{%
    \errmessage{(Inkscape) Color is used for the text in Inkscape, but the package 'color.sty' is not loaded}%
    \renewcommand\color[2][]{}%
  }%
  \providecommand\transparent[1]{%
    \errmessage{(Inkscape) Transparency is used (non-zero) for the text in Inkscape, but the package 'transparent.sty' is not loaded}%
    \renewcommand\transparent[1]{}%
  }%
  \providecommand\rotatebox[2]{#2}%
  \newcommand*\fsize{\dimexpr\f@size pt\relax}%
  \newcommand*\lineheight[1]{\fontsize{\fsize}{#1\fsize}\selectfont}%
  \ifx\svgwidth\undefined%
    \setlength{\unitlength}{76.28478824bp}%
    \ifx\svgscale\undefined%
      \relax%
    \else%
      \setlength{\unitlength}{\unitlength * \real{\svgscale}}%
    \fi%
  \else%
    \setlength{\unitlength}{\svgwidth}%
  \fi%
  \global\let\svgwidth\undefined%
  \global\let\svgscale\undefined%
  \makeatother%
  \begin{picture}(1,0.87347988)%
    \lineheight{1}%
    \setlength\tabcolsep{0pt}%
    \put(0,0){\includegraphics[width=\unitlength,page=1]{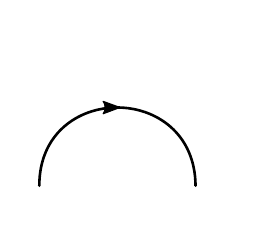}}%
    \put(0.39175841,0.56561114){\color[rgb]{0,0,0}\makebox(0,0)[lt]{\smash{\begin{tabular}[t]{l}$\B$\end{tabular}}}}%
    \put(0.39526925,0.27106461){\color[rgb]{0,0,0}\makebox(0,0)[lt]{\smash{\begin{tabular}[t]{l}$\A$\end{tabular}}}}%
    \put(0,0){\includegraphics[width=\unitlength,page=2]{mu2.pdf}}%
    \put(0.36542337,0.81372231){\color[rgb]{0,0,0}\makebox(0,0)[lt]{\smash{\begin{tabular}[t]{l}$\Ide_\B$\end{tabular}}}}%
    \put(0.10110091,0.02311957){\color[rgb]{0,0,0}\makebox(0,0)[lt]{\smash{\begin{tabular}[t]{l}$F$\end{tabular}}}}%
    \put(0.6908429,0.024875){\color[rgb]{0,0,0}\makebox(0,0)[lt]{\smash{\begin{tabular}[t]{l}$G$\end{tabular}}}}%
  \end{picture}%
\endgroup%
 
         \caption{$\mu_2$}
         \label{fig:mu2}
     \end{subfigure}
     \caption{Diagrams of biadjointness natural  transformations. For  instance, the leftmost diagram is the transformation $\delta_1$ from the identity functor $1_B$ to $FG$. When no  arcs end on a dashed line, then we assign the identity functor, on the category which labels the region, to it. }
     \label{fig1_2}
\end{figure}

The biadjointness  relations (\ref{eq_cond_1}), (\ref{eq_cond_2}) on these four natural transformations are shown in Figure~\ref{fig1_3}. Notice that they  are  just the four isotopy relations on up and  down oriented strands. 

\begin{figure}[htb]
\begin{center}\vspace{5pt}
  \begin{subfigure}[htb]{0.60\textwidth}
        \centering
 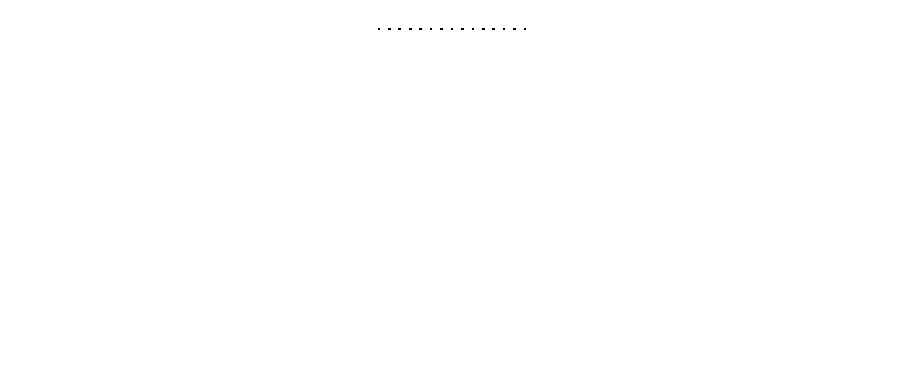
 \caption{Biadjointness equations for $1_F$}
         \label{fig:snake1}
 \end{subfigure}\vspace{10pt}
  \begin{subfigure}[htb]{0.60\textwidth}
        \centering
     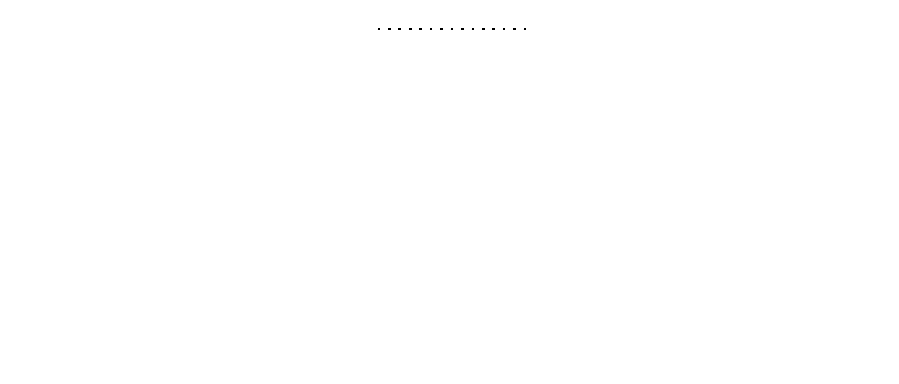
     \caption{Biadjointness equations for $1_G$}
        \label{fig:snake2}
      \end{subfigure}
\caption{Biadjointness relations are the isotopy relations on strands.}
\label{fig1_3}
\end{center}
\end{figure}

\begin{figure}[htb]
\begin{center}
 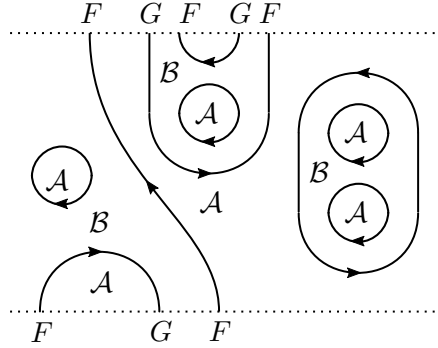
\caption{A diagram built out of the biadjointness transformations.}
\label{fig1_4}
\end{center}
\end{figure}

Arbitrary compositions of the four diagrams depicted in Figure \ref{fig1_3}, modulo isotopy relations, lead to diagrams of oriented  arcs and circles in the strip  $\R\times [0,1]$ of the plane  with regions labelled in  a checkerboard manner by the categories $\mcA$ and  $\mcB$, see Figure~\ref{fig1_4}. 
Edges  are oriented so that the region labelled $\mcA$ is  always to the right as one travels along a line in the direction of its  orientation.
Note  that in the compositions  of  functors,  at the top and bottom of the diagram, $F$ and $G$ always alternate, so these compositions have the form $FGF\dots$ or $GFG\dots$.  The two empty sequences of functors  correspond  to the  identity functors $\Ide_{\A}$ and $\Ide_{\B}$. 
If the rightmost (and semi-infinite) region of the strip is labelled $\mcA$, then the rightmost functor  in the  compositions at  the top  and  bottom is the functor $F$ (or one or both of these compositions is just the identity functor $\Ide_{\A}$). 
If the rightmost  region is labelled $\mcB$, the  rightmost functor is $G$ (or the identity functor $\Ide_{\B}$). 

In the planar diagrams, the top and bottom compositions have  the same parity of the number of terms (functors) $F,G$ appearing in the composition, with $\Ide_{\A}$ and $\Ide_{\B}$ having zero terms.

\vspace{0.1in} 

A closed diagram of nested circles with the outer region labelled by $\A$, respectively, $\B$, defines an element in the center $Z(\A)$ of $\A$, respectively in $Z(\B)$, see Figure~\ref{fig_circlesAB}. 

\begin{figure}
     \centering
     \begin{subfigure}[htb]{0.45\textwidth}
\begingroup%
  \makeatletter%
  \providecommand\color[2][]{%
    \errmessage{(Inkscape) Color is used for the text in Inkscape, but the package 'color.sty' is not loaded}%
    \renewcommand\color[2][]{}%
  }%
  \providecommand\transparent[1]{%
    \errmessage{(Inkscape) Transparency is used (non-zero) for the text in Inkscape, but the package 'transparent.sty' is not loaded}%
    \renewcommand\transparent[1]{}%
  }%
  \providecommand\rotatebox[2]{#2}%
  \newcommand*\fsize{\dimexpr\f@size pt\relax}%
  \newcommand*\lineheight[1]{\fontsize{\fsize}{#1\fsize}\selectfont}%
  \ifx\svgwidth\undefined%
    \setlength{\unitlength}{104.21505956bp}%
    \ifx\svgscale\undefined%
      \relax%
    \else%
      \setlength{\unitlength}{\unitlength * \real{\svgscale}}%
    \fi%
  \else%
    \setlength{\unitlength}{\svgwidth}%
  \fi%
  \global\let\svgwidth\undefined%
  \global\let\svgscale\undefined%
  \makeatother%
  \begin{picture}(1,0.81904437)%
    \lineheight{1}%
    \setlength\tabcolsep{0pt}%
    \put(0.22268509,0.53150462){\color[rgb]{0,0,0}\makebox(0,0)[lt]{\smash{\begin{tabular}[t]{l}$\A$\end{tabular}}}}%
    \put(0,0){\includegraphics[width=\unitlength,page=1]{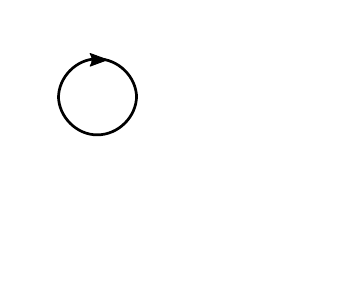}}%
    \put(0.22397171,0.2423527){\color[rgb]{0,0,0}\makebox(0,0)[lt]{\smash{\begin{tabular}[t]{l}$\A$\end{tabular}}}}%
    \put(0,0){\includegraphics[width=\unitlength,page=2]{circlesouterA.pdf}}%
    \put(0.62569783,0.38705711){\color[rgb]{0,0,0}\makebox(0,0)[lt]{\smash{\begin{tabular}[t]{l}$\B$\end{tabular}}}}%
    \put(0,0){\includegraphics[width=\unitlength,page=3]{circlesouterA.pdf}}%
    \put(0.41622397,0.08119948){\color[rgb]{0,0,0}\makebox(0,0)[lt]{\smash{\begin{tabular}[t]{l}$\B$\end{tabular}}}}%
    \put(0,0){\includegraphics[width=\unitlength,page=4]{circlesouterA.pdf}}%
    \put(0.62878343,0.20379957){\color[rgb]{0,0,0}\makebox(0,0)[lt]{\smash{\begin{tabular}[t]{l}$\A$\end{tabular}}}}%
    \put(-0.00477904,0.00846171){\color[rgb]{0,0,0}\makebox(0,0)[lt]{\smash{\begin{tabular}[t]{l}$\A$\end{tabular}}}}%
  \end{picture}%
\endgroup%
 
      \centering
         \caption{An element of $Z(\A)$}
         \label{fig:idF_1}
     \end{subfigure}
     \begin{subfigure}[htb]{0.45\textwidth}
        \centering
\begingroup%
  \makeatletter%
  \providecommand\color[2][]{%
    \errmessage{(Inkscape) Color is used for the text in Inkscape, but the package 'color.sty' is not loaded}%
    \renewcommand\color[2][]{}%
  }%
  \providecommand\transparent[1]{%
    \errmessage{(Inkscape) Transparency is used (non-zero) for the text in Inkscape, but the package 'transparent.sty' is not loaded}%
    \renewcommand\transparent[1]{}%
  }%
  \providecommand\rotatebox[2]{#2}%
  \newcommand*\fsize{\dimexpr\f@size pt\relax}%
  \newcommand*\lineheight[1]{\fontsize{\fsize}{#1\fsize}\selectfont}%
  \ifx\svgwidth\undefined%
    \setlength{\unitlength}{113.72154658bp}%
    \ifx\svgscale\undefined%
      \relax%
    \else%
      \setlength{\unitlength}{\unitlength * \real{\svgscale}}%
    \fi%
  \else%
    \setlength{\unitlength}{\svgwidth}%
  \fi%
  \global\let\svgwidth\undefined%
  \global\let\svgscale\undefined%
  \makeatother%
  \begin{picture}(1,0.41970743)%
    \lineheight{1}%
    \setlength\tabcolsep{0pt}%
    \put(0.16411078,0.19311216){\color[rgb]{0,0,0}\makebox(0,0)[lt]{\smash{\begin{tabular}[t]{l}$\B$\end{tabular}}}}%
    \put(0,0){\includegraphics[width=\unitlength,page=1]{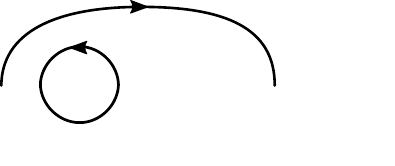}}%
    \put(0.30267698,0.05259536){\color[rgb]{0,0,0}\makebox(0,0)[lt]{\smash{\begin{tabular}[t]{l}$\A$\end{tabular}}}}%
    \put(0.46110853,0.19584835){\color[rgb]{0,0,0}\makebox(0,0)[lt]{\smash{\begin{tabular}[t]{l}$\B$\end{tabular}}}}%
    \put(0,0){\includegraphics[width=\unitlength,page=2]{circlesouterB.pdf}}%
    \put(0.86060279,0.19584949){\color[rgb]{0,0,0}\makebox(0,0)[lt]{\smash{\begin{tabular}[t]{l}$\A$\end{tabular}}}}%
    \put(0,0){\includegraphics[width=\unitlength,page=3]{circlesouterB.pdf}}%
    \put(0.72875897,0.00516956){\color[rgb]{0,0,0}\makebox(0,0)[lt]{\lineheight{1.25}\smash{\begin{tabular}[t]{l}$\B$\end{tabular}}}}%
  \end{picture}%
\endgroup%
 
         \caption{An element of $Z(\B)$}
         \label{fig:delta}
     \end{subfigure}
\caption{Examples of nested circles giving elements in $Z(\A)$ and $Z(\B)$.}
\label{fig_circlesAB}
\end{figure}

More generally, an element $z\in Z(\A)$ can be represented by  a dot labelled by $z$ floating in a region labelled by $\A$. Elements of $Z(\A)$ commute and can float past each other and anywhere in the region labelled $\A$, but generally can not cross the  lines describing the identity maps of $F$  and $G$ and the biadjointness morphisms. 

Wrapping  a clockwise circle around $z\in  Z(\A)$ is the trace map $Z(\A)\lra Z(\B)$ associated to the biadjoint pair $(F,G)$, see  Figure~\ref{fig1_5}, and the discussion in Section \ref{sec_transfer} around \eqref{eq_trace_map}.  Wrapping a counterclockwise circle around  $z'\in Z(\B)$ is the trace map $Z(\B)\lra  Z(\A)$. One reference for trace maps is~\cite{B}. 

\begin{figure}[htb]
\begin{center}
 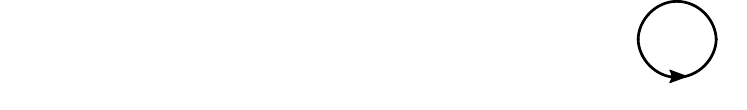
\caption{The trace maps $Z(\A)\lra Z(\B)$ and  $Z(\B)\lra Z(\A)$.}
\label{fig1_5}
\end{center}
\end{figure}

\begin{remark*}
If the categories $\mcA$ and  $\mcB$ happen to coincide, so that  $F,G$ are endofunctors  of $\mcA$, then all regions are colored by $\mcA$, and this label can be removed. Instead, orientations of  lines are then used to differentiate between $F$ and  $G$  and their corresponding transformations.  This allows diagrams for natural 
transformations  between arbitrary, not just alternating, compositions of functors $F$ and $G$, such as $FFGGGF$, etc. 
\end{remark*}


\subsection{Diagrammatics for a self-adjoint functor and the category \texorpdfstring{$\wU$}{U}}\label{sec-wU}
\quad

\smallskip

Following Do\v{s}en and Petri\'c~\cite{DP1,DP2}, in this section we explain how a  self-adjoint  functor   gives rise to a monoidal category $\wU$ described by collections  of circles and arcs in the plane, up  to rel boundary isotopies. We call such  collections \emph{$\wU$-diagrams} or \emph{arc and circle diagrams}.

\ssubsection{Self-adjoint  functor diagrammatics.}
Suppose that an endofunctor $F\colon  \mcA\lra \mcA$ on a category $\mcA$ is self-adjoint. This means that natural isomorphisms 
\begin{equation}
    \HomA(FM,N) \cong \HomA(M,FN), \ \  M,N\in\ObA
\end{equation}
have been fixed, over all pairs of objects in $\mcA$. Equivalently, one fixes natural transformations  
\begin{equation}
    \delta\colon \Ide_{\A}\Longrightarrow FF , \ \  \mu\colon FF\Longrightarrow\Ide_{\A},
\end{equation}
subject to the conditions
\begin{equation}\label{eq_cond}
    (1_F \, \mu )\circ (\delta \,1_F) = 1_F, \ \  ( \mu \, 1_F)\circ (1_F\, \delta) =1_F.
\end{equation}

\begin{figure}
     \centering
     \begin{subfigure}[htb]{0.2\textwidth}
\begingroup%
  \makeatletter%
  \providecommand\color[2][]{%
    \errmessage{(Inkscape) Color is used for the text in Inkscape, but the package 'color.sty' is not loaded}%
    \renewcommand\color[2][]{}%
  }%
  \providecommand\transparent[1]{%
    \errmessage{(Inkscape) Transparency is used (non-zero) for the text in Inkscape, but the package 'transparent.sty' is not loaded}%
    \renewcommand\transparent[1]{}%
  }%
  \providecommand\rotatebox[2]{#2}%
  \newcommand*\fsize{\dimexpr\f@size pt\relax}%
  \newcommand*\lineheight[1]{\fontsize{\fsize}{#1\fsize}\selectfont}%
  \ifx\svgwidth\undefined%
    \setlength{\unitlength}{45.00000283bp}%
    \ifx\svgscale\undefined%
      \relax%
    \else%
      \setlength{\unitlength}{\unitlength * \real{\svgscale}}%
    \fi%
  \else%
    \setlength{\unitlength}{\svgwidth}%
  \fi%
  \global\let\svgwidth\undefined%
  \global\let\svgscale\undefined%
  \makeatother%
  \begin{picture}(1,1.55455641)%
    \lineheight{1}%
    \setlength\tabcolsep{0pt}%
    \put(0.41573097,1.35195243){\color[rgb]{0,0,0}\makebox(0,0)[lt]{\smash{\begin{tabular}[t]{l}$F$\end{tabular}}}}%
    \put(0,0){\includegraphics[width=\unitlength,page=1]{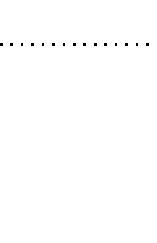}}%
    \put(0.41560265,0.01959638){\color[rgb]{0,0,0}\makebox(0,0)[lt]{\smash{\begin{tabular}[t]{l}$F$\end{tabular}}}}%
    \put(0,0){\includegraphics[width=\unitlength,page=2]{idF2.pdf}}%
  \end{picture}%
\endgroup%
 
      \centering
         \caption{$1_F$}
         \label{fig:idF_2}
     \end{subfigure}
     \begin{subfigure}[htb]{0.2\textwidth}
        \centering
\begingroup%
  \makeatletter%
  \providecommand\color[2][]{%
    \errmessage{(Inkscape) Color is used for the text in Inkscape, but the package 'color.sty' is not loaded}%
    \renewcommand\color[2][]{}%
  }%
  \providecommand\transparent[1]{%
    \errmessage{(Inkscape) Transparency is used (non-zero) for the text in Inkscape, but the package 'transparent.sty' is not loaded}%
    \renewcommand\transparent[1]{}%
  }%
  \providecommand\rotatebox[2]{#2}%
  \newcommand*\fsize{\dimexpr\f@size pt\relax}%
  \newcommand*\lineheight[1]{\fontsize{\fsize}{#1\fsize}\selectfont}%
  \ifx\svgwidth\undefined%
    \setlength{\unitlength}{73.78704627bp}%
    \ifx\svgscale\undefined%
      \relax%
    \else%
      \setlength{\unitlength}{\unitlength * \real{\svgscale}}%
    \fi%
  \else%
    \setlength{\unitlength}{\svgwidth}%
  \fi%
  \global\let\svgwidth\undefined%
  \global\let\svgscale\undefined%
  \makeatother%
  \begin{picture}(1,0.95710789)%
    \lineheight{1}%
    \setlength\tabcolsep{0pt}%
    \put(0,0){\includegraphics[width=\unitlength,page=1]{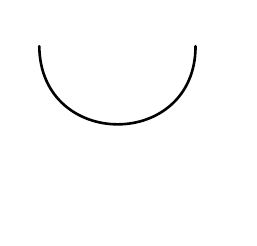}}%
    \put(0.10270737,0.83354708){\color[rgb]{0,0,0}\makebox(0,0)[lt]{\smash{\begin{tabular}[t]{l}$F$\end{tabular}}}}%
    \put(0.71277671,0.83354708){\color[rgb]{0,0,0}\makebox(0,0)[lt]{\smash{\begin{tabular}[t]{l}$F$\end{tabular}}}}%
    \put(0,0){\includegraphics[width=\unitlength,page=2]{delta.pdf}}%
    \put(0.35419728,0.01917734){\color[rgb]{0,0,0}\makebox(0,0)[lt]{\smash{\begin{tabular}[t]{l}$\Ide_\A$\end{tabular}}}}%
  \end{picture}%
\endgroup%
 
         \caption{$\delta$}
         \label{fig:delta3}
     \end{subfigure}
     \begin{subfigure}[htb]{0.2\textwidth}
        \centering
\begingroup%
  \makeatletter%
  \providecommand\color[2][]{%
    \errmessage{(Inkscape) Color is used for the text in Inkscape, but the package 'color.sty' is not loaded}%
    \renewcommand\color[2][]{}%
  }%
  \providecommand\transparent[1]{%
    \errmessage{(Inkscape) Transparency is used (non-zero) for the text in Inkscape, but the package 'transparent.sty' is not loaded}%
    \renewcommand\transparent[1]{}%
  }%
  \providecommand\rotatebox[2]{#2}%
  \newcommand*\fsize{\dimexpr\f@size pt\relax}%
  \newcommand*\lineheight[1]{\fontsize{\fsize}{#1\fsize}\selectfont}%
  \ifx\svgwidth\undefined%
    \setlength{\unitlength}{73.89416297bp}%
    \ifx\svgscale\undefined%
      \relax%
    \else%
      \setlength{\unitlength}{\unitlength * \real{\svgscale}}%
    \fi%
  \else%
    \setlength{\unitlength}{\svgwidth}%
  \fi%
  \global\let\svgwidth\undefined%
  \global\let\svgscale\undefined%
  \makeatother%
  \begin{picture}(1,0.90173871)%
    \lineheight{1}%
    \setlength\tabcolsep{0pt}%
    \put(0,0){\includegraphics[width=\unitlength,page=1]{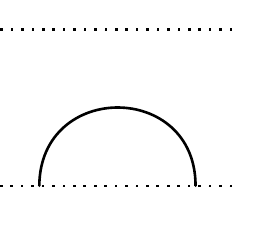}}%
    \put(0.35368384,0.84004787){\color[rgb]{0,0,0}\makebox(0,0)[lt]{\smash{\begin{tabular}[t]{l}$\Ide_\A$\end{tabular}}}}%
    \put(0.10437172,0.02386754){\color[rgb]{0,0,0}\makebox(0,0)[lt]{\smash{\begin{tabular}[t]{l}$F$\end{tabular}}}}%
    \put(0.71319306,0.02567975){\color[rgb]{0,0,0}\makebox(0,0)[lt]{\smash{\begin{tabular}[t]{l}$F$\end{tabular}}}}%
  \end{picture}%
\endgroup%

         \caption{$\mu$}
         \label{fig:mu}
     \end{subfigure}
\caption{Identity  transformation $1_F$ and generating natural  transformations $\delta$ and $\mu$ for a self-adjoint functor. Every region of these diagram is labelled  by the category $\A$.}
\label{fig2_1}
\end{figure}

\begin{figure}[htb]
\begin{center}
 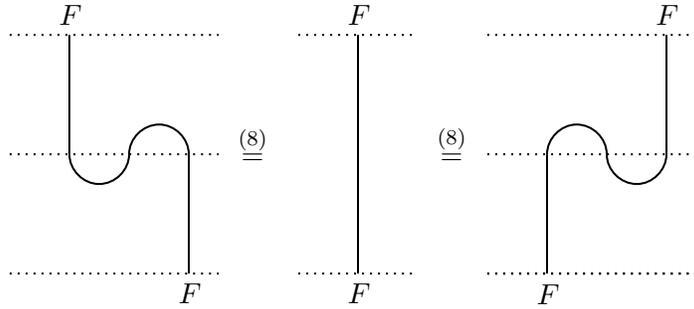
\caption{Defining relations of self-adjointness are the isotopy relations.}
\label{fig2_2}
\end{center}
\end{figure}

\begin{figure}[htb]
\begin{center}
\begingroup%
  \makeatletter%
  \providecommand\color[2][]{%
    \errmessage{(Inkscape) Color is used for the text in Inkscape, but the package 'color.sty' is not loaded}%
    \renewcommand\color[2][]{}%
  }%
  \providecommand\transparent[1]{%
    \errmessage{(Inkscape) Transparency is used (non-zero) for the text in Inkscape, but the package 'transparent.sty' is not loaded}%
    \renewcommand\transparent[1]{}%
  }%
  \providecommand\rotatebox[2]{#2}%
  \newcommand*\fsize{\dimexpr\f@size pt\relax}%
  \newcommand*\lineheight[1]{\fontsize{\fsize}{#1\fsize}\selectfont}%
  \ifx\svgwidth\undefined%
    \setlength{\unitlength}{198.75001039bp}%
    \ifx\svgscale\undefined%
      \relax%
    \else%
      \setlength{\unitlength}{\unitlength * \real{\svgscale}}%
    \fi%
  \else%
    \setlength{\unitlength}{\svgwidth}%
  \fi%
  \global\let\svgwidth\undefined%
  \global\let\svgscale\undefined%
  \makeatother%
  \begin{picture}(1,0.70633817)%
    \lineheight{1}%
    \setlength\tabcolsep{0pt}%
    \put(0,0){\includegraphics[width=\unitlength,page=1]{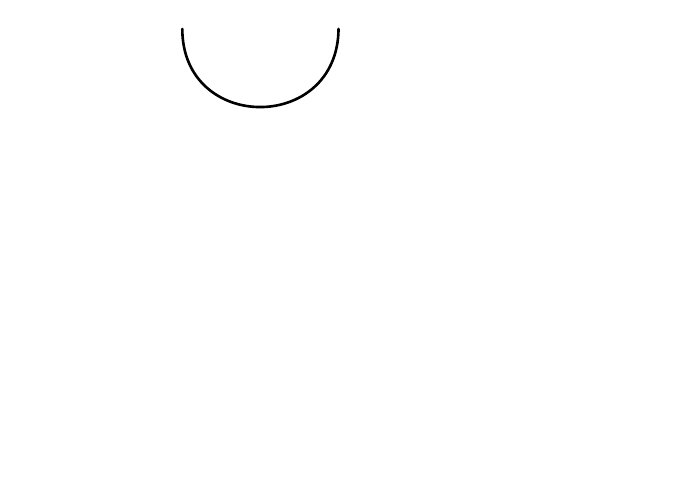}}%
    \put(0.24534085,0.68340185){\color[rgb]{0,0,0}\makebox(0,0)[lt]{\smash{\begin{tabular}[t]{l}$F$\end{tabular}}}}%
    \put(0.47183248,0.68340185){\color[rgb]{0,0,0}\makebox(0,0)[lt]{\smash{\begin{tabular}[t]{l}$F$\end{tabular}}}}%
    \put(0,0){\includegraphics[width=\unitlength,page=2]{expldiag2.pdf}}%
    \put(0.18879866,0.0053899){\color[rgb]{0,0,0}\makebox(0,0)[lt]{\smash{\begin{tabular}[t]{l}$F$\end{tabular}}}}%
    \put(0,0){\includegraphics[width=\unitlength,page=3]{expldiag2.pdf}}%
    \put(0.1885837,0.68340185){\color[rgb]{0,0,0}\makebox(0,0)[lt]{\smash{\begin{tabular}[t]{l}$F$\end{tabular}}}}%
    \put(0.52884267,0.00443691){\color[rgb]{0,0,0}\makebox(0,0)[lt]{\smash{\begin{tabular}[t]{l}$F$\end{tabular}}}}%
    \put(0,0){\includegraphics[width=\unitlength,page=4]{expldiag2.pdf}}%
    \put(0.52935093,0.68248695){\color[rgb]{0,0,0}\makebox(0,0)[lt]{\smash{\begin{tabular}[t]{l}$F$\end{tabular}}}}%
    \put(0,0){\includegraphics[width=\unitlength,page=5]{expldiag2.pdf}}%
  \end{picture}%
\endgroup%

\caption{A natural transformation from $F^2$ to  $F^4$ described by a diagram of 3 arcs and 7 embedded circles. Any such  system can be further transformed via isotopies into a form where each  circle has a unique maximum and minimum for the projection on  the $y$-axis and each arc has at most  one such extremum (none if it connects a top $F$  with a  bottom $F$).}
\label{fig2_3}
\end{center}
\end{figure}

Using the natural transformations $1_F,\delta,\mu$ and their  planar compositions  one can  build various natural transformations $F^n\Longrightarrow F^m$, $n,m\ge 0$ between powers of $F$, including the case $n=0$ or $m=0$, where the  corresponding functor is $F^0=\Ide_{\A}$. 
These natural transformations from $F^n$ to $F^m$  are encoded  by isotopy classes  of collections of  $\frac{n+m}{2}$ properly embedded arcs and finitely many circles in the strip $\R\times[0,1]$ of the plane, see Figure~\ref{fig2_3} for an example  of  a natural transformation from  $F^2$ to $F^4$. 

\vspace{0.1in}

Specializing to $n=m=0$, each diagram of nested circles in the plane defines an element of the center of $\A$. Wrapping a circle around such a diagram is the trace map, as in Figure~\ref{fig1_5}, where now we do not orient the circles.

\ssubsection{The monoidal category $\wU$.}
Denote by $\wU^m_n$ the set of isotopy classes of planar diagrams discussed above, with $\frac{n+m}{2}$ arcs connecting in pairs $n$ points on the  bottom  line and  $m$ points on the top line and some number of circles (possibly none). There is an associative composition
\begin{equation}\label{Ucomp}
    \wU^k_m \times \wU^m_n \lra \wU^k_n 
\end{equation}
given by stacking and concatenating two digrams along their common $m$ boundary points, see Figure~\ref{fig2_4} for an example. 

\begin{figure}[htb]
\begin{center}
 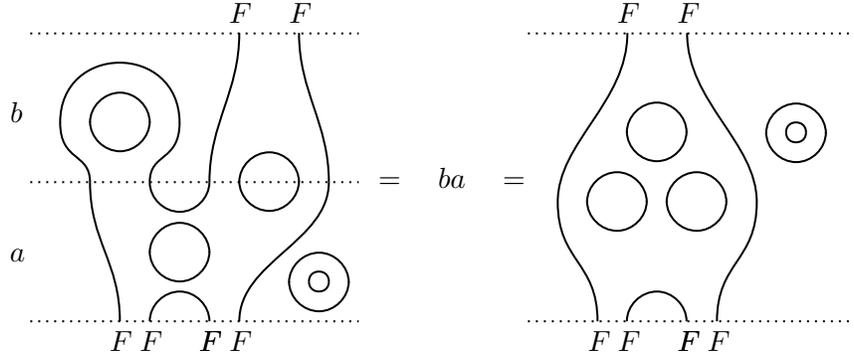
\caption{The composition $ba$ of diagrams $b$ and $a$.}
\label{fig2_4}
\end{center}
\end{figure}

The composition in \eqref{Ucomp} is associative. We  can form a category $\wU$ having non-negative  integers $n$ as objects and  morphisms from $n$  to $m$ given by elements  of $\wU^m_n$.  The unit  morphism  $1_n$ is given by  the diagram of $n$ vertical arcs.  

Denote by $\wU^m_n(k)$ the set of diagrams  in $\wU^m_n$ with $k$ circles. Then
\begin{equation}
    \wU^m_n = \bigsqcup_{k\ge 0} \wU^m_n(k). 
\end{equation}

The category $\wU$ is strict monoidal, with the tensor product given on  objects  by $n\otimes m = n+m$ and  on morphisms by  stacking them next to each other on the plane. This monoidal structure is  rigid, with  self-dual objects  $n^{\ast}\cong n$. In fact, the category $\wU$  is  free as a strict monoidal category over the self-dual object $1$ \cite[Proposition~9.4]{Del}. We note that the category  $\wU$ is not symmetric.  

\vspace{0.1in}

A self-adjointness datum $(\delta,\mu)$ for the endofunctor $F$ of $\mcA$ as above  gives a monoidal functor 
\begin{equation}\label{eq_mon_fun} 
    \mathcal{F}\colon  \wU  \lra \mcEnd(\mcA)
\end{equation}
from the category $\wU$  to the monoidal category of endofunctors  of $\mcA$ that assigns $F^n$  to the object $n$ of $\wU$ and natural transformations $\delta$ and  $\mu$ to the cup and  cap  diagrams as in Figure~\ref{fig2_1}. Conversely, a monoidal functor as in (\ref{eq_mon_fun}) determines a  functor  $F=\mathcal{F}(1)$ and biadjointness data $(\delta,\mu)$ by applying $\mathcal{F}$ to the cup and cap diagrams in $\wU_0^2$ and $\wU_2^0$, correspondingly.

\ssubsection{The forgetful functor to the category $\wB$ of crossingless  matchings.}
Given a $\wU$-diagram as studied above, forgetting the circles gives us  a crossingless (planar) matching  of $n$ points  on the  bottom  and  $m$  points on top of the strip. Denote by $\wB^m_n$ the set of such matchings (as usual, we choose one representative  diagram from the  corresponding isotopy  class for each matching). We refer to an element of $\wB_n^m$ as a \emph{$\wB$-diagram}, or \emph{arc diagram}.

\vspace{0.1in}

We denote by $\wB$ the category, with objects $n\in\Z_+$, morphisms from
$n$ to $m$  in $\wB$ given by  the set $\wB_n^m$ of crossingless matching diagrams in  the strip, and   composition of morphisms  given by  concatenation of diagrams followed by removal of any circles that  may  appear. The category $\wB$ is rigid monoidal, similarly to $\wU$, with tensor product $n\otimes m=n+m$.

\vspace{0.1in}

Given a diagram $u\in \wU_n^m$ denote by $\mcr(u)$ the number of circles in $u$ and by $\arc(u)\in \wB^m_n$ the diagram obtained by  removing circles from $u$.
For $u\in \wU^m_n(k)$ we have $\mcr(u)=k$. 
The  map $\arc$ that  forgets the circles, 
\begin{equation}
    \arc \colon  \wU_n^m \lra\wB^m_n,
\end{equation}
extends to a functor
\begin{equation}\label{UBforget}
    \arc \colon  \wU \lra \wB.
\end{equation}
In other words, the functor $\arc$ turns an arc-circle diagram into an arc diagram by removing all circles. The functor $\arc$ is monoidal, full, and a bijection on objects $n\in\Z_+$ of both categories.  We refer the reader to \cite[Section~2.2]{Kh1} for details on the category $\wB$ and the functor $\arc$.

\ssubsection{Temperley--Lieb categories}

Let $R$ be a commutative ring. The monoidal category $\wB$ can  be linearized   by  forming arbitrary  linear combinations  of morphisms from $n$ to $m$ with coefficients in $R$. The resulting category  $R\wB$ is equivalent to the \emph{Temperley--Lieb  category} $TL(d)$, where  the value $d$ of the  circle is one, 
\begin{equation}
    R\wB \cong TL(1). 
\end{equation}
More generally, the Temperley--Lieb category $TL(d)$, for $d\in R$, is the pre-additive $R$-linear monoidal  category
with  objects  $n\in\Z_+$  and morphisms from $n$  to $m$ being $R$-linear combinations of crossingless matchings in $\wB^m_n$. Upon  composition (concatenation) of two  matchings each resulting circle  is removed simultaneously with  multiplying the remaining expression by $d\in R$ for each removed circle. 

The endomorphism rings $TL_n(d)=\End_{TL(d)}(n)$ are the \emph{Temperley--Lieb algebras} \cite{Jo1,Jo4,Jo5}. These algebras have interesting connections to knot theory, quantum groups, and statistical mechanics, see \cite{Jo2,Jo5,Jo6,Ka} and references therein. One usually specializes  $R$  to a field $\kk$ (often $\C$ of $\Q(q)$) and sets $d=\pm(q+q^{-1})$ or a similar expression, see \cite{CFS,KaL}.

We also linearize the category $\wU$ to  a category $R\wU$ by keeping the same objects  $n\in\Z_+$ and allowing arbitrary finite $R$-linear  combinations of morphisms. The category $R\wU$ is a pre-additive monoidal $R$-linear category. Picking $d\in R$ gives  a monoidal functor  
\begin{equation}
    \arc_d \ \colon   \ R\wU \lra TL(d)
\end{equation}
that takes a diagram in $\wU$ and evaluates each circle to $d$  while keeping the  collection of  arcs the same. A diagram $u\in \wU^m_n(k)$ with  $k$ circles and underlying diagram $\arc(u)\in\wB_n^m$ of arcs  goes under  $\arc_d$ to $d^k \, \arc(u)$: 
\begin{equation}
    \arc_d (u) =   d^{\,\mcr(u)}\,\arc(u)= d^{\,k}\, \arc(u),
\end{equation}
where we denote by $\mcr(u)$  the  number of circles in $u$. 

The functor $\arc_d$ forgets an enormous amount of information, since to evaluate it on a diagram $u\in \wU^m_n(k)$ one only needs to know the underlying arc diagram and the  number of circles $k$ in $u$. In this paper we explore more subtle ways to linearize $\wU$ and produce functors from it to categories with finite-dimensional morphism spaces.

\ssubsection{Circle diagrams.}
Endomorphisms $\End_{\wU}(0)=\wU^0_0$ of  the unit object $0$ of $\wU$ constitute a commutative monoid. 
Elements of $\wU^0_0$ are isotopy classes of finite collections of circles  in plane (it is convenient to fix a representative  for each isotopy class). 
Do\v{s}en and Petri\'c~\cite{DP1,DP2}   call elements  of $\wU^0_0$ \emph{circular forms}. We will also call them \emph{circle diagrams} or \emph{closed $\wU$-diagrams}.
Reflection in the plane takes any closed diagram to  itself, see Proposition~\ref{prop_refl}.

Notice  that the monoid $\wU^0_0$ is commutative since we can slide one group of circles  past the other next to it. The monoid $\wU^0_0$ is isomorphic to the free commutative  monoid on  the following countable set $\wUcirc$. 

 We call a circle $c$ in $u\in \wU^0_0$ \emph{exterior} (or \emph{outer}) if it borders the infinite region of the diagram $u$. Equivalently, a circle $c$ is exterior if  it can be connected to the boundary of the strip $\R\times[0,1]$ by an arc disjoint from other circles.  
In the diagram in Figure~\ref{fig2_5} there are four exterior circles, each marked by the letter $e$ next to  it.

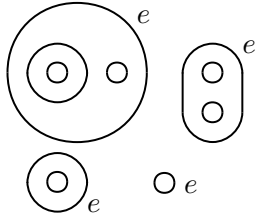
\begin{figure}[htb]
\begin{center}
\begingroup%
  \makeatletter%
  \providecommand\color[2][]{%
    \errmessage{(Inkscape) Color is used for the text in Inkscape, but the package 'color.sty' is not loaded}%
    \renewcommand\color[2][]{}%
  }%
  \providecommand\transparent[1]{%
    \errmessage{(Inkscape) Transparency is used (non-zero) for the text in Inkscape, but the package 'transparent.sty' is not loaded}%
    \renewcommand\transparent[1]{}%
  }%
  \providecommand\rotatebox[2]{#2}%
  \newcommand*\fsize{\dimexpr\f@size pt\relax}%
  \newcommand*\lineheight[1]{\fontsize{\fsize}{#1\fsize}\selectfont}%
  \ifx\svgwidth\undefined%
    \setlength{\unitlength}{96.09571642bp}%
    \ifx\svgscale\undefined%
      \relax%
    \else%
      \setlength{\unitlength}{\unitlength * \real{\svgscale}}%
    \fi%
  \else%
    \setlength{\unitlength}{\svgwidth}%
  \fi%
  \global\let\svgwidth\undefined%
  \global\let\svgscale\undefined%
  \makeatother%
  \begin{picture}(1,0.82951556)%
    \lineheight{1}%
    \setlength\tabcolsep{0pt}%
    \put(0,0){\includegraphics[width=\unitlength,page=1]{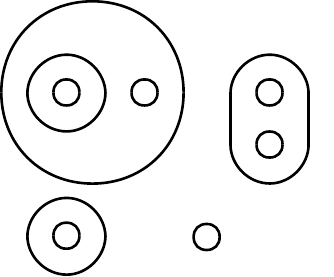}}%
    \put(0.5112089,0.74756593){\color[rgb]{0,0,0}\makebox(0,0)[lt]{\smash{\begin{tabular}[t]{l}$e$\end{tabular}}}}%
    \put(0.92485926,0.63049515){\color[rgb]{0,0,0}\makebox(0,0)[lt]{\smash{\begin{tabular}[t]{l}$e$\end{tabular}}}}%
    \put(0.3160911,0.00611776){\color[rgb]{0,0,0}\makebox(0,0)[lt]{\smash{\begin{tabular}[t]{l}$e$\end{tabular}}}}%
    \put(0.69734845,0.08026958){\color[rgb]{0,0,0}\makebox(0,0)[lt]{\smash{\begin{tabular}[t]{l}$e$\end{tabular}}}}%
  \end{picture}%
\endgroup%
 
\caption{A diagram $u\in \wU^0_0$ with four exterior circles each labelled  by the letter  $e$.}
\label{fig2_5}
\end{center}
\end{figure}

Denote  by  $\wUcirc$ the subset  of $\wU^0_0$ consisting of  diagrams with  only one exterior circle. Elements of  $\wUcirc$  may be called  \emph{$\circ$-diagrams} or \emph{outer diagrams}.
Note  that  the empty  diagram $\emptyset$  with  no circles (which is the  unit element  of the monoid  $\End_{\wU}(0)$) is not in $\wUcirc$.

A diagram $u\in  \wU^0_0$ is in $\wUcirc$  if and only if it has exactly one exterior circle. Let 
\begin{equation}\label{eq_def_omega_1} 
    \omega \ \colon  \ \wU^0_0 \lra \wUcirc
\end{equation}
be a bijection of sets  that  takes a closed  diagram and  wraps a circle around it, making it a diagram in $\wUcirc$ (a  $\circ$-diagram). Starting with the empty diagram $\emptyset$ and iteratively applying $\omega$ and forming unions of diagrams one can generate  any diagram in $\wU^0_0$, see Figure~\ref{fig2_4_2} for  examples. 

\begin{figure}[htb]
\begin{center}
\begingroup%
  \makeatletter%
  \providecommand\color[2][]{%
    \errmessage{(Inkscape) Color is used for the text in Inkscape, but the package 'color.sty' is not loaded}%
    \renewcommand\color[2][]{}%
  }%
  \providecommand\transparent[1]{%
    \errmessage{(Inkscape) Transparency is used (non-zero) for the text in Inkscape, but the package 'transparent.sty' is not loaded}%
    \renewcommand\transparent[1]{}%
  }%
  \providecommand\rotatebox[2]{#2}%
  \newcommand*\fsize{\dimexpr\f@size pt\relax}%
  \newcommand*\lineheight[1]{\fontsize{\fsize}{#1\fsize}\selectfont}%
  \ifx\svgwidth\undefined%
    \setlength{\unitlength}{233.22579539bp}%
    \ifx\svgscale\undefined%
      \relax%
    \else%
      \setlength{\unitlength}{\unitlength * \real{\svgscale}}%
    \fi%
  \else%
    \setlength{\unitlength}{\svgwidth}%
  \fi%
  \global\let\svgwidth\undefined%
  \global\let\svgscale\undefined%
  \makeatother%
  \begin{picture}(1,0.16400529)%
    \lineheight{1}%
    \setlength\tabcolsep{0pt}%
    \put(0,0){\includegraphics[width=\unitlength,page=1]{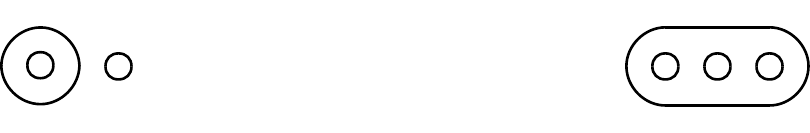}}%
    \put(0.21079856,0.06592381){\color[rgb]{0,0,0}\makebox(0,0)[lt]{\smash{\begin{tabular}[t]{l}$,$\end{tabular}}}}%
    \put(0,0){\includegraphics[width=\unitlength,page=2]{omegadiags.pdf}}%
    \put(0.66230366,0.06592381){\color[rgb]{0,0,0}\makebox(0,0)[lt]{\smash{\begin{tabular}[t]{l}$,$\end{tabular}}}}%
  \end{picture}%
\endgroup%
 
\caption{Diagrams $\omega^2(\emptyset)\omega(\emptyset),   \omega(\omega(\omega(\emptyset)^2)\omega(\emptyset)), $ and $\omega^2(\omega(\emptyset)^3)$ of $\wU^0_0$. The first diagram is not in $\wUcirc$ while the second and third diagrams are.}
\label{fig2_4_2}
\end{center}
\end{figure}

\ssubsection{Reflection involutions on $\wU$.}
The category $\wU$ carries an involution $\rho_v$, referred to as \emph{vertical reflection}, that takes the object $n$ to $n$  and reflects a diagram about a vertical axis. The category $\wU$  also carries a contravariant involution $\rho_h$, called \emph{horizontal reflection}, also denoted  by the  bar symbol, taking  $a$ to $\overline{a}$. The latter takes $n$  to $n$ and reflects a diagram about a horizontal line. 
These  relections give monoidal functors 
\begin{equation}\label{reflections}
    \rho_h \colon (\wU,\otimes)\to (\wU^{\oop},\otimes), \quad \quad     \rho_v \colon (\wU,\otimes)\to (\wU,\otimes^{\oop}),
\end{equation}
where $f\otimes^{\oop} g=g\otimes f$ for morphisms.

Note that the involutions $\rho_v$ and $\rho_h$ commute (in the strict sense). Furthermore, on a  closed diagram (a diagram in  $\wU^0_0$),  horizontal  and vertical reflection have the same effect, since we  consider  diagrams up to isotopies, so that $\rho_v = \rho_h$ on endomorphisms of $0$. In fact, a stronger statement holds. 

\begin{prop} \label{prop_refl}
Horizontal refection (and hence also vertical reflection) is the identity on the set $\wU^0_0$ of (isomorphism classes) of circle diagrams.
\end{prop} 

\begin{proof}
To prove this, we observe that any diagram in $\wU^0_0$ can be represented  by a collection of circles in the  plane such  that each circle has exactly two (generic) intersection points with the chosen horizontal axis, see an example in  Figure~\ref{fig2_4_0}. This is easy to show by induction on the number  of circles  in the  diagram. The property clearly holds for the empty diagram and  a one-circle diagrams. If it  holds for diagrams $u_1,u_2$, it holds  for  their union, since one  can place $u_1,u_2$ along the horizontal axis as required and  away from each other. If the discussed symmetry property holds for a diagram $u$, it holds for the diagram $\omega(u)$, the diagram given  by placing  the circle (which is itself symmetric about the horizontal axis) around $u$.
All diagrams in $\wU_0^0$ can be constructed inductively using disjoint union and the operation $\omega$. Thus, the proposition  follows. 
\end{proof}

\begin{figure}[htb]
\begin{center}
 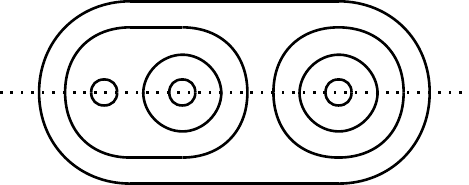
\caption{A diagram invariant  under the reflection  about the (dotted) horizontal axis.}
\label{fig2_4_0}
\end{center}
\end{figure}

\ssubsection{Diagrams in the disk and annulus}

A diagram $u\in \wU^{n}_m$ can alternatively be described as a diagram in a disk $D^2$ with $n+m$ marked points on the outer circle.
To remember that $x$ comes from a morphism from $n$ to $m$, we can arrange $n$ points to be on the lower half-circle of the disk and $m$ points on the upper half-circle --- see Figure \ref{fig2_4_0a} for an example.  

\begin{figure}[htb]
\begin{center}
\begingroup%
  \makeatletter%
  \providecommand\color[2][]{%
    \errmessage{(Inkscape) Color is used for the text in Inkscape, but the package 'color.sty' is not loaded}%
    \renewcommand\color[2][]{}%
  }%
  \providecommand\transparent[1]{%
    \errmessage{(Inkscape) Transparency is used (non-zero) for the text in Inkscape, but the package 'transparent.sty' is not loaded}%
    \renewcommand\transparent[1]{}%
  }%
  \providecommand\rotatebox[2]{#2}%
  \newcommand*\fsize{\dimexpr\f@size pt\relax}%
  \newcommand*\lineheight[1]{\fontsize{\fsize}{#1\fsize}\selectfont}%
  \ifx\svgwidth\undefined%
    \setlength{\unitlength}{192.75009484bp}%
    \ifx\svgscale\undefined%
      \relax%
    \else%
      \setlength{\unitlength}{\unitlength * \real{\svgscale}}%
    \fi%
  \else%
    \setlength{\unitlength}{\svgwidth}%
  \fi%
  \global\let\svgwidth\undefined%
  \global\let\svgscale\undefined%
  \makeatother%
  \begin{picture}(1,0.3687918)%
    \lineheight{1}%
    \setlength\tabcolsep{0pt}%
    \put(0,0){\includegraphics[width=\unitlength,page=1]{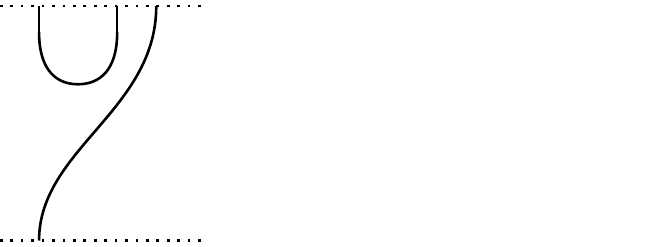}}%
    \put(0.39688636,0.17342152){\color[rgb]{0,0,0}\makebox(0,0)[lt]{\smash{\begin{tabular}[t]{l}$\longrightarrow$\end{tabular}}}}%
    \put(0,0){\includegraphics[width=\unitlength,page=2]{Uzdiagtocirc.pdf}}%
  \end{picture}%
\endgroup%

\caption{An example of representing an element of $\wU_n^m$ as a diagram in the disk with marked circle points.}
\label{fig2_4_0a}
\end{center}
\end{figure}

Take the semi-open annulus $\wA=\R^2\setminus  \mathrm{int}(D^2)$, the complement in $\R^2$ of the  interior of the  disk $D^2$. We refer to $\partial\wA \cong \SS^1\cong \partial D^2$ as the inner circle of $\wA$. 

We introduce the set $\wUout{2n}$ of \emph{outer circle diagrams} in the annulus $\wA$ for $2n$ points on the inner circle of the annulus. 
Outer circle diagrams are isotopy classes of collections of $n$ disjoint arcs in $\wA$ connecting these $2n$  points  in  pairs, together with circular forms (i.e., elements of $\wU_0^0$) that float in the regions of the annulus separated by the arcs. Figure \ref{fig2_4_0b} depicts an example of an outer circle diagram. 
There,  we have placed a mark {\scriptsize$\times$} on the circle between the leftmost top and leftmost bottom point. The mark corresponds to the position  of  the  left  edge of  the box --- it helps to easily separate  top and bottom boundary points, if  needed. From this marking, the arc on the circle corresponding to the position of the right  edge of the box can be  recovered if $n$ is  known.

\begin{figure}[htb]
\begin{center}
 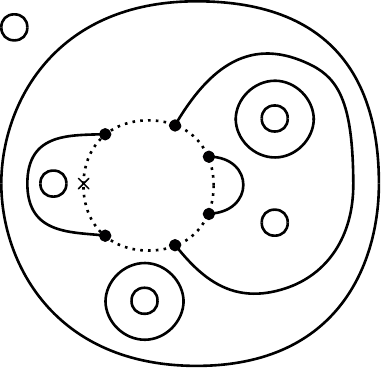
\caption{An outer circle diagram in $\wUout{6}$ with marker {\scriptsize$\times$}.}
\label{fig2_4_0b}
\end{center}
\end{figure}


\subsection{Circular triples obtained from wrapping actions  on centers of categories} \label{actionscenters}

\quad
\smallskip

Recall that a self-adjoint endofunctor $F$ (together with a choice of self-adjunction) on a category $\mcA$ gives rise to a monoidal  functor  
\begin{equation}\label{eq_mon_fun_2} 
    \mathcal{F}\colon  \wU  \lra \mcEnd(\mcA),
\end{equation}
see the discussion around formula (\ref{eq_mon_fun}). To a  diagram $a$ in $\wU^m_n$ the functor $\mathcal{F}$ assigns  a natural transformation  $\mcF(a)\colon F^n\rightarrow F^m$.

In the special case when  $a\in  \wU^0_0$, that is, $a$ is a circle diagram (a diagram of circles in the  plane), $\mcF(a)$ is a natural transformation of  the identity functor $\Ide_{\mcA}$, that is, an element of the  center  $Z(\mcA)$ of $\mcA$. 

The center $Z(\mcB)$ of any category $\mcB$ is a commutative monoid  under composition, with the identity  natural transformation  of $\Ide_{\mcB}$ as the  unit element. 
For $\mcA$ and $F$ as above, the commutative monoid $Z(\mcA)$ also carries a map $\omega$ that wraps a circle around an element $z\in  Z(\mcA)$, 
\begin{equation}\label{eqn_wrapping}
    \omega(z) = \mu \circ (1_F z 1_F) \circ \delta,
\end{equation}
as shown in Figure~\ref{fig2_4_1}. 
This map $\omega$ is usually not a monoid homomorphism and does not take the unit $1$ to $1$. 

\begin{figure}[htb]
\begin{center}
$$z\quad \longmapsto \quad \vcenter{\hbox{
\begingroup%
  \makeatletter%
  \providecommand\color[2][]{%
    \errmessage{(Inkscape) Color is used for the text in Inkscape, but the package 'color.sty' is not loaded}%
    \renewcommand\color[2][]{}%
  }%
  \providecommand\transparent[1]{%
    \errmessage{(Inkscape) Transparency is used (non-zero) for the text in Inkscape, but the package 'transparent.sty' is not loaded}%
    \renewcommand\transparent[1]{}%
  }%
  \providecommand\rotatebox[2]{#2}%
  \newcommand*\fsize{\dimexpr\f@size pt\relax}%
  \newcommand*\lineheight[1]{\fontsize{\fsize}{#1\fsize}\selectfont}%
  \ifx\svgwidth\undefined%
    \setlength{\unitlength}{124.51632552bp}%
    \ifx\svgscale\undefined%
      \relax%
    \else%
      \setlength{\unitlength}{\unitlength * \real{\svgscale}}%
    \fi%
  \else%
    \setlength{\unitlength}{\svgwidth}%
  \fi%
  \global\let\svgwidth\undefined%
  \global\let\svgscale\undefined%
  \makeatother%
  \begin{picture}(1,0.6013424)%
    \lineheight{1}%
    \setlength\tabcolsep{0pt}%
    \put(0,0){\includegraphics[width=\unitlength,page=1]{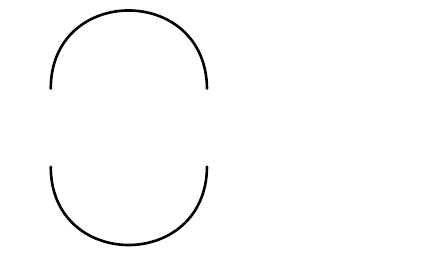}}%
    \put(0.26861879,0.29377855){\color[rgb]{0,0,0}\makebox(0,0)[lt]{\smash{\begin{tabular}[t]{l}$z$\end{tabular}}}}%
    \put(0,0){\includegraphics[width=\unitlength,page=2]{wrap.pdf}}%
    \put(-0.00266657,0.2930793){\color[rgb]{0,0,0}\makebox(0,0)[lt]{\smash{\begin{tabular}[t]{l}$1_F$\end{tabular}}}}%
    \put(0.50785455,0.29383987){\color[rgb]{0,0,0}\makebox(0,0)[lt]{\smash{\begin{tabular}[t]{l}$1_F$\end{tabular}}}}%
    \put(0.41857138,0.57693546){\color[rgb]{0,0,0}\makebox(0,0)[lt]{\smash{\begin{tabular}[t]{l}$\mu$\end{tabular}}}}%
    \put(0.41857138,0.00472139){\color[rgb]{0,0,0}\makebox(0,0)[lt]{\smash{\begin{tabular}[t]{l}$\delta$\end{tabular}}}}%
    \put(0.6655271,0.29383987){\color[rgb]{0,0,0}\makebox(0,0)[lt]{\smash{\begin{tabular}[t]{l}$=$\end{tabular}}}}%
    \put(0.8847328,0.28217761){\color[rgb]{0,0,0}\makebox(0,0)[lt]{\smash{\begin{tabular}[t]{l}$z$\end{tabular}}}}%
    \put(0,0){\includegraphics[width=\unitlength,page=3]{wrap.pdf}}%
  \end{picture}%
\endgroup%
 }}$$
\caption{The ``wrapping" action of $\omega$ on an element  $z\in Z(\mcA)$.}
\label{fig2_4_1}
\end{center}
\end{figure}

The functor $\cF$ restricts to a homorphism of commutative monoids 
$$\cF\colon \wU_0^0\longrightarrow Z(\A).$$
Explicitly, the restriction $\cF$ is constructed as follows. To each  circular form $u\in \wU^0_0$ one assigns an element $\mc{F}(u)$ of $Z(\mcA)$, constructed inductively on the number of  circles in $u$.

\begin{enumerate}
    \item If $u=u_1 u_2 $ is a product of circular forms $u_1,u_2$ (that is, $u$ is given by placing $u_1$ and  $u_2$ next  to  each other), then 
    \begin{equation}
    \mcF(u)=\mcF(u_1)\mcF( u_2).
    \end{equation} 
    \item If  $u$ is given by a diagram $v$  with a  circle wrapped around it  (which we write as $u=\omega(v)$), then
    \begin{equation}
    \mcF(u) = \omega(\mcF(v)).
    \end{equation}
\end{enumerate}

We see  that iterating multiplication in $Z(\mcA)$ with applying the map $\omega$ gives an endomorphism of the identity functor for any circular  diagram, see examples in Figure~\ref{fig2_4_2b}.

\begin{figure}[htb]
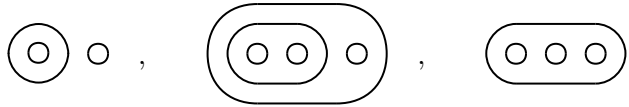

\begin{center}

\caption{The elements $\omega^2(1)\omega(1),$ $\omega(\omega(\omega(1)^2)\omega(1)), $ and $\omega^2(\omega(1)^3)$ of $Z(\mcA)$, cf. Figure \ref{fig2_4_2} where the same diagrams are interpreted in $\wU_0^0$.
}
\label{fig2_4_2b}
\end{center}
\end{figure}

The above diagrammatics  can  be enhanced by choosing  a subset  $S$ of $Z(\mcA)$ (perhaps a set of  generators  or just all elements of $Z(\mcA)$) and adding dots  in the regions of  the plane labelled by elements of $S$. A dot can float in its region but cannot jump into  another  region. Wrapping a circle around a dot $s\in S$ results in the  element  $\omega(s)$, see examples  in Figure~\ref{fig2_4_3}. If  $S$ is multiplicative, we can add the product relation for two dots in  the same region, merging dots labelled by $s_1,s_2$ into a single dot labelled $s_1s_2$, see Figure~\ref{fig2_4_3}.

\begin{figure}
\begin{center}
     \begin{subfigure}[htb]{0.3\textwidth}
      \centering
\begingroup%
  \makeatletter%
  \providecommand\color[2][]{%
    \errmessage{(Inkscape) Color is used for the text in Inkscape, but the package 'color.sty' is not loaded}%
    \renewcommand\color[2][]{}%
  }%
  \providecommand\transparent[1]{%
    \errmessage{(Inkscape) Transparency is used (non-zero) for the text in Inkscape, but the package 'transparent.sty' is not loaded}%
    \renewcommand\transparent[1]{}%
  }%
  \providecommand\rotatebox[2]{#2}%
  \newcommand*\fsize{\dimexpr\f@size pt\relax}%
  \newcommand*\lineheight[1]{\fontsize{\fsize}{#1\fsize}\selectfont}%
  \ifx\svgwidth\undefined%
    \setlength{\unitlength}{23.2872651bp}%
    \ifx\svgscale\undefined%
      \relax%
    \else%
      \setlength{\unitlength}{\unitlength * \real{\svgscale}}%
    \fi%
  \else%
    \setlength{\unitlength}{\svgwidth}%
  \fi%
  \global\let\svgwidth\undefined%
  \global\let\svgscale\undefined%
  \makeatother%
  \begin{picture}(1,0.97353991)%
    \lineheight{1}%
    \setlength\tabcolsep{0pt}%
    \put(0,0){\includegraphics[width=\unitlength,page=1]{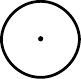}}%
    \put(0.59473074,0.33754141){\color[rgb]{0,0,0}\makebox(0,0)[lt]{\smash{\begin{tabular}[t]{l}$s$\end{tabular}}}}%
    \put(0,0){\includegraphics[width=\unitlength,page=2]{omegas.pdf}}%
  \end{picture}%
\endgroup%

         \caption{$\omega(s)$}
         \label{fig:omegas3}
     \end{subfigure}
     \begin{subfigure}[htb]{0.35\textwidth}
        \centering
\begingroup%
  \makeatletter%
  \providecommand\color[2][]{%
    \errmessage{(Inkscape) Color is used for the text in Inkscape, but the package 'color.sty' is not loaded}%
    \renewcommand\color[2][]{}%
  }%
  \providecommand\transparent[1]{%
    \errmessage{(Inkscape) Transparency is used (non-zero) for the text in Inkscape, but the package 'transparent.sty' is not loaded}%
    \renewcommand\transparent[1]{}%
  }%
  \providecommand\rotatebox[2]{#2}%
  \newcommand*\fsize{\dimexpr\f@size pt\relax}%
  \newcommand*\lineheight[1]{\fontsize{\fsize}{#1\fsize}\selectfont}%
  \ifx\svgwidth\undefined%
    \setlength{\unitlength}{92.43258529bp}%
    \ifx\svgscale\undefined%
      \relax%
    \else%
      \setlength{\unitlength}{\unitlength * \real{\svgscale}}%
    \fi%
  \else%
    \setlength{\unitlength}{\svgwidth}%
  \fi%
  \global\let\svgwidth\undefined%
  \global\let\svgscale\undefined%
  \makeatother%
  \begin{picture}(1,0.73837597)%
    \lineheight{1}%
    \setlength\tabcolsep{0pt}%
    \put(0,0){\includegraphics[width=\unitlength,page=1]{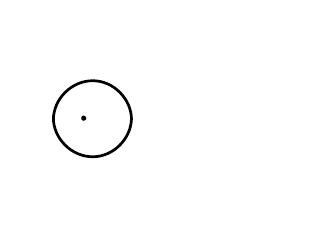}}%
    \put(0.28115694,0.34528775){\color[rgb]{0,0,0}\makebox(0,0)[lt]{\smash{\begin{tabular}[t]{l}$s_1$\end{tabular}}}}%
    \put(0,0){\includegraphics[width=\unitlength,page=2]{omegas2.pdf}}%
    \put(0.51176108,0.34368673){\color[rgb]{0,0,0}\makebox(0,0)[lt]{\smash{\begin{tabular}[t]{l}$s_2$\end{tabular}}}}%
    \put(0,0){\includegraphics[width=\unitlength,page=3]{omegas2.pdf}}%
    \put(0.87674717,0.34165822){\color[rgb]{0,0,0}\makebox(0,0)[lt]{\smash{\begin{tabular}[t]{l}$s_3$\end{tabular}}}}%
    \put(0,0){\includegraphics[width=\unitlength,page=4]{omegas2.pdf}}%
  \end{picture}%
\endgroup%
 
         \caption{$\omega^2(\omega(s_1)s_2)\omega(1)s_3$}
         \label{fig:omegas4}
     \end{subfigure}
     \begin{subfigure}[htb]{0.3\textwidth}
        \centering
\begingroup%
  \makeatletter%
  \providecommand\color[2][]{%
    \errmessage{(Inkscape) Color is used for the text in Inkscape, but the package 'color.sty' is not loaded}%
    \renewcommand\color[2][]{}%
  }%
  \providecommand\transparent[1]{%
    \errmessage{(Inkscape) Transparency is used (non-zero) for the text in Inkscape, but the package 'transparent.sty' is not loaded}%
    \renewcommand\transparent[1]{}%
  }%
  \providecommand\rotatebox[2]{#2}%
  \newcommand*\fsize{\dimexpr\f@size pt\relax}%
  \newcommand*\lineheight[1]{\fontsize{\fsize}{#1\fsize}\selectfont}%
  \ifx\svgwidth\undefined%
    \setlength{\unitlength}{70.481436bp}%
    \ifx\svgscale\undefined%
      \relax%
    \else%
      \setlength{\unitlength}{\unitlength * \real{\svgscale}}%
    \fi%
  \else%
    \setlength{\unitlength}{\svgwidth}%
  \fi%
  \global\let\svgwidth\undefined%
  \global\let\svgscale\undefined%
  \makeatother%
  \begin{picture}(1,0.07807008)%
    \lineheight{1}%
    \setlength\tabcolsep{0pt}%
    \put(0,0){\includegraphics[width=\unitlength,page=1]{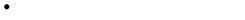}}%
    \put(0.05429194,0.01775495){\color[rgb]{0,0,0}\makebox(0,0)[lt]{\smash{\begin{tabular}[t]{l}$s_1$\end{tabular}}}}%
    \put(0,0){\includegraphics[width=\unitlength,page=2]{s1s2.pdf}}%
    \put(0.26407364,0.01338451){\color[rgb]{0,0,0}\makebox(0,0)[lt]{\smash{\begin{tabular}[t]{l}$s_2$\end{tabular}}}}%
    \put(0.50092476,0.02890012){\color[rgb]{0,0,0}\makebox(0,0)[lt]{\smash{\begin{tabular}[t]{l}$=$\end{tabular}}}}%
    \put(0,0){\includegraphics[width=\unitlength,page=3]{s1s2.pdf}}%
    \put(0.74425336,0.01528466){\color[rgb]{0,0,0}\makebox(0,0)[lt]{\smash{\begin{tabular}[t]{l}$s_1s_2$\end{tabular}}}}%
    \put(0,0){\includegraphics[width=\unitlength,page=4]{s1s2.pdf}}%
  \end{picture}%
\endgroup%

         \caption{Merging dots $s_1s_2$}
         \label{fig:mu3}
     \end{subfigure}
\caption{Examples of enhanced circle diagrams with labels from $S\subseteq Z(\A)$.  If the set $S$ is multiplicative, dots can be merged, as in (c).}
\label{fig2_4_3}
\end{center}
\end{figure}

For a specific category $\mcA$ and a  self-adjoint functor  $F$ these diagrammatics for elements of the center  $Z(\mcA)$ may have additional relations that will depend on the choice $(\mcA,F)$. We will discuss the formal construction of monoidal categories of such diagrams in the next subsection.

\vspace{0.1in} 

Assume now that the category $\mcA$ and  the functor $F$ are  $R$-linear, for a commutative ring  $R$. This means  that morphism spaces in  $\mcA$ are $R$-modules, composition of maps is $R$-bilinear, and $F$ respects the $R$-linear structure of $\mcA$. 

Then the center $Z(\mcA)$ is a commutative $R$-algebra, and the endomorphism  $\omega$ of the center is an $R$-linear map. Usually, $\omega$ does  not commute with the multiplication in $Z=Z(\mcA)$, that is, $\omega(ab)\not= \omega(a)\omega(b), a,b\in Z$. 

\begin{definition}\label{def_circular}
We call  a  triple $(R,Z,\omega)$ of a commutative ring $R$,  a unital commutative $R$-algebra $Z$ and  an $R$-linear map $\omega\colon Z\lra Z$ a \emph{circular triple}. 

A circular triple is \email{$\omega$-generated} if $Z$ is the only $R$-subalgebra of $Z$ that contains $1$ and is closed under $\omega$.
\end{definition}

With this terminology, an $R$-linear category $\A$ with a self-adjunction $(F,\delta,\mu)$ gives a circular triple $(R,Z(\A),\omega)$.

\vspace{0.1in} 

Taking any commutative  algebra $Z$ and  a linear  map $\omega$ on it gives a large number of examples of circular triples. 
The map $\omega$ may have special properties, such as be a derivation, $\omega(ab)=\omega(a)b+a\omega(b)$, or, more generally, a differential operator, or just any linear  map on a commutative algebra (see Section \ref{sec:derdiags}). Another interesting example is the Frobenius endomorphism $\sigma$ of a  commutative ring $A$ over a characteristic $p$ field, see the end of Section~\ref{sec_adj_ex}.

 
 \subsection{The skein category \texorpdfstring{$\wSU_{\omega}$}{SUZ} of a circular triple}\label{sec-BZ}
 
 \quad
 \smallskip

Given a circular triple $(R,Z,\omega)$, we now construct a monoidal category $\wSU_{Z,\omega}$, also denoted $\wSU_\omega$, for short, and a category $\wU_Z$ which does not depend on $\omega$.

\ssubsection{The construction of $\wSU_{Z,\omega}$}
 Earlier, in \eqref{UBforget}, we considered the forgetful functor $\wU\lra \wB$ that removes  circles from diagrams  in $\wU^m_n$  producing diagrams in $\wB^m_n$. Given a circular  triple $(R,Z,\omega)$, instead of removing circles, we can evaluate them to elements in $Z$ to construct the \emph{skein category}  $\wSU_\omega=\wSU_{Z,\omega}$. 
 
 The objects of the category $\wSU_{Z,\omega}$ are  non-negative integers $n\in\Z_+$ and the space of morphism  from $n$ to $m$ is given by all finite $R$-linear combinations of  diagrams  of  crossingless matchings $b\in \wB^m_n$, now enhanced  by allowing elements of $Z$ to float in regions of $\wB^m_n$. Examples of diagrams in $\wSU_{Z,\omega}$ are given in Figure \ref{figBZ1}. Recall that $b$ has $\frac{n+m}{2}+1$ regions into which it separates the strip.  We impose the following rules on these diagrams: 
 \begin{itemize}
     \item Elements $z_1,z_2$  floating in the same region can merge into a single element $z_1z_2.$
     \item An element $rz$, with  $r\in R$, floating  in the  region can be changed by removing $r$ from the  region  and  multiplying   the coefficient of  the diagram by $r.$
     \item A diagram which has  a sum $z_1+z_2$ in a region equals the sum of the corresponding diagrams with $z_1$  and $z_2$ in that region. 
 \end{itemize}
 In this  way, the  space of diagrams for a given crossingless matching $b$ can be identified with  the tensor power $Z^{\otimes_R k}$, for $k=\frac{n+m}{2}+1$, one copy of $Z$ for each  region of $b$. 
 
 \begin{figure}
\begin{center}
     \begin{subfigure}[htb]{0.4\textwidth}
      \centering
\begingroup%
  \makeatletter%
  \providecommand\color[2][]{%
    \errmessage{(Inkscape) Color is used for the text in Inkscape, but the package 'color.sty' is not loaded}%
    \renewcommand\color[2][]{}%
  }%
  \providecommand\transparent[1]{%
    \errmessage{(Inkscape) Transparency is used (non-zero) for the text in Inkscape, but the package 'transparent.sty' is not loaded}%
    \renewcommand\transparent[1]{}%
  }%
  \providecommand\rotatebox[2]{#2}%
  \newcommand*\fsize{\dimexpr\f@size pt\relax}%
  \newcommand*\lineheight[1]{\fontsize{\fsize}{#1\fsize}\selectfont}%
  \ifx\svgwidth\undefined%
    \setlength{\unitlength}{97.49999055bp}%
    \ifx\svgscale\undefined%
      \relax%
    \else%
      \setlength{\unitlength}{\unitlength * \real{\svgscale}}%
    \fi%
  \else%
    \setlength{\unitlength}{\svgwidth}%
  \fi%
  \global\let\svgwidth\undefined%
  \global\let\svgscale\undefined%
  \makeatother%
  \begin{picture}(1,0.39232052)%
    \lineheight{1}%
    \setlength\tabcolsep{0pt}%
    \put(0,0){\includegraphics[width=\unitlength,page=1]{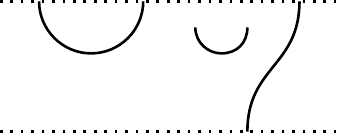}}%
    \put(0.42307682,0.15769867){\color[rgb]{0,0,0}\makebox(0,0)[lt]{\smash{\begin{tabular}[t]{l}$z_4$\end{tabular}}}}%
    \put(0.61401092,0.31154498){\color[rgb]{0,0,0}\makebox(0,0)[lt]{\smash{\begin{tabular}[t]{l}$z_5$\end{tabular}}}}%
    \put(0,0){\includegraphics[width=\unitlength,page=2]{Bzdiag2.pdf}}%
    \put(0.7692307,0.04231404){\color[rgb]{0,0,0}\makebox(0,0)[lt]{\smash{\begin{tabular}[t]{l}$z_6$\end{tabular}}}}%
  \end{picture}%
\endgroup%

         \caption{A morphism $u$ from $1$ to $5$.}
         \label{fig:omegas}
     \end{subfigure}
     \begin{subfigure}[htb]{0.4\textwidth}
        \centering
\begingroup%
  \makeatletter%
  \providecommand\color[2][]{%
    \errmessage{(Inkscape) Color is used for the text in Inkscape, but the package 'color.sty' is not loaded}%
    \renewcommand\color[2][]{}%
  }%
  \providecommand\transparent[1]{%
    \errmessage{(Inkscape) Transparency is used (non-zero) for the text in Inkscape, but the package 'transparent.sty' is not loaded}%
    \renewcommand\transparent[1]{}%
  }%
  \providecommand\rotatebox[2]{#2}%
  \newcommand*\fsize{\dimexpr\f@size pt\relax}%
  \newcommand*\lineheight[1]{\fontsize{\fsize}{#1\fsize}\selectfont}%
  \ifx\svgwidth\undefined%
    \setlength{\unitlength}{97.49999906bp}%
    \ifx\svgscale\undefined%
      \relax%
    \else%
      \setlength{\unitlength}{\unitlength * \real{\svgscale}}%
    \fi%
  \else%
    \setlength{\unitlength}{\svgwidth}%
  \fi%
  \global\let\svgwidth\undefined%
  \global\let\svgscale\undefined%
  \makeatother%
  \begin{picture}(1,0.39232049)%
    \lineheight{1}%
    \setlength\tabcolsep{0pt}%
    \put(0,0){\includegraphics[width=\unitlength,page=1]{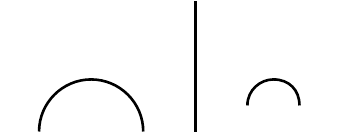}}%
    \put(0.23076893,0.0423138){\color[rgb]{0,0,0}\makebox(0,0)[lt]{\smash{\begin{tabular}[t]{l}$z_1$\end{tabular}}}}%
    \put(0.76923077,0.19526487){\color[rgb]{0,0,0}\makebox(0,0)[lt]{\smash{\begin{tabular}[t]{l}$z_2$\end{tabular}}}}%
    \put(0.76670744,0.04141768){\color[rgb]{0,0,0}\makebox(0,0)[lt]{\smash{\begin{tabular}[t]{l}$z_3$\end{tabular}}}}%
    \put(0,0){\includegraphics[width=\unitlength,page=2]{Bzdiag1.pdf}}%
  \end{picture}%
\endgroup%
 
         \caption{A morphism $v$ from $5$ to $3$.}
         \label{fig:omegas2}
     \end{subfigure}
\caption{Examples of morphisms in $\wSU_\omega$, written as crossingless matchings with the regions labelled by elements $z_i\in Z$. Linear combinations of diagrams are allowed as well.}
\label{figBZ1}
\end{center}
\end{figure}

Now, the space of morphisms from $n$ to  $m$ in $\wSU_{Z,\omega}$ can be identified with the direct sum  of $c_{\frac{n+m}{2}}$ copies of  that  tensor power  $Z^{\otimes_R k}$, one for  each  crossingless matching $b\in  \wB^m_n$. Here  $c_r$ denotes the $r$-th Catalan number.   
 
To define composition in $\wSU_{Z,\omega}$ we compose  diagrams $a$ and  $b$ representing morphisms from  $m$ to $k$ and from $n$  to $m$, respectively. 
We then inductively simplify the composed  diagram $ab$. If there  is  a circle in $ab$ that  wraps around element $z\in Z$, we remove the circle and its interior and place  $\omega(z)$ in that place of the diagram instead.  Starting with the innermost  circles, we can eventually remove  all  circles from $ab$. If the interior of the  circle is empty, we replace it  with $\omega(1)\in Z$. This composition rule is extended to $R$-linear combinations. An example of a composition of diagrams in $\wSU_{Z,\omega}$ can be found in Figure \ref{figBZ2}.
For simplicity, we often write 
\begin{align}
\wSU_{\omega}=\wSU_{Z,\omega}.
\end{align}

 \begin{figure}[htb]
\begin{center}
    \import{Graphics/}{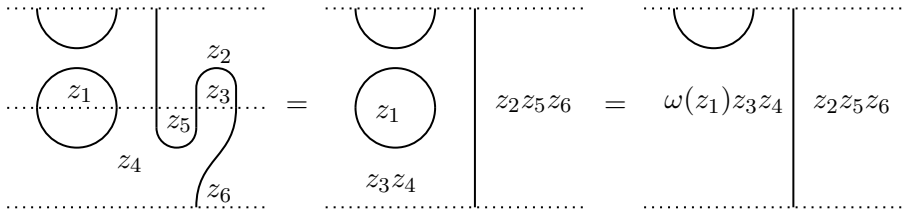} 
\caption{The composition $vu$ of the diagrams from Figure \ref{figBZ1} in $\wSU_\omega$.}
\label{figBZ2}
\end{center}
\end{figure}

\ssubsection{The monoidal structure on $\wSU_\omega$.}
Morphism spaces $\wSU_{\omega,n}^m:=\Hom_{\wSU_\omega}(n,m)$ in $\wSU_\omega$ carry a natural structure of $Z\otimes Z$-modules by placing an element $z\in Z$ into the unique leftmost, respectively, rightmost region. This structure is compatible with composition.

We thus obtain a tensor product 
$$\otimes \colon \wSU_{\omega,n}^m\times \wSU_{\omega,r}^s\longrightarrow \wSU_{\omega,n+r}^{m+s},\quad (u, v)\longmapsto u\otimes_Z v.$$
which is defined using the horizontal multiplication (tensor product) of the underlying elements in $\wB$, which joins the rightmost region of $u$ and the leftmost region of $v$ in $u\otimes v$. In addition, all region labels from $Z$ are kept, and the ones in the middle region are multiplied. An example of a tensor product is given in Figure \ref{figBZ3}. This tensor product turns $\wSU_\omega$ into an $R$-linear monoidal category, which is not, in general, symmetric or braided.

\begin{figure}
\begin{center}
     \begin{subfigure}[htb]{0.28\textwidth}
      \centering
\begingroup%
  \makeatletter%
  \providecommand\color[2][]{%
    \errmessage{(Inkscape) Color is used for the text in Inkscape, but the package 'color.sty' is not loaded}%
    \renewcommand\color[2][]{}%
  }%
  \providecommand\transparent[1]{%
    \errmessage{(Inkscape) Transparency is used (non-zero) for the text in Inkscape, but the package 'transparent.sty' is not loaded}%
    \renewcommand\transparent[1]{}%
  }%
  \providecommand\rotatebox[2]{#2}%
  \newcommand*\fsize{\dimexpr\f@size pt\relax}%
  \newcommand*\lineheight[1]{\fontsize{\fsize}{#1\fsize}\selectfont}%
  \ifx\svgwidth\undefined%
    \setlength{\unitlength}{56.27424795bp}%
    \ifx\svgscale\undefined%
      \relax%
    \else%
      \setlength{\unitlength}{\unitlength * \real{\svgscale}}%
    \fi%
  \else%
    \setlength{\unitlength}{\svgwidth}%
  \fi%
  \global\let\svgwidth\undefined%
  \global\let\svgscale\undefined%
  \makeatother%
  \begin{picture}(1,0.67972916)%
    \lineheight{1}%
    \setlength\tabcolsep{0pt}%
    \put(0,0){\includegraphics[width=\unitlength,page=1]{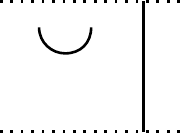}}%
    \put(0.19991382,0.23276816){\color[rgb]{0,0,0}\makebox(0,0)[lt]{\smash{\begin{tabular}[t]{l}$z_2$\end{tabular}}}}%
    \put(0.26417162,0.53977857){\color[rgb]{0,0,0}\makebox(0,0)[lt]{\smash{\begin{tabular}[t]{l}$z_1$\end{tabular}}}}%
    \put(0,0){\includegraphics[width=\unitlength,page=2]{Bzdiag5.pdf}}%
    \put(0.79727489,0.24228747){\color[rgb]{0,0,0}\makebox(0,0)[lt]{\smash{\begin{tabular}[t]{l}$z_3$\end{tabular}}}}%
  \end{picture}%
\endgroup%

         \caption{A morphism $u_1\colon 1\to 3$.}
         \label{fig:BZ5}
     \end{subfigure}
     \begin{subfigure}[htb]{0.28\textwidth}
      \centering
\begingroup%
  \makeatletter%
  \providecommand\color[2][]{%
    \errmessage{(Inkscape) Color is used for the text in Inkscape, but the package 'color.sty' is not loaded}%
    \renewcommand\color[2][]{}%
  }%
  \providecommand\transparent[1]{%
    \errmessage{(Inkscape) Transparency is used (non-zero) for the text in Inkscape, but the package 'transparent.sty' is not loaded}%
    \renewcommand\transparent[1]{}%
  }%
  \providecommand\rotatebox[2]{#2}%
  \newcommand*\fsize{\dimexpr\f@size pt\relax}%
  \newcommand*\lineheight[1]{\fontsize{\fsize}{#1\fsize}\selectfont}%
  \ifx\svgwidth\undefined%
    \setlength{\unitlength}{67.5000085bp}%
    \ifx\svgscale\undefined%
      \relax%
    \else%
      \setlength{\unitlength}{\unitlength * \real{\svgscale}}%
    \fi%
  \else%
    \setlength{\unitlength}{\svgwidth}%
  \fi%
  \global\let\svgwidth\undefined%
  \global\let\svgscale\undefined%
  \makeatother%
  \begin{picture}(1,0.56668507)%
    \lineheight{1}%
    \setlength\tabcolsep{0pt}%
    \put(0,0){\includegraphics[width=\unitlength,page=1]{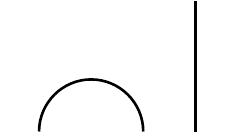}}%
    \put(0.33333285,0.06111992){\color[rgb]{0,0,0}\makebox(0,0)[lt]{\smash{\begin{tabular}[t]{l}$z_4$\end{tabular}}}}%
    \put(0.61111102,0.25228713){\color[rgb]{0,0,0}\makebox(0,0)[lt]{\smash{\begin{tabular}[t]{l}$z_5$\end{tabular}}}}%
    \put(0,0){\includegraphics[width=\unitlength,page=2]{Bzdiag4.pdf}}%
  \end{picture}%
\endgroup%

         \caption{A morphism $u_2\colon 3\to 3$.}
         \label{fig:BZ4}
     \end{subfigure}
          \begin{subfigure}[htb]{0.4\textwidth}
      \centering
\begingroup%
  \makeatletter%
  \providecommand\color[2][]{%
    \errmessage{(Inkscape) Color is used for the text in Inkscape, but the package 'color.sty' is not loaded}%
    \renewcommand\color[2][]{}%
  }%
  \providecommand\transparent[1]{%
    \errmessage{(Inkscape) Transparency is used (non-zero) for the text in Inkscape, but the package 'transparent.sty' is not loaded}%
    \renewcommand\transparent[1]{}%
  }%
  \providecommand\rotatebox[2]{#2}%
  \newcommand*\fsize{\dimexpr\f@size pt\relax}%
  \newcommand*\lineheight[1]{\fontsize{\fsize}{#1\fsize}\selectfont}%
  \ifx\svgwidth\undefined%
    \setlength{\unitlength}{116.25000945bp}%
    \ifx\svgscale\undefined%
      \relax%
    \else%
      \setlength{\unitlength}{\unitlength * \real{\svgscale}}%
    \fi%
  \else%
    \setlength{\unitlength}{\svgwidth}%
  \fi%
  \global\let\svgwidth\undefined%
  \global\let\svgscale\undefined%
  \makeatother%
  \begin{picture}(1,0.32904296)%
    \lineheight{1}%
    \setlength\tabcolsep{0pt}%
    \put(0,0){\includegraphics[width=\unitlength,page=1]{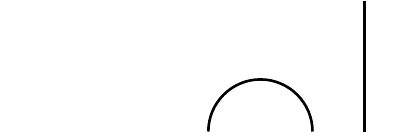}}%
    \put(0.61290293,0.03548899){\color[rgb]{0,0,0}\makebox(0,0)[lt]{\smash{\begin{tabular}[t]{l}$z_4$\end{tabular}}}}%
    \put(0,0){\includegraphics[width=\unitlength,page=2]{Bzdiag6.pdf}}%
    \put(0.09677428,0.11267773){\color[rgb]{0,0,0}\makebox(0,0)[lt]{\smash{\begin{tabular}[t]{l}$z_2$\end{tabular}}}}%
    \put(0.12788007,0.26129517){\color[rgb]{0,0,0}\makebox(0,0)[lt]{\smash{\begin{tabular}[t]{l}$z_1$\end{tabular}}}}%
    \put(0,0){\includegraphics[width=\unitlength,page=3]{Bzdiag6.pdf}}%
    \put(0.38594444,0.1322628){\color[rgb]{0,0,0}\makebox(0,0)[lt]{\smash{\begin{tabular}[t]{l}$z_3z_5$\end{tabular}}}}%
  \end{picture}%
\endgroup%

         \caption{The tensor product $u_1\otimes u_2$.}
         \label{fig:BZ6}
     \end{subfigure}
\caption{The tensor product of two morphisms in $\wSU_\omega$.}
\label{figBZ3}
\end{center}
\end{figure}

Now assume that $Z$  is a free $R$-module and a basis $B$ of $Z$ is chosen. Recall that as an $R$-module, $\Hom_{\wSU_\omega}(n,m)$ is isomorphic to a direct sum of $c_{k-1}$ copies of $Z^{\otimes k}$, for $k=\frac{n+m}{2}+1$. This way, we obtain an $R$-basis for $\Hom_{\wSU_\omega}(n,m)$ consisting of all diagrams $u$ in $\wB_n^m$ with regions labelled by elements of the basis $B$. 

\begin{example}\label{ex:UasBZ}
\begin{enumerate}
\item[(1)]
We may set $Z=R\wU_0^0$, with multiplication given by taking disjoint union of diagrams, and map $\omega$ be the operation of wrapping a circle around a diagram. This results in the monoidal category $R\wU$, which is the $R$-linearization of the monoidal category $\wU$ defined in Section $\ref{sec-wU}$, as a special case of $\wSU_\omega$.
\item[(2)]
For the circular triple $(R,R,\omega=\ide_R)$ we recover $\wSU_\omega=R\wB$, the $R$-linearization of the category $\wB$ of crossingless matchings. 
\end{enumerate}
\end{example}

\begin{example}\label{ex:TL}
The Temperley--Lieb category $TL(d)$ is the skein category $\wSU_\omega$ for  $Z=R$ and $\omega(1)=d\in R$. Note that $TL(1)=R\wB$.
\end{example}

We observe that the category $\wSU_\omega$ is rigid and every object is isomorphic to its left and right dual. The duality morphisms are the same as those in $\wB$, i.e. labelled with the $1\in Z$ in all regions. As a monoidal category, $\wSU_\omega$ is generated by the object $1$. Morphisms are generated over $R$ by the following elementary morphisms: 
\begin{itemize}
    \item The cap and cup morphisms from $2$ to $0$, respectively, $0$ to $2$.
    \item Endomorphisms of $0$ corresponding to $R$-algebra generators of $Z$.
\end{itemize}
These morphisms are subject to the following relations
\begin{itemize}
    \item Identities of the cap and cup morphisms displaying $1$ as a self-dual object.
    \item Compatiblity of composition of  morphisms coming from elements of $Z$ with multiplication in $Z$.
    \item Topological relations on the underlying diagrams of crossingless matchings in $\wB$.
    \item Moving labels $z\in Z$ within the regions.
\end{itemize}

\ssubsection{The monoidal category $\wU_Z$}

Assume given a commutative $R$-algebra $Z$. Consider diagrams in $\wU^0_{0}$ with additional labels from elements of $Z$. This defines a ring $\wU^0_{Z,0}$, similarly to $\wSU_{\omega,0}^0$ where the difference is that circles are \emph{not} evaluated using $\omega$. Define the monoidal category
\begin{equation}\label{wUZ}
\wU_Z:=\wSU_{(\wU_{Z,0}^0,\,\omega)},
\end{equation}
confer Example \ref{ex:UasBZ}. 
Thus, $\wU_Z$ has the same objects as $\wU$ and morphisms are given by morphisms in $\wU$ together with labels from $Z$ in all regions. 

Given a circular triple $(R,Z,\omega)$, any diagram in $\wU^0_{Z,0}$ can be evaluated to an element in $Z$. Here, the circle wrapping around a label $z\in Z$ is mapped to $\omega(z)$. This definition extends inductively, by sending disjoint unions of diagrams to products in $Z$, to a morphism of $R$-algebras
\begin{align}\label{morphismFZ}
    \cF_Z\colon \wU_{Z,0}^0\longrightarrow Z.
\end{align}
More generally, we obtain a diagram of  monoidal $R$-linear functors 
\begin{equation}\wU\hookrightarrow \wU_Z\stackrel{\cF_Z}{\twoheadrightarrow} \wSU_\omega.\label{functorsUZ}
\end{equation}
Here, the second functor $\cF_Z$ is obtained from applying the morphism $\cF_Z$ to the labels of regions, which contain elements of $\wU_{Z,0}^0$. See also the more general discussion in Section \ref{sec:circadj}.


\subsection{Universality of \texorpdfstring{$R\wU_0^0$}{RU} and \texorpdfstring{$\omega$}{w}-evaluation}\label{sec_universal}
\quad
\smallskip

\ssubsection{Universality of $R\wU_0^0$}
The $R$-algebra $R\wU_0^0$ is universal among algebras with a distinguished $R$-linear morphism $\omega$. In fact, given any choice of a circular triple $(R,Z,\omega)$, consider the restriction of the morphism $\cF_Z$ from Equation \eqref{morphismFZ} to the submonoid $\wU_0^0$. This gives a map
\begin{align}
    \mcF_Z \ \colon  \  \wU^0_0 \lra Z 
\end{align}
which takes the  empty diagram to $1$, intertwines  the  action of $\omega$ on $Z$ and  $\wU^0_0$ 
and takes the disjoint union of  diagram to the product  of  corresponding elements, i.e.,
\begin{eqnarray*}
    \mcF_Z(ab) & = &  \mcF_Z(a)\mcF_Z(b), \ \  a,b\in \wU^0_0, \\
    \mcF_Z(\emptyset) &= & 1 , \\
    \mcF_Z(\omega(a)) & =&  \omega(\mcF_Z(a)). 
\end{eqnarray*}
The $R$-linear extension of $\mcF_Z$ is a unital homomorphism  of commutative $R$-algebras
\begin{equation}\label{eq_R_Z} 
\cF_Z\colon R\wU^0_0\lra  Z
\end{equation}  
that intertwines the action of $R$-linear map $\omega$ on both algebras. We refer to this homomorphism as \emph{$\omega$-evaluation}. We say  that $Z$ is \emph{$\omega$-generated} if the
map (\ref{eq_R_Z}) is surjective. Equivalently,    
the smallest $R$-subalgebra of  $Z$ that  contains the unit element $1\in Z$ and is closed under $\omega$ equals $Z$.

\vspace{0.1in}

To summarize, we have the following result.

\begin{proposition}\label{prop_universal}
The commutative monoid ($R$-algebra)  $\wU_0^0$ (respectively, $R\wU_0^0$) is initial among commutative monoids with a distinguished map (respectively, a distingished $R$-linear map).
\end{proposition}


\subsection{An adjunction of circular triples and monoidal categories with a self-adjoint endofunctor}\label{sec:circadj}
\quad
\smallskip

Expanding on the previous subsection, we now explain how the constructions from Section~\ref{sec-BZ} are functorial and how $\wSU_\omega$ can be regarded as a free monoidal category with a prescribed circular triple datum recoverable from its endomorphism ring of the tensor unit. 

\vspace{0.1in}

We define $\CTrip_R$ to be the \emph{category of circular triples} over $R$. That is, objects are circular triples $(R,Z,\omega)$, often just displayed as the pair $(Z,\omega)$. A morphism $\phi\colon (Z_1,\omega_1)\to (Z_2,\omega_2)$ in $\CTrip_R$ is given by  a homomorphism of $R$-algebras $\phi \colon Z_1\to Z_2$ that intertwines the maps $\omega_1$, $\omega_2$, i.e.,  such that the diagram
\begin{align}
    \xymatrix{Z_1\ar[d]^{\phi}\ar[rr]^{\omega_1}&&Z_1\ar[d]^{\phi}\\
    Z_2\ar[rr]^{\omega_2}&&Z_2
    }
\end{align}
commutes. 

\vspace{0.1in}

We have already seen examples of morphisms of circular triples. For example, the map $\cF_Z\colon R\wU_0^0\to Z$ from \eqref{eq_R_Z} is a map of circular triples. It extends to a map of circular triples
$$\wU_{Z,0}^0\longrightarrow Z$$
that sends a dot labelled with an element $z\in Z$ to the corresponding element of $Z$. 
\vspace{0.1in}

We observe that a morphism $\phi\colon Z_1\to Z_2$ in $\CTrip_R$ induces a monoidal functor 
\begin{align}\label{eq:SUphi}
\wSU_\phi\colon \wSU_{\omega_1}\longrightarrow \wSU_{\omega_2}
\end{align}
defined by applying $\phi$ to all labels. For example, the morphism $R\wU_0^0\to R$ that sends any closed circle diagram to $1\in R$ induces the forgetful functor $\wU\to \wB$ from earlier.

The assignment 
$$(Z,\omega)\longmapsto \wSU_\omega, \qquad \phi \longmapsto \wSU_\phi$$
obtained this way defines a functor from the category of circular triples (with underlying ring $R$) to $R$-linear monoidal categories (i.e., monoidal categories enriched in $R$-modules). We denote this functor by $\wSU_{(-)}$. 

We also consider the assignment $Z'(-)$ which associates to an $R$-linear monoidal category with a self-adjoint functor the endomorphism ring $Z'(\A):=\End_\A(\one)$ of the tensor unit $\one$.  Note that $Z'(\A)$ embeds into $Z(\A)$ by sending $z$ to $z\Ide_A$. Hence, to the category $\A$ we associate the circular triple $(R,Z'(\A), \omega_F)$, where $\omega_F$ is the restriction of the wrapping operation defined in \eqref{eqn_wrapping}.

Since the category $\wSU_\omega$ has self-dual objects, we observe that the endo-functor $F_1$ of tensoring by the generating object $1$, i.e. 
$$F_1\colon \wSU_\omega\longrightarrow \wSU_\omega, \quad n\mapsto n\otimes 1$$
is a self-adjoint functor. The functor $\wSU_{(-)}$ can now be enhanced to a functor to the category whose objects are pairs of $R$-linear monoidal categories together with a self-adjoint endofunctor. Morphisms in this category are $R$-linear monoidal  functors $G\colon \A\to \B$ that intertwine the corresponding self-adjoint functors in the sense that
$$G \circ F_\A = F_\B\circ G,\qquad G \mu^{F_\A}=\mu^{F_\B}, \qquad G\delta^{F_\A}=\delta^{F_\B},$$
where $(F_\A,\mu^{F_\A},\delta^{F_\A})$ (or, $(F_\B,\mu^{F_\B},\delta^{F_\B})$) is the self-adjoint endofunctor on $\A$ (respectively, $\B$). Note that the self-adjoint functors $F_\A, F_\B$ are not required to be monoidal functors.

The assignment $Z'(\A)$ extends to a functor from $R$-linear monoidal categories with a self-adjoint functor to $\CTrip_R$. The functors required in this category are compatible with the self-adjunctions as $G$ above and come with a choice of isomorphism $G(\one)\to \one$ but do not need to be monoidal functors. Such functors commute with the wrapping action. We show that this functor $Z'$ is right adjoint to the functor $\wSU_{(-)}$.



\begin{proposition}\label{alltriplescenters}
The functor $\wSU_{(-)}$ is left adjoint to the functor $Z'(-)$ that associates to an $R$-linear monoidal category with a self-adjoint functor the circular triple $(R,Z'(\A),\omega)$. The unit natural transformation consists of isomorphisms $(Z,\omega)\to (Z'(\wSU_\omega),\omega_{F_1})$ in $\CTrip_R$.
\end{proposition}
\begin{proof}
Let $\A$ be an $R$-linear monoidal category with a self-adjoint endofunctor $F$, and $(Z,\omega)=(R,Z,\omega)$ a circular triple.  
Consider the natural transformation 
$$\eta_{(Z,\omega)}\colon (Z,\omega)\lra (Z'(\wSU_\omega),\omega_F)$$
given by sending $z$ to the empty crossingless matching with only region labelled by $z$. This clearly defines an element of $Z'(\wSU_\omega)=Z$ and commutes with the wrapping maps $\omega,\omega_F$. By  construction of $\wSU_\omega$, it follows that $\eta_{(Z,\omega)}$ defines a natural isomorphism.

We further construct a natural transformation $\epsilon\colon \wSU_{Z'(-)}\to (-).$
That is, for each pair $(\A,F)$, where $\A$ is an $R$-linear monoidal category with a self-adjoint $R$-linear functor $F$, we construct an $R$-linear functor  $$\epsilon_{\A,F} \colon \wSU_{Z'(\A),\omega_F}\to \A.$$ 
This functor sends the object $n$ of $\wSU_{\omega_F}$ to $F^n(\one)$. The cap morphism $2\to 0$ is sent to the morphism $\mu_\one\colon F^2(\one)\to \one$ obtained from self-adjointness of $F$, and the cup morphism $0\to 2$ is sent to $\delta_\one\colon \one\to F^2(\one)$. An endomorphism of $0$ labelled by $z\in Z'(\A)$ is sent to the corresponding element of $z\in \End_\A(\one)$. Since $\wSU_{\omega_F}$ is generated by these morphisms as an $R$-linear monoidal category we can inductively extend this assignment to the entire category $\wSU_{\omega_F}$ following the rules
\begin{itemize}
    \item $\epsilon_{\A,F}(\ide_n)=\ide_{F^n(\one)}$;
    \item $\epsilon_{\A,F}(u\circ v)=\epsilon_{\A,F}(u)\circ\epsilon_{\A,F}(v)$, for compatible $u,v\in\wSU_{\omega_F}$;
    \item $\epsilon_{\A,F}(1_n\otimes x \otimes 1_m)=F^m(x_{F^n(\one)})$, for $x=z\in Z'(\A)$, or $x$ equal to the cap or cup morphism.
\end{itemize}
and taking $R$-linear combinations. For example, the disjoint union of two cup diagrams $0\to 2$ is sent to the morphism $F^2(\delta_{\one})\circ\delta_\one\colon \one\to F^4(\one)$, the disjoint union of two cap diagrams $2\to 0$ is sent to the morphism $\mu_\one\circ \mu_{F^2(\one)}\colon F^4(\one)\to \one$. This assignment respects all the relations on the generators in the category $\wSU_{\omega_F}$ and thus gives an $R$-linear functor as required.

The endofunctor $F_1$ on $\wSU_{\omega_F}$ sends an object $n$ to $n\otimes 1$. Thus, 
$$\epsilon_{\A,F} \circ F_1(n)=F^{n+1}(\one)=F(F^n(\one))=F(\epsilon_{\A,F}).$$
This equality extends to morphisms of $\wSU_{\omega_F}$. E.g., for the generating cap morphism $c$ and $z\in Z'(\A)$ we have
\begin{align*}
        \epsilon_{\A,F}\circ F_1(c)=\epsilon_{\A,F}(c\otimes 1)=F(\mu_\one) =F\circ \epsilon_{\A,F}(c),\\
        \epsilon_{\A,F}\circ F_1(z)=\epsilon_{\A,F}(z\otimes 1)=F(z)=F\circ \epsilon_{\A,F}(z).
    \end{align*}
Thus, $\epsilon_{\A,F}$ is a morphism of $R$-linear categories equipped with a self-adjoint functor. One now verifies the adjunction identities from Equation \ref{eq_cond_1} for $\eta$ and $\varepsilon$. For this, we note that $Z'(\epsilon_{\A,F})$ and $\eta_{(Z'(\A),\omega_F)}$ are both the identity on $Z'(\A)$. 
Further, $\epsilon_{\wSU_\omega,F_1}$ sends $n$ to the object $n$ and is the identity on morphisms under identification of $Z'(\wSU_\omega)$ and $Z$. Similarly, $\wSU_{\eta_{(Z,\omega)}}$ is the functor induced from this identification so the two functors are mutually inverse.  
\end{proof}

 In the case of the category $\wU$, a similar universal property was given in \cite[Section 11]{DP1}. Namely, $\wU$ is the free category with a self-adjoint endofunctor. 
 Mapping the generating object $1$ of $\wU$ to any object in a category $\A$ equipped with a self-adjunction determines a unique functor between these categories respecting the self-adjunctions.

%
%

 \section{Pairings and monoidal envelopes} \label{sec_pairings}


\subsection{Pairings, negligible morphisms, and the gligible quotient category \texorpdfstring{$\wU_\omega$}{Uomega} of \texorpdfstring{$\wSU_\omega$}{SUomega}}\label{subsec_pairings_neg}

\quad 
\smallskip

In this section, we describe the quotient of $\wSU_\omega$ 
by the ideal of negligible morphisms.
To describe the latter  ideal, 
we introduce the set $\wBout{2n}$ of \emph{outer matchings} in an annulus of $2n$ points on  the inner circle of the annulus. Take an annulus $\wA$
and place $2n$ points on the inner circle of $\wA$. Outer matchings  are isotopy  classes of collections of $n$ disjoint arcs in $\wA$ connecting these $2n$  points  in  pairs. 
It is easy to check that $|\wBout{2n}|=\binom{2n}{n}.$ Some examples  of outer matchings are shown in Figure~\ref{fig:outermatch}. The set $\wBout{0}$  consists  of  the empty matching and  $\wBout{2}$ consists  of two matchings, see Figure {\scshape\ref{fig:outmatch1}}. 

 \begin{figure}
\begin{center}
     \begin{subfigure}[htb]{0.45\textwidth}
      \centering
       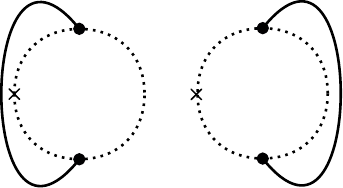
         \caption{The two elements of $\wBout{2}$.}
         \label{fig:outmatch1}
     \end{subfigure}
          \begin{subfigure}[htb]{0.45\textwidth}
      \centering
       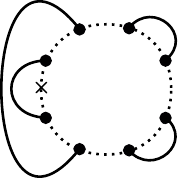
         \caption{An element of $\wBout{8}$.}
         \label{fig:outmatch2}
     \end{subfigure}
\caption{Examples of outer matchings.}
\label{fig:outermatch}
\end{center}
\end{figure}

We may take a diagram $x$ describing a morphism  in $\Hom_{\wSU_\omega}(n,m)$ and represent it as a box with  $n$ bottom and $m$ top endpoints. Alternatively, we can  visualize  $x$  as a diagram in a  disk with  $n+m$ boundary points. To remember that $x$ comes from a morphism from $n$ to $m$, we can arrange $n$ points to be on the lower half-circle of the disk and $m$ points on the upper half-circle --- see Figure {\scshape\ref{fig:BZdiagtocirc}} for an example.  As in Section \ref{sec-wU} for morphisms in $\wUout{2n}$, we may  place a mark {\scriptsize$\times$} corresponding to the position  of  the  left  edge of  the box. The position of the right  edge of the box can then be  recovered if $n$ is known.
\begin{figure}
    \centering
\begingroup%
  \makeatletter%
  \providecommand\color[2][]{%
    \errmessage{(Inkscape) Color is used for the text in Inkscape, but the package 'color.sty' is not loaded}%
    \renewcommand\color[2][]{}%
  }%
  \providecommand\transparent[1]{%
    \errmessage{(Inkscape) Transparency is used (non-zero) for the text in Inkscape, but the package 'transparent.sty' is not loaded}%
    \renewcommand\transparent[1]{}%
  }%
  \providecommand\rotatebox[2]{#2}%
  \newcommand*\fsize{\dimexpr\f@size pt\relax}%
  \newcommand*\lineheight[1]{\fontsize{\fsize}{#1\fsize}\selectfont}%
  \ifx\svgwidth\undefined%
    \setlength{\unitlength}{139.50061697bp}%
    \ifx\svgscale\undefined%
      \relax%
    \else%
      \setlength{\unitlength}{\unitlength * \real{\svgscale}}%
    \fi%
  \else%
    \setlength{\unitlength}{\svgwidth}%
  \fi%
  \global\let\svgwidth\undefined%
  \global\let\svgscale\undefined%
  \makeatother%
  \begin{picture}(1,0.29451246)%
    \lineheight{1}%
    \setlength\tabcolsep{0pt}%
    \put(0,0){\includegraphics[width=\unitlength,page=1]{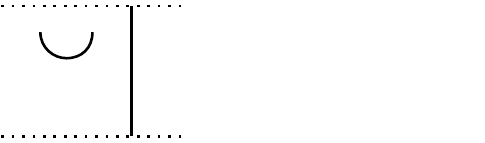}}%
    \put(0.0833372,0.10425186){\color[rgb]{0,0,0}\makebox(0,0)[lt]{\smash{\begin{tabular}[t]{l}$z_2$\end{tabular}}}}%
    \put(0.10925865,0.22809919){\color[rgb]{0,0,0}\makebox(0,0)[lt]{\smash{\begin{tabular}[t]{l}$z_1$\end{tabular}}}}%
    \put(0,0){\includegraphics[width=\unitlength,page=2]{Bzdiagtocirc.pdf}}%
    \put(0.29800668,0.10809193){\color[rgb]{0,0,0}\makebox(0,0)[lt]{\smash{\begin{tabular}[t]{l}$z_3$\end{tabular}}}}%
    \put(0.4973131,0.13209195){\color[rgb]{0,0,0}\makebox(0,0)[lt]{\smash{\begin{tabular}[t]{l}$\longrightarrow$\end{tabular}}}}%
    \put(0,0){\includegraphics[width=\unitlength,page=3]{Bzdiagtocirc.pdf}}%
    \put(0.77726649,0.21225834){\color[rgb]{0,0,0}\makebox(0,0)[lt]{\smash{\begin{tabular}[t]{l}$z_1$\end{tabular}}}}%
    \put(0.78168265,0.10079564){\color[rgb]{0,0,0}\makebox(0,0)[lt]{\smash{\begin{tabular}[t]{l}$z_2$\end{tabular}}}}%
    \put(0.90053774,0.07516223){\color[rgb]{0,0,0}\makebox(0,0)[lt]{\smash{\begin{tabular}[t]{l}$z_3$\end{tabular}}}}%
    \put(0,0){\includegraphics[width=\unitlength,page=4]{Bzdiagtocirc.pdf}}%
  \end{picture}%
\endgroup%

    \caption{An example of representing an element of $\wSU_\omega$ in the disk with marked circle points.}
\label{fig:BZdiagtocirc}
\end{figure}

\vspace{0.1in}

Consider all possible closures of $x$ via diagrams $y$ with $n+m$ boundary points in an annulus. Such a diagram $y$ consists of an outer matching in  $\wBout{n+m}$ together  with elements of  $Z$ sprinkled  over  the $\frac{n+m}{2}+1$ regions  of $y$  into which $\frac{n+m}{2}$ arcs of $y$ split the annulus. We call such a diagram $y$   a \emph{$Z$-decorated outer matching}. The closure $yx$ is  a planar  diagram  of circles and elements of $Z$ written in the regions, and it evaluates to an element of $Z$ that we also denote  $yx$. The evaluation is the same one  that simplifies diagrams in $\wSU^0_{Z,0}$  to elements of  $Z$.

 \begin{figure}
\begin{center}
     \begin{subfigure}[htb]{0.4\textwidth}
      \centering
       $x=\vcenter{\hbox{
\begingroup%
  \makeatletter%
  \providecommand\color[2][]{%
    \errmessage{(Inkscape) Color is used for the text in Inkscape, but the package 'color.sty' is not loaded}%
    \renewcommand\color[2][]{}%
  }%
  \providecommand\transparent[1]{%
    \errmessage{(Inkscape) Transparency is used (non-zero) for the text in Inkscape, but the package 'transparent.sty' is not loaded}%
    \renewcommand\transparent[1]{}%
  }%
  \providecommand\rotatebox[2]{#2}%
  \newcommand*\fsize{\dimexpr\f@size pt\relax}%
  \newcommand*\lineheight[1]{\fontsize{\fsize}{#1\fsize}\selectfont}%
  \ifx\svgwidth\undefined%
    \setlength{\unitlength}{39.42905721bp}%
    \ifx\svgscale\undefined%
      \relax%
    \else%
      \setlength{\unitlength}{\unitlength * \real{\svgscale}}%
    \fi%
  \else%
    \setlength{\unitlength}{\svgwidth}%
  \fi%
  \global\let\svgwidth\undefined%
  \global\let\svgscale\undefined%
  \makeatother%
  \begin{picture}(1,1.04029068)%
    \lineheight{1}%
    \setlength\tabcolsep{0pt}%
    \put(0,0){\includegraphics[width=\unitlength,page=1]{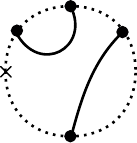}}%
    \put(0.21196539,0.74927434){\color[rgb]{0,0,0}\makebox(0,0)[lt]{\smash{\begin{tabular}[t]{l}$z_1$\end{tabular}}}}%
    \put(0.22758981,0.3549176){\color[rgb]{0,0,0}\makebox(0,0)[lt]{\smash{\begin{tabular}[t]{l}$z_2$\end{tabular}}}}%
    \put(0.64810098,0.26422619){\color[rgb]{0,0,0}\makebox(0,0)[lt]{\smash{\begin{tabular}[t]{l}$z_3$\end{tabular}}}}%
  \end{picture}%
\endgroup%
}}$
         \caption{An element $x$ of $\wSU_\omega$.}
         \label{fig:x}
     \end{subfigure}
 \begin{subfigure}[htb]{0.4\textwidth}
      \centering
       $y=\vcenter{\hbox{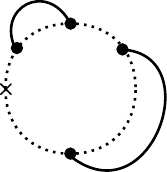}}$
         \caption{An outer matching $y$ in $\wBout{4}$.}
         \label{fig:y}
     \end{subfigure}
 \begin{subfigure}[htb]{0.6\textwidth}
      \centering
       $yx=\vcenter{\hbox{
\begingroup%
  \makeatletter%
  \providecommand\color[2][]{%
    \errmessage{(Inkscape) Color is used for the text in Inkscape, but the package 'color.sty' is not loaded}%
    \renewcommand\color[2][]{}%
  }%
  \providecommand\transparent[1]{%
    \errmessage{(Inkscape) Transparency is used (non-zero) for the text in Inkscape, but the package 'transparent.sty' is not loaded}%
    \renewcommand\transparent[1]{}%
  }%
  \providecommand\rotatebox[2]{#2}%
  \newcommand*\fsize{\dimexpr\f@size pt\relax}%
  \newcommand*\lineheight[1]{\fontsize{\fsize}{#1\fsize}\selectfont}%
  \ifx\svgwidth\undefined%
    \setlength{\unitlength}{116.67902423bp}%
    \ifx\svgscale\undefined%
      \relax%
    \else%
      \setlength{\unitlength}{\unitlength * \real{\svgscale}}%
    \fi%
  \else%
    \setlength{\unitlength}{\svgwidth}%
  \fi%
  \global\let\svgwidth\undefined%
  \global\let\svgscale\undefined%
  \makeatother%
  \begin{picture}(1,0.42405391)%
    \lineheight{1}%
    \setlength\tabcolsep{0pt}%
    \put(0,0){\includegraphics[width=\unitlength,page=1]{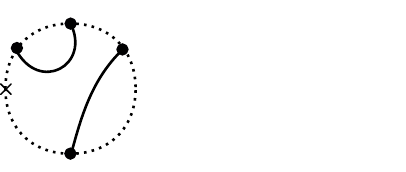}}%
    \put(0.07162895,0.28275918){\color[rgb]{0,0,0}\makebox(0,0)[lt]{\smash{\begin{tabular}[t]{l}$z_1$\end{tabular}}}}%
    \put(0.07690887,0.14949518){\color[rgb]{0,0,0}\makebox(0,0)[lt]{\smash{\begin{tabular}[t]{l}$z_2$\end{tabular}}}}%
    \put(0.21901118,0.11884805){\color[rgb]{0,0,0}\makebox(0,0)[lt]{\smash{\begin{tabular}[t]{l}$z_3$\end{tabular}}}}%
    \put(0,0){\includegraphics[width=\unitlength,page=2]{closure.pdf}}%
    \put(0.46327107,0.19242443){\color[rgb]{0,0,0}\makebox(0,0)[lt]{\smash{\begin{tabular}[t]{l}$=$\end{tabular}}}}%
    \put(0.61111183,0.18599667){\color[rgb]{0,0,0}\makebox(0,0)[lt]{\smash{\begin{tabular}[t]{l}$z_1$\end{tabular}}}}%
    \put(0,0){\includegraphics[width=\unitlength,page=3]{closure.pdf}}%
    \put(0.75183738,0.19449036){\color[rgb]{0,0,0}\makebox(0,0)[lt]{\smash{\begin{tabular}[t]{l}$z_2$\end{tabular}}}}%
    \put(0.90036691,0.18599667){\color[rgb]{0,0,0}\makebox(0,0)[lt]{\smash{\begin{tabular}[t]{l}$z_3$\end{tabular}}}}%
    \put(0,0){\includegraphics[width=\unitlength,page=4]{closure.pdf}}%
  \end{picture}%
\endgroup%
}}=\omega(z_1)z_2\omega(z_3)$
         \caption{The closure $yx$ of $x$ by $y$.}
         \label{fig:yx}
     \end{subfigure}
\caption{Examples of the closure of a diagram in $\wSU_\omega$ by a compatible outer matching.}
\label{fig:closure}
\end{center}
\end{figure}

We say that a finite $R$-linear combination $x=\sum_i \lambda_ix_i$ of  diagrams $x_i$  as above with coefficients in $R$ is \emph{negligible} if and only if
\begin{equation}
\sum_i \lambda_i \, yx_i=0\in Z
\end{equation} 
for  any outer diagram $y$. 
In other  words, a morphism $x$  from $n$  to $m$ in  $\wSU_\omega$ is negligible if  $yx=0$ for any  way  to  close $x$ on the outside via a $Z$-decorated outer matching $y$. Note that $x$ is  a linear combination of diagrams, and  $yx$ is the corresponding linear  combination of  closed diagrams.  

\begin{proposition}  The set of negligible morphisms constitutes a  two-sided monoidal ideal in  $\wSU_\omega$. 
\end{proposition}  
\begin{proof}
The proof is straightforward.
\end{proof}

Denote  by  $\wU_\omega$ the quotient  of $\wSU_\omega$ by the two-sided  ideal of negligible  morphisms. 
We refer to $\wU_\omega$ as the \emph{state category of $(Z,\omega)$} or the \emph{gligible quotient of $\wSU_\omega$}.

\vspace{0.1in} 

The category $\wU_\omega$  is  an $R$-linear rigid monoidal category. Its objects are non-negative  integers $n$. The  morphism spaces  $\Hom_{\wU_\omega}(n,m)$ are naturally $Z$-bimodules, just like the morphism spaces $\Hom_{\wSU_\omega}(n,m)$. If $Z$ is a finitely-generated $R$-module, the morphism spaces $\Hom_{\wU_\omega}(n,m)$ are  finitely-generated $R$-modules as well.


\subsection{A commutative square of categories} \label{sec:commsquare}
\quad
\smallskip

In this subsection, we consider the Karoubi closures of $\wSU_\omega$ and its gligible quotient $\wU_\omega$. We give an overview diagram of the monoidal categories considered in this section in Equation~\eqref{overview-diag}.

\ssubsection{The Karoubi closure of $\wSU_\omega$.}
Given an $R$-linear category $\A$, one considers $\Kar(\A^{\oplus})$ which is the \emph{Karoubi closure} (or, \emph{idempotent completion}) of the category $\A$. This category is constructed in two steps. First, we formally adjoin finite direct sums of objects in $\A$. That is, we construct a category $\A^{\oplus}$ where objects are direct sums $\bigoplus_{i=1}^nA_i$, together with morphisms 
\begin{align}\iota_i\colon A_i\to \bigoplus_{i=1}^nA_i,\qquad \pi_i\colon \bigoplus_{i=1}^nA_i\to A_i\end{align}
satisfying the relations
\begin{align}\pi_j\iota_i=\delta_{i,j}\ide_{A_i}.\end{align}
The category $\Kar(\A^{\oplus})$ consists of pairs $A_e:=(A,e)$, where $A$ is an object in $\A^{\oplus}$ and $e\colon A\to A$ an idempotent endomorphism, i.e. $e\circ e=e$. Morphisms are given by 
$$\Hom_{\Kar(\A^{\oplus})}\big(A_e,B_f\big)=f\circ \Hom_{\A^{\oplus}}(A,B)\circ e.$$

\ssubsection{Generalized Deligne--Karoubi categories}
Define the  Deligne (or Deligne--Karoubi) category $\wDSU_{\omega}$ associated to $\wSU_\omega$ as  the additive  Karoubi closure  of the latter, i.e.
\begin{equation} 
\wDSU_\omega :=\Kar(\wSU_{\omega}^{\oplus}). 
\end{equation} 
This is an idempotent-complete  $R$-linear monoidal category with duals. 

\begin{example} 
Let us specialize to $Z=R$ and $d=\omega(1)$  an invertible element of $R$, so that $\wSU_d= TL(d)$, the Temperley--Lieb category. Then the  additive Karoubi closure $\wDSU_\omega$ contains the idempotent $e=\tfrac{1}{d}\, \vcenter{\hbox{
\begingroup%
  \makeatletter%
  \providecommand\color[2][]{%
    \errmessage{(Inkscape) Color is used for the text in Inkscape, but the package 'color.sty' is not loaded}%
    \renewcommand\color[2][]{}%
  }%
  \providecommand\transparent[1]{%
    \errmessage{(Inkscape) Transparency is used (non-zero) for the text in Inkscape, but the package 'transparent.sty' is not loaded}%
    \renewcommand\transparent[1]{}%
  }%
  \providecommand\rotatebox[2]{#2}%
  \newcommand*\fsize{\dimexpr\f@size pt\relax}%
  \newcommand*\lineheight[1]{\fontsize{\fsize}{#1\fsize}\selectfont}%
  \ifx\svgwidth\undefined%
    \setlength{\unitlength}{8.24999902bp}%
    \ifx\svgscale\undefined%
      \relax%
    \else%
      \setlength{\unitlength}{\unitlength * \real{\svgscale}}%
    \fi%
  \else%
    \setlength{\unitlength}{\svgwidth}%
  \fi%
  \global\let\svgwidth\undefined%
  \global\let\svgscale\undefined%
  \makeatother%
  \begin{picture}(1,1.36363653)%
    \lineheight{1}%
    \setlength\tabcolsep{0pt}%
    \put(0,0){\includegraphics[width=\unitlength,page=1]{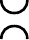}}%
  \end{picture}%
\endgroup%
}}\colon 2\to 2$, an endomorphism of object $2$. The splitting object $(2,e)$ is a proper subobject of $2$, for instance since $\rank_{R}\,\End((2,e))=1\not= 2=\rank_{R}\,\End(2)$. The object $(2,e)$ is also not isomorphic to $1$ since there are no non-zero morphisms $1\to 2$ in $TL(d)$.
\end{example}

\ssubsection{The Karoubi closure of the gligible quotient}

We define the Karoubi closure of the gligible quotient category by
\begin{equation}
    \wDU_{\omega}:= \Kar(\wU_{\omega}^{\oplus})
\end{equation}

If the ground ring $R$ is a  field  $\kk$  and $\dim_{\kk}Z< \infty$, then  the  morphism spaces in categories  $\wDU_\omega$  and $\wU_\omega$ are  finite-dimensional. Consequently, taking the gligible quotient commutes with passing  to the additive Karoubi closure  in this case,  and  there  is a commutative diagram of monoidal categories and monoidal functors 
\begin{align}\label{overview-diag}
 \vcenter{\hbox{ \xymatrix{\wU\ar@{^{(}->}[rr]&&R\wU\ar@{^{(}->}[rr]&&\wU_Z\ar@{>>}[rr]\ar@{>>}[rrd]&&\wSU_{\omega}\ar@{^{(}->}[rr]\ar@{>>}[d]&&\wDSU_{\omega}  \ar@{>>}[d] \\
    &&&&&&\wU_{\omega}\ar@{^{(}->}[rr]&& \wDU_{\omega}.
    }}}
\end{align}
We note that the gligible quotient of $\wU_Z$ is equivalent (even isomorphic) to $\wU_\omega$.  

Under weaker conditions (if $R$ is not a field or $Z$ has infinite rank over $R$), there are potentially two distinct categories in place of $\wDU_{\omega}$: the gligible quotient of $\wDSU_{\omega}$ and the Karoubi envelope of $\wU_{\omega}$. There is a functor from the first category to the second, but it is not clear when this functor is an equivalence. With these weak conditions, it is natural to define $\wDU_{\omega}$ as the Karoubi envelope of the gligible quotient $\wU_{\omega}$, while keeping in mind the above caveat: the square is still commutative, but the gligible quotient of $\wDSU_{\omega}$ may not be equivalent to $\wDU_{\omega}$.  
 
We can  think of the four categories in the right square of diagram \eqref{overview-diag} as various monoidal envelopes of  the circular triple  $(R,Z,\omega)$. Objects of all of these categories have two-sided duals.

\vspace{0.1in} 

\emph{Summary of these categories}
\begin{itemize}
    \item $\wU$ has non-negative integers as objects. Morphisms from $n$ to $m$ in $\wU$   are  isotopy classes of planar diagrams of arcs and circles in the strip $\R\times [0,1]$ with
    $n$ bottom and $m$ top boundary points. 
    \item $R\wU$, for a commutative ring $R$, is a \emph{linearization} or \emph{pre-linearization} of $\wU$. It has non-negative integers as objects. Morphism from $n$ to $m$ in $\wU$ are 
    finite $R$-linear combinations of morphisms in $\wU$. 
    \item $\wU_Z$ is associated to a commutative $R$-algebra $Z$. It has the same objects as $\wU$, that is, non-negative integers. Compared to $R\wU$, morphisms in $\wU_Z$ are enriched by allowing elements of $Z$ to  float in the regions of a diagram.  
    \item $\wSU_{\omega}$ is associated to a circular triple $(R,Z,\omega)$, with $\omega$ an $R$-linear endomorphism of $Z$. It is a quotient of $\wU_Z$ by the relation that a circle wrapping around $z\in Z$ evaluates to $\omega(z)$. 
    \item $\wU_{\omega}$ is the quotient of $\wSU_{\omega}$ by the ideal of negligible morphisms (the \emph{gligible quotient} of $\wSU_{\omega}$). Since our categories are only monoidal and  not symmetric, the definition of a  negligible morphism requires converting its diagram in $\R\times[0,1]$ to a diagram in $D^2$ and then evaluating closings of this diagram via all possible annular diagrams. 
    \item $\wDSU_{\omega}$ is the additive Karoubi closure of $\wSU_{\omega}$. It is the counterpart of the Deligne category of the symmetric group. (Although note that  $\wDSU_{\omega}$ is not symmetric monoidal.)
    \item $\wDU_{\omega}$ is the gligible quotient of $\wDSU_{\omega}$. When $R=\kk$ is a field and $Z$ is finite-dimensional over $\kk$, the category $\wDU_{\omega}$ is also equivalent to the additive Karoubi closure of $\wU_{\omega}$, making the  square in (\ref{overview-diag}) commutative.
\end{itemize}

\vspace{0.1in} 

Consider the functor  from  $R\wU$  to $\wSU_\omega$ and the composite functor to $\wU_\omega$. In the special case when $Z$ is $\omega$-generated over $R$ (see Section \ref{sec_universal}), then the functor from $R\wU$  to $\wSU_\omega$, and hence also the composition to $\wU_\omega$, are full. In this case, we  can use circular forms and have them float in regions of crossingless matching diagrams in place of elements  of $Z$ in order to display morphisms  in $\wSU_\omega$.

\vspace{0.1in} 

Let $Z$ be a commutative $R$-algebra with a set of elements 
$\{s_i\}_{i\in I}, s_i\in Z$ that \emph{$\omega$-generates} $Z$ over $R$. In other words, any element of $Z$ is a $R$-linear combination of iterated products of applications of $\omega$ to the elements $s_i$.  In this case any morphism in $\wSU_\omega$ is a linear combination of  crossingless matchings with diagrams of circles and dots labelled by the $s_i$ floating in regions.

 
 \subsection{Spherical triples, evaluations, state spaces and categories}\label{subsec_arcs}
 \quad
 \smallskip
 
 For the gligible quotient category $\wU_\omega$ from Section~\ref{subsec_pairings_neg} to be monoidal we needed to adopt an asymmetric set-up which pairs diagrams in a disk with diagrams in an annulus to define the correct quotient space  $\wU_{\omega,n}^m$ for the spaces  of diagrams $\wSU_{\omega,n}^m$ in a disk. This pairing is different from similar pairings in~\cite{Kh2,KS,KQR}  where one works with manifolds (sometimes with defects) rather than planar diagrams, and the ambient category is symmetric rather than just monoidal, which is the case with the planar diagrams considered here.

 \ssubsection{Spherical triples.}
 
 One can limit the consideration to the symmetric case and consider pairings of diagrams in a disk when the underlying triple $(R,Z,\omega)$ is \emph{spherical}. This means that the evaluation of a diagram in the plane $\R^2$ depends only on the isotopy class of the corresponding diagram in $\SS^2$. 
 
 \begin{definition}\label{def_Zspherical}
 A triple $(R,Z,\omega)$ is called $Z$-\emph{spherical} if $\omega(z_1)z_2 = z_1 \omega(z_2)$ for any $z_1,z_2\in Z.$
 \end{definition}
 
 This property is equivalent to the condition that any planar diagram, when evaluated to an element of $Z$, depends only on the isotopy class of the diagram in $\SS^2$. Such an isotopy of $\SS^2$ is a composition of  isotopies in $\R^2$ and moving an arc from a circle in the diagram through the infinite point of $\SS^2$. For a move of a latter type, that circle splits the diagram in $\SS^2$ into two disks, and the diagrams there may be evaluated to elements $z_1,z_2\in Z$, respectively. The relation $\omega(z_1)z_2 = z_1 \omega(z_2)$ in the above definition says that moving a circle bounding $z_1,z_2$ on the two sides through the infinite point of $\SS^2$ does not change the evaluation, see Figure~\ref{fig_spher}.

 \begin{figure}[htb]
\begin{center}
    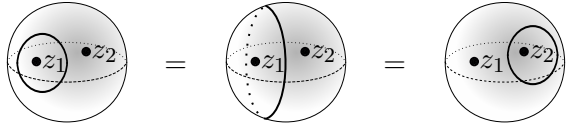
\caption{The $Z$-sphericality condition for $\omega$. The dashed equatorial circle is shown to emphasize that the $\omega$-circle and $z_1,z_2$-dots are placed on a 2-sphere.}
\label{fig_spher}
\end{center}
\end{figure}
 
 The $Z$-spherical condition is equivalent to the condition that  
 $\omega(z)= z\omega(1)$ for all $z\in Z$. In particular, $\omega$ is determined by $\omega(1)\in Z$, and, vice versa, any element $z_0\in Z$ gives rise to a spherical triple with $\omega(z)=z z_0$. Hence, for $Z$-spherical $\omega$, to evaluate a diagram, count the number $k$ of its circles, remove the circles, and then multiply the evaluation of what remains by $z_0^k$. 

 \vspace{0.1in} 
 
 In the $Z$-spherical case, labels $z\in Z$ in the regions of an arc-circle diagram in $\wU_\omega$ may freely move between the regions. 
 
 \begin{lemma}\label{lem_Zspherical}
 Let $(R,Z,\omega)$ be a $Z$-spherical circular triple. For any $z\in Z$, the relation
\begin{equation}\label{eqn_Zspherical1}
    z\cdot \ide_1= \ide_1\cdot z 
\end{equation}
holds in $\wU_\omega$. More generally, for any morphism $f$ in $\wU_\omega$, we have $z\cdot f=f\cdot z$.
In particular, the set $\{z\cdot u\}$, where $u\in \wB_n^m$ and $z$ runs over a basis of $Z$ constitutes a generating set of $\Hom_{\wU_\omega}(n,m)$.
 \end{lemma}
 \begin{proof}
 Closing the morphism $z\cdot \ide_1- \ide_1\cdot z$ gives the relation of $Z$-sphericality from Definition \ref{def_Zspherical}. Thus, this relation holds in the gligible quotient $\wU_\omega$. This implies that we can move all $Z$-labels in regions of an arc-circle diagram $u\in \wB_n^m$ to the leftmost region (or the rightmost region). The remaining statements follow.
 \end{proof}

 \begin{figure}[htb]
\begin{center}
    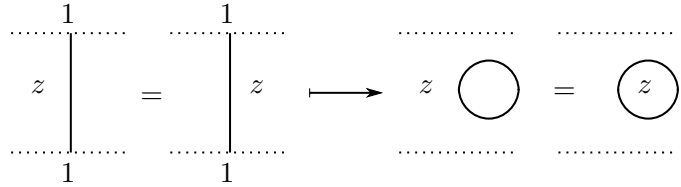
    \caption{Relation \eqref{eqn_Zspherical1} holds in $\wU_\omega$ in the $Z$-spherical case.}
     \label{fig:my_label}
\end{center}
 \end{figure}

The $Z$-spherical condition is very restrictive, since the combinatorics of nested circles in the plane or on the 2-sphere is lost. One can instead refine it by adding a trace map from $Z$ to a smaller commutative ground ring to evaluate spherical diagrams. Without aiming for full generality, let us define $R$-spherical or simply \emph{spherical} triples. 

\begin{definition}\label{def_Rspherical}
A circular triple $(R,Z,\omega)$ is called \emph{$R$-spherical} (or \emph{spherical} for short)  if it comes equipped with a non-degenerate $R$-linear trace map $\varepsilon\colon Z\lra R$ such that 
  \begin{equation}
      \varepsilon(z_1\omega(z_2)) = \varepsilon(\omega(z_1)z_2), \ \ z_1,z_2\in Z. 
  \end{equation}
\end{definition}


 \vspace{0.1in} 
 
 When pairing a diagram in a disk to a diagram in an annulus, for a $Z$-spherical triple $(R,Z,\omega)$ or an $R$-spherical triple $(R,Z,\omega,\varepsilon)$, the annular diagram may be reduced to one in a disk by moving some of its arcs through the infinite point of $\SS^2$. The two $R$-modules on the two sides of the pairing can be made isomorphic, and the pairing is then symmetric. We will look at some examples in Section~\ref{sec_adj_ex} and now discuss a related setup when only the evaluations of spherical diagrams in $R$ (taken to be a field $\kk$, for simplicity) are given.

\ssubsection{Pairings on circle diagrams}
 
 Recall from Section \ref{sec-wU} that the category $\wU$ has non-negative integers $n$ as objects  and morphisms  from $n$ to  $m$ are isotopy classes of planar diagram with circles  and  arcs, the latter connecting $n+m$ points on the boundary of  the diagram in pairs via  a crossingless  matching.
 Thus, the set of  morphisms 
 $\wU^m_n$ from $n$  to $m$ in $\wU$ is the set of isotopy classes of diagrams of circles and  arcs in the  strip $\R\times [0,1]$ with $m$ top and  $n$ bottom endpoints. 
 
 Elements of $\wU^m_n$ are  in a bijective correspondence  with the following data, see also Section~\ref{sec-wU}. Each $u \in \wU^m_n$ defines a crossingless matching $\arc(u)\in \wB^m_n$ given by erasing the circles of $u$. The diagram $\arc(u)$ partitions the strip into $\frac{n+m}{2}+1$ contractible regions. The intersection of $u$ with the interior of each region is a diagram of circles, thus an element of  $\wU^0_0$. Hence, we see that elements of $\wU^m_n$ are in a bijection with crossingless matching  in $\wB^m_n$ together  with a choice of a diagram in $\wU^0_0$ (a closed diagram) for each of the $\frac{n+m}{2}+1$ regions. 

 \vspace{0.1in} 
 
Given a  diagram $u\in \wU^m_n$ denote by $\overline{u}\in \wU^n_m$ the reflection of $u$ about the horizontal line through the middle of the strip. As we have seen in \eqref{reflections}, this operation extends to a contravariant  involution on the category $\wU$. 

The set $\wU_n^m$  is empty unless $n+m$ is even, and $\wU^n_0$ is empty unless $n$ is even.  
For two elements $a,b\in \wU^{n}_0$ the composition $\overline{b}a$ is a closed diagram in $\wU^0_0$, see  Figure~\ref{fig2_11} for an example.

\begin{figure}
\begin{center}
     \centering
     \begin{subfigure}[htb]{0.25\textwidth}
     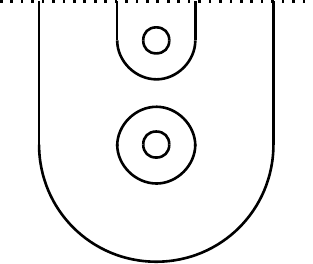 
      \centering
         \caption{$a$}
         \label{fig:ainU}
     \end{subfigure}
     \begin{subfigure}[htb]{0.25\textwidth}
        \centering
     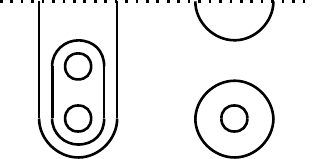 
         \caption{$b$}
         \label{fig:binU}
     \end{subfigure}
     \begin{subfigure}[htb]{0.25\textwidth}
        \centering
    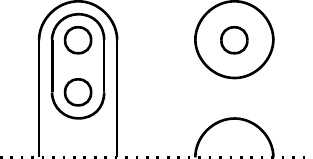
         \caption{$\ov{b}$}
         \label{fig:bbarb}
     \end{subfigure}
\caption{Diagrams $a$, $b\in \wU^4_0$ and the reflection $\overline{b}\in \wU^0_4$.}
\label{fig2_10}
\end{center}
\end{figure}

\begin{figure}[htb]
\begin{center}
\begingroup%
  \makeatletter%
  \providecommand\color[2][]{%
    \errmessage{(Inkscape) Color is used for the text in Inkscape, but the package 'color.sty' is not loaded}%
    \renewcommand\color[2][]{}%
  }%
  \providecommand\transparent[1]{%
    \errmessage{(Inkscape) Transparency is used (non-zero) for the text in Inkscape, but the package 'transparent.sty' is not loaded}%
    \renewcommand\transparent[1]{}%
  }%
  \providecommand\rotatebox[2]{#2}%
  \newcommand*\fsize{\dimexpr\f@size pt\relax}%
  \newcommand*\lineheight[1]{\fontsize{\fsize}{#1\fsize}\selectfont}%
  \ifx\svgwidth\undefined%
    \setlength{\unitlength}{254.33164281bp}%
    \ifx\svgscale\undefined%
      \relax%
    \else%
      \setlength{\unitlength}{\unitlength * \real{\svgscale}}%
    \fi%
  \else%
    \setlength{\unitlength}{\svgwidth}%
  \fi%
  \global\let\svgwidth\undefined%
  \global\let\svgscale\undefined%
  \makeatother%
  \begin{picture}(1,0.47477832)%
    \lineheight{1}%
    \setlength\tabcolsep{0pt}%
    \put(0,0){\includegraphics[width=\unitlength,page=1]{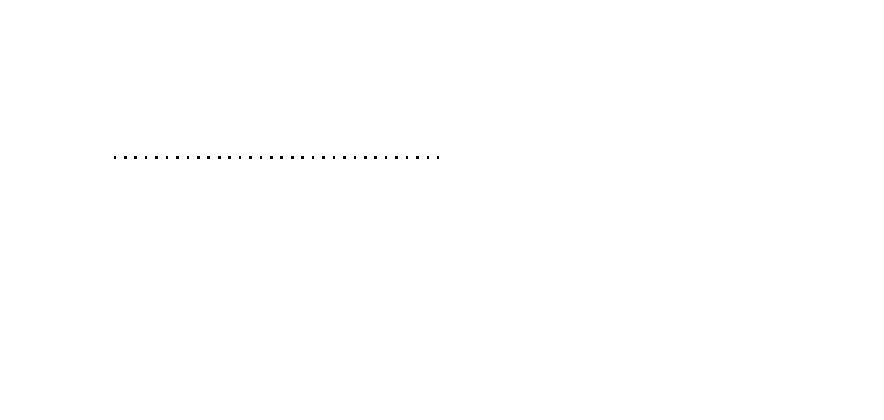}}%
    \put(-0.00130551,0.28683601){\color[rgb]{0,0,0}\makebox(0,0)[lt]{\smash{\begin{tabular}[t]{l}$\ov{b}a \; =$\end{tabular}}}}%
    \put(0.54144283,0.28758073){\color[rgb]{0,0,0}\makebox(0,0)[lt]{\smash{\begin{tabular}[t]{l}$=$\end{tabular}}}}%
    \put(0,0){\includegraphics[width=\unitlength,page=2]{bbara.pdf}}%
  \end{picture}%
\endgroup%

\caption{The composition $\overline{b}a\in \wU^0_0$ of the diagrams $a,b$ from Figure \ref{fig2_10}.}
\label{fig2_11}
\end{center}
\end{figure}

The map of  sets  
\begin{equation}
 \wU^n_0 \times \wU^n_0 \lra \wU^0_0
\end{equation}
taking $a\times b$  to  $\ov{b}a$ is symmetric, $\ov{b}a=\ov{a}b$ since  for  $c\in \wU^0_0$ we have  $\ov{c}=c$ by Proposition \ref{prop_refl}.

\ssubsection{Spherical evaluations}

Let us first discuss a minimalist approach to the construction of state spaces and state categories in the \emph{spherical} framework (Section~\ref{subsec_circ_s} will expand on this approach in the non-spherical case as well). We specialize $R$ to a field $\kk$ and assume given a map 
\begin{equation}\label{eq_alpha_spher2} 
    \alpha\colon  \wU^0_0 \lra \kk , \ \  \alpha(\omega(u_1)u_2)=\alpha(\omega(u_2)u_1), \ \  u_1,u_2 \in \wU^0_0.    
\end{equation}
The condition says that evaluation $\alpha$ of circle diagrams is \emph{spherical}, that is, depends only on the isotopy class of the diagram in $\SS^2$. Alternatively, one can define the  quotient set $\wU^{s,0}_{0}$ of $\wU^0_0$ by identifying two diagrams if they  are isotopic as diagrams in  $\SS^2$ and define $\alpha$ as the composition $\wU^0_0\lra \wU^{s,0}_{0} \lra\kk$ for some evaluation map $\wU^{s,0}_{0} \lra\kk$.  

\vspace{0.1in} 

Consider the bilinear form $(\,,\,)_{\alpha}$ on $\kk \wU^{2n}_0$ defined on the basis of circle diagrams by 
\begin{equation}\label{eq_bil_symm} 
    (a,b)_{\alpha} \ := \ \alpha(\ov{b}a), \ \  a,b\in \wU^{2n}_0
\end{equation}
and extended to the entire vector space $\kk\wU^{2n}_0$ bilinearly.  Define the state space 
\begin{equation}
    \wU^{2n}_{\alpha} \ := \ \kk \wU^{2n}_0 /\mathrm{ker}((\,,\,)_{\alpha})
\end{equation}
as the quotient of the vector space $\kk \wU^{2n}_0$ of diagrams by the kernel of the bilinear form $(\,,\,)_{\alpha}$. The bilinear form is symmetric by  Proposition~\ref{prop_refl}. 

This definition makes sense for any function $\alpha\colon  \wU^0_0\lra \kk$ but gives a better behaved collection of spaces $\wU^{2n}_{\alpha}$ when $\alpha$ is spherical, so that $\alpha$ factors through the quotient set $\wU^{s,0}_0$. 

\begin{definition}
 A spherical evaluation $\alpha$ as in (\ref{eq_alpha_spher2}) is called \emph{recognizable} if the state spaces $\wU^{2n}_{\alpha}$ are finite-dimensional for all $n$. 
\end{definition}

\begin{prop}\label{prop_rational} A spherical evaluation $\alpha$ is recognizable  if and only if $\wU^0_{\alpha}$ is finite-dimensional. 
\end{prop}
\begin{proof}
The proof is straightforward: if $\wU^0_{\alpha}$ is finite-dimensional then a closed diagram, when part of any larger diagram to be evaluated, can be reduced to a linear combination of diagrams from a fixed finite subset $S\subset \wU^0_0$. Consequently, any diagram $u$ in $\wU^{2n}_0$ can be reduced, in $\wU^0_{\alpha}$, to a linear combination of diagrams with the $n$ arcs as in $u$ and a diagrams from $S$ in each of $n+1$ regions cut out by the arcs in the lower half-plane. 
\end{proof}

\vspace{0.1in} 

More generally, when $R$ is a commutative ring rather than a field $\kk$, the evaluation $\alpha\colon \wU^0_0\lra R$ is called \emph{recognizable} if $\wU^0_{\alpha}$ is a finitely generated $R$-module. This is equivalent to all state spaces $\wU^{2n}_{\alpha}$ being finitely generated $R$-modules. 

\vspace{0.1in} 

From here on we restrict to spherical $\alpha$ for the rest of this section. The state space $\wU^0_{\alpha}$ is naturally a commutative algebra  with the non-degenerate trace form $\alpha$. This  algebra is finite-dimensional exactly when $\alpha$ is recognizable. Wrapping a circle around a diagram induces an $\kk$-linear map $\omega\colon  \wU^0_{\alpha} \lra \wU^0_{\alpha}$. 

\vspace{0.1in}

Converting from the  lower half-plane to a disk or a plane strip  with $n$ bottom and top $m$ boundary points, we can define the state space $\wU^m_{\alpha,n}$  of the evaluation $\alpha$ for such diagrams. This can be done, for instance, by bending the bottom $n$ points  via $n$ arcs to get  diagrams in $\wU^{n+m}_0$ and using the bilinear form on that space, or by directly gluing together two such diagrams  into a spherical diagram and applying $\alpha$. There is an isomorphism of vector spaces  $\wU^m_{\alpha,n}\cong  \wU^{n+m}_{\alpha}$ as long as  we fix an arc diagram to move the bottom points to the top, see Figure~\ref{fig_bending}.

\begin{figure}[htb]
\begin{center}
\begingroup%
  \makeatletter%
  \providecommand\color[2][]{%
    \errmessage{(Inkscape) Color is used for the text in Inkscape, but the package 'color.sty' is not loaded}%
    \renewcommand\color[2][]{}%
  }%
  \providecommand\transparent[1]{%
    \errmessage{(Inkscape) Transparency is used (non-zero) for the text in Inkscape, but the package 'transparent.sty' is not loaded}%
    \renewcommand\transparent[1]{}%
  }%
  \providecommand\rotatebox[2]{#2}%
  \newcommand*\fsize{\dimexpr\f@size pt\relax}%
  \newcommand*\lineheight[1]{\fontsize{\fsize}{#1\fsize}\selectfont}%
  \ifx\svgwidth\undefined%
    \setlength{\unitlength}{335.79455668bp}%
    \ifx\svgscale\undefined%
      \relax%
    \else%
      \setlength{\unitlength}{\unitlength * \real{\svgscale}}%
    \fi%
  \else%
    \setlength{\unitlength}{\svgwidth}%
  \fi%
  \global\let\svgwidth\undefined%
  \global\let\svgscale\undefined%
  \makeatother%
  \begin{picture}(1,0.35715911)%
    \lineheight{1}%
    \setlength\tabcolsep{0pt}%
    \put(0,0){\includegraphics[width=\unitlength,page=1]{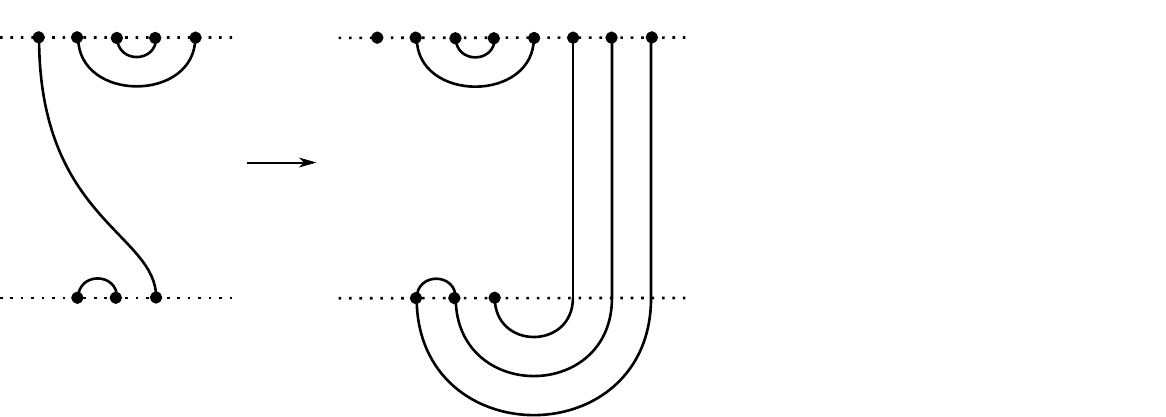}}%
    \put(0.6030473,0.24680322){\color[rgb]{0,0,0}\makebox(0,0)[lt]{\lineheight{1.25}\smash{\begin{tabular}[t]{l}$=$\end{tabular}}}}%
    \put(0,0){\includegraphics[width=\unitlength,page=2]{bending.pdf}}%
    \put(0.03126916,0.34731111){\color[rgb]{0,0,0}\makebox(0,0)[lt]{\lineheight{1.25}\smash{\begin{tabular}[t]{l}$\overbrace{\qquad \qquad\,}^{m}$\end{tabular}}}}%
    \put(0.06477178,0.07338846){\color[rgb]{0,0,0}\makebox(0,0)[lt]{\lineheight{1.25}\smash{\begin{tabular}[t]{l}$\underbrace{\qquad\,}_{n}$\end{tabular}}}}%
    \put(0.3208277,0.34810875){\color[rgb]{0,0,0}\makebox(0,0)[lt]{\lineheight{1.25}\smash{\begin{tabular}[t]{l}$\overbrace{\qquad \qquad\qquad \quad\,}^{n+m}$\end{tabular}}}}%
    \put(0.67818889,0.34810875){\color[rgb]{0,0,0}\makebox(0,0)[lt]{\lineheight{1.25}\smash{\begin{tabular}[t]{l}$\overbrace{\qquad \qquad\qquad \quad\,}^{n+m}$\end{tabular}}}}%
    \put(0,0){\includegraphics[width=\unitlength,page=3]{bending.pdf}}%
  \end{picture}%
\endgroup%

\caption{An isomorphism of $\wU^{n+m}_{\alpha}$ and   $\wU^m_{\alpha,n}$ given by bending the $n$ bottom points to the top. }
\label{fig_bending}
\end{center}
\end{figure}

\vspace{0.1in} 

The benefit of using spherical evaluations with the bilinear form (\ref{eq_bil_symm}) is that the spaces $\wU^m_{\alpha,n}$ can  be arranged into a monoidal category $\wU_{\alpha}$. The objects of that category are non-negative integers $n\geq 0$, and the vector space $\wU^m_{\alpha,n}$ describes the space of morphisms from $n$ to $m$. Composition is given by concatenation of diagrams, and compatibility with the quotient construction is clear.  That $\alpha$ is spherical implies that the category $\wU_{\alpha}$ is (strict) monoidal. On objects, the tensor product is given by $n\otimes  m = n+m.$ On morphisms, the tensor  product is given by placing diagrams in parallel. 

To get a monoidal category in this way from a non-spherical evaluation $\alpha\colon \wU^0_0\lra \kk$ one needs to use the asymmetric setup and couple diagrams in a disk to diagrams in an annulus, similar to that  in Section~\ref{subsec_pairings_neg}. In the later Section \ref{subsec_circ_s} we discuss this construction in more detail. 

\vspace{0.1in} 

Given a spherical  evaluation  $\alpha$ and the associated monoidal category $\wU_{\alpha}$, the most  natural case to consider is  that of recognizable $\alpha$.  The following observation is immediate. 

\begin{prop} A spherical evaluation $\alpha$ is recognizable if and only if the morphism spaces in the  category $\wU_{\alpha}$ are finite-dimensional.
\end{prop} 

Given a recognizable $\alpha$  the commutative $\kk$-algebra  $Z:=\wU^0_{\alpha}$ is finite-dimensional and the the trace form $\alpha$ turns  it into a commutative  Frobenius algebra. It satisfies the sphericality condition with respect to the $\kk$-linear map $\omega\colon Z\lra Z$ given by wrapping a circle around an element of $\wU^0_\alpha$, i.e.
\begin{equation}\label{eq_alpha_spher}
    \alpha(\omega(u_1)u_2)=\alpha(u_1\omega(u_2)), \ \ u_1,u_2\in Z. 
\end{equation}
Furthermore, $Z$ is $\omega$-generated (cf. Definition \ref{def_circular}),  which is  a  stability condition on that data. 

\begin{prop} \label{prop_bij_spherical} There is a natural bijection between recognizable spherical evaluations $\alpha$ and isomorphism classes of commutative  Frobenius algebras $Z$ with a trace form $\varepsilon\colon Z\lra\kk$ and a linear map $\omega\colon Z\lra Z$ such that $Z$ is the only $\omega$-stable subalgebra of $Z$ and the spherical condition (\ref{eq_alpha_spher2}) holds for $\varepsilon$.  
\end{prop} 
\begin{proof}
The proof is immediate.
\end{proof}

One can repeat the construction of categories and functors as in~\cite{KS,KQR,Kh3} and \eqref{overview-diag} and consider the  following categories and functors.
\begin{equation}\label{diagUalpha}
    \vcenter{\hbox{\xymatrix{
    \wU\ar@{^{(}->}[r]&\kk \wU\ar@{->>}[r]
    &\wSU_\alpha\ar@{->>}[d]\ar@{^{(}->}[r]&\wDSU_\alpha\ar@{->>}[d]\\
   &&\wU_\alpha\ar@{^{(}->}[r] & \wDU_\alpha
    }}}
\end{equation}

We start with the category $\wU$. First, allow any finite $\kk$-linear combinations of morphisms in $\wU$ to form the category $\kk\wU$ with the same objects $n\geq 0$ as $\wU$. Next, we pick a spherical evaluation $\alpha$ and consider the commutative algebra $\wU^0_{\alpha}$  as  described above. Introduce the skein category $\wSU_{\alpha}$ as the quotient of $\kk\wU$ by adding   the relations that an endomorphism $x$ of the $0$ object in $\wSU_{\alpha}$ is zero in $\Hom_{\wSU_{\alpha}}(0,0)$ if its image in  $\wU^0_{\alpha}$ is zero. Such an endomorphism is a $\kk$-linear combination of closed diagrams (diagrams in $\wU^0_0$). Choose a  set $S$ of diagrams in $\wU^0_0$ that project to a basis of the algebra $\wU^0_{\alpha}$. 

Thus, when passing from $\kk\wU$ to $\wSU$ we only add relations on closed diagrams, not on diagrams with arcs. Relations on closed diagrams allow to simplify any diagram in $\wU^m_n$, when viewed as a morphism from $n$ to $m$ in $\wSU_{\alpha}$, into a linear combination of diagrams with a circular form from $S$ in each region. The vector space of morphisms from $n$ to $m$ in $\wSU_{\alpha}$ has a basis given by a choice of a crossingless matching of $n+m$ points on the boundary of a strip together with a choice of an element of $S$ in each of the  $\frac{n+m}{2}$ regions cut out by the arcs of the matching. In particular, morphism spaces in $\wSU_{\alpha}$ are finite-dimensional if and only if $\alpha$ is rational (recognizable). Note that the category $\wSU_\alpha$ is equivalent to the category $\wSU_{\wU_\alpha^0,\,\omega}$ discussed in Section~\ref{sec-BZ}.

We think of $\wSU_{\alpha}$ as a type of
\emph{skein category}, similar to those in~\cite{KS,KQR,Kh3}, where one has control over the size of morphism spaces and can write down a basis in each. 

Assume now that $\alpha$ is a recognizable spherical evaluation. Then morphism spaces in $\wSU_{\alpha}$ are finite-dimensional. 
This allows us to build a commutative square of categories and functors. We can pass from $\wSU_{\alpha}$ to the \emph{gligible quotient} category $\wU_{\alpha}$. The morphism spaces in $\wU_{\alpha}$ are quotients of those in $\kk\wU$ or $\wSU_{\alpha}$ by the bilinear pairings on $\wU^m_n$, gluing two diagrams with identical boundaries into a diagram of circles on the 2-sphere, as already discussed. 

In the category $\wU_{\alpha}$ the morphism space $\Hom(0,2n)$ is naturally isomorphic to the state space $\wU^{2n}_{\alpha}$. Thus, the gligible quotient category is equivalent to the category obtained from the state spaces $\wU_\alpha$.

The square of four categories and functors on the right side of the diagram in \eqref{diagUalpha} is commutative in the strong sense: Taking the additive Karoubi envelope $\wDSU_\alpha$ of $\wSU_{\alpha}$ and then the gligible quotient results in the category $\wDU_\alpha$ equivalent to that of first taking the gligible quotient and then forming the additive Karoubi envelope.

\begin{prop} \label{prop_all_monoidal} All the categories and functors in the above diagram \eqref{diagUalpha} are monoidal. In each of these categories the functor of tensoring  with the object $1$ is self-adjoint, via the canonical natural transformations given by the cup and the cap diagrams. 
\end{prop} 



Recall the skein categories $\wSU_{Z,\omega}=\wSU_\omega$ defined for a circular triple $(R,Z,\omega)$ in Section \ref{sec-BZ}. 
Assume given a $R$-spherical datum $(R,Z,\omega,\varepsilon)$ of a circular triple with $R$-linear map $\varepsilon\colon Z\to R$. We may consider the composition of $R$-linear maps 
$$\alpha \colon R\wU_0^0\xrightarrow{\cF_Z} Z\xrightarrow{\varepsilon} R.$$
The pairing $(\,,\,)_{\alpha}\colon R\wU_0^0\times R\wU_0^0\to R$ factors through $\varepsilon m\colon Z\times Z\to R$, using the $R$-algebra map $\cF_Z$ defined in \eqref{eq_R_Z}.

If $Z$ is $\omega$-generated, $Z\cong R\wU_0^0/\ker \cF_Z$ and, using that $\cF_Z$ is an algebra map on $R\wU_0^0$,  we see that $\wU_\alpha^0$ is isomorphic to a quotient of the commutative $R$-algebra $Z$.
Hence, $Z$ serves as a preliminary reduction of the space $R\wU_0^0$ and can detect whether $\alpha$ is recognizable. Thus, the following lemma is a consequence of Proposition~\ref{prop_rational}.

\begin{lemma}Let $\kk$ be a field and $(\kk,Z,\omega)$ a spherical triple with a $\kk$-linear spherical trace $\varepsilon\colon Z\to \kk$.
If $Z$ is finite-dimensional over $\kk$, then $\alpha$ is recognizable. 
\end{lemma}

Note that in the $\omega$-generated case, the quotient map  
$p\colon Z\twoheadrightarrow \wU_\alpha^0$ gives a morphism of circular triples. Thus, it induces a full monoidal functor  
$$\wSU_p\colon \wSU_\omega\lra \wSU_\alpha,$$
using \eqref{eq:SUphi}.
The maps $\Hom_{\wSU_\omega}(0,2n)\to \Hom_{\wSU_{\alpha}}(0,2n)$ fit into the  commutative diagram
\begin{align*}
    \xymatrix@R=10pt{
    \kk \wU_{0}^{2n}\times \kk \wU_{0}^{2n}\ar[rr]^-{(~ ,~)}\ar[dd]&& \kk \wU_{0}^{0}\ar[dd]^{\cF_Z}\ar[drr]^{\alpha}&& \\
    &&&&\kk.\\
    \kk \wSU_{\omega,0}^{2n}\times \kk \wSU_{\omega,0}^{2n}\ar[rr]^-{(~ ,~)}&&Z\ar[urr]^{\varepsilon}&&
    }
\end{align*}
This implies surjective maps  $\Hom_{\wU_\omega}(n,m)\to \Hom_{\wU_{\alpha}}(n,m)$ for all $n,m$ using duality as in Figure~\ref{fig_bending}. Thus,  in the presence of a trace map $\varepsilon \colon Z\to \kk$, $\wU_\alpha$ is a quotient of $\wU_\omega=\wU_{Z,\omega}$.

The sphericality condition is analogous to the \emph{spherical trace} property of Barrett--Westbury~\cite{BW}. 

%
%

\subsection{Circular series and recognizable series}\label{subsec_circ_s} 

In this section, we generalize the construction of state spaces to non-spherical circular series.

\ssubsection{Circular series} 

We work over a ground field $\kk$, for simplicity,  but the constructions below generalize to a ground commutative ring $R$  and, more generally, to a ground commutative semiring. 

A  \emph{circular series} $\alpha$ is defined to be a map $\alpha\colon \wU^0_0\rightarrow \kk$ that assigns an element $\alpha(u)$ of $\kk$ to each circular form $u$. We  alternatively write 
 \begin{equation}
     \alpha = \{\alpha(u)\}_{u\in \wU^0_0} =  \sum_{u \in \wU^0_0} \alpha(u)\, u,
 \end{equation}
 and can also view $\alpha$ as a linear map $\kk\wU^0_0\lra \kk$. 
 To build state spaces and a category from circular series $\alpha$ we use an approach similar to the one in Section~\ref{subsec_pairings_neg}. 
 
 Throughout the paper we use $\wU^{2n}_0$ to denote the set  of circular diagrams with $2n$ endpoints in both a disk and in the lower half-plane, via the standard identification of these sets (see an example in Figure~\ref{fig2_4_0a} and the discussion in that section). $\wU^{2n}_0$ can be viewed as the set of matchings of $2n$ points on the unit circle by $n$ arcs  in the unit disk, with possibly nested circles floating in the regions of the diagram. 
 
Recall from Section~\ref{sec-wU} that  $\wUout{2n}$ is the set on pairings  of $2n$ points on the unit circle by $n$ arcs in the outside annulus, possibly with additional nested collections of circles, also see Figure~\ref{fig2_4_0b} for an example of such an annular diagram.
   
 The pairing of diagrams 
\begin{equation}
    \wU^{2n}_0 \times \wUout{2n} \lra \wU^0_0, \ \ a, b \longmapsto ab , \ \  a\in \wU_0^{2n}, \ b\in \wUout{2n},
\end{equation}
is given by gluing diagrams $a$ and $b$ from the corresponding sets into a circular diagram $ab$ in the plane. 

 The series (or evaluation) $\alpha$ allows us to define a bilinear form 
 \begin{equation}\label{eq_bilin_form} 
    (\,,\,)_\alpha\colon \kk \wU_0^{2n}\times \kk \wUout{2n}\lra \kk, \qquad (a,b)\longmapsto \alpha(ab), 
 \end{equation}
 on the vector spaces with these sets as bases, by evaluating closed planar diagram $ab$ via $\alpha$. If needed, one can write $(,)_{n,\alpha}$ instead of $(,)_{\alpha}$ to emphasize dependence on $n$. Note that both of these spaces are infinite-dimensional, including when $n=0$, since circles can be nested in infinitely many ways. 
 
 \vspace{0.1in} 
 
 Define the \emph{left kernel} of $(,)_{\alpha}$ as 
 \begin{equation*}
    \ker_{\ell}(\alpha) \ : = \ \{x\in \kk\wU^{2n}_0 | (x,y)_{\alpha}=0 \  \  \forall y \in  \kk\wUout{2n} \}
\end{equation*}

We define the \emph{state  space}  $A_{\alpha}(n)$, or just $A(n)$, as the quotient of $\kk\wU^{2n}_0$ by $\ker_{\ell}(\alpha)$,
\begin{equation}
    A_{\alpha}(n)  \ := \ \kk\wU^{2n}_0/\ker_{\ell}(\alpha), 
\end{equation}
 and call it the state space of $2n$ points (on the boundary of a disk) for the evaluation $\alpha$. 
 
 \vspace{0.1in}

\ssubsection{Recognizable series}

\begin{definition}\label{def_rec_circ}
  A circular series $\alpha$ is called \emph{recognizable} iff the state spaces  $A(n)$ are finite-dimensional for all  $n\ge 0$. 
\end{definition} 

Alternatively, we can say that $\alpha$ is of finite rank. 

\begin{prop}\label{prop:alpha_rec} A circular series $\alpha$ is recognizable if and only if $A(0)=A_{\alpha}(0)$ is a finite-dimensional $\kk$-vector space. 
\end{prop} 

\begin{proof} Clearly, $A(0)$ needs to be finite-dimensional for  $\alpha$ to be  recognizable.  
If $A(0)$ is finite-dimensional, there are finitely many circular forms $v_1, \dots, v_k$ (where $k=\dim_{\kk} A(0)$) such that any circular form is obtained by their linear combination modulo an element of $\ker((,)_{0,\alpha})$. 

A diagram $w\in\wU^{2n}_0$ is determined by the crossingless matching on its $2n$ endpoints together with choices  of circular forms to place in $n+1$ regions of the disk separated by  the arcs of  the matching. Modulo $\ker((,)_{n,\alpha})$, one can reduce to placing one of $v_1,\dots, v_k$ in each region of the disk. In particular,
\[ \dim_{\kk}A(n) \le c_n k^{n+1} = \frac{1}{n+1} \binom{2n}{n} k^{n+1},
\] 
where $c_n$ is the number of crossingless matchings of $2n$ points on the boundary of a disk. The proposition follows. 
\end{proof}

The state space $A(0)$ is naturally a unital commutative algebra. The multiplication comes from that on $\wU^0_0$ given by placing diagrams (circular forms) next to each other. The unit element $1$ is the empty circular form $\emptyset$. The operator $\omega$  of wrapping a circle around a diagram preserves the left kernel of  the bilinear form $(,)_{0,\alpha}$ and descends to a linear map, also denoted $\omega$, on $A(0)$. The trace form $ \varepsilon\colon A(0)\lra \kk$ comes from the evaluation $\alpha$ of closed diagrams, $\varepsilon(a):= \alpha(a)$, for $a\in \wU^0_0$.

We see that  $A(0)$ is a unital commutative algebra equipped with a $\kk$-linear map $\omega\colon  A(0)\lra A(0)$ and a trace form $\varepsilon$.  

The triple  $(A(0),\omega,\varepsilon)$  is non-degenerate in the following weak sense. For any $x\in A(0), x\not= 0$ there exists $k\ge 0$ and  a  sequence $x_1,\dots, x_k\in A(0)$ such that  
\begin{equation}\label{eq_vareps} 
\varepsilon(x_k\omega (x_{k-1}\dots \omega (x_2 \omega (x_1x)))\dots )\not=0 .
\end{equation} 
We call such a data $(A,\omega,\varepsilon)$  a \emph{commutative weakly Frobenius triple}. In the pictorial language, we start with the diagram $x$ and iterate between placing $x_i$, $i=1,\dots, k$, next to the previous diagram and enveloping the diagram by a circle (application of $\omega$). At the end the trace form $\varepsilon$ is applied. 

\begin{prop} \label{prop_bij_A} There is a bijection between recognizable circular series $\alpha$ and isomorphism classes of commutative finite-dimensional algebras $A$ with the trace form $\varepsilon$ and a linear endomorphism $\omega$ subject to weakly Frobenius property above and to the stability condition that $A$ is the only subalgebra of $A$ that contains $1$ and is closed under $\omega$. 
\end{prop} 
\begin{proof}
The proof is straightforward.
\end{proof}

This proposition is very similar in spirit to Proposition~\ref{prop_bij_spherical}, except that the spherical condition (\ref{eq_alpha_spher2}) is dropped here and the bilinear pairing needed to define $A(n)$ is asymmetric, with a bigger space on the other side of the pairing (\ref{eq_bilin_form}).  

\vspace{0.1in} 

The dihedral group $D_{2n}$ of symmetries of a regular $2n$-gon acts on $A(n)$, via rotations and reflections. Note that reflection of diagrams respects the left kernel of the form, since any diagram in $\wU^0_0$ is invariant under the plane's reflection, so that reflection descends to an invertible linear map $r\colon A(n)\lra A(n)$.  

Placing diagrams with $2n$ and $2m$ endpoints, respectively, next to each other induces multiplication maps 
\begin{equation}
    A(n)\otimes A(m) \lra A(n+m), 
\end{equation}
that respect the reflection maps $r$ as above, in the sense that $r(xy)=r(y)r(x)$, for $x\in A(n)$ and $y\in A(m)$.

\ssubsection{Circular triples and quadruples}
 For a field $\kk$, 
 a circular triple $(\kk,Z,\omega)$ allows us to  associate an element $\mc{F}_Z(u)$ of a  commutative $\kk$-algebra  $Z$  to each circular form $u\in \wU^0_0$. Assume that $Z$ comes with a $\kk$-linear trace map $\varepsilon\colon Z\lra \kk$. Then composing with $\varepsilon$ allows us to  assign to a circular form $u$  the element $\varepsilon(\mc{F}_Z(u))$ of  $\kk$. 
 
 Assume that $Z$ is finite-dimensional. Then evaluating a circular form $u$ to $\varepsilon(\mc{F}_Z(u))$ gives a recognizable circular series $\alpha$. To understand the  space $A(0)$, first pass to the smallest subalgebra $Z'$ of $Z$ that contains $1$ and is closed under $\omega$. The trace map $\varepsilon$ may be degenerate on $Z'$ (as well as on $Z$), in the interated compositions sense as discussed right before Proposition~\ref{prop_bij_A}. Consider the subspace $K\subset Z'$ that consists of  $x$  such that the evaluations on the left  hand side  of (\ref{eq_vareps}) are zero for any sequence $x_1,\dots, x_k\in Z'$. The space $A(0)$ is naturally isomorphic to the quotient,   $A(0)\cong Z'/K$. 
 
 \vspace{0.1in} 
 
 In particular, we  see that 
 a series $\alpha$ is recognizable if and only if there exists a circular triple $(\kk,Z,\omega)$ together with a $\kk$-linear  map $\varepsilon\colon  Z \lra \kk$
 such that 
 \begin{itemize}
     \item $\dim_{\kk} Z  <   \infty$, that is,  $Z$ is finite-dimensional, 
     \item $\alpha(u) =    \varepsilon(\mcF_Z(u))$ for all circular forms $u$.
 \end{itemize}
 A   data $(\kk,Z,\omega,\varepsilon)$ with the above properties may be called a \emph{circular quadruple}. Furthermore, given recognizable $\alpha$, such circular quadruple can be chosen so that $Z$ is $\omega$-generated and $(Z,\omega,\varepsilon)$ is a commutative weakly Frobenius triple. Given $\alpha$, such a \emph{minimal circular quadruple} is unique up to isomorphism, see  Proposition~\ref{prop_bij_A}. 
 
 \vspace{0.1in} 
 
 One can think of circular quadruples as describing ``inner-to-outer''  evaluations of  circular  diagrams. 
 
 \ssubsection{A diagram of categories and functors associated to $\alpha$} 
 
 To a spherical recognizable series $\alpha$ we have associated a diagram of categories of functors, see (\ref{diagUalpha}). This construction extends immediately to arbitrary circular recognizable series $\alpha$, so essentially the same diagram is reproduced below.  
 
 \begin{equation}\label{diagUalpha2}
    \vcenter{\hbox{\xymatrix{
    \wU\ar@{^{(}->}[r]&\kk \wU\ar@{->>}[r]
    &\wSU_\alpha\ar@{->>}[d]\ar@{^{(}->}[r]&\wDSU_\alpha\ar@{->>}[d]\\
   &&\wU_\alpha\ar@{^{(}->}[r] & \wDU_\alpha
    }}}
\end{equation}

The skein category $\wSU_{\alpha}$ is given by including all relations on closed planar diagrams (relations in $A(0)$). In this category the dimension of the morphism space from $n$ to $m$ is $c_{k}\cdot\dim A(0)^{k+1}$, where $k=(n+m)/2$ and $c_k$ is the $k$-the Catalan number. The category $\wU_{\alpha}$ is the gligible quotient of $\wSU_{\alpha}$, with the same objects $n\ge 0$. The categories on the far right are Karoubi additive closures of the categories $\wSU_{\alpha}$ and $\wU_{\alpha}$, respectively. The square commutes in the strong sense, see the discussion preceeding Proposition~\ref{prop_all_monoidal}. 

\vspace{0.1in} 

Each recognizable circular series $\alpha$ gives rise to a collection of finite-dimensional $\kk$-algebras 
\begin{equation*}
TL_{\alpha,n}:=\End_{\wU_{\alpha}}(n), 
\end{equation*}
the endomorphism rings of objects $n\in \mathbb{N}$ of the category $\wU_{\alpha}$. These algebras generalize the Jones quotients of  Temperley--Lieb algebras~\cite{Jo1,Ka}.

%
%

\section{Trees, forests, their series and  relation to circular forms.}
\label{sec_trees} 

In this section we explain the standard correspondence between trees (forests) and circular forms, allowing one to flip between these two types of combinatorial objects when forming the corresponding series. 

\subsection{Trees, forests, and circular forms} 

By a \emph{tree} we mean a finite connected unoriented graph $\Gamma$ without multiple edges, cycles and with a preferred vertex (called \emph{root}). The  empty  graph is excluded. Denote by $\trees$ the  set of trees, up to isomorphisms; we pick one representative from each isomorphism class. Trees are often depicted by planar  diagrams with the root at the top and vertices  at distance $k$ from the root placed $k$ steps below the root. Examples of trees are shown in Figure~\ref{fig2_6}. 

\begin{figure}[htb]
\begin{center}
\begingroup%
  \makeatletter%
  \providecommand\color[2][]{%
    \errmessage{(Inkscape) Color is used for the text in Inkscape, but the package 'color.sty' is not loaded}%
    \renewcommand\color[2][]{}%
  }%
  \providecommand\transparent[1]{%
    \errmessage{(Inkscape) Transparency is used (non-zero) for the text in Inkscape, but the package 'transparent.sty' is not loaded}%
    \renewcommand\transparent[1]{}%
  }%
  \providecommand\rotatebox[2]{#2}%
  \newcommand*\fsize{\dimexpr\f@size pt\relax}%
  \newcommand*\lineheight[1]{\fontsize{\fsize}{#1\fsize}\selectfont}%
  \ifx\svgwidth\undefined%
    \setlength{\unitlength}{290.48693656bp}%
    \ifx\svgscale\undefined%
      \relax%
    \else%
      \setlength{\unitlength}{\unitlength * \real{\svgscale}}%
    \fi%
  \else%
    \setlength{\unitlength}{\svgwidth}%
  \fi%
  \global\let\svgwidth\undefined%
  \global\let\svgscale\undefined%
  \makeatother%
  \begin{picture}(1,0.17369143)%
    \lineheight{1}%
    \setlength\tabcolsep{0pt}%
    \put(0,0){\includegraphics[width=\unitlength,page=1]{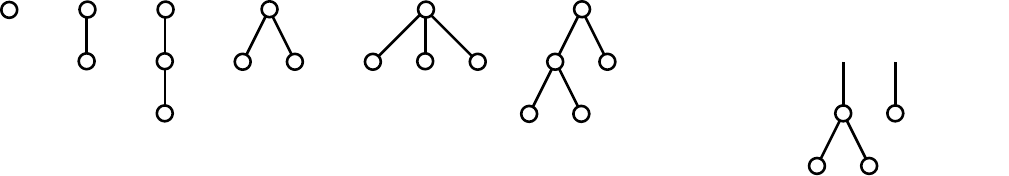}}%
    \put(0.62867796,0.10515952){\color[rgb]{0,0,0}\makebox(0,0)[lt]{\smash{\begin{tabular}[t]{l}$=$\end{tabular}}}}%
    \put(0,0){\includegraphics[width=\unitlength,page=2]{trees2.pdf}}%
  \end{picture}%
\endgroup%
 
\caption{Examples of trees; note that a planar presentation  of  a tree  is rarely unique. For the sixth tree from the left  two different  presentations are depicted. A presentation of a tree  can be made unique by picking a total order on trees  and placing subtrees below each node from left to right in the decreasing order direction.}
\label{fig2_6}
\end{center}
\end{figure}

A forest $w$ is a graph which is a disjoint union of  finitely many trees. The empty graph  is allowed. Each component of  $w$ carries a  preferred vertex (root  of the corresponding tree). The order  of  trees when listing a forest does  not matter. When choosing a set of graphs to represent  forests, we pick one representative for each  isomorphism class of forests. 

Denote  by $\forests$ the set  of  forests and by  $\forests_k$ the set of forests with $k$  components  (trees), so that 
\begin{equation}
    \forests = \bigsqcup_{k\ge 0} \forests_k.
\end{equation}
The set $\forests_1$ is  in a bijection  with $\trees$, the set $\forests_0$ consists of  the empty forest. The set  $\forests_k$ can be  identified  with the $k$-th  symmetric power of $\forests_1$,  
\begin{equation}
    \forests_k \cong S^k(\forests_1). 
\end{equation}

We now describe well-known mutually-inverse bijections, denoted by $\formap$ and $\circles$, between the set $\wU^0_0$ of closed planar diagrams and the set $\forests$  of forests, which restrict  to mutually-inverse bijections between  the set $\wUcirc$ of  $\circ$-diagrams   (diagrams with a single outer circle) and the set $\trees $ of trees. That is, the bijections fit into the commutative diagrams
\begin{align}\label{eq_bij_1}
\vcenter{\hbox{
\xymatrix{
\wUcirc\ar[rr]^{\formap}_{\sim}\ar@{^{(}->}[d]&&\trees\ar@{^{(}->}[d]\\
\wU_0^0\ar[rr]^{\formap}_{\sim}&&\forests,
}}}&&
\vcenter{\hbox{
\xymatrix{
\trees\ar[rr]^{\circles}_{\sim}\ar@{^{(}->}[d]&&\wUcirc\ar@{^{(}->}[d]\\
\forests\ar[rr]^{\circles}_{\sim}&&\wU_0^0.
}}}
\end{align}

First, to a tree $t$ we assign an element  in $\wUcirc$ denoted $\circles(t)$, a collection of circles with one exterior circle. Define this map
\begin{equation}\label{eq_map_circles}  
    \circles \colon  \trees\lra \wUcirc
\end{equation}
by induction on the number of nodes in $t$. To the unique tree with a single node assign the diagram with a single circle, see Figure~\ref{fig2_7} left. Given a tree $t$, denote by  $t_1,\dots ,  t_k$ the trees obtained by removing the root $r(t)$ of $t$ together with all adjacent edges and making the vertices adjacent to $r(t)$ in $t$ the roots $r(t_1),\dots, r(t_k)$ of the trees  $t_1, \dots, t_k$. Now place diagrams  $\circles(t_1),\dots,  \circles(t_k)$, already defined by induction, to float inside a circle, 
see Figure~\ref{fig2_7} right. 

\begin{figure}[htb]
\begin{center}
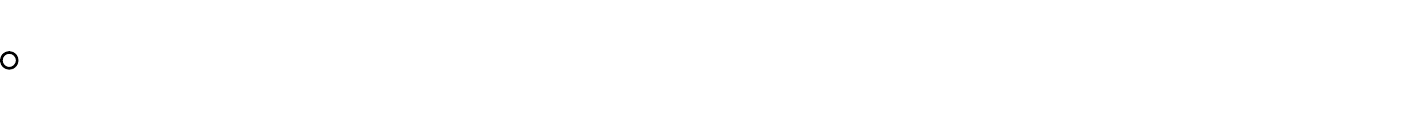 
\caption{Inductive  construction of  the  map  $\circles$.}
\label{fig2_7}
\end{center}
\end{figure}

This map $\circles$ in (\ref{eq_map_circles})  is clearly a  bijection. Consider the inverse bijection
\begin{equation}
    \formap\colon   \wUcirc \stackrel{\cong}{\lra} \trees, \quad\quad   \formap \circ \circles  = \mathrm{id}_{\trees}, \ \  \circles\circ \formap  = \mathrm{id}_{\wUcirc}.  
\end{equation}
The inverse bijection  takes a $\circ$-diagram $u$   and builds a tree $\formap(u)$ with nodes in bijection with circles of  $u$. The unique exterior circle $c$  of $u$ gives the root  node $\formap(u)$. A circle $c_2$  nested immediately inside a circle $c_1$ gives a child node $\formap(c_2)$  to that of  $\formap(c_1)$.  Examples of circle configurations in $\wUcirc$ and  associated trees are shown  in  Figure~\ref{fig2_8}.

\begin{figure}[htb]
\begin{center}
\begingroup%
  \makeatletter%
  \providecommand\color[2][]{%
    \errmessage{(Inkscape) Color is used for the text in Inkscape, but the package 'color.sty' is not loaded}%
    \renewcommand\color[2][]{}%
  }%
  \providecommand\transparent[1]{%
    \errmessage{(Inkscape) Transparency is used (non-zero) for the text in Inkscape, but the package 'transparent.sty' is not loaded}%
    \renewcommand\transparent[1]{}%
  }%
  \providecommand\rotatebox[2]{#2}%
  \newcommand*\fsize{\dimexpr\f@size pt\relax}%
  \newcommand*\lineheight[1]{\fontsize{\fsize}{#1\fsize}\selectfont}%
  \ifx\svgwidth\undefined%
    \setlength{\unitlength}{456.58436703bp}%
    \ifx\svgscale\undefined%
      \relax%
    \else%
      \setlength{\unitlength}{\unitlength * \real{\svgscale}}%
    \fi%
  \else%
    \setlength{\unitlength}{\svgwidth}%
  \fi%
  \global\let\svgwidth\undefined%
  \global\let\svgscale\undefined%
  \makeatother%
  \begin{picture}(1,0.24466405)%
    \lineheight{1}%
    \setlength\tabcolsep{0pt}%
    \put(0,0){\includegraphics[width=\unitlength,page=1]{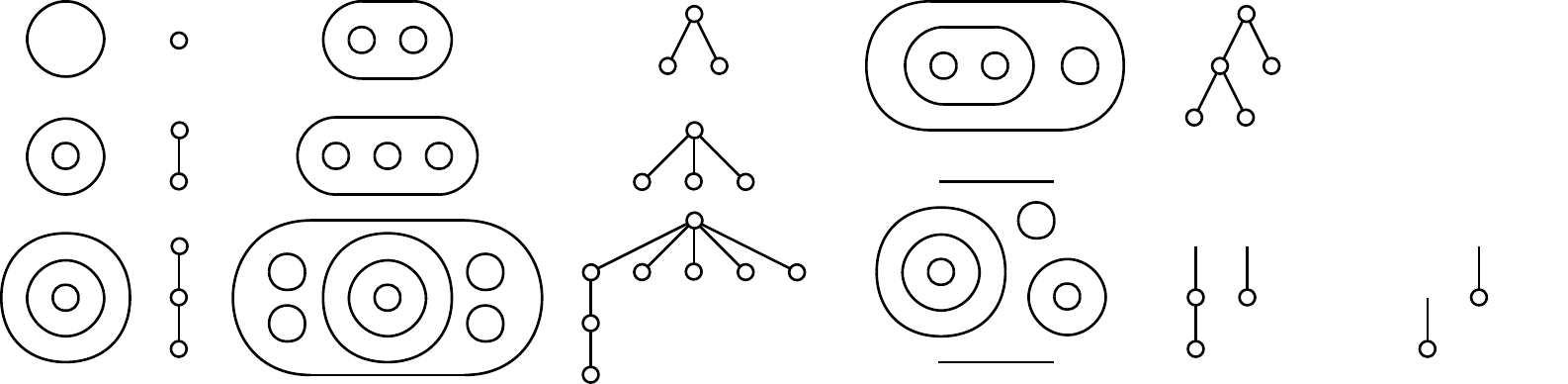}}%
    \put(0.82605709,0.19799437){\color[rgb]{0,0,0}\makebox(0,0)[lt]{\smash{\begin{tabular}[t]{l}$=$\end{tabular}}}}%
    \put(0,0){\includegraphics[width=\unitlength,page=2]{crfor.pdf}}%
    \put(0.84264281,0.07127702){\color[rgb]{0,0,0}\makebox(0,0)[lt]{\smash{\begin{tabular}[t]{l}$=$\end{tabular}}}}%
    \put(0,0){\includegraphics[width=\unitlength,page=3]{crfor.pdf}}%
    \put(0.95792039,0.07215699){\color[rgb]{0,0,0}\makebox(0,0)[lt]{\smash{\begin{tabular}[t]{l}$=\ldots$\end{tabular}}}}%
    \put(0,0){\includegraphics[width=\unitlength,page=4]{crfor.pdf}}%
  \end{picture}%
\endgroup%
 
\caption{Examples of  the  correspondence  between trees and $\circ$-diagrams.}
\label{fig2_8}
\end{center}
\end{figure}

Figure~\ref{fig2_9} shows a more  complicated configuration  in $\wUcirc$ and the associated tree. 

\begin{figure}[htb]
\begin{center}
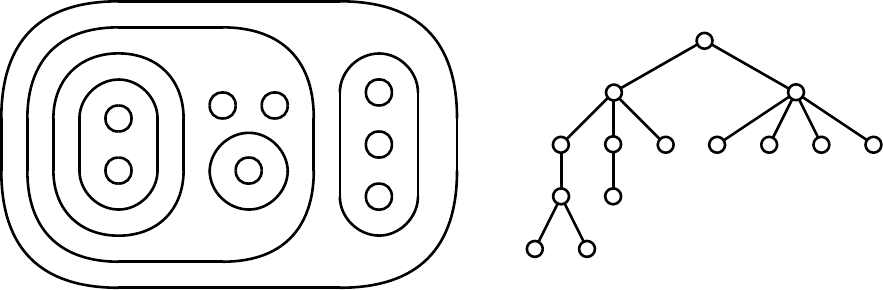 
\caption{A more  complicated example of a tree (on the right)
associated to a $\circ$-diagram (on the left).}
\label{fig2_9}
\end{center}
\end{figure}

Extending these bijections  to disjoint unions of trees (i.e., forests) on one side and unions of $\circ$-diagrams floating in the plane (elements of $\wU^0_0$) gives the mutually-inverse bijections in (\ref{eq_bij_1}). The empty  forest corresponds to the diagram with  no circles. Under these bijections the number of nodes in a forest equals the number of circles in the corresponding planar diagram. The exterior circles of a diagram correspond to the roots of the trees of the associated forest as in Figure \ref{fig2_12}.

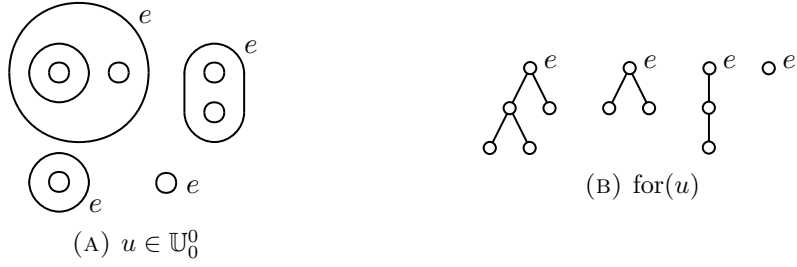
\begin{figure}[htb]
\begin{center}
     \centering
     \begin{subfigure}[htb]{0.4\textwidth}
      
      \centering
         \caption{$u\in \wU^0_0$}
         \label{fig:u}
     \end{subfigure}
     \begin{subfigure}[htb]{0.4\textwidth}
        \centering
\begingroup%
  \makeatletter%
  \providecommand\color[2][]{%
    \errmessage{(Inkscape) Color is used for the text in Inkscape, but the package 'color.sty' is not loaded}%
    \renewcommand\color[2][]{}%
  }%
  \providecommand\transparent[1]{%
    \errmessage{(Inkscape) Transparency is used (non-zero) for the text in Inkscape, but the package 'transparent.sty' is not loaded}%
    \renewcommand\transparent[1]{}%
  }%
  \providecommand\rotatebox[2]{#2}%
  \newcommand*\fsize{\dimexpr\f@size pt\relax}%
  \newcommand*\lineheight[1]{\fontsize{\fsize}{#1\fsize}\selectfont}%
  \ifx\svgwidth\undefined%
    \setlength{\unitlength}{120.21519021bp}%
    \ifx\svgscale\undefined%
      \relax%
    \else%
      \setlength{\unitlength}{\unitlength * \real{\svgscale}}%
    \fi%
  \else%
    \setlength{\unitlength}{\svgwidth}%
  \fi%
  \global\let\svgwidth\undefined%
  \global\let\svgscale\undefined%
  \makeatother%
  \begin{picture}(1,0.29754638)%
    \lineheight{1}%
    \setlength\tabcolsep{0pt}%
    \put(0,0){\includegraphics[width=\unitlength,page=1]{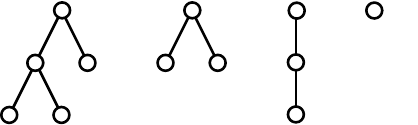}}%
    \put(0.1914067,0.27226619){\color[rgb]{0,0,0}\makebox(0,0)[lt]{\smash{\begin{tabular}[t]{l}$e$\end{tabular}}}}%
    \put(0.50400857,0.27179522){\color[rgb]{0,0,0}\makebox(0,0)[lt]{\smash{\begin{tabular}[t]{l}$e$\end{tabular}}}}%
    \put(0.7530896,0.2714786){\color[rgb]{0,0,0}\makebox(0,0)[lt]{\smash{\begin{tabular}[t]{l}$e$\end{tabular}}}}%
    \put(0.93993519,0.27179522){\color[rgb]{0,0,0}\makebox(0,0)[lt]{\smash{\begin{tabular}[t]{l}$e$\end{tabular}}}}%
  \end{picture}%
\endgroup%
 
         \caption{$\formap(u)$}
         \label{fig:foru}
     \end{subfigure}
\caption{{\scshape (a):} a diagram $u\in \wU^0_0$ with four exterior circles;  the latter are  labelled  by the letter  $e$ next  to them. {\scshape (b):} the forest $\formap(u)$ associated to $u$.}
\label{fig2_12}
\end{center}
\end{figure}

\vspace{0.1in} 

Recall that  $\wU^m_n$ denotes the set of isotopy classes of diagrams of circles and  arcs in the  strip $\R\times [0,1]$ with $m$ top and  $n$ bottom endpoints. Elements of $\wU^m_n$ are  in a bijective correspondence  with the following data. Each $u \in \wU^m_n$ defines a crossingless matching $\arc(u)\in \wB^m_n$ given by erasing the circles of $u$. Diagram $\arc(u)$ partitions the strip into $\frac{n+m}{2}+1$ contractible regions. The intersection of $u$ with the interior of each region is a diagram of circles, thus an element of  $\wU^0_0$. Thus, elements of $\wU^m_n$ are in a bijection with crossingless matchings in $\wB^m_n$ together  with a choice of a diagram in $\wU^0_0$ for each of  $\frac{n+m}{2}+1$ regions. Equivalently, elements  of $\wU^m_n$ are in a bijection with elements  $a\in\wB^m_n$ together with a choice of a forest for each region of $a$.   

\begin{figure}[htb]
\begin{center}
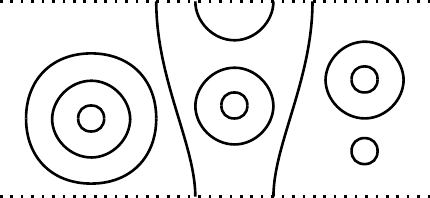 
\caption{A diagram in $\wU_2^4$ corresponds to an element of $\wB_2^4$ together with circle diagrams in each region.}
\label{fig2_13}
\end{center}
\end{figure}

For example, in Figure~\ref{fig2_13}, $n=2,m=4$, and there are four regions, which carry configurations of $0,2,3,3$ circles, respectively. Two of these configrations correspond to  trees, one is the empty forest, and one  is a 2-component forest. 

\vspace{0.1in}

\subsection{Tree and forest series and tree automata}

 Suppose  given a function
\begin{equation} 
\alpha\colon  \wU^0_0 \lra \kk, 
\end{equation} 
that is, a map from the set of \emph{circular forms} to  $\kk$.  Then
$\alpha$ can be thought of as  formal  series (circular form series) 
\begin{equation}
    Z_{\alpha} =\sum_{u \in \wU^0_0} \alpha(u)\,u ,
\end{equation}
that  is, a formal sum, usually with infinitely many non-zero terms, over  circular forms. 

Composing $\alpha$ with  the bijection  $\circles$ from \eqref{eq_map_circles} gives us a map
\begin{equation}
   \alpha\circ \circles \colon  \forests\lra \kk ,
\end{equation}
which we may also refer to as $\alpha$, when it is unambiguous. Such a function  $\alpha$  can be called a  \emph{tree (or forest) series}, since it can formally be encoded by the generating function over forests 
\begin{equation}
    Z^{\formap}_{\alpha} = \sum_{f \in \forests} \alpha(\circles(f))\, f  
\end{equation}
as the corresponding   formal sum  over forests. 
\vspace{0.1in} 

This bijection between forests and circular forms descends to a bijection between a suitable quotient set of $\forests$ and the set of spherical circular forms. We leave the details to an interested reader. 

\vspace{0.1in} 

Converting from circular forms to forests gives a simple combinatorial encoding of isotopy classes of collections of disjoint curves on the plane. It would also be a starting point to investigate connections between universal theories and categories built from $\alpha$, see equation (\ref{diagUalpha2}) for instance, and weighted tree automata. 
Tree and weighted tree automata generalize finite state automata (FSA) and weighted FSA and are studied in depth in computer science, see~\cite{DV,FV, NP} and references therein. Noncommutative power series generalize to  series for weighted tree automata, and our series associated to an evaluation $\alpha$ of circular forms constitute examples of tree series. Trees and forests associated to circular forms are less general that those that appear in  arbitrary tree automata. Some of this gap can be bridged by adding labels to the circles in circular forms and adding other defects to planar configurations. After reducing the gap, similarities between tree automata and weighted tree automata on one side and universal theories for planar diagrams of labelled circles with various decorations and planar graphs on the other side appear worthy of 
further investigation. 

\vspace{0.1in} 

A commutative weakly Frobeninus algebra $A(0)$ with the trace map $\varepsilon$ and the linear endomorphism $\omega$ can be viewed as a $\kk$-linear \emph{bottom-to-top} tree automaton that evaluates a closed diagram starting with the innermost circles (that evaluate to $\omega(1)$), computing unions of closed diagrams via multiplication in $A(0)$, circle wrapping given by $\omega$, combined with the trace map to $\kk$ to end the computation. 

\vspace{0.1in} 

In this paper we mostly work over a field $\kk$. Extending definitions to an arbitrary commutative ring $R$ is straightforward. It is also direct to extend our constructions to a ground commutative semiring $R$. Then, for instance, state spaces $A(0)$ and morphism spaces in the gligible categories $\wU_{\alpha}$ become semimodules over $R$. 

A cursory examination of semimodules,  over the boolean semiring $\wB=\{0,1\}$ with $1+1=1$, for instance, show that they are harder to deal with than modules over rings. In some cases general semimodules may be hidden, and one can instead work with free semimodules or with just a set of their generators, reducing the structures to set-theoretical ones and substantially simplifying the theory --- this approach seems implicit in some standard textbook material on weighted FSA and tree automata. 

The approach of this paper and related papers~\cite{Kh2,KS,Kh3,KQR,KKO}, if rewritten over a semiring, would combine the theory of  semimodules over commutative semirings with monoidal or symmetric monoidal categories, making it harder to stay within set-theoretical structures. This may be an interesting extension of weighted tree automata and related constructions in the general automata theory to explore.

\subsection{The set-theoretical version} \label{subsec_settheory} 

In a set-theoretical version of the story, the underlying categories are neither additive nor pre-additive. In particular, the center $Z$ of a category $\A$ is only a commutative monoid. One can think of its elements as floating in a region of the plane labelled by $\A$. Multiplication in $Z$ corresponds to placing the elements next to each other. Wrapping a circle around $z\in Z$ is a map of sets $\omega\colon  Z \lra Z$. 

One can now start with this data $(Z,\omega)$: a commutative monoid and a map of sets $\omega$. Given a collection $u$ of nested circles in the plane, one can recursively evaluate $(Z,\omega)$ on $u$ to get an element $\alpha(u)\in Z$. 

To $(Z,\omega)$ one can assign several monoidal categories similarly to \eqref{overview-diag}. The skein category $\wSU_\omega$ has objects $n\ge 0$ and morphisms given by planar diagrams of arcs up to isotopy, with elements of $Z$ floating in the regions of the diagram. 

From $\wSU_{\omega}$ we can pass to the gligible quotient category $\wU_{\omega}$ by identifying morphisms $u_1,u_2$ if closing them by any annular diagram $v$ gives equal elements of $Z$, with $\alpha(vu_1)=\alpha(vu_2)$. One can then form Karoubi closures as well.  

Endomorphisms of the object $n$ in $\wSU_{\omega}$ constitute a monoid which is analogous to or generalizes the Temperley--Lieb monoid. The latter has $(n,n)$-crossingless matchings as its elements, with the product given by concatenation with consequent removal  of closed components (circles). 

\vspace{0.1in} 

A spherical case of this construction would consist of data $(Z,\omega,\varepsilon)$ with $Z$ and $\omega$ as before, and a map of sets $\varepsilon\colon  Z\lra W$, from $Z$ to a set $W$, subject to the sphericality condition
\begin{equation}
    \varepsilon(\omega(z_1)z_2)= \varepsilon(\omega(z_2)z_1), \ \ \ z_1,z_2 \in Z. 
\end{equation}
The pairing now becomes a symmetric pairing on the product of a set with itself, with both diagrams in the disk (rather than one a disk diagram the other an annular diagram). We leave the details to the reader. 

\vspace{0.1in} 

Likewise, the analogue of a series will be a formal sum 
\begin{equation} \label{eq_alpha_s}
  \alpha = \sum_{u\in \wU^0_0} \alpha(u) u, \ \  \alpha(u)\in Z,
\end{equation} 
We define the notion of \emph{recognizable tree (or forest) series} by requiring $A(0)$ to be finitely-generated over $Z$. In the spherical case, one can replace $Z$ in (\ref{eq_alpha_s}) by a set $W$. These structures are related to special cases of bottom-to-top tree automata, see the references in the previous section, and to suitable monoidal envelopes of such automata, analogous to monoidal envelopes of FSA sketched in~\cite{Kh3}. 

\vspace{0.1in} 

Assume given $(Z,\omega)$ as above, with a finite commutative monoid $Z$.  
Given elements $a,b\in \wB^{2n}_0$ (two crosssingless matchings of $2n$ points) their pairing $\overline{b}a$ is a collection of circles in the plane. This collection has a (generally non-unique) minimal presentation, 
$\overline{b}a=\overline{c}c$ for some $c \in \wB^{2m}_0$ with $m\le n$ and $m$ minimal with this property. Figure~\ref{fig_matchingminimal} shows an example of such a minimal presentation.

\begin{figure}[htb]
\begin{center}
\begingroup%
  \makeatletter%
  \providecommand\color[2][]{%
    \errmessage{(Inkscape) Color is used for the text in Inkscape, but the package 'color.sty' is not loaded}%
    \renewcommand\color[2][]{}%
  }%
  \providecommand\transparent[1]{%
    \errmessage{(Inkscape) Transparency is used (non-zero) for the text in Inkscape, but the package 'transparent.sty' is not loaded}%
    \renewcommand\transparent[1]{}%
  }%
  \providecommand\rotatebox[2]{#2}%
  \newcommand*\fsize{\dimexpr\f@size pt\relax}%
  \newcommand*\lineheight[1]{\fontsize{\fsize}{#1\fsize}\selectfont}%
  \ifx\svgwidth\undefined%
    \setlength{\unitlength}{308.25020126bp}%
    \ifx\svgscale\undefined%
      \relax%
    \else%
      \setlength{\unitlength}{\unitlength * \real{\svgscale}}%
    \fi%
  \else%
    \setlength{\unitlength}{\svgwidth}%
  \fi%
  \global\let\svgwidth\undefined%
  \global\let\svgscale\undefined%
  \makeatother%
  \begin{picture}(1,0.35821997)%
    \lineheight{1}%
    \setlength\tabcolsep{0pt}%
    \put(0,0){\includegraphics[width=\unitlength,page=1]{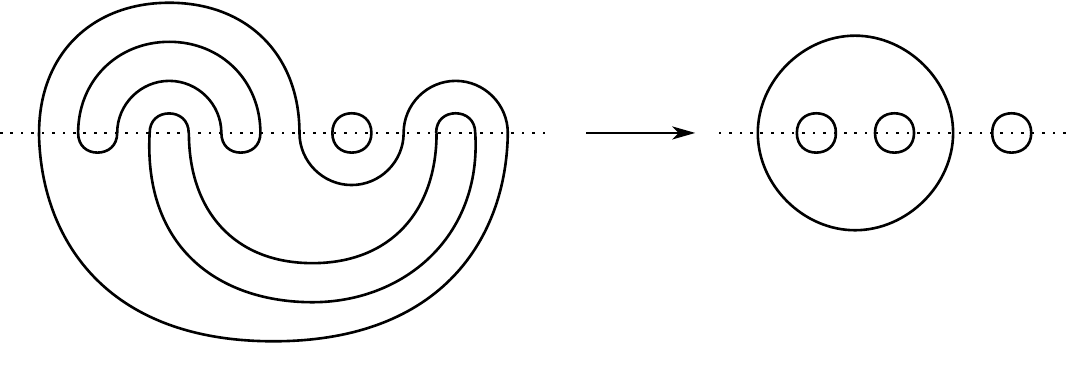}}%
    \put(0.30413604,0.31916195){\color[rgb]{0,0,0}\makebox(0,0)[lt]{\lineheight{1.25}\smash{\begin{tabular}[t]{l}$\overline{b}$\end{tabular}}}}%
    \put(0.74280218,0.34343133){\color[rgb]{0,0,0}\makebox(0,0)[lt]{\lineheight{1.25}\smash{\begin{tabular}[t]{l}$\overline{c}$\end{tabular}}}}%
    \put(0.76642287,0.08801851){\color[rgb]{0,0,0}\makebox(0,0)[lt]{\lineheight{1.25}\smash{\begin{tabular}[t]{l}$c$\end{tabular}}}}%
    \put(0.35279765,0.00286078){\color[rgb]{0,0,0}\makebox(0,0)[lt]{\lineheight{1.25}\smash{\begin{tabular}[t]{l}$a$\end{tabular}}}}%
  \end{picture}%
\endgroup%
 
\caption{A minimal presentation for the diagram on the left involving crossingless matchings of $2n$ points, with $n=7$ and $m=4$.}
\label{fig_matchingminimal}
\end{center}
\end{figure}
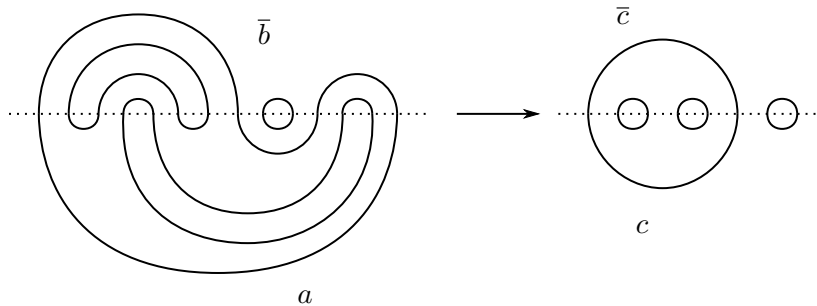

The data $(Z,\omega)$ can then be used to evaluate $\overline{b}a$ to an element 
$\alpha(\overline{b}a)\in Z$. In this sense, $\overline{b}a=\overline{c}c$ can be viewed as a toy instance of computation data, with just two operations: commutative multiplication in $Z$ and the endomorphism $\omega$. 
Although an efficient way to record this data is via a  crossingless matching $c$, storing the same data in two separate locations as matchings $a$ and $b$ allows to hide the original program or intended computation, until $a$ and $b$ are brought together into $\overline{b}a$. 
Starting with $c\in \wB^{2m}_0$, one can randomly represent $\overline{c}c$ as $\overline{b}a$ for $a,b\in \wB^{2n}_0$ with $n=\lfloor{\lambda m}\rfloor$ for some $\lambda>1$. 
It is clear that $c$ and the evaluation $\alpha(\overline{b}a)$ will be hard to guess given access to only $a$ or $b$. 

It may  be interesting to find and study similar  factorizations of programs or computations beyond this toy case, including for arbitrary boolean networks.

%
%

\section{Examples and derivation diagrammatics} \label{sec_adj_ex}

In this section, we consider examples of categories $\wSU_\omega$ and $\wU_\omega$ in more detail: 
\begin{itemize}
    \item Section~\ref{sec_TL} treats the case of a one-dimensional $\kk$-algebra $Z$, when the categories we consider are the Temperley--Lieb categories and their quotients by negligible ideals. 
    \item Section~\ref{subset_ss_2d} and~\ref{subset_ss_spherical} consider the class of examples where the evaluation is spherical and $Z$ is a semisimple algebra (product of base fields). 
\end{itemize}
Then, in Section~\ref{sec:derdiags}, we briefly discuss toy examples of diagrammatics for rings of operators on commutative rings.

\subsection{The one-dimensional case: Temperley--Lieb algebras, meander determinants, and quantum \texorpdfstring{$sl(2)$}{sl(2)}}\label{sec_TL}

Let $\kk$ be an algebraically closed field of characteristic zero and consider the case of the Tem\-per\-ley--Lieb category $TL(d)$ from Example \ref{ex:TL}, with endomorphism rings of objects isomorphic the Temperley--Lieb algebras \cite{Jo1, Jo5}. $TL(d)$ is the skein category $\wSU_{d}$ and associated to the circular triple $(\kk,\kk,\omega)$, where $\omega$ is the multiplication by $d\in \kk$. We want to study the gligible quotient category $\wU_{d}$ of $\wSU_d$ and the state spaces for this class of examples. Note that the circular triple $(\kk,\kk,d)$ is $\kk$-spherical, see Section \ref{subsec_arcs}.

The evaluation $\alpha_d$ associated to this circular triple is given by $\alpha_d(u)=d^{\kappa(u)}$, where $\kappa(u)$ is the number of circles in $u$. For brevity, we shorten the associated category $\wU_{\alpha_d}$ to $\wU_d$. 

Consider the pairing on $\wU_0^{2n}\times \wU_0^{2n}$  given by 
\begin{align}
(a,b)=\cF_\kk(\ov{a}b)=\alpha_d(\ov{a}b)=d^{\kappa(\ov{a}b)},
\end{align}
where $\kappa(\ov{b}a)$ is the number of circles in the object $\ov{b}a\in \wU_0^0$. This pairing extends linearly to $\Bbbk\wU_0^{2n}\times \Bbbk\wU_0^{2n}$. 

A spanning set for $\wU_{d}^{2n}$ is given by the elements of $\wB_0^{2n}$, the arc diagrams of $2n$ points. There are $c_n$ such diagrams. We can restrict the pairing $(\,,\,)$ to this spanning set $\wB_0^{2n}$. The associated matrix of the pairing $(\,,\,)$ on $\wB_0^{2n}$ is given by the \emph{Meander matrix} $\cG_{2n}(d)$, where
\begin{equation*}
    \big(\cG_{2n}(d)\big)_{a,b}=d^{\kappa(\ov{b}a)}, \ \ \ \  \text{for $a,b\in \wB_0^{2n}$.}
\end{equation*}
We refer to \cite{DiF,DGG} for introductions to meander matrices and their determinants. The following result appears in \cite{DiF,DGG}.

\begin{theorem}\label{prop_DiF}
The set $\wB_0^{2n}$ is a basis for $\wU_d^{2n}$ if and only if $d\neq q+q^{-1}$ for $q$ a root of unity of order less or equal to $n+1$. In this case, $\wU_d=\wSU_d$.
\end{theorem}
\begin{proof}
An element $v\in \Bbbk \wB_0^{2n}$ is in the kernel of the pairing if and only if it is in the kernel of the matrix $\cG_{2n}(q)$. It was shown in \cite{Ma}, \cite[Theorem~1]{DiF} that the determinant of this matrix is a product of the Chebychev polynomials $U_m(q)$, for $m\leq n$, with certain powers detailed in \cite[Section~5.2]{DGG}. Thus, the pairing is non-degenerate if and only if $q$ is not a root of one of these polynomials. Using \cite[Equation (5.2)]{DGG}, the roots of $U_m(x)$ are given by
$$2\cos\Big( \frac{k}{m+1}\pi\Big),\qquad m=1,\ldots, n,\qquad k=1,\ldots, m,$$
as claimed.

The morphism spaces $\wU_{d,n}^m$ of the gligble quotient category are isomorphic to $\wU_d^{n+m}$, see Section \ref{subsec_arcs}, and thus zero if $n+m$ is odd and given by the state spaces if $n+m$ is even. If $q$ is not a root of unity of order at most $n+1$ then the pairing $(\,,\,)$ on these state is non-degenerate and hence $\wU_{d,n}^m=\wSU_{d,n}^m$.
\end{proof}

In degenerate case, with $d=q+q^{-1}$, for $q^n=1$, the algebra $TL_n(d)$ is
isomorphic to the algebra of endomorphisms of $V_1^{\otimes n}$ for the fundamental representation $V_1$ of the small quantum group $u_q(\mathfrak{sl}_2)$, see \cite[Section~1.3]{FK} for details. The category $TL(d)=\wSU_d$ is a non-semisimple category and $\wU_d$ is its semisimplification. This semisimplification is used in the Witten--Reshetikhin--Turaev topological field theory \cite{Wi,RT} and related to the Jones polynomial of knots \cite{Jo2,Jo6}.


\subsection{The semisimple two-dimensional \texorpdfstring{$\kk$-spherical}{k-spherical} case}\label{subset_ss_2d}
Assume given matrices 
\begin{equation}
    a = \begin{pmatrix}
     a_{11} &  a_{12} \\
     a_{21} & a_{22}
    \end{pmatrix}, \ \  b= \begin{pmatrix}
     b_1 &  0 \\
     0 & b_2
    \end{pmatrix},
\end{equation}
with  entries in $\kk$. Consider a 2-dimensional semisimple commutative Frobenius algebra $Z=\kk e_1\times \kk e_2$, where  $e_1,e_2$ are mutually-orthogonal idempotents and the trace map 
$\varepsilon(e_i)=b_i$, $i=1,2.$ The Frobenius condition is equivalent  to $b_i\not= 0$, $i=1,2$. Assume that the endomorphism $\omega$ of the vector space $Z$ is given by the matrix $a$ in the basis $(e_1,e_2)$, so that 
\[ 
\omega(e_1,e_2) = (e_1,e_2)a = 
(a_{11}e_1 + a_{21}e_2, a_{12}e_1 + a_{22} e_2).
\] 
The condition that this data is $\kk$-spherical is equivalent to symmetricity of the matrix 
\[  ba = \begin{pmatrix}
     b_1 &  0 \\
     0 & b_2
    \end{pmatrix}
    \begin{pmatrix}
     a_{11} &  a_{12} \\
     a_{21} & a_{22}
    \end{pmatrix} = 
    \begin{pmatrix}
     b_1 a_{11} &  b_1 a_{12} \\
     b_2 a_{21} & b_2 a_{22}
    \end{pmatrix}. 
\] 
In turn, this is equivalent to the single equation
\begin{equation} \label{eq_ab} b_1 a_{12} = b_2 a_{21}. 
\end{equation} 

A crossingless matching $u\in \wB_0^{2n}$ together with an assignment $u'$ of numbers $1$ or $2$ to each of the $n+1$ regions of $c$ in the disk gives an element of the morphism space $\Hom(0,2n)$ of the skein category $\wSU_{\alpha}$. Assigning the number $i$ means placing the idempotent $e_i$ in the corresponding region of the diagram. 

The above elements constitute a basis of the space of morphisms from  $0$ to $2n$ in  the skein category for $\alpha$. In particular, 
\begin{equation}\label{dimHomSU}
\dim (\Hom_{\wSU_{\alpha}}(0,2n)) = \frac{2^{n+1}}{n+1}\binom{2n}{n}
\end{equation}
Using duality, the same formula gives dimension of morphism spaces in the skein category from $m$ to $2n-m$, for $0\le m\le 2n$. 
Furthermore, the elements $(u,u')$ as above constitute a spanning set of the morphism space from $0$ to $2n$ in the gligible quotient category $\wU_{\alpha}$ or, equivalently, a spanning set of the state space $A(n)$. 

The inner product $(a,b)$ of two such diagrams is zero unless the idempotents in the regions of  $a$ and $b$ match on each interval of the common boundary circle. There are $2n$ intervals there, and each one carries an induced coloring by an element of  $\{1,2\}$ coming from the labels of the regions. Both $a$ and $b$ induce such a coloring of the $2n$ intervals, and $(a,b)=0$ unless the induced colorings coincide. 

Consequently the Gram matrix for the pairing in this spanning set is block-diagonal, with $2^{2n}$ blocks. Each block is a  square matrix of the size at most the $n$-th Catalan number. This block decomposition simplifies the computation of the determinant and of the state spaces $A(n)$. In particular, $A(n)$ decomposes into a direct sum of $2^{2n}$ subspaces (some may be trivial), one for each $\{1,2\}$ colorings of the $2n$ segments on the circle. Note that, for some sequences or colors, such as $1122$ (which we may also write as $1^22^2$), no matching respects the sequence in the sense that the corresponding element in $\wSU_{0}^{2n}$ is zero. One can either not consider these cases or list them as giving $0\times 0$ blocks each with determinant $1$.

For example, the state space $A(1)$ has a spanning set $v_{11},v_{12},v_{21},v_{22}$ that consists of elements shown in Figure~\ref{fig_cm3}. Each one is a single arc, constituting the unique matching of two points, together with an assignment of $1$ or $2$ to each of the two regions of the  disk or lower half-plane.

\begin{figure}[htb]
\begin{center}
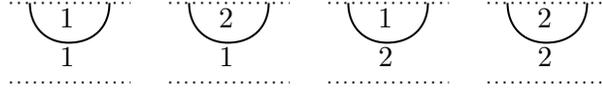
\caption{The vectors  $v_{11},v_{12},v_{21},v_{22}$ that span $A(1)$. The labels $1,2$ denote idempotents $e_1,e_2$ placed in the corresponding regions. }
\label{fig_cm3}
\end{center}
\end{figure}

The Gram matrix is diagonal in this basis and given by 
\[  
\begin{pmatrix} b_1 a_{11} & 0  &  0 & 0 \\
0 & b_1 a_{12} &  0 & 0 \\
0 & 0 & b_2 a_{21}  & 0 \\
0 & 0 & 0 & b_2 a_{22}
\end{pmatrix}.
\]
Notice that the two middle diagonal entries are equal, due to (\ref{eq_ab}).
We see that the vectors $v_{11},v_{12},v_{21},v_{22}$ constitute a basis of $A(1)$, unless one of $a_{ij}$ is zero.

If one of $a_{12},a_{21}$ is zero, the other one is zero as well, $\omega$ stabilizes each of $\kk e_i$, $i=1,2$, and the system fully decouples and becomes the direct sum of two one-dimensional systems, each described by the Templerley--Lieb category with parameters $a_{00}$ and $a_{11}$.

\begin{figure}[htb]
\begin{center}
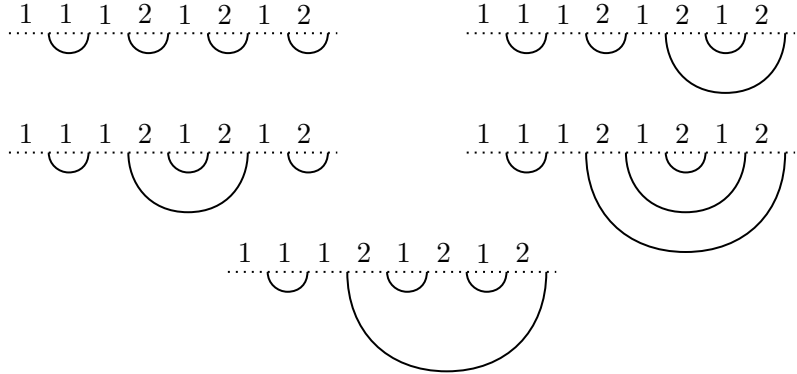
\caption{The five crossingless matchings compatible with the  sequence $1^321212$.}
\label{fig_cm4}
\end{center}
\end{figure}

\begin{table}[tbh]
\begin{center}
\begin{tabular}{ c| c c }
Labels & \# (non-zero diagrams) &Gram determinant\\
\hline
$1^4$ & $2$ &       $b_1^2a_{11}^2(a_{11}-1)(a_{11}+1)$\\ 
$1^32$ & $1$ &      $b_1a_{11}a_{12}$ \\  
 $1^2 2^2$ & $0$ &  $1$  \\
 $1212$ & $2$ & $b_1a_{12}^2(a_{12}a_{21}-1)$
\end{tabular}
\end{center}
    \caption{Determinants of the blocks of the Gram matrix for $n=2$.}
    \label{tab:gramn=2}
\end{table}


\begin{table}[tbh]
\begin{center}
\begin{tabular}{ c| c c }
Labels & \# (non-zero diagrams) &Gram determinant\\
\hline
$1^6$           & 5 & $b_1^5a_{11}^5(a_{11}-1)^4(a_{11}+1)^4(a_{11}^2-2)$ \\ 
$1^5 2$         & 2 & $b_1^2a_{11}^2a_{12}^2(a_{11}-1)(a_{11}+1)$\\  
$1^4 2^2$       & 0 & $1$ \\
$1^3 2 1 2$     & 2 & $b_1^2a_{11}^2a_{12}^2(a_{11}-1)(a_{11}+1)$  \\
$1^2 2 1^2 2$   & 1 & $b_1a_{11}a_{12}^2$\\
$1^3 2^3    $   & 1 & $b_1a_{11}a_{12}a_{22}$\\
$1^2 2 1 2^2$   & 0 & $1$\\
$121212$        & 5 & $b_1^5a_{12}^5(a_{12}a_{21}-1)^4(a_{12}a_{21}-2)$
\end{tabular}
\end{center}
    \caption{Determinants of the blocks of the Gram matrix for $n=3$.}
    \label{tab:gramn=3}
\end{table}

We have computed the determinants of Gram matrices for $n=2,3,4,5$ (corresponding to diagrams with $4,6,8,10$ endpoints, respectively) and all possible length $2n$ sequences of $1,2$, up to cyclic order and reflection (since these transformations do not change the determinants, nor the state spaces). As an example, the four non-zero diagrams for the sequence $1^321212$ are shown  in Figure~\ref{fig_cm4}. 
Furthermore, the symmetry interchanging $1,2$ in a sequence corresponds to transposing indices $1,2$ in all  $a_{ij}$ and $b_i$ that appear in the Gram determinant for the sequence --- this symmetry is also taken into account to reduce the number of cases in the tables. The determinants of the blocks of the Gram matrices for $n=2,3,4$ are summarized in Tables \ref{tab:gramn=2}, \ref{tab:gramn=3}, \ref{tab:gramn=4}, and a partial list of Gram determinants for $n=5$ in Table \ref{tab:gramn=5}. Relation (\ref{eq_ab}) allows one to rewrite the terms in the product formulas for some determinants in several different ways.


\begin{table}[tbh]
\begin{center}
\begin{tabular}{ c| c c }
Labels & \# (non-zero diagrams) &Gram determinant\\
\hline
$1^8$             & 14&  \begin{tabular}{@{}c@{}}$b_{{1}}^{14}a_{{11}}^{14}\left( a_{{11}}^{2}+a_{{11}}-1 \right)  \left( a_{{11}}^{2}-a_{{11}}-1 \right)$\\$\cdot   \left( a_{{11}}-1 \right) ^{13} \left( a_{{11}}+1 \right) ^{13}\left( a_{{11}}^{2}-2 \right) ^{6}$ \end{tabular}\\
$1^72$            & 5 & $b_{{1}}^{5}a_{{11}}^{5}a_{{12}}^{5} 
\left( a_{{11}}-1 \right) ^{4} \left( a_{{11}}+1 \right) ^{4}\left( a_{{11}}^{2}-2 \right) $ \\
$1^62^2$          & 0 & $1$\\
$1^5 2 1 2$       & 4 & $b_{{1}}^{4}a_{{11}}^{4}a_{{12}}^{4}\left( a_{{11}}-1 \right) ^{2} \left( a_{{11}}+1 \right) ^{2}  \left( a_{{12}}a_{{21}}-1 \right) ^{2}$\\
$1^4 2 1^2 2$     & 2 & $b_{{1}}^{2} a_{{11}}^{2}a_{{12}}^{4} \left( a_{{11}}^{2}-1 \right)$\\
$1^3 2 1^3 2$     & 3 & $b_{{1}}^{3}a_{{11}}^{4}a_{{12}}^{4} \left( a_{{11}}-1 \right) \left( a_{{11}}+1 \right)  \left( a_{{12}}a_{{21}}-1 \right)$\\
$1^52^3$          & 2 & $b_{{1}}^{2}a_{{11}}^{2}a_{{12}}^{2}a_{{22}}^{2}
\left( a_{{11}}-1 \right)\left( a_{{11}}+1 \right) $\\
$1^42^212$        & 0 & $1$\\
$1^321^22^2$      & 0 & $1$\\
$1^321212$        & 5 & $b_{{1}}^{5}a_{{11}}^{5}a_{{12}}^{5}\left( a_{{12}}a_{{21}}-1 \right) ^{4} \left( a_{{12}}a_{{21}}-2 \right) $\\
$1^221^2212$      & 2 & $b_{{1}}^{2}a_{{11}}^{2}a_{{12}}^{4} \left( a_{{12}}a_{{21}}-1 \right) $\\
$1^4 2^4$         & 0 & $1$\\
$1^3 2 1 2^3$     & 2 & $b_{{1}}^{2}a_{{11}}^{2}a_{{12}}^{2}a_{{22}}^{2} \left( a_{{12}}a_{{21}}-1 \right)$\\
$1^2 2 1^2 2^3$   & 1 & $b_{{1}}a_{{11}}a_{{12}}^{2}a_{{22}}$ \\
$1^3 2^2 1 2^2$   & 1 & $b_{{1}}a_{{11}}a_{{12}}a_{{22}}a_{{21}}$ \\
$1^22^2 1^22^2$   & 0 & $1$\\
$1^2 2^2 1212$    & 0 & $1$ \\
$1^2212^212$      & 0 & $1$ \\
$12121212$        & 14 & $b_{{1}}^{14}a_{{12}}^{14}  \left( a_{{12}}a_{{21}}-1 \right) ^{13} \left( a_{{12}}a_{{21}}-2 \right) ^{6}\left(a_{{12}}^{2}a_{{21}}^{2}-3\,a_{{12}}a_{{21}}+1 \right)$ \\
\end{tabular}
\end{center}
    \caption{Determinants of the blocks of the Gram matrix for $n=4$.}
    \label{tab:gramn=4}
\end{table}

\begin{table}[tbh]
\begin{center}
\begin{tabular}{ c| c c }
Labels & \# (non-zero diagrams) &Gram determinant\\
\hline
 $1^{10}$            & $42$ &\begin{tabular}{@{}c@{}}$b_1^{42}a_{{1,1}}^{42}
\left( a_{{1,1}}-1 \right) ^{41} \left( a_{{1,1}}+1 \right) ^{41}\left( a_{{1,1}}^{2}-2 \right) ^{26}\left( a_{{1,1}}^{2}-3 \right)  $\\$\cdot \left( a_{{1,1}}^{2}+a_{{1,1}}-1 \right) ^{8}\left( a_{{1,1}}^{2}-a_{{1,1}}-1 \right) ^{8}
$\end{tabular} \\
$1^92$            & $14$ &\begin{tabular}{@{}c@{}}
$b_1^{14}a_{{1,1}}^{14}a_{{1,2}}^{14}\left( a_{{1,1}}^{2}+a_{{1,1}}-1 \right)  \left( a_{{1,1}}^{2}-a_{{1,1}}-1 \right) $\\$\cdot
\left( a_{{1,1}}^{2}-2 \right) ^{6} \left( a_{{1,1}}-1 \right) ^{13} \left( a_{{1,1}}+1 \right) ^{13}$\end{tabular}\\
$1^521212$            & $10$ &\begin{tabular}{@{}c@{}}
$b_1^{10}a_{{1,1}}^{10}a_{{1,2}}^{10}\left( a_{{1,1}}-1 \right) ^{5} \left( a_{{1,1}}+1 \right) ^{5}\left( a_{{1,2}}a_{{2,1}}-2 \right) ^{2} \left( a_{{1,2}}a_{{2,1}}-1 \right) ^{8}$\end{tabular}\\
$1^321^3212$            & $7$ &\begin{tabular}{@{}c@{}}
$b_1^{7}a_{{1,1}}^{10}a_{{1,2}}^{9} \left( a_{{1,2}}a_{{2,1}}-2 \right) \cdot\left( a_{{1,2}}a_{{2,1}}-1 \right) ^{5} \left( a_{{1,1}}-1 \right) ^{2}\left( a_{{1,1}}+1 \right) ^{2}$\end{tabular}\\$1^32121212$            & $14$ &\begin{tabular}{@{}c@{}}
$b_1^{14}
a_{{1,1}}^{14}a_{{1,2}}^{14}\left( a_{{1,2}}^{2}a_{{2,1}}^{2}-3\,a_{{1,2}}a_{{2,1}}+1 \right)  \left( a_{{1,2}}a_{{2,1}}-2 \right) ^{6}$\\$
\cdot \left( a_{{1,2}}a_{{2,1}}-1 \right) ^{13}$\end{tabular}\\
 $1212121212$            & $42$ &\begin{tabular}{@{}c@{}}$b_{{1}}^{42}a_{{1,2}}^{42}\left( a_{{1,2}}a_{{2,1}}-3 \right)  \left(a_{{1,2}}^{2}a_{{2,1}}^{2}-3\,a_{{1,2}}a_{{2,1}}+1 \right) ^{8}$\\ $\cdot
 \left( a_{{1,2}}a_{{2,1}}-2 \right) ^{26} \left( a_{{1,2}}a_{{2,1}}-1 \right) ^{41}
$\end{tabular} 
\end{tabular}
\end{center}
    \caption{Determinants of some of the blocks of the Gram matrix for $n=5$.}
    \label{tab:gramn=5}
\end{table}

\subsection{The semisimple spherical case} 
\label{subset_ss_spherical} 

Let $Z=\oplus_{i=1}^k \kk e_i$ be a semisimple  $\kk$-algebra of dimension $k$ over  an algebraically closed field $\kk$ of characteristic zero, with minimal idemptotents $e_1,\ldots, e_k$. Assume given a $\kk$-spherical circular triple $(\kk,Z,\omega)$ with $\omega$ described by the $k\times k$ matrix $a=(a_{ij})$, so  that $\omega(e_j)=\sum_i a_{ij}e_i$,  and the trace form $\varepsilon$ given by $b=(b_i)$ in the basis $(e_i)_{i=1}^k$. 

Denote by $\alpha$ the map $\kk\wU_0^0\xrightarrow{\cF_Z} Z\xrightarrow{\varepsilon} \kk$. The $\kk$-sphericality condition from Definition \ref{def_Rspherical} is equivalent to 
\begin{align}\label{aijspherical}
    b_ia_{ij}&=b_ja_{ji}, &\text{for all }i,j.
\end{align}

Consider the ring $R'=\kk[a_{ij},b_i^{\pm 1}]/J$ of polynomials in $a_{ij}$ and Laurent polynomials in $b_i$ modulo the ideal $J$  generated by $b_ia_{ij}-b_ja_{ji}$ for $1\le i<j\le k$. The ring $R'$ is naturally  isomorphic to the ring $\kk[a_{ij}',b_i^{\pm 1}]$ for variables $a_{ij}'$ with $1\le i\le  j\le k$ and $b_i$ for $1\le i\le k$, by denoting $a_{ij}'=b_i a_{ij} = b_j a_{ji}$. In particular, $R'$ is an integral domain. The Gram determinants of the associated pairings in Proposition~\ref{prop_ssTL} can be viewed as an element of $R'$. The ring $R'$ is bigraded, with $\deg(a_{ij}')=(1,0)$ and $\deg(b_i)=(0,1)$. Consequently, $\deg(a_{ij})=(1,-1)$. 

\begin{proposition}\label{prop_ssTL} 
If $a_{ij}\in \kk,b_i\in \kk\setminus\{0\}$ are generic elements of an algebraically closed field $\kk$, the skein category $\wSU_\alpha$ is isomorphic to the gligible quotient category $\wU_\alpha$. Equivalently, for each $k$ the Gram determinant is a non-zero element of the integral domain $R'$ defined above. 
\end{proposition}
\begin{proof}
The coefficients $a_{ij}, b_i$ being generic implies that $Z$ is $\omega$-generated.
This is equivalent to being able to write each $e_i$ as a linear combination of closed circular diagrams. (In the non-generic case, we can pass to the subalgebra of $Z$ generated by $1$ and closed under $\omega$ to achieve this.) 

 A diagram $u$ in $\wB^{2n}_0$  (a crossingless matching of $2n$ points) gives rise to an element  of $\wU^{2n}_{\alpha}$, also denoted by $u$. Given $u$ and a region $r$ of $u$, denote by $(u,r(i))$ the diagram $u$ with the idempotent $e_i$ placed in this region $r$. We view $(u,r(i))$ as a vector in $\wU_{\alpha}^{2n}$, so that $u=\sum_{i=1}^k (u,r(i))$, where $k=\dim(Z)$.

Minimal idempotents $e_i$ may be placed in more than one region of $u$. In particular, minimal idempotents may be assigned to all regions of $u$, resulting in a corresponding vector in $\wU^{2n}_{\alpha}$. Since  a  crossingless matching of $2n$  points has $n+1$ regions, the element $u$ can then be written as a sum of $k^{n+1}$ terms, each one carrying an assignment of minimal idempotents to all $n+1$ regions of $u$. 

\medspace

The space $\wU_{\alpha,0}^{2n}$ has a spanning set given by diagrams of crossingless matchings $u\in \wB^{2n}_0$ together with a choice of idempotent $e_i$ for each region of $u$. We denote such a  vector by $(u,c)$, where $c$ is the idempotent assignment, and can write  $u=\sum_c (u,c)$, the sum over all $k^{n+1}$ assignments.  

Each region of a matching $u$ contains one or more segments on the boundary of $u$. Labelling these regions by idempotents $e_i$ induces a labelling of the corresponding segments by the same index $i$ (or by the idempotent $e_i$).

An assignment $c$ of minimal idempotents to all regions of $u$  induces a sequence of indices $(i_1,\dots, i_{2n})$, the labels of the $2n$ segments on the boundary of $u$, going from left to  right. Here we view the boundary as the real line and crossingless matching $u$ as lying in the bottom half-plane, see Figure~\ref{fig_matchingseq} for an example.

\begin{figure}[htb]
\begin{center}
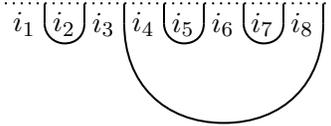
\caption{Labelling of a crossingless matching of $8$ points by a sequence $(i_1,\ldots,i_8)$ corresponding to idempotents placed in the boundary regions.}
\label{fig_matchingseq}
\end{center}
\end{figure}

Given two such vectors $(u,c)$ and $(v,c')$, with $u,v\in \wB^{2n}_0$, their inner product is $0$ unless the idempotent assignments $c,c'$ give rise to the same sequence $(i_1,\dots, i_{2n})$ of indices on the boundaries of $u$ and $v$. Consider the Gram matrix of the bilinear form in the spanning set $\{(u,c)\}$ of $\wU^{2n}_0$ for all possible matchings $u$ and labelings $c$. This matrix is block-diagonal, where we sort the rows and columns into blocks according to the induced sequences $(i_1,\dots, i_{2n})$ of labels of boundary segments. Here $1\le i_1,\dots, i_{2n}\le k$.  

There may not be any vectors $(u,c)$ inducing a particular sequence ($(1,2,2,1)$ is an example of such a sequence, for $n=2$). These sequences can be ignored, with the corresponding blocks of size $0 \times 0$ of determinant $1$ by convention. 

\smallskip

Suppose that the pairing $((u,c),(v,c'))$ is non-zero, so that, in particular, $c$ and $c'$ induce the same sequence on the boundary. The circular form $\overline{v}u$ has at most $n$ circles, and has $n$ circles if and only if $u=v$. The labeling $c$ induces a labelling of regions of $\overline{v}u$, also denoted $c$. The element $\alpha((\overline{v}u,c))$ can  be evaluated  inductively on the number of circles in $\overline{v}u$, starting with the innermost circles, 
using the matrix $a=(a_{ij})$ to compute the action of $\omega$ and multiplication in the semisimple algebra $Z$ to reduce the result, at each step, to a  linear combination of circular  forms with one less  circle each and a full idempotent assignment to the regions. When no circles are left, we use the vector $b=(b_i)$ to evaluate each of the  resulting diagrams of the empty plane with an idempotent $e_i$ in it.

\smallskip

We see that the evaluation $\alpha((\overline{v}u,c))$ is a polynomial with each  term of degree one in the $b_i$'s and degree  $m$ in the $a_{ij}$'s, where  $m$ is the number of circles of $\overline{v}u$. Notice that $m\le  n$, with equality occurring if and only if $v=u$. Consequently, each of the diagonal terms of the Gram matrix has degree $n$ in the $a_{ij}'s$, while all off-diagonal terms have strictly lower degrees. 

Switching to the bidegrees, as earlier defined, 
 each diagonal term is a homogeneous element of the ring $R'$  of bidegree $(n,1-n)$. Each off-diagonal term is homogeneous of bidegree $(m,1-m)$ for $m<n$.

\smallskip

To show that the determinant of the Gram matrix is non-zero for generic values of the parameters, it suffices to check that each diagonal entry of the matrix is  non-zero for some $a_{ij}$'s and $b_i$'s. This would imply that each diagonal entry is a non-zero polynomial, necessarily of bidegree $(n,1-n)$. Collapsing bidegree $(n_1,n_2)$ into a single degree $n_1-n_2$ would tell us that the determinant is a polynomial of degree $2n-1$, with a nontrivial top homogeneous term, implying the proposition.

\smallskip

For a given diagonal entry $\alpha(\overline{u}u,c)$ set $a_{ij}=1$ for all $i,j$ and $b_i=1$ for all $i$. Then 
$\alpha(\overline{u}u,c)=1\not= 0$. The proposition follows. 
\end{proof}

We formulate the following conjecture, generalizing the statement for Temperley--Lieb categories and Meander determinants from Theorem \ref{prop_DiF}, regarding degeneracies in the Gram determinants of the paring associated to a  $\kk$-spherical triple for a semisimple $\kk$-algebra. This conjecture have been verified computationally using Maple\textsuperscript{\tiny{TM}} for morphism spaces $\Hom(n,m)$ with $n+m\leq 10$ for a two-dimensional algebra (see Tables \ref{tab:gramn=2}--\ref{tab:gramn=5}).

\begin{con} In the setup of Proposition \ref{prop_ssTL},
if all $b_i\neq 0$, then the factors of the Gram determinants are Chebychev polynomials of the second kind in the variables
    $c_{ij}:=\sqrt{b_i/b_j}a_{ij}.$
    In particular, the only degeneracies occur when $c_{ij}=q+q^{-1}$ for $q$ a root of unity.
\end{con}

The case of a non-semisimple algebra appears very interesting and we hope to look into it in the future.

%
%

 \subsection{Example: derivation diagrammatics}\label{sec:derdiags}
 \quad
 \smallskip
 
Additional properties of  $\omega$ may lead to rules for manipulation and  simplification of  such  diagrams. For instance, assume that set $S\subset A$ generates $A$, that  is, the algebra homomorphism $R[S]\lra A$ is surjective. Furthermore, assume that $\omega=\partial$ is a  derivation, that is, $\partial(ab)=\partial(a)b+a\partial(b)$, see Figure \ref{fig2_4_4a}. Then any diagram reduces to a linear combination of products  of diagrams representing $\partial^n(s)$, for  various $s\in S$ and $n\in \Z_+$,  see  Figure~\ref{fig2_4_4b} for examples. 

\begin{figure}[htb]
\begin{center}
  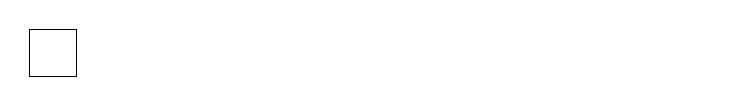
\caption{The Leibniz rule for derivations, where $a$ and $b$ represent arbitrary diagrams.}
\label{fig2_4_4a}
\end{center}
\end{figure}

\begin{figure}
\begin{center}
     \begin{subfigure}[htb]{0.45\textwidth}
      \centering
\begingroup%
  \makeatletter%
  \providecommand\color[2][]{%
    \errmessage{(Inkscape) Color is used for the text in Inkscape, but the package 'color.sty' is not loaded}%
    \renewcommand\color[2][]{}%
  }%
  \providecommand\transparent[1]{%
    \errmessage{(Inkscape) Transparency is used (non-zero) for the text in Inkscape, but the package 'transparent.sty' is not loaded}%
    \renewcommand\transparent[1]{}%
  }%
  \providecommand\rotatebox[2]{#2}%
  \newcommand*\fsize{\dimexpr\f@size pt\relax}%
  \newcommand*\lineheight[1]{\fontsize{\fsize}{#1\fsize}\selectfont}%
  \ifx\svgwidth\undefined%
    \setlength{\unitlength}{38.25001224bp}%
    \ifx\svgscale\undefined%
      \relax%
    \else%
      \setlength{\unitlength}{\unitlength * \real{\svgscale}}%
    \fi%
  \else%
    \setlength{\unitlength}{\svgwidth}%
  \fi%
  \global\let\svgwidth\undefined%
  \global\let\svgscale\undefined%
  \makeatother%
  \begin{picture}(1,1.00000067)%
    \lineheight{1}%
    \setlength\tabcolsep{0pt}%
    \put(0.55909754,0.40176275){\color[rgb]{0,0,0}\makebox(0,0)[lt]{\smash{\begin{tabular}[t]{l}$s$\end{tabular}}}}%
    \put(0,0){\includegraphics[width=\unitlength,page=1]{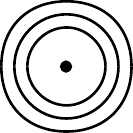}}%
  \end{picture}%
\endgroup%

         \caption{The diagram $\partial^3(s)$}
         \label{fig:del3s}
     \end{subfigure}
          \begin{subfigure}[htb]{0.45\textwidth}
      \centering
\begingroup%
  \makeatletter%
  \providecommand\color[2][]{%
    \errmessage{(Inkscape) Color is used for the text in Inkscape, but the package 'color.sty' is not loaded}%
    \renewcommand\color[2][]{}%
  }%
  \providecommand\transparent[1]{%
    \errmessage{(Inkscape) Transparency is used (non-zero) for the text in Inkscape, but the package 'transparent.sty' is not loaded}%
    \renewcommand\transparent[1]{}%
  }%
  \providecommand\rotatebox[2]{#2}%
  \newcommand*\fsize{\dimexpr\f@size pt\relax}%
  \newcommand*\lineheight[1]{\fontsize{\fsize}{#1\fsize}\selectfont}%
  \ifx\svgwidth\undefined%
    \setlength{\unitlength}{129.52904552bp}%
    \ifx\svgscale\undefined%
      \relax%
    \else%
      \setlength{\unitlength}{\unitlength * \real{\svgscale}}%
    \fi%
  \else%
    \setlength{\unitlength}{\svgwidth}%
  \fi%
  \global\let\svgwidth\undefined%
  \global\let\svgscale\undefined%
  \makeatother%
  \begin{picture}(1,0.23739864)%
    \lineheight{1}%
    \setlength\tabcolsep{0pt}%
    \put(0,0){\includegraphics[width=\unitlength,page=1]{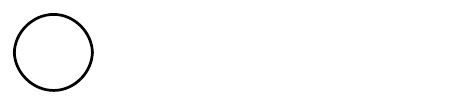}}%
    \put(0.11378178,0.10433215){\color[rgb]{0,0,0}\makebox(0,0)[lt]{\smash{\begin{tabular}[t]{l}$s_1$\end{tabular}}}}%
    \put(0,0){\includegraphics[width=\unitlength,page=2]{derdiags.pdf}}%
    \put(0.37434308,0.1084698){\color[rgb]{0,0,0}\makebox(0,0)[lt]{\smash{\begin{tabular}[t]{l}$s_1$\end{tabular}}}}%
    \put(0,0){\includegraphics[width=\unitlength,page=3]{derdiags.pdf}}%
    \put(0.63495855,0.10422537){\color[rgb]{0,0,0}\makebox(0,0)[lt]{\smash{\begin{tabular}[t]{l}$s_2$\end{tabular}}}}%
    \put(0,0){\includegraphics[width=\unitlength,page=4]{derdiags.pdf}}%
    \put(0.86650859,0.10640809){\color[rgb]{0,0,0}\makebox(0,0)[lt]{\smash{\begin{tabular}[t]{l}$s_3^2$\end{tabular}}}}%
    \put(0,0){\includegraphics[width=\unitlength,page=5]{derdiags.pdf}}%
  \end{picture}%
\endgroup%

         \caption{The diagram $\partial^2(s_1)\partial(s_1)\partial(s_2)s_3^2$}
         \label{fig:derdiags}
     \end{subfigure}
\caption{Examples of derivation diagrams.}
\label{fig2_4_4b}
\end{center}
\end{figure}

\ssubsection{The first Weyl algebra and an $\mathfrak{sl}_2$-action on polynomials}
Consider a basic  example when the center $Z(\mcA)=\kk[x]$ is the  polynomial ring in one variable $x$  over a field $\kk$. For this diagrammatics we do not need a pair $(\mcA,F)$  of a category and a self-adjoint  functor and  can  just work with  a  commutative algebra $Z$ and a linear  operator  on it. For  this  example, $Z=\kk[x]$  and  the operator $\omega$ is the differentiation 
\begin{equation}\label{eq_A1}
    \partial\colon  Z  \lra Z, \ \partial(x)=1, \ \partial(ab)=\partial(a)b + a\partial(b) , \ a,b\in  Z. 
\end{equation}
Then any  diagram  of  nested  circles  and dots reduces to  a polynomial in $Z$ upon  repeated  simplification. 

Let 
\begin{equation}
A_1 = Z\langle \partial \rangle = \kk\langle x,\partial \rangle /(\partial x - x\partial -1)
\end{equation} 
be the first  Weyl algebra, that is, the  algebra of polynomial differential operators on  one  variable  $x$. This algebra acts on $Z=\kk[x]$, with $x$  acting by adding a dot outside of a  diagram and  $\partial$ acting by wrapping a circle around a diagram, with the  relation in Figure~\ref{fig2_4_5}.

Note that to any commutative algebra and a linear operator on it we can associate a monoidal category as in Section~\ref{sec_pairings}. We do not know the defining relations in that category for the pair $(\kk[x],\partial)$. 

\begin{figure}
\begin{center}
      \begin{subfigure}[htb]{0.45\textwidth}
      \centering
\begingroup%
  \makeatletter%
  \providecommand\color[2][]{%
    \errmessage{(Inkscape) Color is used for the text in Inkscape, but the package 'color.sty' is not loaded}%
    \renewcommand\color[2][]{}%
  }%
  \providecommand\transparent[1]{%
    \errmessage{(Inkscape) Transparency is used (non-zero) for the text in Inkscape, but the package 'transparent.sty' is not loaded}%
    \renewcommand\transparent[1]{}%
  }%
  \providecommand\rotatebox[2]{#2}%
  \newcommand*\fsize{\dimexpr\f@size pt\relax}%
  \newcommand*\lineheight[1]{\fontsize{\fsize}{#1\fsize}\selectfont}%
  \ifx\svgwidth\undefined%
    \setlength{\unitlength}{114.76500371bp}%
    \ifx\svgscale\undefined%
      \relax%
    \else%
      \setlength{\unitlength}{\unitlength * \real{\svgscale}}%
    \fi%
  \else%
    \setlength{\unitlength}{\svgwidth}%
  \fi%
  \global\let\svgwidth\undefined%
  \global\let\svgscale\undefined%
  \makeatother%
  \begin{picture}(1,0.33329008)%
    \lineheight{1}%
    \setlength\tabcolsep{0pt}%
    \put(0.16080925,0.14645675){\color[rgb]{0,0,0}\makebox(0,0)[lt]{\smash{\begin{tabular}[t]{l}$f$\end{tabular}}}}%
    \put(0.36339795,0.15124094){\color[rgb]{0,0,0}\makebox(0,0)[lt]{\smash{\begin{tabular}[t]{l}$=$\end{tabular}}}}%
    \put(0.62526851,0.15124131){\color[rgb]{0,0,0}\makebox(0,0)[lt]{\smash{\begin{tabular}[t]{l}$+$\end{tabular}}}}%
    \put(0,0){\includegraphics[width=\unitlength,page=1]{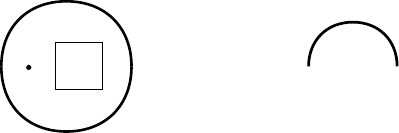}}%
    \put(0.48756389,0.14645675){\color[rgb]{0,0,0}\makebox(0,0)[lt]{\smash{\begin{tabular}[t]{l}$f$\end{tabular}}}}%
    \put(0.84699399,0.14704017){\color[rgb]{0,0,0}\makebox(0,0)[lt]{\smash{\begin{tabular}[t]{l}$f$\end{tabular}}}}%
    \put(0,0){\includegraphics[width=\unitlength,page=2]{delrel1.pdf}}%
  \end{picture}%
\endgroup%

         \caption{The relation $\partial(xf)=f+x\partial(f)$}
         \label{fig:delrel1}
     \end{subfigure}
           \begin{subfigure}[htb]{0.45\textwidth}
      \centering
\begingroup%
  \makeatletter%
  \providecommand\color[2][]{%
    \errmessage{(Inkscape) Color is used for the text in Inkscape, but the package 'color.sty' is not loaded}%
    \renewcommand\color[2][]{}%
  }%
  \providecommand\transparent[1]{%
    \errmessage{(Inkscape) Transparency is used (non-zero) for the text in Inkscape, but the package 'transparent.sty' is not loaded}%
    \renewcommand\transparent[1]{}%
  }%
  \providecommand\rotatebox[2]{#2}%
  \newcommand*\fsize{\dimexpr\f@size pt\relax}%
  \newcommand*\lineheight[1]{\fontsize{\fsize}{#1\fsize}\selectfont}%
  \ifx\svgwidth\undefined%
    \setlength{\unitlength}{75.43097061bp}%
    \ifx\svgscale\undefined%
      \relax%
    \else%
      \setlength{\unitlength}{\unitlength * \real{\svgscale}}%
    \fi%
  \else%
    \setlength{\unitlength}{\svgwidth}%
  \fi%
  \global\let\svgwidth\undefined%
  \global\let\svgscale\undefined%
  \makeatother%
  \begin{picture}(1,0.34839818)%
    \lineheight{1}%
    \setlength\tabcolsep{0pt}%
    \put(0.20527027,0.20402636){\color[rgb]{0,0,0}\makebox(0,0)[lt]{\smash{\begin{tabular}[t]{l}$n$\end{tabular}}}}%
    \put(0.44265639,0.14437054){\color[rgb]{0,0,0}\makebox(0,0)[lt]{\smash{\begin{tabular}[t]{l}$=$\end{tabular}}}}%
    \put(0.65145605,0.14437054){\color[rgb]{0,0,0}\makebox(0,0)[lt]{\smash{\begin{tabular}[t]{l}$n$\end{tabular}}}}%
    \put(0.85031327,0.20367142){\color[rgb]{0,0,0}\makebox(0,0)[lt]{\smash{\begin{tabular}[t]{l}$n-1$\end{tabular}}}}%
    \put(0,0){\includegraphics[width=\unitlength,page=1]{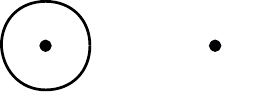}}%
  \end{picture}%
\endgroup%

         \caption{The formula  $\partial(x^n)=n x^{n-1}$}
         \label{fig:delrel2}
     \end{subfigure}
\caption{Derivation diagrammatics, multiplication by $\bullet=x$}
\label{fig2_4_5}
\end{center}
\end{figure}

The Lie algebra  $\mfsl_2$ acts via  polynomial derivations 
\begin{equation}
    E \mapsto x^2 \partial, \ \ H \mapsto 2 x \partial , \ \ F \mapsto -\partial
\end{equation}
on $\mathbb{P}^1$, for $\partial=\frac{\partial}{\partial x}$, and  this  action extends to a homomorphism from $U(\mfsl_2)$ to  $A_1$. The diagrammatic counterpart of the action is shown in Figure~\ref{fig2_4_7}. 

\begin{figure}[htb]
\begin{center}
 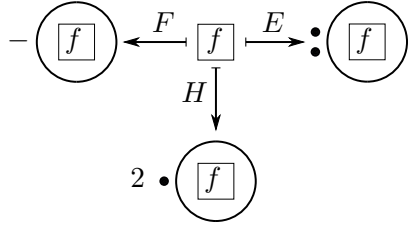
\caption{The action of $E,H,F$ via annular diagrams.}
\label{fig2_4_7}
\end{center}
\end{figure}

A more sophisticated example of an $\mfsl_2$ action via annular diagrams can be found in~\cite{BHLZ}, which comes from the annular closure of categorified quantum $\mfsl_2$. There the action extends to an action of the current algebra. It may be interesting to compare it with the more elementary example here. 

\vspace{0.1in}

\begin{figure}[htb]
\begin{center}
   \begin{subfigure}[htb]{0.45\textwidth}
      \centering
\begingroup%
  \makeatletter%
  \providecommand\color[2][]{%
    \errmessage{(Inkscape) Color is used for the text in Inkscape, but the package 'color.sty' is not loaded}%
    \renewcommand\color[2][]{}%
  }%
  \providecommand\transparent[1]{%
    \errmessage{(Inkscape) Transparency is used (non-zero) for the text in Inkscape, but the package 'transparent.sty' is not loaded}%
    \renewcommand\transparent[1]{}%
  }%
  \providecommand\rotatebox[2]{#2}%
  \newcommand*\fsize{\dimexpr\f@size pt\relax}%
  \newcommand*\lineheight[1]{\fontsize{\fsize}{#1\fsize}\selectfont}%
  \ifx\svgwidth\undefined%
    \setlength{\unitlength}{57.93423127bp}%
    \ifx\svgscale\undefined%
      \relax%
    \else%
      \setlength{\unitlength}{\unitlength * \real{\svgscale}}%
    \fi%
  \else%
    \setlength{\unitlength}{\svgwidth}%
  \fi%
  \global\let\svgwidth\undefined%
  \global\let\svgscale\undefined%
  \makeatother%
  \begin{picture}(1,0.27107638)%
    \lineheight{1}%
    \setlength\tabcolsep{0pt}%
    \put(0,0){\includegraphics[width=\unitlength,page=1]{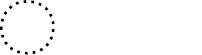}}%
    \put(0.03684507,0.07173839){\color[rgb]{0,0,0}\makebox(0,0)[lt]{\smash{\begin{tabular}[t]{l}$M$\end{tabular}}}}%
    \put(0,0){\includegraphics[width=\unitlength,page=2]{A1mod1.pdf}}%
    \put(0.77582677,0.09713224){\color[rgb]{0,0,0}\makebox(0,0)[lt]{\smash{\begin{tabular}[t]{l}$m$\end{tabular}}}}%
    \put(0.45957319,0.07120159){\color[rgb]{0,0,0}\makebox(0,0)[lt]{\smash{\begin{tabular}[t]{l}$,$\end{tabular}}}}%
  \end{picture}%
\endgroup%

         \caption{Holes in the plane labelled by $M$, $m$}
         \label{fig:A1mod1}
     \end{subfigure}
        \begin{subfigure}[htb]{0.45\textwidth}
      \centering
\begingroup%
  \makeatletter%
  \providecommand\color[2][]{%
    \errmessage{(Inkscape) Color is used for the text in Inkscape, but the package 'color.sty' is not loaded}%
    \renewcommand\color[2][]{}%
  }%
  \providecommand\transparent[1]{%
    \errmessage{(Inkscape) Transparency is used (non-zero) for the text in Inkscape, but the package 'transparent.sty' is not loaded}%
    \renewcommand\transparent[1]{}%
  }%
  \providecommand\rotatebox[2]{#2}%
  \newcommand*\fsize{\dimexpr\f@size pt\relax}%
  \newcommand*\lineheight[1]{\fontsize{\fsize}{#1\fsize}\selectfont}%
  \ifx\svgwidth\undefined%
    \setlength{\unitlength}{174.5360832bp}%
    \ifx\svgscale\undefined%
      \relax%
    \else%
      \setlength{\unitlength}{\unitlength * \real{\svgscale}}%
    \fi%
  \else%
    \setlength{\unitlength}{\svgwidth}%
  \fi%
  \global\let\svgwidth\undefined%
  \global\let\svgscale\undefined%
  \makeatother%
  \begin{picture}(1,0.13521657)%
    \lineheight{1}%
    \setlength\tabcolsep{0pt}%
    \put(0,0){\includegraphics[width=\unitlength,page=1]{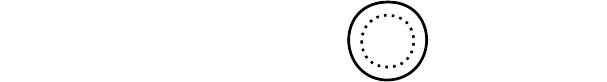}}%
    \put(0.61542337,0.05458139){\color[rgb]{0,0,0}\makebox(0,0)[lt]{\smash{\begin{tabular}[t]{l}$m$\end{tabular}}}}%
    \put(0,0){\includegraphics[width=\unitlength,page=2]{A1mod2.pdf}}%
    \put(0.85052451,0.04563911){\color[rgb]{0,0,0}\makebox(0,0)[lt]{\smash{\begin{tabular}[t]{l}$\del m$\end{tabular}}}}%
    \put(0.7467892,0.05786857){\color[rgb]{0,0,0}\makebox(0,0)[lt]{\smash{\begin{tabular}[t]{l}$=$\end{tabular}}}}%
    \put(0,0){\includegraphics[width=\unitlength,page=3]{A1mod2.pdf}}%
    \put(0.09423414,0.05102129){\color[rgb]{0,0,0}\makebox(0,0)[lt]{\smash{\begin{tabular}[t]{l}$m$\end{tabular}}}}%
    \put(0,0){\includegraphics[width=\unitlength,page=4]{A1mod2.pdf}}%
    \put(0.33474067,0.04993615){\color[rgb]{0,0,0}\makebox(0,0)[lt]{\smash{\begin{tabular}[t]{l}$xm$\end{tabular}}}}%
    \put(0.22683928,0.06216562){\color[rgb]{0,0,0}\makebox(0,0)[lt]{\smash{\begin{tabular}[t]{l}$=$\end{tabular}}}}%
    \put(0.50615099,0.04927435){\color[rgb]{0,0,0}\makebox(0,0)[lt]{\smash{\begin{tabular}[t]{l}$,$\end{tabular}}}}%
  \end{picture}%
\endgroup%

         \caption{The actions of $x$, $\del$}
         \label{fig:A1mod2}
     \end{subfigure}
\caption{Diagrammatics with labels from an $A_1$-module $M$, an element $m$ of $M$, and the action of $x$ and  $\partial$  on $m$. }
\label{fig2_4_6}
\end{center}
\end{figure}

\ssubsection{More general derivation diagrammatics}

Instead of acting on polynomials, we can take any module $M$  over $A_1$, make a hole in the plane to  insert elements $m$ of a module $M$ and have $A_1$ act diagrammatically, via adding a dot (the action of $x$) and circle wrapping (the action of $\partial$), see Figure~\ref{fig2_4_6}.  The relation in Figure~{\scshape \ref{fig:delrel1}}  holds  with an element $m$ of $M$ in place of $f$. 
This  diagrammatics may be useful, for instance,  for analyzing bilinear forms on $M$ with suitable compatibility  conditions on the action of $A_1$. For this application,  one would  glue two planar diagrams  describing the action of $A_1$ on two copies of $M$ into one diagram on the sphere. 

\vspace{0.1in}

More  generally, generalized diagrammatics of this sort may be useful for studying rings of operators acting on commutative algebras. For the Weyl algebra \begin{equation}
A_n := A_1^{\otimes n} = \kk\langle x_1,\dots ,x_n,\partial_1,\dots, \partial_n\rangle / (x_ix_j-x_jx_i,\partial_i \partial_j - \partial_j \partial_i, \partial_i x_j - x_j\partial_i - \delta_{i,j})
\end{equation} 
the diagrammatics will consists of $n$ types of dots colored by $\{1,\dots, n\}$, one for each generator $x_i$, and  $n$ types of wrapping circles, also labelled by numbers from $1$ to $n$, one for each $\partial_i$, with some relations shown in Figure~\ref{fig2_4_8}. For $A_n$ diagrammatics, one can, in addition, allow generic intersections  of  circles of different colors, with circles sliding freely through each other via relations in Figures~\ref{fig2_4_9} and~\ref{fig2_4_10}. Doing this allows for a local diagrammatic interpretation of the  commutativity  of derivations,  $\partial_i\partial_j=\partial_j\partial_i$,  see Figure~\ref{fig2_4_9} on the right. 

A diagram of  overlapping circles and dots on $n$ colors can be evaluated  by splitting it into $n$ diagrams each containing only dots and circles of one color, evaluating each diagram via iterated derivations and taking the product of evaluations.
From this  evaluation rule one can obtain  defining relations for this  toy example. An additional relation  is  that  for   $i\not=j$ an $i$-dot can pass through the boundary of a $j$-circle.

\begin{figure}
\begin{center}
  \begin{subfigure}[htb]{0.4\textwidth}
      \centering
\begingroup%
  \makeatletter%
  \providecommand\color[2][]{%
    \errmessage{(Inkscape) Color is used for the text in Inkscape, but the package 'color.sty' is not loaded}%
    \renewcommand\color[2][]{}%
  }%
  \providecommand\transparent[1]{%
    \errmessage{(Inkscape) Transparency is used (non-zero) for the text in Inkscape, but the package 'transparent.sty' is not loaded}%
    \renewcommand\transparent[1]{}%
  }%
  \providecommand\rotatebox[2]{#2}%
  \newcommand*\fsize{\dimexpr\f@size pt\relax}%
  \newcommand*\lineheight[1]{\fontsize{\fsize}{#1\fsize}\selectfont}%
  \ifx\svgwidth\undefined%
    \setlength{\unitlength}{90.41672609bp}%
    \ifx\svgscale\undefined%
      \relax%
    \else%
      \setlength{\unitlength}{\unitlength * \real{\svgscale}}%
    \fi%
  \else%
    \setlength{\unitlength}{\svgwidth}%
  \fi%
  \global\let\svgwidth\undefined%
  \global\let\svgscale\undefined%
  \makeatother%
  \begin{picture}(1,0.25865053)%
    \lineheight{1}%
    \setlength\tabcolsep{0pt}%
    \put(0.44001049,0.15482763){\color[rgb]{0,0,0}\makebox(0,0)[lt]{\smash{\begin{tabular}[t]{l}$n$\end{tabular}}}}%
    \put(0,0){\includegraphics[width=\unitlength,page=1]{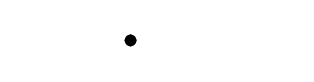}}%
    \put(0.33217648,0.06576057){\color[rgb]{0,0,0}\makebox(0,0)[lt]{\smash{\begin{tabular}[t]{l}$i$\end{tabular}}}}%
    \put(0,0){\includegraphics[width=\unitlength,page=2]{Angen.pdf}}%
    \put(-0.00367223,0.06871395){\color[rgb]{0,0,0}\makebox(0,0)[lt]{\smash{\begin{tabular}[t]{l}$i$\end{tabular}}}}%
    \put(0,0){\includegraphics[width=\unitlength,page=3]{Angen.pdf}}%
    \put(0.70613638,0.00650201){\color[rgb]{0,0,0}\makebox(0,0)[lt]{\smash{\begin{tabular}[t]{l}$i$\end{tabular}}}}%
  \end{picture}%
\endgroup%

         \caption{The elements $x_i,x_i^n$, and $\partial_i(\;\;)$}
         \label{fig:Angen}
     \end{subfigure}
       \begin{subfigure}[htb]{0.5\textwidth}
      \centering
\begingroup%
  \makeatletter%
  \providecommand\color[2][]{%
    \errmessage{(Inkscape) Color is used for the text in Inkscape, but the package 'color.sty' is not loaded}%
    \renewcommand\color[2][]{}%
  }%
  \providecommand\transparent[1]{%
    \errmessage{(Inkscape) Transparency is used (non-zero) for the text in Inkscape, but the package 'transparent.sty' is not loaded}%
    \renewcommand\transparent[1]{}%
  }%
  \providecommand\rotatebox[2]{#2}%
  \newcommand*\fsize{\dimexpr\f@size pt\relax}%
  \newcommand*\lineheight[1]{\fontsize{\fsize}{#1\fsize}\selectfont}%
  \ifx\svgwidth\undefined%
    \setlength{\unitlength}{130.42746753bp}%
    \ifx\svgscale\undefined%
      \relax%
    \else%
      \setlength{\unitlength}{\unitlength * \real{\svgscale}}%
    \fi%
  \else%
    \setlength{\unitlength}{\svgwidth}%
  \fi%
  \global\let\svgwidth\undefined%
  \global\let\svgscale\undefined%
  \makeatother%
  \begin{picture}(1,0.18084022)%
    \lineheight{1}%
    \setlength\tabcolsep{0pt}%
    \put(0,0){\includegraphics[width=\unitlength,page=1]{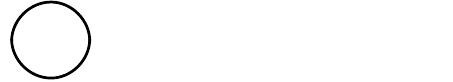}}%
    \put(-0.00254572,0.00450741){\color[rgb]{0,0,0}\makebox(0,0)[lt]{\smash{\begin{tabular}[t]{l}$i$\end{tabular}}}}%
    \put(0,0){\includegraphics[width=\unitlength,page=2]{Anrel1.pdf}}%
    \put(0.06569515,0.05760321){\color[rgb]{0,0,0}\makebox(0,0)[lt]{\smash{\begin{tabular}[t]{l}$j$\end{tabular}}}}%
    \put(0.25574192,0.0730358){\color[rgb]{0,0,0}\makebox(0,0)[lt]{\smash{\begin{tabular}[t]{l}$=\, \delta_{i,j}$\end{tabular}}}}%
    \put(0,0){\includegraphics[width=\unitlength,page=3]{Anrel1.pdf}}%
    \put(0.63010519,0.00604218){\color[rgb]{0,0,0}\makebox(0,0)[lt]{\smash{\begin{tabular}[t]{l}$i$\end{tabular}}}}%
    \put(0.88839282,0.07457057){\color[rgb]{0,0,0}\makebox(0,0)[lt]{\smash{\begin{tabular}[t]{l}$=\, 0$\end{tabular}}}}%
  \end{picture}%
\endgroup%

         \caption{The relations $\partial_i(x_j)=\delta_{i,j}$ and $\partial_i(1)=0$}
         \label{fig:Anrel1}
     \end{subfigure}
\caption{Diagrammatic interpretation of some relations in  the  Weyl algebra $A_n$. The generators $x_i$ and $\partial_i$ are given by a dot and  circle labelled $i$. The circle wraps about the area to be differentiated. If one needs  to keep track of  powers  of a  dot, a possible convention is  to write  the color to the lower left and  the power to the upper right of the dot.}
\label{fig2_4_8}
\end{center}
\end{figure}

\begin{figure}[htb]
\begin{center}
   \begin{subfigure}[htb]{0.4\textwidth}
      \centering
\begingroup%
  \makeatletter%
  \providecommand\color[2][]{%
    \errmessage{(Inkscape) Color is used for the text in Inkscape, but the package 'color.sty' is not loaded}%
    \renewcommand\color[2][]{}%
  }%
  \providecommand\transparent[1]{%
    \errmessage{(Inkscape) Transparency is used (non-zero) for the text in Inkscape, but the package 'transparent.sty' is not loaded}%
    \renewcommand\transparent[1]{}%
  }%
  \providecommand\rotatebox[2]{#2}%
  \newcommand*\fsize{\dimexpr\f@size pt\relax}%
  \newcommand*\lineheight[1]{\fontsize{\fsize}{#1\fsize}\selectfont}%
  \ifx\svgwidth\undefined%
    \setlength{\unitlength}{54.03265419bp}%
    \ifx\svgscale\undefined%
      \relax%
    \else%
      \setlength{\unitlength}{\unitlength * \real{\svgscale}}%
    \fi%
  \else%
    \setlength{\unitlength}{\svgwidth}%
  \fi%
  \global\let\svgwidth\undefined%
  \global\let\svgscale\undefined%
  \makeatother%
  \begin{picture}(1,0.54936015)%
    \lineheight{1}%
    \setlength\tabcolsep{0pt}%
    \put(0,0){\includegraphics[width=\unitlength,page=1]{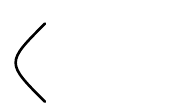}}%
    \put(0.19735306,0.49311523){\color[rgb]{0,0,0}\makebox(0,0)[lt]{\smash{\begin{tabular}[t]{l}$j$\end{tabular}}}}%
    \put(0,0){\includegraphics[width=\unitlength,page=2]{braidrel2.pdf}}%
    \put(-0.00614501,0.49237156){\color[rgb]{0,0,0}\makebox(0,0)[lt]{\smash{\begin{tabular}[t]{l}$i$\end{tabular}}}}%
    \put(0.41993311,0.18749708){\color[rgb]{0,0,0}\makebox(0,0)[lt]{\smash{\begin{tabular}[t]{l}$=$\end{tabular}}}}%
    \put(0.89137789,0.49311523){\color[rgb]{0,0,0}\makebox(0,0)[lt]{\smash{\begin{tabular}[t]{l}$j$\end{tabular}}}}%
    \put(0.68787982,0.49237156){\color[rgb]{0,0,0}\makebox(0,0)[lt]{\smash{\begin{tabular}[t]{l}$i$\end{tabular}}}}%
    \put(0,0){\includegraphics[width=\unitlength,page=3]{braidrel2.pdf}}%
  \end{picture}%
\endgroup%

         \caption{Crossing relation of strands for $i\neq j$}
         \label{fig:braidrel2}
     \end{subfigure}
       \begin{subfigure}[htb]{0.55\textwidth}
      \centering
    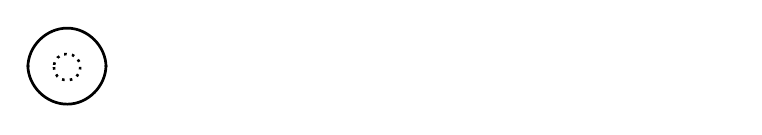
         \caption{Commutativity of $\del_i,\del_j$}
         \label{fig:circlepass}
     \end{subfigure}
\caption{The relation (A) allows differently colored circles to cross. It can be viewed as the simplest ``virtual crossings" simplification relation. (B) is a local derivation of commutativity of the operators   $\partial_i,\partial_j$. Here, the  middle equality  is an isotopy of diagrams.}
\label{fig2_4_9}
\end{center}
\end{figure}

\begin{figure}[htb]
\begin{center}
\begingroup%
  \makeatletter%
  \providecommand\color[2][]{%
    \errmessage{(Inkscape) Color is used for the text in Inkscape, but the package 'color.sty' is not loaded}%
    \renewcommand\color[2][]{}%
  }%
  \providecommand\transparent[1]{%
    \errmessage{(Inkscape) Transparency is used (non-zero) for the text in Inkscape, but the package 'transparent.sty' is not loaded}%
    \renewcommand\transparent[1]{}%
  }%
  \providecommand\rotatebox[2]{#2}%
  \newcommand*\fsize{\dimexpr\f@size pt\relax}%
  \newcommand*\lineheight[1]{\fontsize{\fsize}{#1\fsize}\selectfont}%
  \ifx\svgwidth\undefined%
    \setlength{\unitlength}{70.52495745bp}%
    \ifx\svgscale\undefined%
      \relax%
    \else%
      \setlength{\unitlength}{\unitlength * \real{\svgscale}}%
    \fi%
  \else%
    \setlength{\unitlength}{\svgwidth}%
  \fi%
  \global\let\svgwidth\undefined%
  \global\let\svgscale\undefined%
  \makeatother%
  \begin{picture}(1,0.42124938)%
    \lineheight{1}%
    \setlength\tabcolsep{0pt}%
    \put(0,0){\includegraphics[width=\unitlength,page=1]{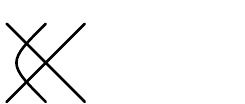}}%
    \put(-0.004708,0.3775881){\color[rgb]{0,0,0}\makebox(0,0)[lt]{\smash{\begin{tabular}[t]{l}$i$\end{tabular}}}}%
    \put(0.15443251,0.37777808){\color[rgb]{0,0,0}\makebox(0,0)[lt]{\smash{\begin{tabular}[t]{l}$j$\end{tabular}}}}%
    \put(0.31471009,0.37815737){\color[rgb]{0,0,0}\makebox(0,0)[lt]{\smash{\begin{tabular}[t]{l}$k$\end{tabular}}}}%
    \put(0,0){\includegraphics[width=\unitlength,page=2]{braidrel.pdf}}%
    \put(0.58019174,0.37758803){\color[rgb]{0,0,0}\makebox(0,0)[lt]{\smash{\begin{tabular}[t]{l}$i$\end{tabular}}}}%
    \put(0.73933194,0.37777804){\color[rgb]{0,0,0}\makebox(0,0)[lt]{\smash{\begin{tabular}[t]{l}$j$\end{tabular}}}}%
    \put(0.8996089,0.37815729){\color[rgb]{0,0,0}\makebox(0,0)[lt]{\smash{\begin{tabular}[t]{l}$k$\end{tabular}}}}%
    \put(0.42067313,0.13299384){\color[rgb]{0,0,0}\makebox(0,0)[lt]{\smash{\begin{tabular}[t]{l}$=$\end{tabular}}}}%
  \end{picture}%
\endgroup%

\caption{Another  standard relation on virtual crossings, with $i,j,k$ pairwise distinct.}
\label{fig2_4_10}
\end{center}
\end{figure}

In a more subtle example in~\cite{BHLZ} (also see~\cite{BHLW} for the multi-color generalization) one allows intersections of circles of the  same color (the $\mfsl_2$-case corresponds to one color) as well as self-intersections of a circle. Circles, in addition, carry dots on them and the ring of operators given by annular diagrams acts on a ring 
isomorphic to the ring of symmetric functions in infinitely many variables.

\ssubsection{Divided power differentials}
Let us go back  to our original example of  one variable $x$ represented by a dot and derivation $\partial$ represented by a circle around  the polynomial which is being differentiated, see the discussion around \eqref{eq_A1}. If we change the ground  ring from $\kk$ to $\Z$, the corresponding ring of polynomials $\Z[x]$ admits divided powers differentiations 
\begin{equation*}
    \partial^{(n)} \ := \ \frac{\partial^n}{n!}, \qquad  
    \partial^{(m)}\partial^{(n)}= \binom{n+m}{n} \partial^{(n+m)}.
\end{equation*} 
To describe them, we can enhance the diagrammatics by letting a circle of thickness $n$ denote  $\partial^{(n)}$, see Figure~\ref{fig2_4_11}. 

\begin{figure}[htb]
\begin{center}
   \begin{subfigure}[htb]{0.4\textwidth}
      \centering
    $\vcenter{\hbox{
\begingroup%
  \makeatletter%
  \providecommand\color[2][]{%
    \errmessage{(Inkscape) Color is used for the text in Inkscape, but the package 'color.sty' is not loaded}%
    \renewcommand\color[2][]{}%
  }%
  \providecommand\transparent[1]{%
    \errmessage{(Inkscape) Transparency is used (non-zero) for the text in Inkscape, but the package 'transparent.sty' is not loaded}%
    \renewcommand\transparent[1]{}%
  }%
  \providecommand\rotatebox[2]{#2}%
  \newcommand*\fsize{\dimexpr\f@size pt\relax}%
  \newcommand*\lineheight[1]{\fontsize{\fsize}{#1\fsize}\selectfont}%
  \ifx\svgwidth\undefined%
    \setlength{\unitlength}{50.0890468bp}%
    \ifx\svgscale\undefined%
      \relax%
    \else%
      \setlength{\unitlength}{\unitlength * \real{\svgscale}}%
    \fi%
  \else%
    \setlength{\unitlength}{\svgwidth}%
  \fi%
  \global\let\svgwidth\undefined%
  \global\let\svgscale\undefined%
  \makeatother%
  \begin{picture}(1,0.63529877)%
    \lineheight{1}%
    \setlength\tabcolsep{0pt}%
    \put(0.22622548,0.26820725){\color[rgb]{0,0,0}\makebox(0,0)[lt]{\smash{\begin{tabular}[t]{l}$f$\end{tabular}}}}%
    \put(0,0){\includegraphics[width=\unitlength,page=1]{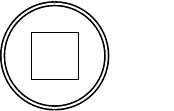}}%
    \put(0.54467808,0.57462557){\color[rgb]{0,0,0}\makebox(0,0)[lt]{\smash{\begin{tabular}[t]{l}$\scr{(n)}$\end{tabular}}}}%
  \end{picture}%
\endgroup%
}}\hspace{-10pt}=\;\del^{(n)}f={\displaystyle\text{``$\left(\frac{\del^nf}{n!}\right)$"}}$
         \caption{Diagrammatics for $\del^{(n)}f$}
         \label{fig:deldivn}
     \end{subfigure}
       \begin{subfigure}[htb]{0.5\textwidth}
      \centering
\begingroup%
  \makeatletter%
  \providecommand\color[2][]{%
    \errmessage{(Inkscape) Color is used for the text in Inkscape, but the package 'color.sty' is not loaded}%
    \renewcommand\color[2][]{}%
  }%
  \providecommand\transparent[1]{%
    \errmessage{(Inkscape) Transparency is used (non-zero) for the text in Inkscape, but the package 'transparent.sty' is not loaded}%
    \renewcommand\transparent[1]{}%
  }%
  \providecommand\rotatebox[2]{#2}%
  \newcommand*\fsize{\dimexpr\f@size pt\relax}%
  \newcommand*\lineheight[1]{\fontsize{\fsize}{#1\fsize}\selectfont}%
  \ifx\svgwidth\undefined%
    \setlength{\unitlength}{160.55779026bp}%
    \ifx\svgscale\undefined%
      \relax%
    \else%
      \setlength{\unitlength}{\unitlength * \real{\svgscale}}%
    \fi%
  \else%
    \setlength{\unitlength}{\svgwidth}%
  \fi%
  \global\let\svgwidth\undefined%
  \global\let\svgscale\undefined%
  \makeatother%
  \begin{picture}(1,0.28963006)%
    \lineheight{1}%
    \setlength\tabcolsep{0pt}%
    \put(0.25764766,0.25239307){\color[rgb]{0,0,0}\makebox(0,0)[lt]{\smash{\begin{tabular}[t]{l}$\scr{(m)}$\end{tabular}}}}%
    \put(0.3098276,0.13724234){\color[rgb]{0,0,0}\makebox(0,0)[lt]{\smash{\begin{tabular}[t]{l}$=$\end{tabular}}}}%
    \put(0.40497504,0.13783258){\color[rgb]{0,0,0}\makebox(0,0)[lt]{\smash{\begin{tabular}[t]{l}$\displaystyle\binom{n+m}{n}$\end{tabular}}}}%
    \put(0,0){\includegraphics[width=\unitlength,page=1]{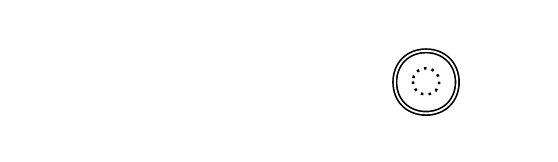}}%
    \put(0.81279861,0.19152706){\color[rgb]{0,0,0}\makebox(0,0)[lt]{\smash{\begin{tabular}[t]{l}$\scr{(n+m)}$\end{tabular}}}}%
    \put(0,0){\includegraphics[width=\unitlength,page=2]{deldivprod.pdf}}%
    \put(0.19149641,0.19157043){\color[rgb]{0,0,0}\makebox(0,0)[lt]{\smash{\begin{tabular}[t]{l}$\scr{(n)}$\end{tabular}}}}%
    \put(0,0){\includegraphics[width=\unitlength,page=3]{deldivprod.pdf}}%
  \end{picture}%
\endgroup%

         \caption{The composition $\del^{(m)}\del^{(n)}$}
         \label{fig:deldivprod}
     \end{subfigure}
\caption{A circle of thickness $n$ is denoted by a double circle with label $n$ on the top right and symbolizes the $n$-th divided power differentiation $\del^{(n)}$ applied to the label (function) inside the circle.}
\label{fig2_4_11}
\end{center}
\end{figure}

The Leibniz rule 
\begin{equation}
    \partial^{(n)}(fg) = \sum_{k=0}^n \partial^{(k)}(f)\partial^{(n-k)}(g)
\end{equation}
translates into the diagrammatics in Figure~\ref{fig2_4_12}.

\begin{figure}[htb]
\begin{center}
    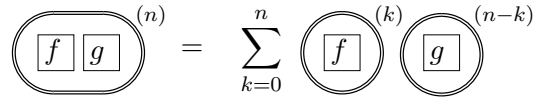
\caption{The Leibniz rule.}
\label{fig2_4_12}
\end{center}
\end{figure}

Two-dimensional diagrammatics for manipulation of these  divided powers may appear contrived, but it may  potentially lead to new extensions of this or related  algebraic constructions. We may try, for  instance, to  allow the lines for different divided  powers of  $\partial$ to split and merge. Figure~{\scshape\ref{fig:circlemerge}} depicts  an example where a double line representing  $\partial^{(2)}$ splits into two single  lines representing $\partial$, which then merge back into the double line.

\begin{figure}[htb]
\begin{center}
  \begin{subfigure}[htb]{0.3\textwidth}
      \centering
     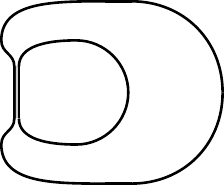
         \caption{Two circles partially merging into a double circle}
         \label{fig:circlemerge}
     \end{subfigure}
 \begin{subfigure}[htb]{0.65\textwidth}
      \centering
     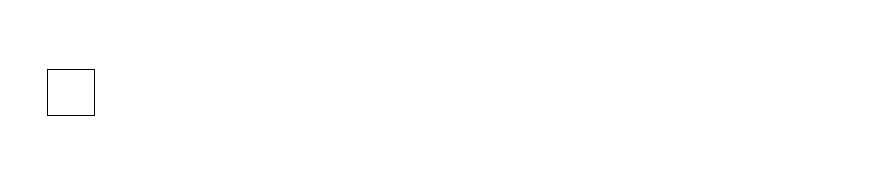
         \caption{Diagrammatics for $\del^{(1,1)}(f,g)$}
         \label{fig:del11}
     \end{subfigure}
\caption{The operator $\del^{(1,1)}$, given by formula (\ref{eq_part_3}), applied to a pair of functions $f,g$.}
\label{fig2_4_13}
\end{center}
\end{figure}


We can place polynomials $f$ and  $g$ into two bounded regions separated by the lines  of derivatives in the plane. In the left diagram of Figure~{\scshape\ref{fig:del11}}, $g$ is nested deeper  than $f$, and we  apply $\partial^{(2)}$ to it, while only applying $\partial$ to $f$. We then  modify the  Leibniz rule for $\partial^{(2)}$ applied to the product
\begin{equation}\label{eq_part_2}
    \partial^{(2)}(fg) = f\,\partial^{(2)}(g) + \partial(f)\,\partial(g) + \partial^{(2)}(f)\, g
\end{equation}
to the rule that the left diagram in Figure~{\scshape \ref{fig:del11}} evaluates to 
\begin{equation}\label{eq_part_3}
    \partial^{(1,1)}(f,g)\ := \ f\partial^{(2)}(g) + \partial(f)\partial(g). 
\end{equation}
The  expression is similar to that in (\ref{eq_part_2}) for $\partial^{(2)}$, but the last term is dropped. The  new operator $\partial^{(1,1)}$ applied to a pair $(f,g)$ differs from $\partial^{(2)}$ applied to $fg$ in that the term $\partial^{(2)}$ is not applied to $f$, for it is not nested inside both circles into which we can  split the graph in Figure~{\scshape\ref{fig:circlemerge}}. 
The arrow in Figure~{\scshape\ref{fig:del11}} 
depicts this modification of the Leibniz rule. 

\vspace{0.1in} 

Generalizing, consider a line of thickness $n$ that splits into $r$ lines of thickness $n_1,\dots, n_r$ that run  in parallel, with $n=n_1+\dots +n_r$. The lines then merge back into the original line, see Figure~\ref{fig2_4_14} on the left.  Denote the sequence by 
$\underline{n}=(n_1,\dots, n_r)$. Diagrammatically, the corresponding operation $\del^{\underline{n}}(f_1,\ldots, f_r)$ 
is given by the diagram that envelops $r$ bounded regions with polynomials $f_1,\dots, f_r$ in them as in  Figure~\ref{fig2_4_14} on the left.

\begin{figure}[htb]
\begin{center}
 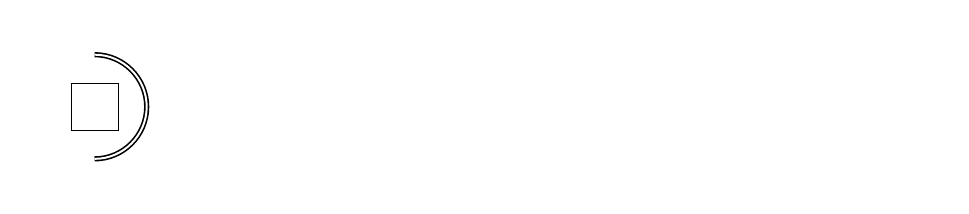
\caption{The $\un{n}$-differentiation operator $\del^{\un{n}}$ acting on $(f_1,\dots, f_r)$.}
\label{fig2_4_14}
\end{center}
\end{figure}

We can then modify the Leibniz  rule for the diagram in Figure~\ref{fig2_4_14} and define $\underline{n}$-differentiation, for $\underline{n}=(n_1,\dots, n_r)$, by  
\begin{equation}
    \partial^{\underline{n}}(f_1,\dots, f_r) := \sum_{\underline{k}} 
    \partial^{(k_1)}(f_1)\dots \partial^{(k_r)}(f_r),
\end{equation}
where  the sum is taken over all sequences $\underline{k}=(k_1,\dots, k_r)$ with $k_i\in \Z_+$, such that
\begin{equation*}
    k_1+k_2+\dots+k_r =n, \ k_1 \le n_1, \ k_1+k_2\le n_1+n_2 , \ \dots , k_1+\dots + k_{r-1}\le n_1+\dots + n_{r-1}.
\end{equation*}
In this situation, we also write $\underline{k}\leq \underline{n}$.
By a \emph{region of depth}  $m$ we mean a region separated by lines of  total thickness $m$ from the outer region. When  distributing divided powers of $\partial$ to act on polynomials in various regions of the diagram, on a polynomial located in  a region of depth $m$ the divided power of degree at most  $m$ may act. 

It may be interesting to see whether this or related  diagrammatics can be developed further  to  justify such two-dimensional manipulation rules for various systems of operators acting on commutative rings, beyond the examples  coming from trace reductions of categorified quantum groups and categorifications of the Heisenberg  algebra, where commutative rings on which the operators act are rings of symmetric functions in finitely many  variables and their tensor products. 


\ssubsection{Frobenius endomorphism}
Given a commutative ring $A$ of characteristic $p$, the Frobenius endomorphism $\sigma\colon A\lra A$ acts by $\sigma(a)=a^p$, for $a\in A$. Elements of $A$ may be depicted by labelled dots floating in the plane, and the action of $\sigma$ by a circle enveloping a region of a plane, with the relations shown in Figure~\ref{fig_frob}.

\begin{figure}[htb]
\begin{center}
 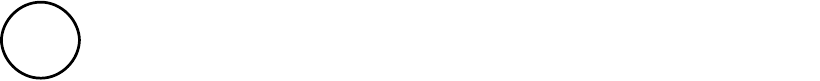
\caption{Graphical calculus for the Frobenius endomorphism of $A$.}
\label{fig_frob}
\end{center}
\end{figure}

By itself, this diagrammatics is very simple; one can try combining it with deeper structures in number theory and algebraic geometry in characteristic $p$ or to convert Frobenius endomorphisms into defect lines on suitable foams.

%
%

\section{Generalizations and transfer maps}\label{sec_miscell}


\subsection{Generalizing to a biadjoint pair \texorpdfstring{$(F,G)$}{(F,G)}}

In this paper we have discussed connections between self-adjoint functors, circular series and forest series evaluations,  and triples $(R,Z,\omega)$ and quadruples $(R,Z,\omega,\varepsilon)$. These connections have a straightforward modification for biadjoint pairs $(F,G)$ between categories $\mcA_0$ and $\mcA_1$,
\begin{equation}\label{eq_biadj_FG}
    F\colon \mcA_0\to \mcA_1,\  G\colon \mcA_1\to \mcA_0.
\end{equation} 
Planar diagrams are now checkerboard colored by $0$ and $1$, the labelling denoting categories $\mcA_0$ and $\mcA_1$. Lines and circles of diagrams obtain induced orientations so that as one travels along an arc in the orientation direction, the region labelled $0$ appears on the right. The resulting systems of arcs and circles are then compatibly oriented. Section \ref{subsec_diag_biadj} explains these diagrammatics, with the categories denoted by $\mcA,\mcB$ there rather than $\mcA_0,\mcA_1$. 

\vspace{0.1in}

\begin{figure}[htb]
    \centering
\begingroup%
  \makeatletter%
  \providecommand\color[2][]{%
    \errmessage{(Inkscape) Color is used for the text in Inkscape, but the package 'color.sty' is not loaded}%
    \renewcommand\color[2][]{}%
  }%
  \providecommand\transparent[1]{%
    \errmessage{(Inkscape) Transparency is used (non-zero) for the text in Inkscape, but the package 'transparent.sty' is not loaded}%
    \renewcommand\transparent[1]{}%
  }%
  \providecommand\rotatebox[2]{#2}%
  \newcommand*\fsize{\dimexpr\f@size pt\relax}%
  \newcommand*\lineheight[1]{\fontsize{\fsize}{#1\fsize}\selectfont}%
  \ifx\svgwidth\undefined%
    \setlength{\unitlength}{138.78407052bp}%
    \ifx\svgscale\undefined%
      \relax%
    \else%
      \setlength{\unitlength}{\unitlength * \real{\svgscale}}%
    \fi%
  \else%
    \setlength{\unitlength}{\svgwidth}%
  \fi%
  \global\let\svgwidth\undefined%
  \global\let\svgscale\undefined%
  \makeatother%
  \begin{picture}(1,0.61166296)%
    \lineheight{1}%
    \setlength\tabcolsep{0pt}%
    \put(0,0){\includegraphics[width=\unitlength,page=1]{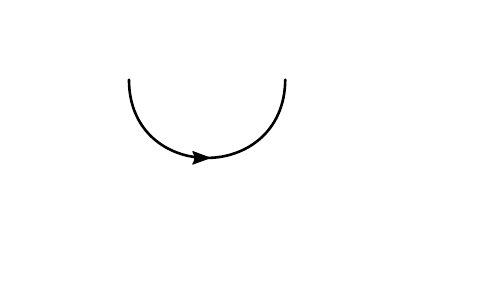}}%
    \put(0.28419686,0.41387312){\color[rgb]{0,0,0}\makebox(0,0)[lt]{\smash{\begin{tabular}[t]{l}$1$\end{tabular}}}}%
    \put(0.40338417,0.17926426){\color[rgb]{0,0,0}\makebox(0,0)[lt]{\smash{\begin{tabular}[t]{l}$0$\end{tabular}}}}%
    \put(0,0){\includegraphics[width=\unitlength,page=2]{biadcircles.pdf}}%
    \put(0.70175685,0.25996734){\color[rgb]{0,0,0}\makebox(0,0)[lt]{\smash{\begin{tabular}[t]{l}$1$\end{tabular}}}}%
    \put(0,0){\includegraphics[width=\unitlength,page=3]{biadcircles.pdf}}%
    \put(0.39135275,0.42112697){\color[rgb]{0,0,0}\makebox(0,0)[lt]{\smash{\begin{tabular}[t]{l}$0$\end{tabular}}}}%
    \put(0.79728163,0.15133714){\color[rgb]{0,0,0}\makebox(0,0)[lt]{\smash{\begin{tabular}[t]{l}$0$\end{tabular}}}}%
    \put(0.79646556,0.36832327){\color[rgb]{0,0,0}\makebox(0,0)[lt]{\smash{\begin{tabular}[t]{l}$0$\end{tabular}}}}%
    \put(0.04000527,0.23992394){\color[rgb]{0,0,0}\makebox(0,0)[lt]{\smash{\begin{tabular}[t]{l}$1$\end{tabular}}}}%
    \put(0,0){\includegraphics[width=\unitlength,page=4]{biadcircles.pdf}}%
  \end{picture}%
\endgroup%

    \caption{An example of a closed diagram for the biadjoint pair $(F,G)$ which defines an element of the center of $\mcA_0$.}
    \label{fig:7.1.1}
\end{figure}

Closed diagrams of this form, see Figure~\ref{fig:7.1.1} for an example, are determined by the underlying circular form (without orientations of circles) and the label ($0$ or $1$) of the outer region. That data determines orientations of all circles and labels for all regions. 

Consequently, an evaluation function $\alpha$ on such oriented circular foams can be described by a pair of circular evaluation functions $(\alpha_0,\alpha_1)$, one for each label of the outside region. We call such series $\alpha=(\alpha_0,\alpha_1)$ an \emph{oriented circular series}. 

The construction of state spaces $A(\mu)$ goes through as earlier, with objects $\mu$ now being alternating sequences of pluses and minuses, describing orientations at boundary points. For the empty sequence, one should additionally specify $0$ or $1$, that is, the label of the region that contains the boundary circle along which the diagrams are paired. One can denote these empty sequences by $\emptyset_0$ and $\emptyset_1$.

To define state spaces, 
one again needs an asymmetric setup, so that diagrams in a disk are paired to similar diagrams in an annulus to produce a planar diagram and evaluate it. 

The diagram of categories and functors (\ref{diagUalpha2}) extends to this ``checkerboard'' case, with the caveat that the various monoidal categories in the diagram become 2-categories with two objects, $0$ and $1$, corresponding to the possible labels of the regions.  

An oriented circular series $\alpha$ (over a field $\kk$) as above is called \emph{recognizable} if all state spaces $A(\mu)$ are finite-dimensional. 

\begin{prop} An oriented circular series $\alpha=(\alpha_0,\alpha_1)$ is recognizable if and only if the state spaces $A(\emptyset_0)$ and $A(\emptyset_1)$ are finite-dimensional.
\end{prop}
\begin{proof}
Proof is  straightforward and left to the reader. This proposition is analogous to Proposition~\ref{prop:alpha_rec}.
\end{proof}

The state spaces $A(\emptyset_0)$ and $A(\emptyset_1)$ are commutative algebras, with trace maps $\varepsilon_0,\varepsilon_1$, respectively. 
Taking a closed diagram with an outer region $0$, respectively $1$, and wrapping a circle around it gives a diagram  with  the  opposite label for the outside region. Consequently, these operations  give rise to linear maps 
\begin{equation}
\omega_0\colon  A_0\lra A_1,\ \  \omega_1\colon  A_1\lra A_0. 
\end{equation}
The commutative algebras $A(\emptyset_0),A(\emptyset_1)$ are Frobenius in the  weak sense, where in analogue with equation (\ref{eq_vareps}) indices of $\omega$ must alternate and $x_i$'s must alternate between taking values in the two algebras. For instance, for $x\in A(\emptyset_0)$ there must exist $k\geq 0$ and length $k$ sequences $\{x_i\},\{y_i\}$ such that either
\begin{equation}\label{eq_vareps_3} 
\varepsilon_0 (x_k\omega_1 (y_{k-1}\dots (\omega_0(x_2 \omega_1 (y_1 \omega_0 (x_1x))))\dots )\not=0 .
\end{equation} 
or
\begin{equation}\label{eq_vareps_4} 
\varepsilon_1 (y_k\omega_0 (x_k\dots (\omega_0(x_2 \omega_1 (y_1 \omega_0 (x_1x))))\dots )\not=0 .
\end{equation} 

\begin{prop} \label{prop_rec_pair} Recognizable oriented circular series $\alpha=(\alpha_0,\alpha_1)$ are in a bijection with isomorphism classes of pairs $(A_0,A_1)$ of commutative algebras over $\kk$, with traces $\varepsilon_i\colon  A_i\lra \kk$, $i=0,1$ and linear maps $\omega_0\colon  A_0\lra A_1,$ $\omega_1\colon  A_1\lra A_0$ subject to the above nondegeneracy condition (weak Frobenius property for $(A_i,\omega_i,\varepsilon_i)$, $i=0,1$), and the  stability condition that $(A_0,A_1)$ is the only pair of subalgebras in $A_0,A_1$ that contain unit elements and are closed under $\omega_0,\omega_1$. 
\end{prop}
\begin{proof}
We leave the details of the proof, which is analogous to that of  Proposition~\ref{prop_bij_A},  to the reader. 
\end{proof}

When extending to the oriented checkerboard case, with all categories in 
(\ref{diagUalpha2}) becoming 2-categories with objects $0,1$, the square of what are now 2-functors continues to be commutative in the strong sense, as long as $\alpha$ is recognizable, see an earlier discussion. 

\vspace{0.1in} 

The analogue of a circular triple $(R,Z,\omega)$ in this setup is the data of 
\begin{itemize}
    \item A pair of commutative $R$-algebras $Z_0,Z_1$. 
    \item $R$-linear maps $\omega_{0}\colon  Z_0\lra Z_1$ and $\omega_{1}\colon  Z_1\lra Z_0$. 
    \item  (For the spherical case:) the condition
    \begin{equation}\label{eq_spher}
        \omega_{0}(z_0) z_1 = z_0 \, \omega_{1}(z_1),  \ \  z_0\in Z_0, \ z_1 \in Z_1. 
    \end{equation}
\end{itemize}

It is not difficult to write down the corresponding condition on oriented circular series $\alpha$ to make it spherical.  For spherical series, the bilinear pairing can  be made symmetric: arc and circle diagrams on an annulus now reduce to such diagrams in a disk, and the pairing is applied to two disk diagrams rather than coupling a disk diagram to an annulus diagram. As in the self-adjoint case, discussed at length earlier, this should simplify computations and understanding of the corresponding state spaces and categories.

The generating one-morphisms $+$ and $-$ in each of the corresponding 2-categories with two objects are biadjoint, with the biadjointness 2-morphisms given by oriented cup and cap diagrams (four diagrams in total). Vice versa, from a suitable biadjoint pair of functors $(F,G)$ between pre-additive categories $\mcA_0,\mcA_1$, see (\ref{eq_biadj_FG}), versions of oriented circular series, state spaces, oriented circular triples, etc. can be recovered. Trace maps  $\varepsilon_i\colon  Z(\mcA_i)\lra R$, $i=0,1$ from centers of these categories to the ground ring need to be fixed to write down oriented circular series associated to a biadjoint pair. This way, the notion of a circular quadruple from Section \ref{subsec_circ_s} also extends to this oriented (or checkerboard) case in a straighforward fashion. 

Further generalization, with $\alpha$ taking values in a commutative semiring $R$, is possible. That case would, again, bring the theory closer to that of weighted tree automata. 


\subsection{Systems  of biadjoint pairs and decorated tree series} \label{subsec_systems} 

The two setups, circular series related to a self-adjoint functor and oriented circular series for a biadjoint pair, admit a further generalization. Consider a graph $\Gamma$ with a set of vertices $V(\Gamma)$ and a  set of  edges $E(\Gamma)$, with multiple edges and loops allowed. Some loops may be oriented. 

Such a graph can be associated to a data of categories and biadjoint pairs of functors or self-adjoint functors between them. That data consists of categories $\mcA_v$, for $v\in V(\Gamma)$ and a pair of biadjoint functors $(F_e,G_e)$  between categories $\mcA_{v_1}$, $\mcA_{v_2}$ for an edge $e$ with vertices $v_1,v_2$. For a loop $e$ at vertex $v$, one can consider two cases: 
\begin{itemize}
    \item An oriented loop $e$ corresponds to a biadjoint pair $(F_e,G_e)$ of endofunctors of $\mcA_v$. 
    \item An unoriented loop $e$ corresponds to a self-adjoint endofunctor $F_e$ in $\mcA_v$. 
\end{itemize}

To such a graph $\Gamma$ one can associate the set of $\Gamma$-decorated diagrams of planar embedded circles. We label regions of such a diagram $D$ by vertices of $\Gamma$ and circles by edges. If the two regions on the sides of a circle are 
labelled $v_1,v_2$, the circle must be decorated by an edge with endpoints $v_1,v_2$. If there is only one such edge, the decoration can be omitted or reconstructed from the labels of the regions. If both regions around the circle are labelled by the same vertex $v\in V(\Gamma)$, the circle must be labelled by a loop $e$ at $v$. If the loop $e$ is oriented (and, thus, will correspond to a biadjoint pair $(F_e,G_e)$ of endofunctors), an orientation for the circle must be chosen, to distinguish between $F_e$ and $G_e$ or, vice versa, to reconstruct the functors and categories from the evaluation data, cf. Section \ref{subsec_diag_biadj}. 
For a circle labelled by an unoriented loop no additional decoration at that circle is needed, as in Section \ref{sec-wU}.
Graphs $\Gamma$ with infinitely many vertices or edges may be allowed.

\begin{figure}[htb]
    \centering
    \begin{subfigure}[htb]{0.35\textwidth}
\begingroup%
  \makeatletter%
  \providecommand\color[2][]{%
    \errmessage{(Inkscape) Color is used for the text in Inkscape, but the package 'color.sty' is not loaded}%
    \renewcommand\color[2][]{}%
  }%
  \providecommand\transparent[1]{%
    \errmessage{(Inkscape) Transparency is used (non-zero) for the text in Inkscape, but the package 'transparent.sty' is not loaded}%
    \renewcommand\transparent[1]{}%
  }%
  \providecommand\rotatebox[2]{#2}%
  \newcommand*\fsize{\dimexpr\f@size pt\relax}%
  \newcommand*\lineheight[1]{\fontsize{\fsize}{#1\fsize}\selectfont}%
  \ifx\svgwidth\undefined%
    \setlength{\unitlength}{112.03023984bp}%
    \ifx\svgscale\undefined%
      \relax%
    \else%
      \setlength{\unitlength}{\unitlength * \real{\svgscale}}%
    \fi%
  \else%
    \setlength{\unitlength}{\svgwidth}%
  \fi%
  \global\let\svgwidth\undefined%
  \global\let\svgscale\undefined%
  \makeatother%
  \begin{picture}(1,0.32231587)%
    \lineheight{1}%
    \setlength\tabcolsep{0pt}%
    \put(0,0){\includegraphics[width=\unitlength,page=1]{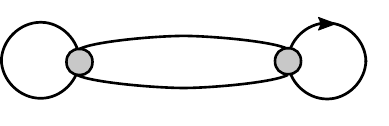}}%
    \put(0.17084677,0.02919719){\color[rgb]{0,0,0}\makebox(0,0)[lt]{\lineheight{1.25}\smash{\begin{tabular}[t]{l}$v_0$\end{tabular}}}}%
    \put(0.70575348,0.0317135){\color[rgb]{0,0,0}\makebox(0,0)[lt]{\lineheight{1.25}\smash{\begin{tabular}[t]{l}$v_1$\end{tabular}}}}%
    \put(0.06540661,0.29518873){\color[rgb]{0,0,0}\makebox(0,0)[lt]{\lineheight{1.25}\smash{\begin{tabular}[t]{l}$e_0$\end{tabular}}}}%
    \put(0.43701745,0.26513138){\color[rgb]{0,0,0}\makebox(0,0)[lt]{\lineheight{1.25}\smash{\begin{tabular}[t]{l}$e_1$\end{tabular}}}}%
    \put(0.43857203,0.00842057){\color[rgb]{0,0,0}\makebox(0,0)[lt]{\lineheight{1.25}\smash{\begin{tabular}[t]{l}$e_2$\end{tabular}}}}%
    \put(0.89499565,0.27929513){\color[rgb]{0,0,0}\makebox(0,0)[lt]{\lineheight{1.25}\smash{\begin{tabular}[t]{l}$e_3$\end{tabular}}}}%
  \end{picture}%
\endgroup%

      \centering
         \caption{The graph $\Gamma$}
         \label{fig:gamma}
    \end{subfigure}
    \begin{subfigure}[htb]{0.6\textwidth}
    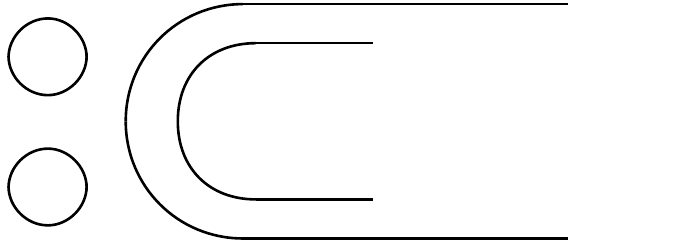
      \centering
         \caption{A $\Gamma$-decorated circular diagram}
         \label{fig:gammacirc}
    \end{subfigure}
    \caption{An example of graph $\Gamma$ is shown in {\sc (a)} and an example of a $\Gamma$-circular form in {\sc (b)}. Since each vertex has a single loop, we omit labels of the corresponding circles in the picture  (these circles have the same region labels on both sides). Orientations are chosen for each circle labelled by the oriented loop $e_3$ (these are the two circles with regions labelled $1$ on both sides) corresponding to the endofunctor of $\mcA_{v_1}$ and its biadjoint. If the graph $\Gamma$ comes with corresponding categories and functors, then this diagram gives an element in $Z(\mcA_{v_0})$. }
    \label{fig:7.1.2}
\end{figure}

Given a graph $\Gamma$ as above, consider all planar $\Gamma$-decorated circular diagrams $\wU(\Gamma)_{\emptyset}$. An example of a $\Gamma$-decorated circular diagram is given in Figure~\ref{fig:7.1.2}. There is 
the 2-category $\wU(\Gamma)$ with the set of objects $V(\Gamma)$, one-morphisms given by paths in the graph $\Gamma$ and 2-morphisms being isotopy classes of $\Gamma$-decorated one-manifolds with boundary embedded in $\R\times[0,1]$. 
Decorations of these planar diagrams, when restricted to the boundaries 
$\R\times\{0\}, \R\times\{1\}$, give paths in the graph $\Gamma$, see Figure~\ref{fig:7.1.3} for an example of a 2-morphism in $\wU(\Gamma)$ for $\Gamma$ as in Figure~{\scshape \ref{fig:gamma}}.

\begin{figure}[htb]
    \centering
   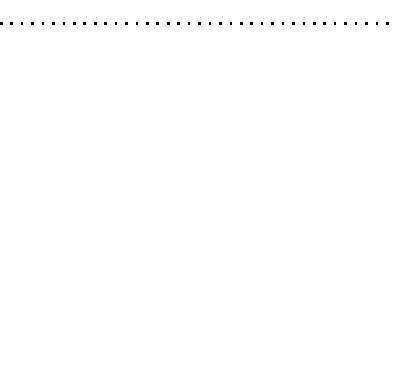
    \caption{A 2-morphism in $\wU(\Gamma)$, with $\Gamma$ as in Figure~{\scshape \ref{fig:gammacirc}}, from one-morphism $e_0e_2e_3e_3^{\ast}$ to $e_0e_1e_1e_2$. Both of these one-morphisms are paths in $\Gamma$. Notation $e_3^{\ast}$ means that we go along $e_3$ in the opposite orientation direction, as part of the corresponding path.}
    \label{fig:7.1.3}
\end{figure}

A closed diagram with an outer region labelled by $v\in V(\Gamma)$ gives an endomorphism of the 1-morphism $(v)$, the latter denoting the length zero path that starts and ends at $v$. The set of endomorphisms of $(v)$ may be denoted $\wU(\Gamma)_{(v)}^{(v)}$ or, simply, $\wU(\Gamma)_v$. More generally, the set of 2-morphisms from the path $p$ to the path $p'$ is denoted $\wU(\Gamma)^{p'}_p$ or $\Hom_{\wU(\Gamma)}(p,p')$. With the earlier notation, 
the set of $\Gamma$-decorated circular diagrams is the union 
\begin{equation}
    \wU(\Gamma)_{\emptyset} \ = \ \bigsqcup_{v\in V(\Gamma)} \wU(\Gamma)_v 
\end{equation}
of diagrams with outer label $v$, over all vertices $v$ of $\Gamma$.

\vspace{0.1in} 

A \emph{$\Gamma$-circular series} $\alpha$ is a map 
\begin{equation}
    \alpha \ \colon  \ \wU(\Gamma)_{\emptyset}  \lra \kk 
\end{equation}
that assigns an element of the ground field $\kk$ to each $\Gamma$-circular diagram. 
Such a series $\alpha$ is called \emph{spherical} if it descends to a map from the set of isotopy classes of $\Gamma$-circular diagrams on the two-sphere rather than the plane.

With a $\Gamma$-circular series $\alpha$ we can repeat the constructions in Section~\ref{subsec_circ_s} and define the  state space $A(p)$, for any closed path $p$ in $\Gamma$, as the space with a basis of $\Gamma$-diagrams in a disk with boundary $p$ modulo the kernel of its pairing with the space of corresponding diagrams in the annulus. The morphism space between two paths $p,p'$ with the same source vertices and same target vertices is  defined as $A(p^{\ast}p')$, where $p^{\ast}$ is the reverse path of $p$. This results in a 2-category $\wU(\Gamma)_{\alpha}$, analogous to the monoidal category $\wU_{\alpha}$ in Section~\ref{subsec_circ_s}. The skein 2-category $\wSU(\Gamma)_{\alpha}$ can 
 be defined by  analogy with the corresponding monoidal category $\wSU_{\alpha}$ in that section, by only imposing the relations on linear combinations of closed diagrams in a disk (no boundary points). 

The resulting 2-categories carry pairs of biadjoint  1-morphisms (or self-adjoint 1-morphisms), one for each unoriented (respectively, oriented) edge in $\Gamma$. 

\vspace{0.1in} 

A $\Gamma$-series $\alpha$ is called \emph{recognizable} if the state spaces $A(p)$ are finite-dimensional for any closed path $p$ in $\Gamma$. 

\begin{prop} The $\Gamma$-series $\alpha$ is recognizable if and only if the state spaces $A((v))$ are finite-dimensional, for any vertex $v$ of $\Gamma$. 
\end{prop} 
\begin{proof}
Recall that $(v)$ is the length zero closed path at vertex $v$. The proof is immediate. 
\end{proof}

We leave it to the reader to write down the analogues of Propositions~\ref{prop_bij_A} and~\ref{prop_rec_pair} for this more general setup. 

\vspace{0.1in} 

For recognizable $\alpha$ one can form a diagram of 2-categories and 2-functors 
 \begin{equation}\label{diagUalpha2gamma}
    \vcenter{\hbox{\xymatrix{
    \wU(\Gamma)\ar@{^{(}->}[r]&\kk \wU(\Gamma)\ar@{->>}[r]
    &\wSU(\Gamma)_\alpha\ar@{->>}[d]\ar@{^{(}->}[r]&\wDSU(\Gamma)_\alpha\ar@{->>}[d]\\
   &&\wU(\Gamma)_\alpha\ar@{^{(}->}[r] & \wDU(\Gamma)_\alpha
    }}}
\end{equation}
analogous to the one in (\ref{diagUalpha2}). The two 2-categories on the right are Karoubi completions of the corresponding 2-categories on the side of the left vertical arrows. The square is commutative in the strong sense, as discussed earlier.


\subsection{Adding defects to lines}  

Another generalization of the circular form construction is to allow lines to carry zero-dimensional defects. On the categorical side, this generalization corresponds to adding functor endomorphisms $F\Rightarrow F$ and denoting them by dots on lines. For a simple example, given a self-adjoint functor $F$, choose a natural transformation $a\colon F\Rightarrow F$ subject to the (strong) ambidexterity condition shown in Figure~\ref{fig_ambid1}.

\begin{figure}[htb]
\begin{center}
     \centering
    \begin{subfigure}[htb]{0.6\textwidth}
\begingroup%
  \makeatletter%
  \providecommand\color[2][]{%
    \errmessage{(Inkscape) Color is used for the text in Inkscape, but the package 'color.sty' is not loaded}%
    \renewcommand\color[2][]{}%
  }%
  \providecommand\transparent[1]{%
    \errmessage{(Inkscape) Transparency is used (non-zero) for the text in Inkscape, but the package 'transparent.sty' is not loaded}%
    \renewcommand\transparent[1]{}%
  }%
  \providecommand\rotatebox[2]{#2}%
  \newcommand*\fsize{\dimexpr\f@size pt\relax}%
  \newcommand*\lineheight[1]{\fontsize{\fsize}{#1\fsize}\selectfont}%
  \ifx\svgwidth\undefined%
    \setlength{\unitlength}{185.94004157bp}%
    \ifx\svgscale\undefined%
      \relax%
    \else%
      \setlength{\unitlength}{\unitlength * \real{\svgscale}}%
    \fi%
  \else%
    \setlength{\unitlength}{\svgwidth}%
  \fi%
  \global\let\svgwidth\undefined%
  \global\let\svgscale\undefined%
  \makeatother%
  \begin{picture}(1,0.2420135)%
    \lineheight{1}%
    \setlength\tabcolsep{0pt}%
    \put(0,0){\includegraphics[width=\unitlength,page=1]{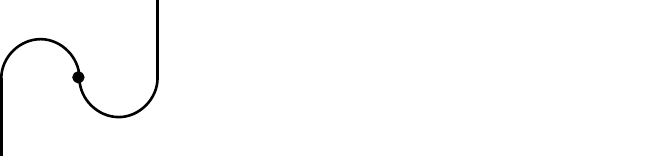}}%
    \put(0.28436902,0.10083893){\color[rgb]{0,0,0}\makebox(0,0)[lt]{\lineheight{1.25}\smash{\begin{tabular}[t]{l}$=$\end{tabular}}}}%
    \put(0,0){\includegraphics[width=\unitlength,page=2]{strongamb1.pdf}}%
    \put(0.88940277,0.10083893){\color[rgb]{0,0,0}\makebox(0,0)[lt]{\lineheight{1.25}\smash{\begin{tabular}[t]{l}$=$\end{tabular}}}}%
    \put(0,0){\includegraphics[width=\unitlength,page=3]{strongamb1.pdf}}%
    \put(0.4762513,0.09939832){\color[rgb]{0,0,0}\makebox(0,0)[lt]{\lineheight{1.25}\smash{\begin{tabular}[t]{l}$=$\end{tabular}}}}%
    \put(0,0){\includegraphics[width=\unitlength,page=4]{strongamb1.pdf}}%
  \end{picture}%
\endgroup%

      \centering
         \caption{The (strong) ambidexterity property}
    \end{subfigure}
    \begin{subfigure}[htb]{0.25\textwidth}
\begingroup%
  \makeatletter%
  \providecommand\color[2][]{%
    \errmessage{(Inkscape) Color is used for the text in Inkscape, but the package 'color.sty' is not loaded}%
    \renewcommand\color[2][]{}%
  }%
  \providecommand\transparent[1]{%
    \errmessage{(Inkscape) Transparency is used (non-zero) for the text in Inkscape, but the package 'transparent.sty' is not loaded}%
    \renewcommand\transparent[1]{}%
  }%
  \providecommand\rotatebox[2]{#2}%
  \newcommand*\fsize{\dimexpr\f@size pt\relax}%
  \newcommand*\lineheight[1]{\fontsize{\fsize}{#1\fsize}\selectfont}%
  \ifx\svgwidth\undefined%
    \setlength{\unitlength}{82.34037161bp}%
    \ifx\svgscale\undefined%
      \relax%
    \else%
      \setlength{\unitlength}{\unitlength * \real{\svgscale}}%
    \fi%
  \else%
    \setlength{\unitlength}{\svgwidth}%
  \fi%
  \global\let\svgwidth\undefined%
  \global\let\svgscale\undefined%
  \makeatother%
  \begin{picture}(1,0.73293552)%
    \lineheight{1}%
    \setlength\tabcolsep{0pt}%
    \put(0,0){\includegraphics[width=\unitlength,page=1]{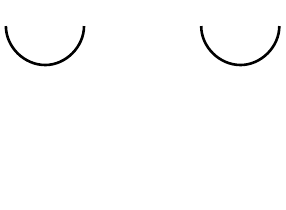}}%
    \put(0.43047602,0.59630694){\color[rgb]{0,0,0}\makebox(0,0)[lt]{\lineheight{1.25}\smash{\begin{tabular}[t]{l}$=$\end{tabular}}}}%
    \put(0,0){\includegraphics[width=\unitlength,page=2]{strongamb2.pdf}}%
    \put(0.43047602,0.09533779){\color[rgb]{0,0,0}\makebox(0,0)[lt]{\lineheight{1.25}\smash{\begin{tabular}[t]{l}$=$\end{tabular}}}}%
    \put(0,0){\includegraphics[width=\unitlength,page=3]{strongamb2.pdf}}%
    \put(0.43047602,0.09533779){\color[rgb]{0,0,0}\makebox(0,0)[lt]{\lineheight{1.25}\smash{\begin{tabular}[t]{l}$=$\end{tabular}}}}%
    \put(0,0){\includegraphics[width=\unitlength,page=4]{strongamb2.pdf}}%
  \end{picture}%
\endgroup%

      \centering
         \caption{Duality relations.}
    \end{subfigure}
\caption{The duality relations in {\scshape(b)} are implied by the property in {\scshape(a)} .}
\label{fig_ambid1}
\end{center}
\end{figure}

The analogue of a circular form in this case is  a circular form together with dots placed on its components. Dots can freely move along the component. The number of dots on each component is then an invariant of a diagram, and isotopy classes of diagrams are in a bijection  with forests as in Section~\ref{sec_trees} with vertices labelled by non-negative  integers.  We call these diagrams \emph{dotted circular forms} and draw some examples in Figure~\ref{fig_dottedcirc}.

\begin{figure}[htb]
\begin{center}
\begingroup%
  \makeatletter%
  \providecommand\color[2][]{%
    \errmessage{(Inkscape) Color is used for the text in Inkscape, but the package 'color.sty' is not loaded}%
    \renewcommand\color[2][]{}%
  }%
  \providecommand\transparent[1]{%
    \errmessage{(Inkscape) Transparency is used (non-zero) for the text in Inkscape, but the package 'transparent.sty' is not loaded}%
    \renewcommand\transparent[1]{}%
  }%
  \providecommand\rotatebox[2]{#2}%
  \newcommand*\fsize{\dimexpr\f@size pt\relax}%
  \newcommand*\lineheight[1]{\fontsize{\fsize}{#1\fsize}\selectfont}%
  \ifx\svgwidth\undefined%
    \setlength{\unitlength}{279.74998486bp}%
    \ifx\svgscale\undefined%
      \relax%
    \else%
      \setlength{\unitlength}{\unitlength * \real{\svgscale}}%
    \fi%
  \else%
    \setlength{\unitlength}{\svgwidth}%
  \fi%
  \global\let\svgwidth\undefined%
  \global\let\svgscale\undefined%
  \makeatother%
  \begin{picture}(1,0.30063957)%
    \lineheight{1}%
    \setlength\tabcolsep{0pt}%
    \put(0,0){\includegraphics[width=\unitlength,page=1]{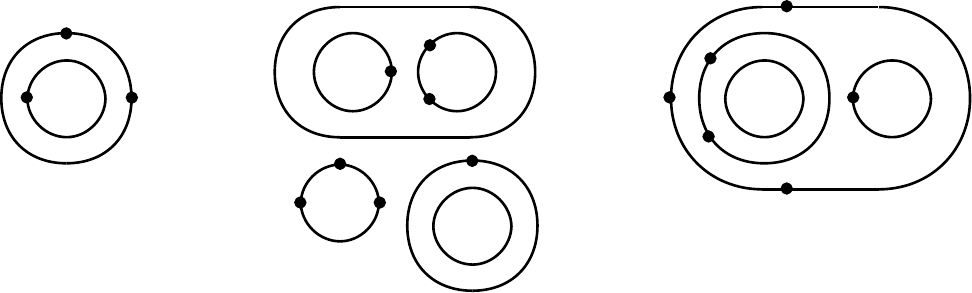}}%
    \put(0.18900801,0.17273066){\color[rgb]{0,0,0}\makebox(0,0)[lt]{\lineheight{1.25}\smash{\begin{tabular}[t]{l}$,$\end{tabular}}}}%
    \put(0.59383373,0.17541166){\color[rgb]{0,0,0}\makebox(0,0)[lt]{\lineheight{1.25}\smash{\begin{tabular}[t]{l}$,$\end{tabular}}}}%
  \end{picture}%
\endgroup%

      \centering
\caption{Examples of dotted circular forms.}
\label{fig_dottedcirc}
\end{center}
\end{figure}

In the ``dotted'' version of the category $\wU$, which we denote by $\wUd$, circles and arcs of a diagram may carry freely floating dots. 
A dotted circular series is a  map 
\begin{equation*}
     \alpha\colon  \ \wUd^0_0 \ \lra \ \kk
\end{equation*}
from the set of dotted circular forms (up to isotopy) to the ground field $\kk$. 

Given $\alpha$, one can define state spaces $A(n)$ for each $n\ge 0$ as before, by pairing dotted diagrams of $n$ arcs and any number of circles in a disk with such diagrams in an annulus, followed by taking the quotient of the disk space by the (left) kernel of the bilinear form, in full analogy with Section~\ref{subsec_circ_s}. See Figure~\ref{fig_dottedpairing} for an example of the pairing of dotted diagrams.  

\begin{figure}[htb]
\begin{center}
\begingroup%
  \makeatletter%
  \providecommand\color[2][]{%
    \errmessage{(Inkscape) Color is used for the text in Inkscape, but the package 'color.sty' is not loaded}%
    \renewcommand\color[2][]{}%
  }%
  \providecommand\transparent[1]{%
    \errmessage{(Inkscape) Transparency is used (non-zero) for the text in Inkscape, but the package 'transparent.sty' is not loaded}%
    \renewcommand\transparent[1]{}%
  }%
  \providecommand\rotatebox[2]{#2}%
  \newcommand*\fsize{\dimexpr\f@size pt\relax}%
  \newcommand*\lineheight[1]{\fontsize{\fsize}{#1\fsize}\selectfont}%
  \ifx\svgwidth\undefined%
    \setlength{\unitlength}{311.29354869bp}%
    \ifx\svgscale\undefined%
      \relax%
    \else%
      \setlength{\unitlength}{\unitlength * \real{\svgscale}}%
    \fi%
  \else%
    \setlength{\unitlength}{\svgwidth}%
  \fi%
  \global\let\svgwidth\undefined%
  \global\let\svgscale\undefined%
  \makeatother%
  \begin{picture}(1,0.26743656)%
    \lineheight{1}%
    \setlength\tabcolsep{0pt}%
    \put(0,0){\includegraphics[width=\unitlength,page=1]{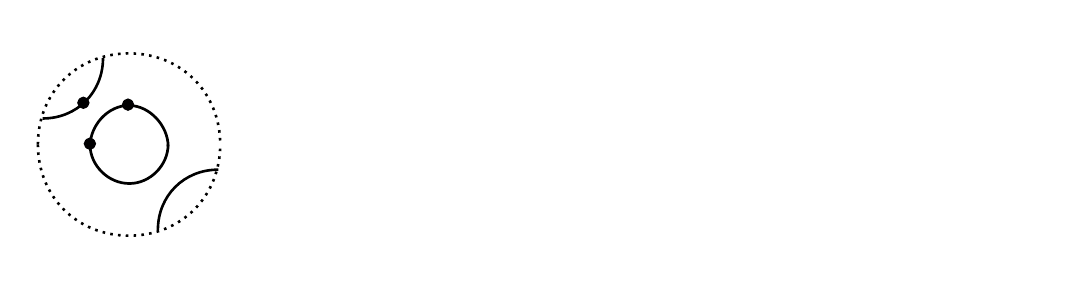}}%
    \put(0.227817,0.12167183){\color[rgb]{0,0,0}\makebox(0,0)[lt]{\lineheight{1.25}\smash{\begin{tabular}[t]{l}$,$\end{tabular}}}}%
    \put(-0.00106662,0.12167183){\color[rgb]{0,0,0}\makebox(0,0)[lt]{\lineheight{1.25}\smash{\begin{tabular}[t]{l}$\Bigg($\end{tabular}}}}%
    \put(0,0){\includegraphics[width=\unitlength,page=2]{dottedpairing.pdf}}%
    \put(0.56524145,0.12179424){\color[rgb]{0,0,0}\makebox(0,0)[lt]{\lineheight{1.25}\smash{\begin{tabular}[t]{l}$\Bigg)$\end{tabular}}}}%
    \put(0,0){\includegraphics[width=\unitlength,page=3]{dottedpairing.pdf}}%
  \end{picture}%
\endgroup%

      \centering
\caption{Pairing of a disk diagram and an outer annular diagram, both equipped with dots on lines.}
\label{fig_dottedpairing}
\end{center}
\end{figure}

Recognizable dotted circular series are defined by the requirement that state spaces $A(n)$ are finite-dimensional for all $n$, see also Definition~\ref{def_rec_circ}. The following generalizes Proposition~\ref{prop:alpha_rec}.

\begin{prop} A dotted circular series $\alpha$ is recognizable if and only if $A(0)$ and $A(1)$ are finite-dimensional. 
\end{prop} 
\begin{proof}
The proof uses that a diagram in ${}^\bullet\wU_0^{2n}$ is given by nesting a cup diagram in ${}^\bullet\wU_0^2$ into an inner  region of a diagram in ${}^\bullet\wU_0^{2(n-1)}$. Using induction on $n$ and assuming that both $A(0)$, $A(1)$ are finite-dimensional, one obtains a finite spanning set for $A(n)$, implying that it is finite-dimensional.
The other implication is clear.
\end{proof}

Wrapping a circle with $n$ dots around a closed diagram induces a $\Bbbk$-linear map $\omega_n\colon A(0)\lra A(0)$, see Figure~\ref{fig_omegan} for an example of such a wrapping map.

\begin{figure}[htb]
\begin{center}
     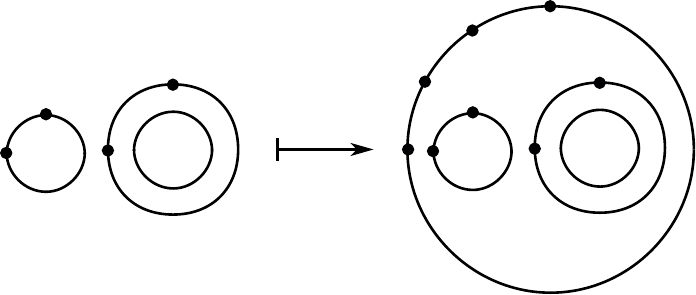
      \centering
\caption{The map $\omega_4$ applied to the dotted diagram on the left.}
\label{fig_omegan}
\end{center}
\end{figure}

A sequence  of linear maps $\omega_{\ast}:=(\omega_0,\omega_1,\dots )$  is called \emph{recurrent} or \emph{linearly recurrent} if 
there exist a non-negative integer $N$ and $b_i\in \kk$, $1\le i \le N$, such that for any $n\ge 0$
\begin{equation}
     \omega_{n+N} = b_1 \omega_{n+N-1}+b_2 \omega_{n+N-2}+\dots + b_N \omega_n. 
\end{equation}
Notice that the last few $b_i$'s may be zero, so that the recursion does not necessarily start in the lowest possible degree. 

\begin{prop}\label{prop_dottedseries} Suppose that $A(0)$ and $A(1)$ are finite-dimensional, for a dotted circular series $\alpha$. Then the sequence $\omega_{\ast}=(\omega_0,\omega_1, \omega_2,\dots)$ of endomorphisms of $A(0)$ is linearly recurrent.
\end{prop}

\begin{proof} Elements of $A(1)$ can be represented as planar boxes with two strands emanating out, denoting a linear combination of dotted circular forms in $\wUd^2_0$. Placing a dot on the left strand is then an endomorphism of $A(1)$, see Figure~\ref{fig_adddot}, that we can denote by $d$. Since $A(1)$ is finite-dimensional, the minimal polynomial for operator $d$ gives us a recurrence relation on powers of $d$, which converts to a recurrence relation for the $\omega_n$'s. 
\end{proof} 

\begin{remark*}
 As a partial converse to Proposition \ref{prop_dottedseries}, if $\omega_{\ast}$ is recurrent and $A(0)$ finite-dimensional then $A(1)$ is also finite-dimensional.
\end{remark*}

\begin{figure}[htb]
\begin{center}
\begingroup%
  \makeatletter%
  \providecommand\color[2][]{%
    \errmessage{(Inkscape) Color is used for the text in Inkscape, but the package 'color.sty' is not loaded}%
    \renewcommand\color[2][]{}%
  }%
  \providecommand\transparent[1]{%
    \errmessage{(Inkscape) Transparency is used (non-zero) for the text in Inkscape, but the package 'transparent.sty' is not loaded}%
    \renewcommand\transparent[1]{}%
  }%
  \providecommand\rotatebox[2]{#2}%
  \newcommand*\fsize{\dimexpr\f@size pt\relax}%
  \newcommand*\lineheight[1]{\fontsize{\fsize}{#1\fsize}\selectfont}%
  \ifx\svgwidth\undefined%
    \setlength{\unitlength}{149.24524438bp}%
    \ifx\svgscale\undefined%
      \relax%
    \else%
      \setlength{\unitlength}{\unitlength * \real{\svgscale}}%
    \fi%
  \else%
    \setlength{\unitlength}{\svgwidth}%
  \fi%
  \global\let\svgwidth\undefined%
  \global\let\svgscale\undefined%
  \makeatother%
  \begin{picture}(1,0.25634525)%
    \lineheight{1}%
    \setlength\tabcolsep{0pt}%
    \put(0,0){\includegraphics[width=\unitlength,page=1]{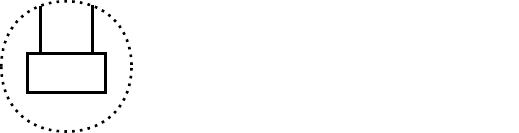}}%
    \put(0.08794242,0.1030207){\color[rgb]{0,0,0}\makebox(0,0)[lt]{\lineheight{1.25}\smash{\begin{tabular}[t]{l}$a$\end{tabular}}}}%
    \put(0,0){\includegraphics[width=\unitlength,page=2]{adddot.pdf}}%
    \put(0.83168507,0.10307171){\color[rgb]{0,0,0}\makebox(0,0)[lt]{\lineheight{1.25}\smash{\begin{tabular}[t]{l}$a$\end{tabular}}}}%
    \put(0,0){\includegraphics[width=\unitlength,page=3]{adddot.pdf}}%
    \put(0.48996531,0.1030207){\color[rgb]{0,0,0}\makebox(0,0)[lt]{\lineheight{1.25}\smash{\begin{tabular}[t]{l}$d(a)=$\end{tabular}}}}%
  \end{picture}%
\endgroup%

      \centering
\caption{The $\kk$-linear map $d\colon A(1)\to A(1)$ of adding a dot on the left strand.}
\label{fig_adddot}
\end{center}
\end{figure}
 
\begin{remark*} The map $d$ placing a dot on the left strand of a diagram is 
 an endomorphism of $A(1)$ viewed as an $A(0)$-bimodule, with the bimodule action given by placing closed dotted diagrams (which are diagrams defining a spanning set of $A(0)$) in the two regions on the two sides of the unique arc in a diagram in $A(1)$. The minimal degree integral relation on powers of the dot with coefficients in $A(0)^{\otimes 2}$ may have lower degree than that on powers of the dot with coefficients in $\kk$. The former can be thought of as  a skein relation to reduce a power of a dot to a linear combination of lower powers times  closed diagrams placed in the  two regions on the sides of the arc that carries powers of the dot.
\end{remark*}

The data  $(A(0),\omega_{\ast},\varepsilon)$  is non-degenerate in the following weak sense. For any $x\in A(0), x\not= 0$ there exists $k\ge 0$ and   sequences $x_1,\dots, x_k\in A(0)$, $i_1,\dots, i_{k-1}\in\Z_+=\{0,1,\dots\}$ such that  
\begin{equation}\label{eq_vareps_2} 
\varepsilon(x_k\,\omega_{i_{k-1}} (x_{k-1}\dots \omega_{i_2} (x_2 \,\omega_{i_1} (x_1x)))\dots )\not=0 .
\end{equation} 
We call a datum $(A,\omega_{\ast},\varepsilon)$ as above with a  $A$ finite-dimensional a \emph{commutative weakly Frobenius $\ast$-triple}. 

\begin{prop} \label{prop_bij_A2}There is a bijection between recognizable dotted circular series $\alpha$ and isomorphism classes of the following data: 
\begin{itemize} 
\item A finite-dimensional commutative $\kk$-algebra $A$ with the trace form $\varepsilon$ and a recurrent sequence of linear maps $\omega_n\colon A\lra A$, $n\ge 0$ subject to (\ref{eq_vareps_2}), that is, a commutative weakly Frobenius $\ast$-triple. 
\item Stability condition: $A$ is the only subalgebra of $A$ that contains $1$ and is closed under $\omega_n$, $n\ge 0$. 
\end{itemize}
\end{prop} 
\begin{proof}
The proof is straightforward.
\end{proof}

\emph{Spherical} dotted circular series can be defined analogously, cf. \eqref{eq_alpha_spher}. In the spherical case, the setup is more symmetric, with the bilinear pairing coming from gluing a pair of dotted diagrams in a disk, rather than dotted diagrams in a disk and an annulus, leading to possible simplifications in the computation of skein relations.

\begin{remark}\label{rem_many_dots}
Instead of a dot of a single type, one can fix a set $S$ of labels for dots and consider dotted circular diagrams with dots labelled by elements of $S$. If the structure of embeddings of circles into the plane or 2-sphere is ignored, one recovers the familiar notion of noncommutative recognizable power series, see~\cite{BR} and its tensor envelope~\cite{Kh3}. Keeping track of the embedding in $\R^2$ would add additional complexity to the theory. We do not attempt to develop it here. 
\end{remark}

Our strong ambidexterity property for an endomorphism of a self-adjoint functor, allows to work with a single type of dot on strands, so that the number of dots on each circle is the only additional information for closed diagrams. It is also well-suited for dealing with the spherical case of such diagrams, considered then as diagrams on a 2-sphere. 

By the \emph{weak} ambidexterity of an endomorphism $a$ of a self-adjoint functor $F$ we mean  the relations shown in Figure~\ref{fig_weakamb} on the right. The diagram in Figure~\ref{fig_weakamb} on the left defines the endomorphism $a^{\ast}$ of $F$. 
\begin{figure}[htb]
\begin{center}
\begingroup%
  \makeatletter%
  \providecommand\color[2][]{%
    \errmessage{(Inkscape) Color is used for the text in Inkscape, but the package 'color.sty' is not loaded}%
    \renewcommand\color[2][]{}%
  }%
  \providecommand\transparent[1]{%
    \errmessage{(Inkscape) Transparency is used (non-zero) for the text in Inkscape, but the package 'transparent.sty' is not loaded}%
    \renewcommand\transparent[1]{}%
  }%
  \providecommand\rotatebox[2]{#2}%
  \newcommand*\fsize{\dimexpr\f@size pt\relax}%
  \newcommand*\lineheight[1]{\fontsize{\fsize}{#1\fsize}\selectfont}%
  \ifx\svgwidth\undefined%
    \setlength{\unitlength}{235.05092021bp}%
    \ifx\svgscale\undefined%
      \relax%
    \else%
      \setlength{\unitlength}{\unitlength * \real{\svgscale}}%
    \fi%
  \else%
    \setlength{\unitlength}{\svgwidth}%
  \fi%
  \global\let\svgwidth\undefined%
  \global\let\svgscale\undefined%
  \makeatother%
  \begin{picture}(1,0.19344629)%
    \lineheight{1}%
    \setlength\tabcolsep{0pt}%
    \put(0.74470926,0.08763301){\color[rgb]{0,0,0}\makebox(0,0)[lt]{\lineheight{1.25}\smash{\begin{tabular}[t]{l}$=$\end{tabular}}}}%
    \put(0.07165473,0.08814992){\color[rgb]{0,0,0}\makebox(0,0)[lt]{\lineheight{1.25}\smash{\begin{tabular}[t]{l}$:=$\end{tabular}}}}%
    \put(0,0){\includegraphics[width=\unitlength,page=1]{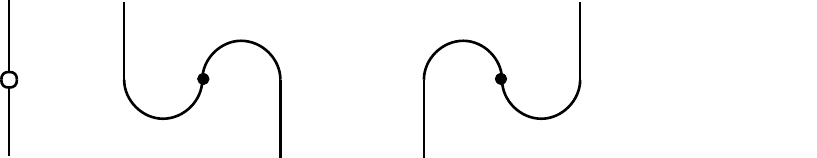}}%
    \put(0.39144257,0.09572404){\color[rgb]{0,0,0}\makebox(0,0)[lt]{\lineheight{1.25}\smash{\begin{tabular}[t]{l}$,$\end{tabular}}}}%
    \put(0,0){\includegraphics[width=\unitlength,page=2]{weakamb.pdf}}%
  \end{picture}%
\endgroup%

      \centering
\caption{The definition of $a^{\ast}$, denoted by a hollow dot, and the weak ambidexterity condition.}
\label{fig_weakamb}
\end{center}
\end{figure}

In general, the endomorphisms $a$ and $a^{\ast}$ of $F$ do not seem to come, with any monomial relations on them. This means that the order of $a$'s and $a^{\ast}$'s in the product of endomorphisms of $F$ is important, and, in a closed diagram, any circle will come with a word in $a$ and $a^{\ast}$, up to overall cyclic order. This gives a lot of freedom in creating possible diagrams and relates this setup to that of noncommutative power series and associated monoidal categories, see the remark above and~\cite{Kh3}. Unlike~\cite{Kh3}, one-manifolds with defects (dots) now lie in the plane and the nesting of these manifolds is part of the structure of such diagrams. One can think of this setup as the $(2,1,0)$-manifold case, with the ambient manifold being $\R^2$ or $\SS^2$, embedded one-manifolds (circles) being codimension one defects and dots on one-manifolds are defects on defects. This can be further extended by coloring one-manifolds as in Section~\ref{subsec_systems} as well as allowing dots (codimension two defects) to float in the regions of the plane or the 2-sphere separated by the circles. The later variation is similar to that in~\cite[Section 8]{KKO}. A further extension is to generalize from diagrams in $\R^2$ and $\SS^2$ to those in more general surfaces. An even further development is to extend from collections of (decorated) circles in the plane or a surface to embedded decorated graphs in $\R^2$ and in more general surfaces. Vaughan Jones' planar algebras~\cite{Jo7} give rise to such planar graph and network evaluations with additional strong unitarity and positivity properties.

\vspace{0.1in}

Given a $\Bbbk$-linear endofunctor $F\colon \A \to \A$, we may consider the monoidal subcategory $\A_F$ of the $\Bbbk$-linear monoidal category $\mathrm{Fun}_{\Bbbk}(\A,\A)$ that $F$ generates. If $F$ is self-adjoint, then $F$ is self-dual as an object of $\A_F$. This self-duality induces a \emph{pivotal structure} on $\A_F$ provided that left and right dualities coincide \cite[Section~1.7]{TV}. Equivalently, $F$ satisfies the \emph{weak ambidexterity property} that
\begin{equation}
    a^{\ast}=(1_F \mu)(1_F a 1_F)(\delta 1_F)=(\mu 1_F)(1_F a 1_F)(1_F \delta)={}^{\ast}a,
\end{equation}
for any endomorphism $a$ of $F$, see Figure \ref{fig_weakamb}. Examples of endofunctors satisfying weak ambidexterity can be given by tensoring with objects from a pivotal category. These are, in general, not self-dual. We refer to \cite{S} for graphical calculus associated to pivotal categories.

\begin{remark}[Reflection involution] 
Recall the reflection involution in the plane, reversing the plane's orientation. This involution fixes the isotopy class of any planar diagram of circles, see Proposition~\ref{prop_refl}. Once dots of several types are allowed on circles, as additional defects,  see Remark~\ref{rem_many_dots}, this involution no longer has to be trivial. Alternatively, if one adds labelled dots floating in the regions of the diagram, the reflection involution may permute nontrivially labels of the dots. With this reflection acting nontrivially, the bilinear form on suitable state spaces may not be symmetric and one needs slight changes in our construction. The ground field $\kk$ may come with an involution $\rho$ to match the reflection involution, with the bilinear pairing hermitian with respect to $(\kk,\rho)$. 
\end{remark} 

\smallskip


\subsection{The transfer map and its  diagrammatics} \label{sec_transfer}
In this section, we discuss diagrammatics of transfer maps in the case of a biadjoint pair of functors.  

\ssubsection{Adding  a  boundary and placing morphisms on it}
Suppose given a  biadjoint pair of functors $(F,G)$ between categories $\A,\B$  as in Section~\ref{subsec_diag_biadj}. The planar string diagrammatics for $(F,G)$, and for collections of biadjoint pairs, can be  enhanced by considering the half-plane to  the left of a vertical line $\mathbb{L}$ that carries objects and  morphisms of the categories $\A$ and  $\B$ as follows. Intervals in $\mathbb{L}$ are labelled by categories $\A$ and $\B$. A  dot  on an interval  of $\mathbb{L}$ labelled by  $\A$ denotes a morphism $f\colon a_1\to a_2$ in $\A$, with a collection  of  consecutive  dots representing the composition $a_1\xrightarrow{f_1} a_2\lra\dots\lra  a_{n-1}\xrightarrow{f_{n-1}} a_n$ of morphisms, see Figure~\ref{fig5_1}.

\begin{figure}[htb]
\begin{center}
\begin{subfigure}[htb]{0.15\textwidth}
\begingroup%
  \makeatletter%
  \providecommand\color[2][]{%
    \errmessage{(Inkscape) Color is used for the text in Inkscape, but the package 'color.sty' is not loaded}%
    \renewcommand\color[2][]{}%
  }%
  \providecommand\transparent[1]{%
    \errmessage{(Inkscape) Transparency is used (non-zero) for the text in Inkscape, but the package 'transparent.sty' is not loaded}%
    \renewcommand\transparent[1]{}%
  }%
  \providecommand\rotatebox[2]{#2}%
  \newcommand*\fsize{\dimexpr\f@size pt\relax}%
  \newcommand*\lineheight[1]{\fontsize{\fsize}{#1\fsize}\selectfont}%
  \ifx\svgwidth\undefined%
    \setlength{\unitlength}{32.41599174bp}%
    \ifx\svgscale\undefined%
      \relax%
    \else%
      \setlength{\unitlength}{\unitlength * \real{\svgscale}}%
    \fi%
  \else%
    \setlength{\unitlength}{\svgwidth}%
  \fi%
  \global\let\svgwidth\undefined%
  \global\let\svgscale\undefined%
  \makeatother%
  \begin{picture}(1,1.97379101)%
    \lineheight{1}%
    \setlength\tabcolsep{0pt}%
    \put(0.22988924,1.03017819){\color[rgb]{0,0,0}\makebox(0,0)[lt]{\smash{\begin{tabular}[t]{l}$\A$\end{tabular}}}}%
    \put(0,0){\includegraphics[width=\unitlength,page=1]{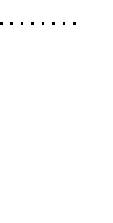}}%
    \put(0.57693683,0.02910167){\color[rgb]{0,0,0}\makebox(0,0)[lt]{\smash{\begin{tabular}[t]{l}$a_1$\end{tabular}}}}%
    \put(0,0){\includegraphics[width=\unitlength,page=2]{morph1.pdf}}%
    \put(0.58106921,1.88003906){\color[rgb]{0,0,0}\makebox(0,0)[lt]{\smash{\begin{tabular}[t]{l}$a_2$\end{tabular}}}}%
    \put(0.80978476,1.03088671){\color[rgb]{0,0,0}\makebox(0,0)[lt]{\smash{\begin{tabular}[t]{l}$f$\end{tabular}}}}%
  \end{picture}%
\endgroup%
 
      \centering
         \caption{}
\end{subfigure}
\begin{subfigure}[htb]{0.15\textwidth}
\begingroup%
  \makeatletter%
  \providecommand\color[2][]{%
    \errmessage{(Inkscape) Color is used for the text in Inkscape, but the package 'color.sty' is not loaded}%
    \renewcommand\color[2][]{}%
  }%
  \providecommand\transparent[1]{%
    \errmessage{(Inkscape) Transparency is used (non-zero) for the text in Inkscape, but the package 'transparent.sty' is not loaded}%
    \renewcommand\transparent[1]{}%
  }%
  \providecommand\rotatebox[2]{#2}%
  \newcommand*\fsize{\dimexpr\f@size pt\relax}%
  \newcommand*\lineheight[1]{\fontsize{\fsize}{#1\fsize}\selectfont}%
  \ifx\svgwidth\undefined%
    \setlength{\unitlength}{38.00974193bp}%
    \ifx\svgscale\undefined%
      \relax%
    \else%
      \setlength{\unitlength}{\unitlength * \real{\svgscale}}%
    \fi%
  \else%
    \setlength{\unitlength}{\svgwidth}%
  \fi%
  \global\let\svgwidth\undefined%
  \global\let\svgscale\undefined%
  \makeatother%
  \begin{picture}(1,1.87992643)%
    \lineheight{1}%
    \setlength\tabcolsep{0pt}%
    \put(0.1960573,0.97511232){\color[rgb]{0,0,0}\makebox(0,0)[lt]{\smash{\begin{tabular}[t]{l}$\A$\end{tabular}}}}%
    \put(0,0){\includegraphics[width=\unitlength,page=1]{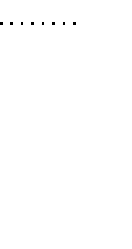}}%
    \put(0.49203121,0.02481888){\color[rgb]{0,0,0}\makebox(0,0)[lt]{\smash{\begin{tabular}[t]{l}$a_1$\end{tabular}}}}%
    \put(0,0){\includegraphics[width=\unitlength,page=2]{morph2.pdf}}%
    \put(0.49555545,1.7999716){\color[rgb]{0,0,0}\makebox(0,0)[lt]{\smash{\begin{tabular}[t]{l}$a_3$\end{tabular}}}}%
    \put(0.6906118,0.67833673){\color[rgb]{0,0,0}\makebox(0,0)[lt]{\smash{\begin{tabular}[t]{l}$f_1$\end{tabular}}}}%
    \put(0.6906118,1.27309642){\color[rgb]{0,0,0}\makebox(0,0)[lt]{\smash{\begin{tabular}[t]{l}$f_2$\end{tabular}}}}%
    \put(0.6906118,0.97571658){\color[rgb]{0,0,0}\makebox(0,0)[lt]{\smash{\begin{tabular}[t]{l}$a_2$\end{tabular}}}}%
  \end{picture}%
\endgroup%
 
      \centering
         \caption{}
\end{subfigure}
\begin{subfigure}[htb]{0.15\textwidth}
     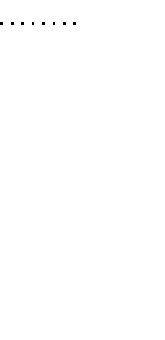 
      \centering
         \caption{}
\end{subfigure}
\begin{subfigure}[htb]{0.15\textwidth}
\begingroup%
  \makeatletter%
  \providecommand\color[2][]{%
    \errmessage{(Inkscape) Color is used for the text in Inkscape, but the package 'color.sty' is not loaded}%
    \renewcommand\color[2][]{}%
  }%
  \providecommand\transparent[1]{%
    \errmessage{(Inkscape) Transparency is used (non-zero) for the text in Inkscape, but the package 'transparent.sty' is not loaded}%
    \renewcommand\transparent[1]{}%
  }%
  \providecommand\rotatebox[2]{#2}%
  \newcommand*\fsize{\dimexpr\f@size pt\relax}%
  \newcommand*\lineheight[1]{\fontsize{\fsize}{#1\fsize}\selectfont}%
  \ifx\svgwidth\undefined%
    \setlength{\unitlength}{47.41600119bp}%
    \ifx\svgscale\undefined%
      \relax%
    \else%
      \setlength{\unitlength}{\unitlength * \real{\svgscale}}%
    \fi%
  \else%
    \setlength{\unitlength}{\svgwidth}%
  \fi%
  \global\let\svgwidth\undefined%
  \global\let\svgscale\undefined%
  \makeatother%
  \begin{picture}(1,1.34938399)%
    \lineheight{1}%
    \setlength\tabcolsep{0pt}%
    \put(0.47351309,0.70428224){\color[rgb]{0,0,0}\makebox(0,0)[lt]{\smash{\begin{tabular}[t]{l}$\A$\end{tabular}}}}%
    \put(0,0){\includegraphics[width=\unitlength,page=1]{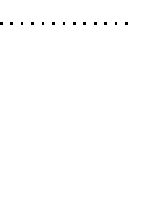}}%
    \put(0.71077248,0.01989538){\color[rgb]{0,0,0}\makebox(0,0)[lt]{\smash{\begin{tabular}[t]{l}$a_1$\end{tabular}}}}%
    \put(0,0){\includegraphics[width=\unitlength,page=2]{morph4.pdf}}%
    \put(0.71359759,1.28529039){\color[rgb]{0,0,0}\makebox(0,0)[lt]{\smash{\begin{tabular}[t]{l}$a_2$\end{tabular}}}}%
    \put(0.86995918,0.70476663){\color[rgb]{0,0,0}\makebox(0,0)[lt]{\smash{\begin{tabular}[t]{l}$f$\end{tabular}}}}%
    \put(0.2363741,1.28436358){\color[rgb]{0,0,0}\makebox(0,0)[lt]{\smash{\begin{tabular}[t]{l}$F$\end{tabular}}}}%
    \put(-0.00100981,0.70485055){\color[rgb]{0,0,0}\makebox(0,0)[lt]{\smash{\begin{tabular}[t]{l}$\B$\end{tabular}}}}%
    \put(0.23625186,0.01989538){\color[rgb]{0,0,0}\makebox(0,0)[lt]{\smash{\begin{tabular}[t]{l}$F$\end{tabular}}}}%
    \put(0,0){\includegraphics[width=\unitlength,page=3]{morph4.pdf}}%
  \end{picture}%
\endgroup%
 
      \centering
         \caption{}
\end{subfigure}
\begin{subfigure}[htb]{0.15\textwidth}
\begingroup%
  \makeatletter%
  \providecommand\color[2][]{%
    \errmessage{(Inkscape) Color is used for the text in Inkscape, but the package 'color.sty' is not loaded}%
    \renewcommand\color[2][]{}%
  }%
  \providecommand\transparent[1]{%
    \errmessage{(Inkscape) Transparency is used (non-zero) for the text in Inkscape, but the package 'transparent.sty' is not loaded}%
    \renewcommand\transparent[1]{}%
  }%
  \providecommand\rotatebox[2]{#2}%
  \newcommand*\fsize{\dimexpr\f@size pt\relax}%
  \newcommand*\lineheight[1]{\fontsize{\fsize}{#1\fsize}\selectfont}%
  \ifx\svgwidth\undefined%
    \setlength{\unitlength}{33.54880409bp}%
    \ifx\svgscale\undefined%
      \relax%
    \else%
      \setlength{\unitlength}{\unitlength * \real{\svgscale}}%
    \fi%
  \else%
    \setlength{\unitlength}{\svgwidth}%
  \fi%
  \global\let\svgwidth\undefined%
  \global\let\svgscale\undefined%
  \makeatother%
  \begin{picture}(1,1.90714378)%
    \lineheight{1}%
    \setlength\tabcolsep{0pt}%
    \put(0.22212677,0.99539309){\color[rgb]{0,0,0}\makebox(0,0)[lt]{\smash{\begin{tabular}[t]{l}$\B$\end{tabular}}}}%
    \put(0,0){\includegraphics[width=\unitlength,page=1]{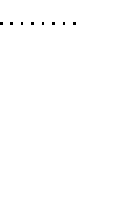}}%
    \put(0.55745592,0.02811902){\color[rgb]{0,0,0}\makebox(0,0)[lt]{\smash{\begin{tabular}[t]{l}$b_1$\end{tabular}}}}%
    \put(0,0){\includegraphics[width=\unitlength,page=2]{morph5.pdf}}%
    \put(0.56144877,1.81655747){\color[rgb]{0,0,0}\makebox(0,0)[lt]{\smash{\begin{tabular}[t]{l}$b_2$\end{tabular}}}}%
    \put(0.78244149,0.99607769){\color[rgb]{0,0,0}\makebox(0,0)[lt]{\smash{\begin{tabular}[t]{l}$g$\end{tabular}}}}%
  \end{picture}%
\endgroup%
 
      \centering
         \caption{}
\end{subfigure}
\begin{subfigure}[htb]{0.15\textwidth}
\begingroup%
  \makeatletter%
  \providecommand\color[2][]{%
    \errmessage{(Inkscape) Color is used for the text in Inkscape, but the package 'color.sty' is not loaded}%
    \renewcommand\color[2][]{}%
  }%
  \providecommand\transparent[1]{%
    \errmessage{(Inkscape) Transparency is used (non-zero) for the text in Inkscape, but the package 'transparent.sty' is not loaded}%
    \renewcommand\transparent[1]{}%
  }%
  \providecommand\rotatebox[2]{#2}%
  \newcommand*\fsize{\dimexpr\f@size pt\relax}%
  \newcommand*\lineheight[1]{\fontsize{\fsize}{#1\fsize}\selectfont}%
  \ifx\svgwidth\undefined%
    \setlength{\unitlength}{38.0956794bp}%
    \ifx\svgscale\undefined%
      \relax%
    \else%
      \setlength{\unitlength}{\unitlength * \real{\svgscale}}%
    \fi%
  \else%
    \setlength{\unitlength}{\svgwidth}%
  \fi%
  \global\let\svgwidth\undefined%
  \global\let\svgscale\undefined%
  \makeatother%
  \begin{picture}(1,1.87568563)%
    \lineheight{1}%
    \setlength\tabcolsep{0pt}%
    \put(0.19561503,0.97291263){\color[rgb]{0,0,0}\makebox(0,0)[lt]{\smash{\begin{tabular}[t]{l}$\B$\end{tabular}}}}%
    \put(0,0){\includegraphics[width=\unitlength,page=1]{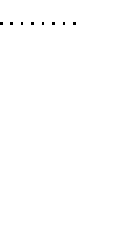}}%
    \put(0.49092127,0.0247629){\color[rgb]{0,0,0}\makebox(0,0)[lt]{\smash{\begin{tabular}[t]{l}$b_1$\end{tabular}}}}%
    \put(0,0){\includegraphics[width=\unitlength,page=2]{morph6.pdf}}%
    \put(0.49443756,1.79591116){\color[rgb]{0,0,0}\makebox(0,0)[lt]{\smash{\begin{tabular}[t]{l}$b_3$\end{tabular}}}}%
    \put(0.68905389,0.67680652){\color[rgb]{0,0,0}\makebox(0,0)[lt]{\smash{\begin{tabular}[t]{l}$g_1$\end{tabular}}}}%
    \put(0.68905389,1.27022453){\color[rgb]{0,0,0}\makebox(0,0)[lt]{\smash{\begin{tabular}[t]{l}$g_2$\end{tabular}}}}%
    \put(0.68905389,0.97351552){\color[rgb]{0,0,0}\makebox(0,0)[lt]{\smash{\begin{tabular}[t]{l}$b_2$\end{tabular}}}}%
  \end{picture}%
\endgroup%
 
      \centering
        \caption{}
\end{subfigure}
\caption{{\scshape (a)}  A  morphism $f\colon a_1\lra a_2$ in  $\A$; {\scshape (b)} a composition $f_2f_1\colon a_1\lra a_3$ of morphisms in $\A$; {\scshape (c)} a composition $f_{n-1}\dots f_1\colon a_1\lra a_n$ in $\A$;  {\scshape (d)} the functor $F$ applied to a morphism $f$ resulting  in the morphism $Ff\colon Fa_1 \lra Fa_2$ in $\B$; {\scshape (e)} a morphism $g\colon b_1\lra b_2$ in $\B$, {\scshape (f)} a composition of morphisms $g_2g_1$ in $\B$.}
\label{fig5_1}
\end{center}
\end{figure}

For a morphism $g\colon Fa\to b$ in $\B$, where $a$ is an object in $\A$, and $b$ an object in $\B$, we use specific diagrammatics of the two lines denoting $F$ and $\ide_a$ merging into a single line denoting $\ide_b$ at a dot. Similar diagrammatics are used for morphisms $f\colon a\to Gb$ in $\A$, see Figure \ref{fig5_2_0}{\scshape a--b}. Similarly, we denote morphism $h\colon Gb\to a$ in $\A$, and $k\colon b\to Fa$ in $\B$, with arrows of opposite orientation, see Figure \ref{fig5_2_0}{\scshape c--d}.

Using that $F$ is left adjoint to $G$, we can display the mutually inverse natural isomorphisms
\begin{gather}\begin{split}
    \Hom_\B(Fa,b)&\stackrel{\sim}{\longleftrightarrow}\Hom_\A(a,Gb),\\
    g\mapsto g^*=Gg(\delta_{2})_a, &\qquad f\mapsto {}^*f=(\mu_2)_bFf,
    \end{split}
\end{gather}
using these diagrammatics, see Figure \ref{fig5_2_1}. Here, $\delta_2,\mu_2$ are the unit and counit of the adjunction $(F,G)$, see Section \ref{subsec_diag_biadj}. We may use similar diagrammatics to express the natural isomorphisms 
\begin{gather}
\Hom_\A(Gb,a)\stackrel{\sim}{\longleftrightarrow}\Hom_\B(b,Fa),
\end{gather}
from $G$ being left adjoint to $F$, using the natural transformations $\delta_1,\mu_1$ from Section \ref{subsec_diag_biadj}.

\begin{figure}[htb]
\begin{center}
 \begin{subfigure}[htb]{0.2\textwidth}
\begingroup%
  \makeatletter%
  \providecommand\color[2][]{%
    \errmessage{(Inkscape) Color is used for the text in Inkscape, but the package 'color.sty' is not loaded}%
    \renewcommand\color[2][]{}%
  }%
  \providecommand\transparent[1]{%
    \errmessage{(Inkscape) Transparency is used (non-zero) for the text in Inkscape, but the package 'transparent.sty' is not loaded}%
    \renewcommand\transparent[1]{}%
  }%
  \providecommand\rotatebox[2]{#2}%
  \newcommand*\fsize{\dimexpr\f@size pt\relax}%
  \newcommand*\lineheight[1]{\fontsize{\fsize}{#1\fsize}\selectfont}%
  \ifx\svgwidth\undefined%
    \setlength{\unitlength}{41.04879464bp}%
    \ifx\svgscale\undefined%
      \relax%
    \else%
      \setlength{\unitlength}{\unitlength * \real{\svgscale}}%
    \fi%
  \else%
    \setlength{\unitlength}{\svgwidth}%
  \fi%
  \global\let\svgwidth\undefined%
  \global\let\svgscale\undefined%
  \makeatother%
  \begin{picture}(1,1.55003149)%
    \lineheight{1}%
    \setlength\tabcolsep{0pt}%
    \put(0.18154283,0.93145794){\color[rgb]{0,0,0}\makebox(0,0)[lt]{\smash{\begin{tabular}[t]{l}$\B$\end{tabular}}}}%
    \put(0,0){\includegraphics[width=\unitlength,page=1]{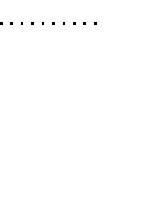}}%
    \put(0.63831277,0.01432175){\color[rgb]{0,0,0}\makebox(0,0)[lt]{\smash{\begin{tabular}[t]{l}$a$\end{tabular}}}}%
    \put(0,0){\includegraphics[width=\unitlength,page=2]{morph7.pdf}}%
    \put(0.64157609,1.47599612){\color[rgb]{0,0,0}\makebox(0,0)[lt]{\smash{\begin{tabular}[t]{l}$b$\end{tabular}}}}%
    \put(0.82219142,0.80542551){\color[rgb]{0,0,0}\makebox(0,0)[lt]{\smash{\begin{tabular}[t]{l}$g$\end{tabular}}}}%
    \put(0,0){\includegraphics[width=\unitlength,page=3]{morph7.pdf}}%
    \put(0.09018673,0.01432175){\color[rgb]{0,0,0}\makebox(0,0)[lt]{\smash{\begin{tabular}[t]{l}$F$\end{tabular}}}}%
    \put(0.36425133,0.3826807){\color[rgb]{0,0,0}\makebox(0,0)[lt]{\smash{\begin{tabular}[t]{l}$\A$\end{tabular}}}}%
  \end{picture}%
\endgroup%
 
      \centering
         \caption{}
\end{subfigure}
\begin{subfigure}[htb]{0.2\textwidth}
\begingroup%
  \makeatletter%
  \providecommand\color[2][]{%
    \errmessage{(Inkscape) Color is used for the text in Inkscape, but the package 'color.sty' is not loaded}%
    \renewcommand\color[2][]{}%
  }%
  \providecommand\transparent[1]{%
    \errmessage{(Inkscape) Transparency is used (non-zero) for the text in Inkscape, but the package 'transparent.sty' is not loaded}%
    \renewcommand\transparent[1]{}%
  }%
  \providecommand\rotatebox[2]{#2}%
  \newcommand*\fsize{\dimexpr\f@size pt\relax}%
  \newcommand*\lineheight[1]{\fontsize{\fsize}{#1\fsize}\selectfont}%
  \ifx\svgwidth\undefined%
    \setlength{\unitlength}{39.91598229bp}%
    \ifx\svgscale\undefined%
      \relax%
    \else%
      \setlength{\unitlength}{\unitlength * \real{\svgscale}}%
    \fi%
  \else%
    \setlength{\unitlength}{\svgwidth}%
  \fi%
  \global\let\svgwidth\undefined%
  \global\let\svgscale\undefined%
  \makeatother%
  \begin{picture}(1,1.59402126)%
    \lineheight{1}%
    \setlength\tabcolsep{0pt}%
    \put(0.186695,0.58210837){\color[rgb]{0,0,0}\makebox(0,0)[lt]{\smash{\begin{tabular}[t]{l}$\A$\end{tabular}}}}%
    \put(0,0){\includegraphics[width=\unitlength,page=1]{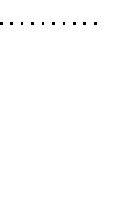}}%
    \put(0.65642804,0.0147282){\color[rgb]{0,0,0}\makebox(0,0)[lt]{\smash{\begin{tabular}[t]{l}$a$\end{tabular}}}}%
    \put(0,0){\includegraphics[width=\unitlength,page=2]{morph8.pdf}}%
    \put(0.65978397,1.51788477){\color[rgb]{0,0,0}\makebox(0,0)[lt]{\smash{\begin{tabular}[t]{l}$b$\end{tabular}}}}%
    \put(0.84552514,0.82828342){\color[rgb]{0,0,0}\makebox(0,0)[lt]{\smash{\begin{tabular}[t]{l}$f$\end{tabular}}}}%
    \put(0,0){\includegraphics[width=\unitlength,page=3]{morph8.pdf}}%
    \put(0.09274622,1.51786527){\color[rgb]{0,0,0}\makebox(0,0)[lt]{\smash{\begin{tabular}[t]{l}$G$\end{tabular}}}}%
    \put(0.37458876,1.14175484){\color[rgb]{0,0,0}\makebox(0,0)[lt]{\smash{\begin{tabular}[t]{l}$\B$\end{tabular}}}}%
  \end{picture}%
\endgroup%
 
      \centering
         \caption{}
\end{subfigure}
 \begin{subfigure}[htb]{0.2\textwidth}
\begingroup%
  \makeatletter%
  \providecommand\color[2][]{%
    \errmessage{(Inkscape) Color is used for the text in Inkscape, but the package 'color.sty' is not loaded}%
    \renewcommand\color[2][]{}%
  }%
  \providecommand\transparent[1]{%
    \errmessage{(Inkscape) Transparency is used (non-zero) for the text in Inkscape, but the package 'transparent.sty' is not loaded}%
    \renewcommand\transparent[1]{}%
  }%
  \providecommand\rotatebox[2]{#2}%
  \newcommand*\fsize{\dimexpr\f@size pt\relax}%
  \newcommand*\lineheight[1]{\fontsize{\fsize}{#1\fsize}\selectfont}%
  \ifx\svgwidth\undefined%
    \setlength{\unitlength}{41.04879464bp}%
    \ifx\svgscale\undefined%
      \relax%
    \else%
      \setlength{\unitlength}{\unitlength * \real{\svgscale}}%
    \fi%
  \else%
    \setlength{\unitlength}{\svgwidth}%
  \fi%
  \global\let\svgwidth\undefined%
  \global\let\svgscale\undefined%
  \makeatother%
  \begin{picture}(1,1.55003149)%
    \lineheight{1}%
    \setlength\tabcolsep{0pt}%
    \put(0.18154283,0.93145794){\color[rgb]{0,0,0}\makebox(0,0)[lt]{\smash{\begin{tabular}[t]{l}$\A$\end{tabular}}}}%
    \put(0,0){\includegraphics[width=\unitlength,page=1]{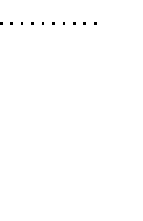}}%
    \put(0.63831277,0.01432175){\color[rgb]{0,0,0}\makebox(0,0)[lt]{\smash{\begin{tabular}[t]{l}$b$\end{tabular}}}}%
    \put(0,0){\includegraphics[width=\unitlength,page=2]{morph7a.pdf}}%
    \put(0.64157609,1.47599612){\color[rgb]{0,0,0}\makebox(0,0)[lt]{\smash{\begin{tabular}[t]{l}$a$\end{tabular}}}}%
    \put(0.82219142,0.80542551){\color[rgb]{0,0,0}\makebox(0,0)[lt]{\smash{\begin{tabular}[t]{l}$h$\end{tabular}}}}%
    \put(0,0){\includegraphics[width=\unitlength,page=3]{morph7a.pdf}}%
    \put(0.09018673,0.01432175){\color[rgb]{0,0,0}\makebox(0,0)[lt]{\smash{\begin{tabular}[t]{l}$G$\end{tabular}}}}%
    \put(0.36425133,0.3826807){\color[rgb]{0,0,0}\makebox(0,0)[lt]{\smash{\begin{tabular}[t]{l}$\B$\end{tabular}}}}%
  \end{picture}%
\endgroup%
 
      \centering
         \caption{}
\end{subfigure}
\begin{subfigure}[htb]{0.2\textwidth}
\begingroup%
  \makeatletter%
  \providecommand\color[2][]{%
    \errmessage{(Inkscape) Color is used for the text in Inkscape, but the package 'color.sty' is not loaded}%
    \renewcommand\color[2][]{}%
  }%
  \providecommand\transparent[1]{%
    \errmessage{(Inkscape) Transparency is used (non-zero) for the text in Inkscape, but the package 'transparent.sty' is not loaded}%
    \renewcommand\transparent[1]{}%
  }%
  \providecommand\rotatebox[2]{#2}%
  \newcommand*\fsize{\dimexpr\f@size pt\relax}%
  \newcommand*\lineheight[1]{\fontsize{\fsize}{#1\fsize}\selectfont}%
  \ifx\svgwidth\undefined%
    \setlength{\unitlength}{40.83004465bp}%
    \ifx\svgscale\undefined%
      \relax%
    \else%
      \setlength{\unitlength}{\unitlength * \real{\svgscale}}%
    \fi%
  \else%
    \setlength{\unitlength}{\svgwidth}%
  \fi%
  \global\let\svgwidth\undefined%
  \global\let\svgscale\undefined%
  \makeatother%
  \begin{picture}(1,1.5583359)%
    \lineheight{1}%
    \setlength\tabcolsep{0pt}%
    \put(0.18251545,0.56907671){\color[rgb]{0,0,0}\makebox(0,0)[lt]{\smash{\begin{tabular}[t]{l}$\B$\end{tabular}}}}%
    \put(0,0){\includegraphics[width=\unitlength,page=1]{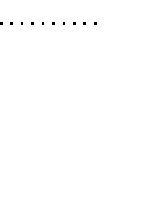}}%
    \put(0.64173258,0.01439848){\color[rgb]{0,0,0}\makebox(0,0)[lt]{\smash{\begin{tabular}[t]{l}$b$\end{tabular}}}}%
    \put(0,0){\includegraphics[width=\unitlength,page=2]{morph8a.pdf}}%
    \put(0.64501339,1.48390388){\color[rgb]{0,0,0}\makebox(0,0)[lt]{\smash{\begin{tabular}[t]{l}$a$\end{tabular}}}}%
    \put(0.82659637,0.80974064){\color[rgb]{0,0,0}\makebox(0,0)[lt]{\smash{\begin{tabular}[t]{l}$k$\end{tabular}}}}%
    \put(0,0){\includegraphics[width=\unitlength,page=3]{morph8a.pdf}}%
    \put(0.09066991,1.48388481){\color[rgb]{0,0,0}\makebox(0,0)[lt]{\smash{\begin{tabular}[t]{l}$F$\end{tabular}}}}%
    \put(0.36620284,1.11619437){\color[rgb]{0,0,0}\makebox(0,0)[lt]{\smash{\begin{tabular}[t]{l}$\A$\end{tabular}}}}%
  \end{picture}%
\endgroup%
 
      \centering
         \caption{}
\end{subfigure}
\caption{{\scshape (a)} A  morphism $g\colon Fa\lra b$ in $\B$, for objects $a\in\A,b\in \B$ and the functor $F$; {\scshape (b)} a morphism $f\colon a\lra Gb$ in $\A$; {\scshape (c)} a morphism $h \colon Gb\to a$ in $\B$; {\scshape (d)} a morphism $k\colon b\to Fa$ in $\A$. }
\label{fig5_2_0}
\end{center}
\end{figure}

\begin{figure}[htb]
    \begin{center}
        \begin{subfigure}[htb]{0.4\textwidth}
     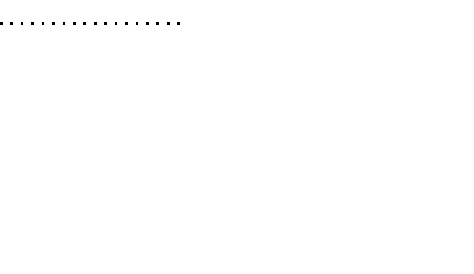 
      \centering
         \caption{}
\end{subfigure}
\begin{subfigure}[htb]{0.4\textwidth}
     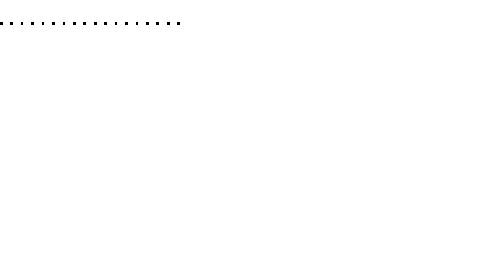 
      \centering
         \caption{}
\end{subfigure}
        \end{center}
    \caption{{\scshape (a)} the dual $g^*$ of a morphism $g$, given by composing with the  natural transformation $\delta_2\colon \Ide_{\A}\lra GF$, see Figure~\ref{fig1_2} and Section~\ref{subsec_diag_biadj}; {\scshape (b)} the inverse operation $f\mapsto {^*}f$. }
    \label{fig5_2_1}
\end{figure}

A natural transformation $\alpha\colon F_1\Longrightarrow F_2$ between functors $F_1,F_2\colon \A\lra \B$ is natural with respect to any morphisms $f\colon a_1\to a_2$ in $\A$, via the commutative diagram
\begin{align}
    \xymatrix{F_1(a_1)\ar[rr]^{F_1(f)}\ar[d]_{\alpha_{a_1}}&& F_1(a_2)\ar[d]_{\alpha_{a_2}}\\
    F_2(a_1)\ar[rr]^{F_2(f)}&& F_2(a_2).
    }
\end{align}
Figure~\ref{fig5_3} expresses this property diagrammatically, as an isotopy condition between dots.

\begin{figure}[htb]
\begin{center}
     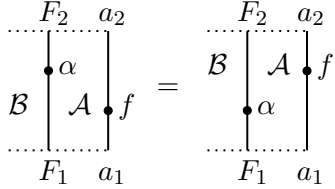 
\caption{The  property of a natural transformation $\alpha\colon F_1\Longrightarrow F_2$ is the isotopy (commutativity) condition that the dots denoting $\alpha$  and $f$ on parallel  vertical lines can slide past each other.}
\label{fig5_3}
\end{center}
\end{figure}

\ssubsection{Transfer maps} 
The \emph{transfer} (or \emph{trace}) map~\cite{G,MMa1,MMa2} is the map 
\begin{equation}\label{eq_transfer}
\begin{split}
    \Tr_F\colon  \Hom_{\B}(Fa_1,Fa_2) &\lra \Hom_{\A}(a_1,a_2),\\
    \big(f\colon Fa_1\to F a_2\big)\,&\longmapsto\, \big( (\mu_1)_{a_2}G(f)(\delta_2)_{a_1}\colon a_1\to a_2\big),
    \end{split}
\end{equation}
built using the natural transformations $\delta_2$, $\mu_1$ from the biadjointness data, cf. Section \ref{subsec_diag_biadj}. Interchanging the roles of $F$ and $G$, we may similarly define 
\begin{equation}\label{eq_transfer2}
\begin{split}
    \Tr_G\colon  \Hom_{\A}(Gb_1,Gb_2) \lra \Hom_{\B}(b_1,b_2),
    \end{split}
\end{equation}
using $\delta_1,\mu_2$ instead. We will focus on the case of $\Tr_F$ here.
Transfer maps appeared already in \cite[Equations (57)--(58)]{M}. There is a characterization of Frobenius pairs of functors in terms of transfer maps \cite[Proposition~45]{CMZ}.

\vspace{0.1in}

A morphism $f$ in $\Hom_{\B}(Fa_1,Fa_2)$ can be denoted by the diagram in Figure~{\scshape \ref{fig:morph11}} as a vertex on the boundary line  bounding intervals for objects $a_1,a_2$ along the boundary and  with the lines, from the functor $F$ on the bottom and top edge of the region, going in and  out of the vertex. 

\begin{figure}[htb]
\begin{center}
  \begin{subfigure}[htb]{0.40\textwidth}
     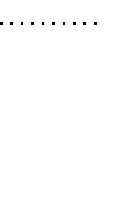 
      \centering
         \caption{The diagrammatic representation of $f$}
        \label{fig:morph11}
\end{subfigure}
\begin{subfigure}[htb]{0.55\textwidth}
     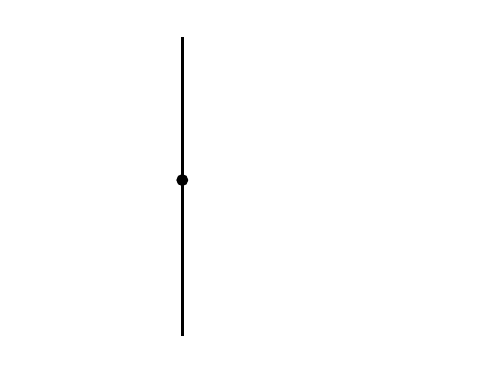 
      \centering
         \caption{The transfer $\Tr(f)$ of $f$}
        \label{fig:transfer}
\end{subfigure}
\caption{Transfer map diagrammatics for a morphism $f\colon Fa_1\lra Fa_2$.}
\label{fig5_4}
\end{center}
\end{figure}

In this diagrammatic language, the transfer map $\Tr_F$ is given by closing up the two ends of $F$ into a loop, see Figure~{\scshape \ref{fig:transfer}}. The result of evaluating the transfer map is a \emph{balloon} attached to a point on the boundary line, see Figure~{\scshape \ref{fig:transfer}} on the right.  

\vspace{0.1in}

The notion of the transfer map can be generalized by placing an element $z\in Z(\B)$ of the center of the category $\B$ inside the  region enveloped by the line for $F$, see Figure~\ref{fig5_5}, giving us a map  
\begin{equation}
    \Tr_F\colon  Z(\B) \times\Hom_{\B}(Fa_1,Fa_2) \lra \Hom_{\A}(a_1,a_2),
\end{equation}
which, alternatively, may be  hidden inside the original map, via the  action
\begin{equation}\label{eq_ZBaction}
     Z(\B) \times\Hom_{\B}(Fa_1,Fa_2) \lra \Hom_{\B}(Fa_1,Fa_2),
\end{equation}
of $Z(\B)$ on morphisms  $\Hom_{\B}(Fa_1,Fa_2)$ given by placing an  element of $Z(\B)$ in  the region  labelled $\B$ in Figure~{\scshape\ref{fig:morph11}}. 

\begin{figure}[htb]
\begin{center}
\begingroup%
  \makeatletter%
  \providecommand\color[2][]{%
    \errmessage{(Inkscape) Color is used for the text in Inkscape, but the package 'color.sty' is not loaded}%
    \renewcommand\color[2][]{}%
  }%
  \providecommand\transparent[1]{%
    \errmessage{(Inkscape) Transparency is used (non-zero) for the text in Inkscape, but the package 'transparent.sty' is not loaded}%
    \renewcommand\transparent[1]{}%
  }%
  \providecommand\rotatebox[2]{#2}%
  \newcommand*\fsize{\dimexpr\f@size pt\relax}%
  \newcommand*\lineheight[1]{\fontsize{\fsize}{#1\fsize}\selectfont}%
  \ifx\svgwidth\undefined%
    \setlength{\unitlength}{51.16474285bp}%
    \ifx\svgscale\undefined%
      \relax%
    \else%
      \setlength{\unitlength}{\unitlength * \real{\svgscale}}%
    \fi%
  \else%
    \setlength{\unitlength}{\svgwidth}%
  \fi%
  \global\let\svgwidth\undefined%
  \global\let\svgscale\undefined%
  \makeatother%
  \begin{picture}(1,2.13002237)%
    \lineheight{1}%
    \setlength\tabcolsep{0pt}%
    \put(0.73196362,0.01843768){\color[rgb]{0,0,0}\makebox(0,0)[lt]{\smash{\begin{tabular}[t]{l}$a_1$\end{tabular}}}}%
    \put(0,0){\includegraphics[width=\unitlength,page=1]{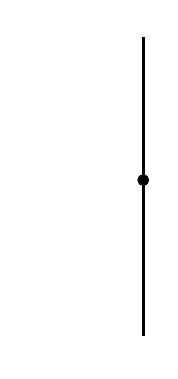}}%
    \put(0.73458174,2.07062478){\color[rgb]{0,0,0}\makebox(0,0)[lt]{\smash{\begin{tabular}[t]{l}$a_2$\end{tabular}}}}%
    \put(0.87948702,1.09287977){\color[rgb]{0,0,0}\makebox(0,0)[lt]{\smash{\begin{tabular}[t]{l}$f$\end{tabular}}}}%
    \put(0,0){\includegraphics[width=\unitlength,page=2]{trace2.pdf}}%
    \put(0.29050848,0.6045166){\color[rgb]{0,0,0}\makebox(0,0)[lt]{\smash{\begin{tabular}[t]{l}$\A$\end{tabular}}}}%
    \put(0,0){\includegraphics[width=\unitlength,page=3]{trace2.pdf}}%
    \put(0.30782339,1.07324769){\color[rgb]{0,0,0}\makebox(0,0)[lt]{\smash{\begin{tabular}[t]{l}$z$\end{tabular}}}}%
    \put(0,0){\includegraphics[width=\unitlength,page=4]{trace2.pdf}}%
  \end{picture}%
\endgroup%
 
\caption{Wrapping $F$ around a central element $z$ of $Z(\B)$, the element $z$ is placed in the region symbolized by a square.}
\label{fig5_5}
\end{center}
\end{figure}

For morphisms $g_1\colon a_1\lra a_2$, $g_2\colon  a_3\lra a_4$  in $\A$ and  $f\colon  Fa_2\lra Fa_3$ in $\B$ 
the relations  
\begin{equation}\label{eq_rel_MMa}
    \Tr_F(f \circ F(g_1)) = \Tr_F(f) \circ g_1, 
    \qquad \Tr_F( F(g_2)\circ f) = g_2\circ \Tr_F(f),
\end{equation}
hold, see~\cite[Proposition 1.8a]{MMa1}. Diagrammatically, they say that the isotopies that change  relative height of the cap and cup points of the balloon, relative to the boundary points representing $g_1,g_2$, do not change the morphism, see Figure~\ref{fig5_6}. 

\begin{figure}[htb]
\begin{center}
 \begin{subfigure}[htb]{0.45\textwidth}
     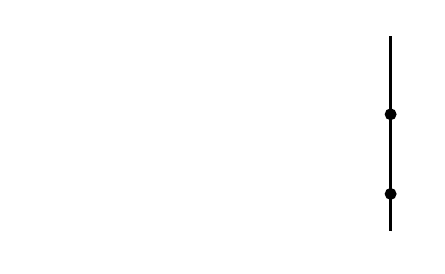 
      \centering
\end{subfigure}
 \begin{subfigure}[htb]{0.45\textwidth}
     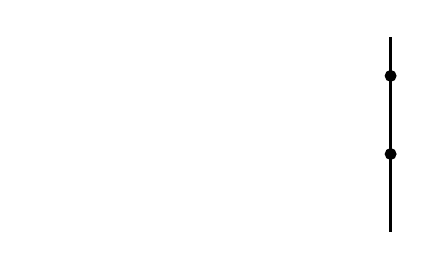 
      \centering
\end{subfigure}
\caption{Equations (\ref{eq_rel_MMa}) are isotopies that change relative heights.}
\label{fig5_6}
\end{center}
\end{figure}

\ssubsection{The trace morphisms}
An endomorphism $\alpha\colon F\longrightarrow F$ can be denoted by a dot on a line labelled by $F$, see Figure~{\scshape\ref{fig:alpha}}. Closing the line into a circle, using biadjointness, see Figure~{\scshape\ref{fig:traceA}}, is the diagrammatic description of the \emph{trace map} of \cite{B}:  
\begin{equation}\label{eq_trace_map}
   \tr_{\A} \ \colon \  \End(F) \lra \End(\Ide_{\A}) \cong Z(\A),
\end{equation}
The trace map can be written as the following composition
\begin{equation*}
   \tr_{\A} \ \colon \  \Ide_{\A} \stackrel{\delta_2}{\lra} GF \xrightarrow{\,(\ide_G) \alpha\,} GF \stackrel{\mu_1}{\lra} \Ide_{\A}. 
\end{equation*}
\vspace{0.1in}
The other trace, 
\begin{equation*}
   \tr_{\B} \ \colon \  \Ide_{\B} \stackrel{\delta_1}{\lra} FG \xrightarrow{\,\alpha(\ide_G)\,} FG \stackrel{\mu_2}{\lra} \Ide_{\B}, 
\end{equation*}
is given by closing the interval with $\alpha$ on the other side, into a clockwise circle in the plane, with the category $\B$ on the outside of the diagram, see Figure~{\scshape\ref{fig:traceB}}. 

\begin{figure}[htb]
\begin{center}
  \begin{subfigure}[htb]{0.28\textwidth}
\begingroup%
  \makeatletter%
  \providecommand\color[2][]{%
    \errmessage{(Inkscape) Color is used for the text in Inkscape, but the package 'color.sty' is not loaded}%
    \renewcommand\color[2][]{}%
  }%
  \providecommand\transparent[1]{%
    \errmessage{(Inkscape) Transparency is used (non-zero) for the text in Inkscape, but the package 'transparent.sty' is not loaded}%
    \renewcommand\transparent[1]{}%
  }%
  \providecommand\rotatebox[2]{#2}%
  \newcommand*\fsize{\dimexpr\f@size pt\relax}%
  \newcommand*\lineheight[1]{\fontsize{\fsize}{#1\fsize}\selectfont}%
  \ifx\svgwidth\undefined%
    \setlength{\unitlength}{45.00000283bp}%
    \ifx\svgscale\undefined%
      \relax%
    \else%
      \setlength{\unitlength}{\unitlength * \real{\svgscale}}%
    \fi%
  \else%
    \setlength{\unitlength}{\svgwidth}%
  \fi%
  \global\let\svgwidth\undefined%
  \global\let\svgscale\undefined%
  \makeatother%
  \begin{picture}(1,1.55455654)%
    \lineheight{1}%
    \setlength\tabcolsep{0pt}%
    \put(0.66262729,0.51445976){\color[rgb]{0,0,0}\makebox(0,0)[lt]{\smash{\begin{tabular}[t]{l}$\A$\end{tabular}}}}%
    \put(0.41573097,1.3519524){\color[rgb]{0,0,0}\makebox(0,0)[lt]{\smash{\begin{tabular}[t]{l}$F$\end{tabular}}}}%
    \put(0.16560267,0.51930028){\color[rgb]{0,0,0}\makebox(0,0)[lt]{\smash{\begin{tabular}[t]{l}$\B$\end{tabular}}}}%
    \put(0,0){\includegraphics[width=\unitlength,page=1]{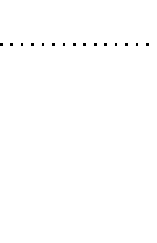}}%
    \put(0.41560265,0.01959635){\color[rgb]{0,0,0}\makebox(0,0)[lt]{\smash{\begin{tabular}[t]{l}$F$\end{tabular}}}}%
    \put(0,0){\includegraphics[width=\unitlength,page=2]{alpha.pdf}}%
    \put(0.5833335,0.82754606){\color[rgb]{0,0,0}\makebox(0,0)[lt]{\smash{\begin{tabular}[t]{l}$\alpha$\end{tabular}}}}%
    \put(0,0){\includegraphics[width=\unitlength,page=3]{alpha.pdf}}%
  \end{picture}%
\endgroup%
 
      \centering
        \caption{An endomorphism $\alpha$ of $F$}
        \label{fig:alpha}
\end{subfigure}
  \begin{subfigure}[htb]{0.33\textwidth}
\begingroup%
  \makeatletter%
  \providecommand\color[2][]{%
    \errmessage{(Inkscape) Color is used for the text in Inkscape, but the package 'color.sty' is not loaded}%
    \renewcommand\color[2][]{}%
  }%
  \providecommand\transparent[1]{%
    \errmessage{(Inkscape) Transparency is used (non-zero) for the text in Inkscape, but the package 'transparent.sty' is not loaded}%
    \renewcommand\transparent[1]{}%
  }%
  \providecommand\rotatebox[2]{#2}%
  \newcommand*\fsize{\dimexpr\f@size pt\relax}%
  \newcommand*\lineheight[1]{\fontsize{\fsize}{#1\fsize}\selectfont}%
  \ifx\svgwidth\undefined%
    \setlength{\unitlength}{62.23927594bp}%
    \ifx\svgscale\undefined%
      \relax%
    \else%
      \setlength{\unitlength}{\unitlength * \real{\svgscale}}%
    \fi%
  \else%
    \setlength{\unitlength}{\svgwidth}%
  \fi%
  \global\let\svgwidth\undefined%
  \global\let\svgscale\undefined%
  \makeatother%
  \begin{picture}(1,0.73508537)%
    \lineheight{1}%
    \setlength\tabcolsep{0pt}%
    \put(0,0){\includegraphics[width=\unitlength,page=1]{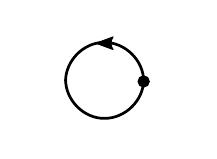}}%
    \put(0.72356557,0.3303065){\color[rgb]{0,0,0}\makebox(0,0)[lt]{\smash{\begin{tabular}[t]{l}$\alpha$\end{tabular}}}}%
    \put(0,0){\includegraphics[width=\unitlength,page=2]{traceB.pdf}}%
    \put(0.41883988,0.30686441){\color[rgb]{0,0,0}\makebox(0,0)[lt]{\smash{\begin{tabular}[t]{l}$\B$\end{tabular}}}}%
    \put(0.09003553,0.30778383){\color[rgb]{0,0,0}\makebox(0,0)[lt]{\smash{\begin{tabular}[t]{l}$\A$\end{tabular}}}}%
  \end{picture}%
\endgroup%
 
      \centering
         \caption{$\tr_{\A}(\alpha)$}
        \label{fig:traceA}
\end{subfigure}
  \begin{subfigure}[htb]{0.33\textwidth}
\begingroup%
  \makeatletter%
  \providecommand\color[2][]{%
    \errmessage{(Inkscape) Color is used for the text in Inkscape, but the package 'color.sty' is not loaded}%
    \renewcommand\color[2][]{}%
  }%
  \providecommand\transparent[1]{%
    \errmessage{(Inkscape) Transparency is used (non-zero) for the text in Inkscape, but the package 'transparent.sty' is not loaded}%
    \renewcommand\transparent[1]{}%
  }%
  \providecommand\rotatebox[2]{#2}%
  \newcommand*\fsize{\dimexpr\f@size pt\relax}%
  \newcommand*\lineheight[1]{\fontsize{\fsize}{#1\fsize}\selectfont}%
  \ifx\svgwidth\undefined%
    \setlength{\unitlength}{59.99999811bp}%
    \ifx\svgscale\undefined%
      \relax%
    \else%
      \setlength{\unitlength}{\unitlength * \real{\svgscale}}%
    \fi%
  \else%
    \setlength{\unitlength}{\svgwidth}%
  \fi%
  \global\let\svgwidth\undefined%
  \global\let\svgscale\undefined%
  \makeatother%
  \begin{picture}(1,0.76251971)%
    \lineheight{1}%
    \setlength\tabcolsep{0pt}%
    \put(0,0){\includegraphics[width=\unitlength,page=1]{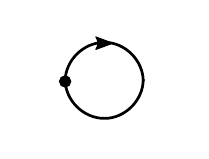}}%
    \put(0.11887267,0.34263397){\color[rgb]{0,0,0}\makebox(0,0)[lt]{\smash{\begin{tabular}[t]{l}$\alpha$\end{tabular}}}}%
    \put(0,0){\includegraphics[width=\unitlength,page=2]{traceA.pdf}}%
    \put(0.4393815,0.31831699){\color[rgb]{0,0,0}\makebox(0,0)[lt]{\smash{\begin{tabular}[t]{l}$\A$\end{tabular}}}}%
    \put(0.74696986,0.31927072){\color[rgb]{0,0,0}\makebox(0,0)[lt]{\smash{\begin{tabular}[t]{l}$\B$\end{tabular}}}}%
  \end{picture}%
\endgroup%
 
      \centering
        \caption{$\tr_{\B}(\alpha)$}
        \label{fig:traceB}
\end{subfigure}
\caption{Diagrammatics of traces of endomorphisms of $F$. Closing a dotted interval into a dotted circle gives the trace morphism $\tr_{\A}$ in {\scshape (b)}, while the other closure is the trace $\tr_{\B}$ in {\scshape (c)}.}
\label{fig5_10}
\end{center}
\end{figure}

Wrapping a central element $z\in Z(\B)$ by  an $F$-labelled counterclockwise oriented circle gives a map $Z(\B)\lra Z(\A)$, see \eqref{fig5_11}. More generally, we obtain the composite morphism 
\begin{align}
   Z(\B)\times \End(F) \xrightarrow{\eqref{eq_ZBaction}} \End(F) \xrightarrow{\tr_\A}Z(\A),
\end{align}
using the action from \eqref{eq_ZBaction}. It wraps a circle with dots labelled by endomorphisms of $F$ around an element $z$ of $Z(\B)$.

\begin{figure}[htb]
\begin{center}
\begin{subfigure}[htb]{0.45\textwidth}
     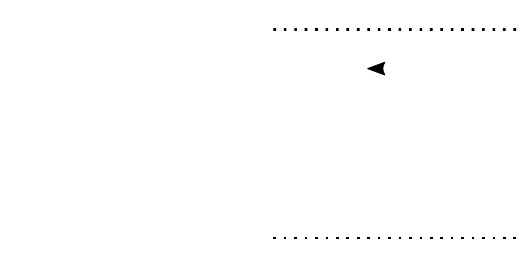 
      \centering
        \caption{}
\end{subfigure}
\begin{subfigure}[htb]{0.45\textwidth}
     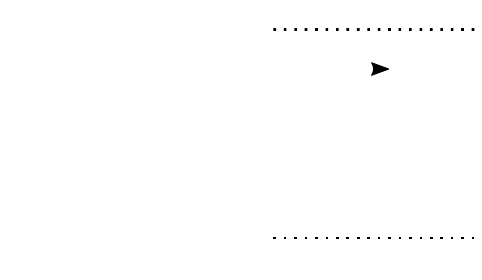 
      \centering
        \caption{}
\end{subfigure}
\caption{In {\scshape (a)}, an element  $z\in Z(\B)$ gives the endomorphism $z\,\ide_F$ of $F$ and the trace $\tr_{\A}(z\,\ide_F)$ is a counterclockwise $F$-bubble wrapped around $z$. In {\scshape (b)}, an  element  $z'\in Z(\A)$ gives the  endomorphism $\ide_F z'$ of $F$ and its other trace, the $\B$-trace, is a clockwise $F$-bubble wrapped around $z'$.}
\label{fig5_11}
\end{center}
\end{figure}

Interpreting the general transfer map  in (\ref{eq_transfer}) diagrammatically requires introducing a peculiar 4-valent vertex on the vertical boundary of  the strip to denote a morphism $Fa_1\lra  Fa_2$, see Figure~\ref{fig5_4} earlier. The input natural transformation $\alpha\colon F\Longrightarrow F$ in the trace morphism (\ref{eq_trace_map}) induces morphisms $\alpha_a\colon Fa\lra Fa$ for all objects $a$ in $\A$. For this morphism $\alpha_a$, the 4-valent vertex can be reduced to placing a vertical line   with  a dot labelled $\alpha$ in  parallel  with  the  boundary labelled by the identity map of $a$, see Figure~{\scshape\ref{fig:morph12}}.

\begin{figure}[htb]
\begin{center}
\begin{subfigure}[htb]{0.45\textwidth}
\begingroup%
  \makeatletter%
  \providecommand\color[2][]{%
    \errmessage{(Inkscape) Color is used for the text in Inkscape, but the package 'color.sty' is not loaded}%
    \renewcommand\color[2][]{}%
  }%
  \providecommand\transparent[1]{%
    \errmessage{(Inkscape) Transparency is used (non-zero) for the text in Inkscape, but the package 'transparent.sty' is not loaded}%
    \renewcommand\transparent[1]{}%
  }%
  \providecommand\rotatebox[2]{#2}%
  \newcommand*\fsize{\dimexpr\f@size pt\relax}%
  \newcommand*\lineheight[1]{\fontsize{\fsize}{#1\fsize}\selectfont}%
  \ifx\svgwidth\undefined%
    \setlength{\unitlength}{48.49529935bp}%
    \ifx\svgscale\undefined%
      \relax%
    \else%
      \setlength{\unitlength}{\unitlength * \real{\svgscale}}%
    \fi%
  \else%
    \setlength{\unitlength}{\svgwidth}%
  \fi%
  \global\let\svgwidth\undefined%
  \global\let\svgscale\undefined%
  \makeatother%
  \begin{picture}(1,1.35217963)%
    \lineheight{1}%
    \setlength\tabcolsep{0pt}%
    \put(0.84850396,0.01818395){\color[rgb]{0,0,0}\makebox(0,0)[lt]{\smash{\begin{tabular}[t]{l}$a$\end{tabular}}}}%
    \put(0,0){\includegraphics[width=\unitlength,page=1]{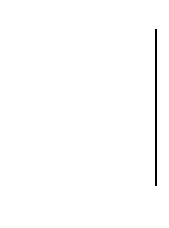}}%
    \put(0.85126619,1.25541667){\color[rgb]{0,0,0}\makebox(0,0)[lt]{\smash{\begin{tabular}[t]{l}$a$\end{tabular}}}}%
    \put(0.58393692,0.63203218){\color[rgb]{0,0,0}\makebox(0,0)[lt]{\smash{\begin{tabular}[t]{l}$\A$\end{tabular}}}}%
    \put(0.09180446,0.35815007){\color[rgb]{0,0,0}\makebox(0,0)[lt]{\smash{\begin{tabular}[t]{l}$\B$\end{tabular}}}}%
    \put(0,0){\includegraphics[width=\unitlength,page=2]{morph12.pdf}}%
    \put(0.38564788,0.01818395){\color[rgb]{0,0,0}\makebox(0,0)[lt]{\smash{\begin{tabular}[t]{l}$F$\end{tabular}}}}%
    \put(0,0){\includegraphics[width=\unitlength,page=3]{morph12.pdf}}%
    \put(0.1391901,0.66130275){\color[rgb]{0,0,0}\makebox(0,0)[lt]{\smash{\begin{tabular}[t]{l}$\alpha$\end{tabular}}}}%
    \put(0,0){\includegraphics[width=\unitlength,page=4]{morph12.pdf}}%
    \put(0.38288609,1.2581789){\color[rgb]{0,0,0}\makebox(0,0)[lt]{\smash{\begin{tabular}[t]{l}$F$\end{tabular}}}}%
  \end{picture}%
\endgroup%
 
      \centering
        \caption{}
        \label{fig:morph12}
\end{subfigure}
\begin{subfigure}[htb]{0.45\textwidth}
\begingroup%
  \makeatletter%
  \providecommand\color[2][]{%
    \errmessage{(Inkscape) Color is used for the text in Inkscape, but the package 'color.sty' is not loaded}%
    \renewcommand\color[2][]{}%
  }%
  \providecommand\transparent[1]{%
    \errmessage{(Inkscape) Transparency is used (non-zero) for the text in Inkscape, but the package 'transparent.sty' is not loaded}%
    \renewcommand\transparent[1]{}%
  }%
  \providecommand\rotatebox[2]{#2}%
  \newcommand*\fsize{\dimexpr\f@size pt\relax}%
  \newcommand*\lineheight[1]{\fontsize{\fsize}{#1\fsize}\selectfont}%
  \ifx\svgwidth\undefined%
    \setlength{\unitlength}{63.54879037bp}%
    \ifx\svgscale\undefined%
      \relax%
    \else%
      \setlength{\unitlength}{\unitlength * \real{\svgscale}}%
    \fi%
  \else%
    \setlength{\unitlength}{\svgwidth}%
  \fi%
  \global\let\svgwidth\undefined%
  \global\let\svgscale\undefined%
  \makeatother%
  \begin{picture}(1,1.26087142)%
    \lineheight{1}%
    \setlength\tabcolsep{0pt}%
    \put(0,0){\includegraphics[width=\unitlength,page=1]{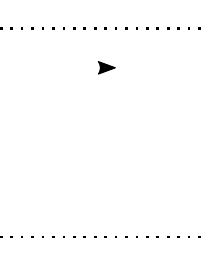}}%
    \put(0.06457629,0.26937923){\color[rgb]{0,0,0}\makebox(0,0)[lt]{\smash{\begin{tabular}[t]{l}$\B$\end{tabular}}}}%
    \put(0.41433167,0.62695472){\color[rgb]{0,0,0}\makebox(0,0)[lt]{\smash{\begin{tabular}[t]{l}$z$\end{tabular}}}}%
    \put(0,0){\includegraphics[width=\unitlength,page=2]{morph13.pdf}}%
    \put(0.88514608,0.00925101){\color[rgb]{0,0,0}\makebox(0,0)[lt]{\smash{\begin{tabular}[t]{l}$b$\end{tabular}}}}%
    \put(0.88514608,1.21304891){\color[rgb]{0,0,0}\makebox(0,0)[lt]{\smash{\begin{tabular}[t]{l}$b$\end{tabular}}}}%
    \put(0,0){\includegraphics[width=\unitlength,page=3]{morph13.pdf}}%
  \end{picture}%
\endgroup%
 
      \centering
        \caption{}
        \label{fig:morph13}
\end{subfigure}
\caption{In {\scshape (a)}, the four-valent vertex  is simplified to parallel lines when the morphism $Fa_1\lra Fa_2$ is $\alpha_a\colon Fa\lra Fa$. In {\scshape (b)}, diagrammatics for the morphism $\tr_{\B}(z) b\colon  b\lra b$.}
\label{fig5_12}
\end{center}
\end{figure}

The trace map bubbles in the presence of boundary describe suitable morphisms in categories $\A$ and $\B$. For instance, the diagram in Figure~{\scshape\ref{fig:morph13}} is the  endomorphism $\tr_{\B}(z) b\colon b \lra b$ for an object $b$ of $\B$ and $z\in  Z(\A)$.


\end{document}